\documentclass[12pts]{amsart}

\usepackage{amscd,latexsym,amsthm,amsfonts,amssymb,amsmath,amsxtra}
\usepackage[mathscr]{eucal}
\newcommand{\BA}{{\mathbb {A}}}

\newcommand{\BC}{{\mathbb {C}}}

\newcommand{\RO}{{\mathrm {O}}}

\newcommand{\diag}{{\mathrm{diag}}}

\newcommand{\GL}{{\mathrm{GL}}}

\newcommand{\Ind}{{\mathrm{Ind}}}

\newcommand{\I}{{\mathrm{I}}}

\renewcommand{\Re}{{\mathrm{Re}}}

\newcommand{\SL}{{\mathrm{SL}}}

\newcommand{\SO}{{\mathrm{SO}}}

\newcommand{\Sp}{{\mathrm{Sp}}}

\newtheorem{thm}{Theorem}[section]
\newtheorem{cor}[thm]{Corollary}
\newtheorem{lem}[thm]{Lemma}
\newtheorem{prop}[thm]{Proposition}

\newtheorem{rmk}[thm]{Remark}

\begin{document}
\renewcommand{\theequation}{\arabic{equation}}
\numberwithin{equation}{section}

\title{Two identities relating Eisenstein series on classical groups}

\author{David Ginzburg}

\address{School of Mathematical Sciences, Sackler Faculty of Exact Sciences, Tel-Aviv University, Israel
69978} \email{ginzburg@post.tau.ac.il}

\thanks{This research was supported by the ISRAEL SCIENCE FOUNDATION
	(grant No. 461/18).}

\author{David Soudry}
\address{School of Mathematical Sciences, Sackler Faculty of Exact Sciences, Tel-Aviv University, Israel
69978} \email{soudry@post.tau.ac.il}





\begin{abstract}
	In this paper we introduce two general identities relating Eisenstein series on split classical groups, as well as double covers of symplectic groups. The first identity can be viewed as an extension of the doubling construction introduced in \cite{CFGK17}. The second identity is a generalization of the descent construction studied in 
	\cite{GRS11}.

\end{abstract}
\maketitle
\setcounter{section}{-1}

\section{Introduction}

For simplicity, in order to avoid general notation in this introduction, we describe the two identities in the case of symplectic groups.

Let $F$ be a number field, and let $\BA$ be its ring of adeles. Let $\tau$ be an irreducible, automorphic, cuspidal representation of $\GL_n(\BA)$. Consider, for a positive integer $\ell$, the Speh representation $\Delta(\tau,\ell)$ of $\GL_{n\ell}(\BA)$. See Sec. 1.2 for the precise definition. Let $m,i$ be positive integers, and assume that $m$ is even. Consider, for a complex number $s$, the parabolic (normalized) induction
$$
\rho_{\Delta(\tau,m+i),s}=\Ind_{Q_{n(m+i)}(\BA)}^{\Sp_{2n(m+i)}(\BA)}\Delta(\tau,m+i)|\det\cdot|^s,
$$
where $Q_{n(m+i)}$ is the Siegel parabolic subgroup of $\Sp_{2n(m+i)}$ (viewed as an algebraic group over $F$). Let $E(f_{\Delta(\tau,m+i),s})$ be an Eisenstein series on $\Sp_{2n(m+i)}(\BA)$, corresponding to a smooth, holomorphic section $f_{\Delta(\tau,m+i),s}$ of $\rho_{\Delta(\tau,m+i),s}$. 

For our first identity, we consider a Fourier coefficient $\mathcal{F}_\psi(E(f_{\Delta(\tau,m+i),s}))$ of our Eisenstein series, along the unipotent radical $U_{m^{n-1}}$ of the standard parabolic subgroup $Q_{m^{n-1}}$, whose Levi part is isomorphic to $\GL_m^{n-1}\times \Sp_{2(m+ni)}$. The corresponding character of $U_{m^{n-1}}(\BA)$ is written explicitly in \eqref{1.3}, and denoted there by $\psi_H$, with $H=\Sp_{2n(m+i)}$. This character is attached to the symplectic partition $((2n-1)^m,1^{2ni+m})$.  It is stabilized by a subgroup $D(\BA)$, where $D$ is an $F$-subgroup of $\Sp_{2n(m+i)}$, which is isomorphic over $F$ to $\Sp_m\times \Sp_{2ni+m}$. See \cite{CM93}, Theorem 6.1.3. For $g\in \Sp_m(\BA)$, $h\in \Sp_{2ni+m}(\BA)$, denote by $t(g,h)$ the corresponding element in $D(\BA)$. We consider the Fourier coefficient above as a function on $\Sp_m(\BA)\times \Sp_{2ni+m}(\BA)$ (it is automorphic),
\begin{equation}\label{0.001}
\mathcal{F}_\psi(E(f_{\Delta(\tau,m+i),s}))(g,h)=
\int\limits_{U_{m^{n-1}}(F)\backslash
	U_{m^{n-1}}(\BA)}E(f_{\Delta(\tau,
	m+i),s},ut(g,h))\psi^{-1}_H(u)du.
\end{equation}
Now, we use $\mathcal{F}_\psi(E(f_{\Delta(\tau,m+i),s}))(g,h)$ as a kernel function and integrate it against cusp forms $\varphi_\sigma$ in a space of an irreducible, automorphic, cuspidal representation $\sigma$ of $\Sp_m(\BA)$, thus obtaining automorphic functions on $\Sp_{2ni+m}(\BA)$,
\begin{equation}\label{0.01}
\mathcal{E}(f_{\Delta(\tau,m+i),s},\varphi_\sigma)(h)=
\int\limits_{\Sp_m(F)\backslash
	\Sp_m(\BA)}\mathcal{F}_\psi(E(f_{\Delta(\tau,
	m+i),s}))(g,h)\varphi_\sigma(g)dg.
\end{equation}
Let $K_{\Sp_{2n(m+i)}(\BA)}=\prod_v K_{\Sp_{2n(m+i)}(F_v)}$, where, for each place $v$, $K_{\Sp_{2n(m+i)}(F_v)}$ is the standard maximal compact subgroup of $\Sp_{2n(m+i)}(F_v)$.\\
\\
{\bf Theorem A.}\quad 
\textit{Assume that $f_{\Delta(\tau,m+i),s}$ is $K_{\Sp_{2n(m+i)}(\BA)}$-finite. Then\\
 $\mathcal{E}(f_{\Delta(\tau,m+i),s},\varphi_\sigma)$ is an Eisenstein series on $\Sp_{2ni+m}(\BA)$, corresponding to the parabolic induction from $\Delta(\tau,i)|\det\cdot|^s\times \sigma^\iota$, where $\sigma^\iota$ is a certain outer conjugation of $\sigma$ by an element of order 2. There is an explicit meromorphic section $\Lambda(f_{\Delta(\tau, m+i),s},\varphi_\sigma)$ of the last parabolic induction, such that, for $Re(s)$ sufficiently large,}
\begin{equation}\label{0.1}
\mathcal{E}(f_{\Delta(\tau,
	m+i),s},\varphi_\sigma)(h)=\sum_{\gamma\in Q_{ni}(F)\backslash
	\Sp_{2ni+m}(F)}\Lambda(f_{\Delta(\tau,
	m+i),s},\varphi_\sigma)(\gamma h).
\end{equation}
\textit{The right hand side of \eqref{0.1} continues to a meromorphic function in the whole plane. Denote it by $E(\Lambda(f_{\Delta(\tau, m+i),s},\varphi_\sigma))$. Then, as meromorphic functions on $\Sp_{2ni+m}(\BA)$,}
\begin{equation}\label{0.1'}
\mathcal{E}(f_{\Delta(\tau,
	m+i),s},\varphi_\sigma)=E(\Lambda(f_{\Delta(\tau, m+i),s},\varphi_\sigma)).
\end{equation}

The proof of Theorem A is carried out in Sec. 2 - 4. In Sec. 4, we also prove a similar identity with normalized Eisenstein series, outside a finite set of places of $F$, containing the Archimedean ones, outside which $\tau$ and $\sigma$ are unramified, as well as $f_{\Delta(\tau,m+i),s}$ and $\varphi_\sigma$ in \eqref{0.1'}. We also prove analogous identities for metaplectic groups and split orthogonal groups. 
We remark here, that we expect the identity \eqref{0.1'} to hold with normalized Eisenstein series at all places, and all $f_{\Delta(\tau,m+i),s}$ and $\varphi_\sigma$, as Shahidi does for generic representations in \cite{S10}, Theorem 6.3.1, and as is done in \cite{JLZ13}, p. 86. In our case, we do not assume that $\sigma$ is a generic representation, and $\Delta(\tau, m+i)$, $\Delta(\tau,i)$ are not generic. The normalization in this case should be carried out by using the local $L$-functions defined in \cite{CFK18}, Sec. 7. In any case, also when $\sigma$ is generic, this will require more work.

The  nice thing about the identity \eqref{0.1'} is that it expresses an Eisenstein series on $\Sp_{2ni+m}(\BA)$, induced from $\Delta(\tau,i)|\det\cdot|^s\times \sigma^\iota$, in a uniform way, for all $\sigma$, in terms of an Eisenstein series on $\Sp_{2n(m+i)}(\BA)$, induced from $\Delta(\tau,m+i)|\det\cdot|^s$. For example, when $i=1$, we get an expression of an Eisenstein series on $\Sp_{2n+m}(\BA)$, induced from the cuspidal representation $\tau|\det\cdot|^s\times \sigma^\iota$, in terms of an Eisenstein series on $\Sp_{2n(m+1)}(\BA)$, induced from $\Delta(\tau,m+1)|\det\cdot|^s$. This shows that Eisenstein series induced from Speh representations play a prominent role. We can use them to produce Eisenstein series induced from cuspidal representations of Levi parts of maximal parabolic subgroups.

We note that the section $\Lambda(f_{\Delta(\tau, m+i),s},\varphi_\sigma)$ has a nice explicit integral expression, and that it is related to the integrals of the generalized doubling method of \cite{CFGK17} for $\Sp_m\times \GL_n$. It is interesting that the identity above exhibits a tower, as in towers of theta lifts, with the tower parametrized by $i$. It will be easy to note that our proofs are valid for $i=0$, as well. In this case, the proof amounts to the Euler product expansion of the generalized doubling integrals for $\Sp_m\times \GL_n$, representing the partial $L$-function $L^S(\sigma\times \tau,s)$, and then $\mathcal{E}(f_{\Delta(\tau, m+i),s},\varphi_\sigma)$ (with $i=0$) defines an element in $\sigma^\iota$. Thus, the tower starts with the cuspidal representation $\sigma^\iota$ on $\Sp_m(\BA)$, and moves up to the Eisenstein series representations on $\Sp_{2ni+m}(\BA)$, induced from $\Delta(\tau,i)|\det\cdot|^s\times \sigma^\iota$, for $i=1,2,...$.\\
We also allow the case $n=1$, where $\tau$ is simply a character of
$F^*\backslash \BA^*$, and then $\Delta(\tau, m+i)=\tau\circ
\det_{\GL_{m+i}}$,  $\rho_{\Delta(\tau,
	m+i),s}=\Ind_{Q_{m+i}(\BA)}^{\Sp_{2(m+i)}(\BA)}\tau(\det\cdot)
|\det\cdot|^s$. In this case, $U_{m^{n-1}}$ is trivial, and the
function $\mathcal{F}_\psi(E(f_{\Delta(\tau,m+i),s}))(g,h)$ is the restriction of the Eisenstein series
$E(f_{\tau\circ \det_{\GL_{m+i}},s})$ to $\Sp_m(\BA)\times
\Sp_{2i+m}(\BA)$. In this case, the Eisenstein
series $\mathcal{E}(f_{\tau\circ\det_{\GL_{m+i}},s},\varphi_\sigma)$ of Theorem A, corresponds to the section $\Lambda(f_{\tau\circ \det_{\GL_{m+i}},s},\varphi_\sigma)$ of the
parabolic induction from $(\tau\circ \det_{\GL_i})|\det\cdot |^s\times \sigma^\iota$ to
$\Sp_{2i+m}(\BA)$. This was considered in
\cite{GPSR97}, Sec. 1, for (arbitrary) orthogonal groups and $i=1$, and generalized by Moeglin in \cite{M97}, for any $i$ and $H$ symplectic, or even orthogonal.

In our second identity, we consider, for $i\geq 1$, the descent of $E(f_{\Delta(\tau,i+1),s})$ from $\Sp_{2n(i+1)}(\BA)$ to the metaplectic group $\Sp^{(2)}_{2ni}(\BA)$, that is we apply to $E(f_{\Delta(\tau,i+1),s})$ a Fourier-Jacobi coefficient corresponding to the partition $(2n,1^{2ni})$. Denote this coefficient by $\mathcal{D}^{(2),\phi}_{\psi,ni}(E(f_{\Delta(\tau,i+1),s}))$, and view it as an automorphic function of $\tilde{h}\in \Sp^{(2)}_{2ni}(\BA)$. Here $\phi\in \mathcal{S}(\BA^{ni})$ is a Schwartz function entering in the theta series on the semi-direct product of the Heisenberg group in $2ni+1$ variables and $\Sp^{(2)}_{2ni}(\BA)$. We consider in a similar way the analogous descent of $E(f_{\Delta(\tau,i+1)\gamma_\psi,s})$ from $\Sp^{(2)}_{2n(i+1)}(\BA)$ to  $\Sp_{2ni}(\BA)$. Here, $\gamma_\psi$ is the Weil factor attached to $\psi$. Denote the corresponding Fourier-Jacobi coefficient by $\mathcal{D}_{\psi,ni}^\phi(E(f_{\Delta(\tau,i+1)\gamma_\psi,s}))$, and view 
it as an automorphic function of $h\in \Sp_{2ni}(\BA)$. We note that the case $i=0$ makes sense. Here, the Fourier-Jacobi coefficients above are simply $\psi$-Whittaker coefficients of our Eisenstein series, which are well known. The second main theorem in this paper is\\
\\
{\bf Theorem B.}\quad 
\textit{$\mathcal{D}^{(2),\phi}_{\psi,ni}(E(f_{\Delta(\tau,i+1),s}))$ is an Eisenstein series on $\Sp^{(2)}_{2ni}(\BA)$, corresponding to the parabolic induction from $\Delta(\tau,i)\gamma_{\psi^{-1}}|\det\cdot|^s$. There is an explicit meromorphic section $\Lambda(f_{\Delta(\tau, i+1)\gamma_\psi,s},\phi)$ of the last parabolic induction, such that for $Re(s)$ sufficiently large,}
\begin{equation}\label{0.2}
\mathcal{D}^{(2),\phi}_{\psi,ni}(E(f_{\Delta(\tau,i+1),s}))(\tilde{h})=\sum_{\gamma\in Q_{ni}(F)\backslash
	\Sp_{2ni}(F)}\Lambda(f_{\Delta(\tau,
	i+1),s},\phi)(\gamma \tilde{h}).
\end{equation} 
\textit{The right hand side of \eqref{0.2} continues to a meromorphic function in the whole plane. Denote it by $E(\Lambda(f_{\Delta(\tau, i+1),s},\phi))(\tilde{h})$. Then, as meromorphic functions on $\Sp^{(2)}_{2ni}(\BA)$,}
\begin{equation}\label{0.3}
\mathcal{D}^{(2),\phi}_{\psi,ni}(E(f_{\Delta(\tau,i+1),s}))=E(\Lambda(f_{\Delta(\tau, i+1),s},\phi)).
\end{equation}
Theorem B is proved in Sec. 7. We can prove a similar identity with normalized Eisenstein series, outside a finite set of places of $F$, containing the Archimedean places, outside which $\tau$ is unramified, as well as $f_{\Delta(\tau,i+1),s}$ and $\phi$. We prove, in Sec. 8, a similar theorem for the descent $\mathcal{D}_{\psi,ni}^\phi(E(f_{\Delta(\tau,i+1)\gamma_\psi,s}))$ from $\Sp^{(2)}_{2n(i+1)}(\BA)$ to  $\Sp_{2ni}(\BA)$, expressing it as an Eisenstein series on $\Sp_{2ni}(\BA)$, corresponding to the parabolic induction from $\Delta(\tau,i)|\det\cdot|^s$. We also prove similar identities for split orthogonal groups. We remark that Theorem B and its analog for the descent $\mathcal{D}_{\psi,ni}^\phi(E(f_{\Delta(\tau,i+1)\gamma_\psi,s}))$, from $\Sp^{(2)}_{2n(i+1)}(\BA)$ to  $\Sp_{2ni}(\BA)$, do not follow from Theorem A. There is one case of Theorem B, for orthogonal groups, which is Theorem A with $m=1$, namely, when we take  the Eisenstein series corresponding to the parabolic induction from $\Delta(\tau,i+1)|\det\cdot|^s$ on $\SO_{2n(i+1)})(\BA)$, and then carry the descent to $\SO_{2ni+1}(\BA)$. See Sec. 5. The descent from the Eisenstein series on $\SO_{2n(i+1)+1}(\BA)$, corresponding the parabolic induction from $\Delta(\tau,i+1)|\det\cdot|^s$, to $\SO_{2ni}(\BA)$, is carried out in Sec. 6.

As in Theorem A, we allow the case $n=1$, where $\tau$ is a character of $F^*\backslash \BA^*$. In this case, Theorem B, for $H(\BA)=\Sp_{2i+2}(\BA)$, is a special case of a theorem of Ikeda \cite{I94}. 

There are several extensions and applications of the above two identities. First, we intend to extend these identities to metaplectic covering groups of split classical groups. In recent years, there has been a growing interest in the structure of automorphic representations of these groups. There is also a growing realization that many  constructions  which work for  linear groups can be extended to metaplectic covering groups. For example, as described in \cite{CFGK16}, the doubling construction for linear groups, given in \cite{CFGK17}, extends in a natural way to metaplectic covering groups. Thus, as in the above case when $i=1$,  we expect to express every Eisenstein series on a metaplectic covering group, parabolically induced from cuspidal data $\tau$ and $\sigma$, in terms of an Eisenstein series involving the representation $\tau$ only. In other words, we expect that an identity, similar to identity \eqref{0.1'} , 
will hold also for metaplectic covering groups. Here, $\tau$ and $\sigma$ are irreducible cuspidal representations of  certain metaplectic coverings of $\GL_n(\BA)$ and $\Sp_m(\BA)$, respectively. In general, we expect that Theorems A and B extend to metaplectic covering groups of split classical groups. 

A second extension, which we intend to study, is a variation  of the first identity. This new identity, which we refer to as the dual identity, uses the fact that the embedding of $\Sp_m(\BA)\times \Sp_{2ni+m}(\BA)$, inside $Sp_{2n(m+i)}(\BA)$, given by $t(g,h)$, is not symmetric with respect to $g$ and $h$. We give some details in the case when $i=1$. Instead of starting with a Fourier coefficient corresponding to the unipotent orbit $((2n-1)^m,1^{2n+m})$ of the group $\Sp_{2n(m+1)}$, we start with a Fourier coefficient corresponding to the orbit $((2n-1)^{2n+m},1^m)$ of the group $\Sp_{2n(2n+m-1)}$. Recall that $m$ is an even number. In both cases, the stabilizer contains the same group  $\Sp_m(\BA)\times \Sp_{2n+m}(\BA)$. However, for $g\in \Sp_m(\BA)$, and  $h\in \Sp_{2n+m}(\BA)$, the embedding now is given by 
$t(h,g)$. As in equation \eqref{0.001}, we can form the Fourier coefficient
\begin{equation}\label{0.03}
\mathcal{F}_\psi(E(f_{\Delta(\tau,2n+m-1),s}))(h,g)=\\ 
\int\limits_{U(F)\backslash
	U(\BA)}E(f_{\Delta(\tau,
	2n+m-1),s},ut(h,g))\psi^{-1}_H(u)du.
\end{equation}
Here, $U$ is a certain unipotent subgroup of $\Sp_{2n(2n+m-i)}$, and $\psi_H$ is a character of $U({\BA})$. Pairing this Fourier coefficient against a function $\varphi_\sigma(g)$, as above,  we obtain the identity
\begin{equation}\label{0.04}
\mathcal{E}(f_{\Delta(\tau,m+1),s},\varphi_\sigma)(h)=
\int\limits_{\Sp_m(F)\backslash
	\Sp_m(\BA)}\mathcal{F}_\psi(E(f_{\Delta(\tau,
	2n+m-1),s}))(h,g)\varphi_\sigma(g)dg.
\end{equation}
In other words, we expect two different ways to obtain the {\sl same} Eisenstein series 
$\mathcal{E}(f_{\Delta(\tau,m+1),s},\varphi_\sigma)$. The first is by using Theorem A, and the second is by interchanging the roles of $g$ and $h$, and then forming a similar construction. 

There are certain applications of the identities given by Theorems A and B, which we intend to study. The first is the study of poles of the various Eisenstein series involved in these identities. Thus, if $E(\Lambda(f_{\Delta(\tau, m+i),s},\varphi_\sigma))$ has a pole at $s=s_0$, then $E(f_{\Delta(\tau,m+i),s})$ must have a pole at $s=s_0$, and then we can relate the leading terms in the Laurent expansions via the identity. We note that the poles of $E(f_{\Delta(\tau,m+i),s})$ are well studied in \cite {JLZ13}. Of course, the pole $s=s_0$ of $E(\Lambda(f_{\Delta(\tau, m+i),s},\varphi_\sigma))$ might come from a pole of the section $\Lambda(f_{\Delta(\tau, m+i),s},\varphi_\sigma)$, and then we have to separate these poles out. Moreover, we hope to use the above two identities to study the  Fourier coefficients which the residue representations support. We believe that  the additional way, as in  \eqref{0.04}, of obtaining the same Eisenstein series  $E(\Lambda(f_{\Delta(\tau, m+i),s},\varphi_\sigma))$, will play a role in the study of these problems.

In many constructions which use certain small representations, such as residues of Eisenstein series, a crucial ingredient is the knowledge of the top unipotent orbit which these representations have. Even for an 
Eisenstein series induced from cuspidal data, attached to a maximal parabolic subgroup, this information is not known in general. For example, consider the Eisenstein series $\mathcal{E}(f_{\Delta(\tau,m+1),s},\varphi_\sigma)$. This Eisenstein series, on  $Sp_{4n+m}({\BA})$, is associated with the parabolic induction from $\tau|\text{det}\cdot|^s\times\sigma^\iota$.
It is well known that this Eisenstein series has a simple pole if the product of partial $L$ - functions $L^S(\tau\times\sigma,s)L^S(\tau,\wedge^2,2s)$ has a simple pole at $s=1/2$. It is conjectured, that the unipotent orbit associated with this residue is
$((2n+m),2n)$. See Conjecture (FC) in \cite{GJR04}. We hope to use Theorem A to study this conjecture.

In a forthcoming work of ours, we will use the identities as in Theorem B to analyze the maximal unipotent orbits supported by the two Eisenstein series $E(f_{\Delta(\tau,i+1)\gamma_\psi,s})$, and $E(f_{\Delta(\tau,i+1),s})$ at their various poles. For example, given a Fourier coefficient, which is nontrivial on the residue at a given pole $s_0$ of the Eisenstein series on the r.h.s. of \eqref{0.3}, and corresponds to a unipotent orbit for $\Sp_{2ni}$, we get a corresponding unipotent orbit for $\Sp_{2n(i+1)}$ supported by the residue at $s_0$ of $E(f_{\Delta(\tau,i+1),s})$. We will also combine both identities of Theorems A, B to find explicit relations of various Rankin-Selberg integrals to those of the doubling method. See \cite{GS19}, for the general philosophy and some concrete examples.

In the next section, we will state precisely the main theorems of this paper. Sections 2-4 deal with Theorem A (and its analogues for metaplectic groups and split orthogonal groups), and Sections 5-8 deal with Theorem B (and its analogues). 

We thank the referee, and appreciate his effort, for his very careful reading throughout the whole paper. He made many important, useful suggestions, comments and queries, regarding the mathematical content, as well as the presentation of the paper.

\section{Statement of the main theorems}

Let $H$ be a symplectic group, viewed as an algebraic group over $F$, or a special orthogonal group, split over $F$. In this section, we consider Eisenstein series on $H(\BA)$, or on the metaplectic double cover of a symplectic group, induced from a Speh representation. We will introduce an explicit kernel function
obtained by considering a certain Fourier coefficient of such an Eisenstein series. In the linear case, it turns out to be an automorphic function on a product $L(\BA)\times R(\BA)$, where $L$ and $R$ are  two split classical groups over $F$, and $L$ is a direct
factor of the Levi part of a maximal parabolic subgroup of $R$. In case of metaplectic groups, $R(\BA)$ and $L(\BA)$ should be replaced by their double covers. Then we pair this kernel function against cusp forms on $L(\BA)$. Our first main theorem will identify such a kernel integral as an Eisenstein series on $R(\BA)$. Later on, we will consider  Bessel coefficients, or  Fourier-Jacobi coefficients on the Eisenstein series above (on $H(\BA)$). They are of the same form as the ones studied in \cite{GRS11} (descent). Viewed as automorphic functions on the adele points of the group $H'$,  stabilizing the character defining the Fourier coefficient, our second main theorem will identify this descent, as a similar type Eisenstein series on $H'(\BA)$, induced from a similar Speh representation, but of length shorter by 1. \\
\\
{\bf 1. The groups}\\

Let $F$ be a number field and $\BA$ its ring of adeles. We will
consider symplectic groups $\Sp_{2k}$ and split orthogonal groups
$\SO_k$ as algebraic groups defined over $F$. Let $v$ be a place of $F$, and $F_v$ the completion of $F$ at $v$. We will realize these groups over $F_v$ as matrix groups in the following standard way. Let $w_k$ denote the $k\times k$
permutation matrix which has $1$ along the main anti-diagonal. Denote by $\bar{F_v}$ an algebraic closure of $F_v$. Then
the corresponding matrix algebraic groups are
\begin{equation}\label{1.1}
 \Sp_{2k}(\bar{F_v})=\{g\in \GL_{2k} (\bar{F_v})\ |\
{}^tg\begin{pmatrix}&w_k\\-w_k\end{pmatrix}g=\begin{pmatrix}&w_k\\-w_k\end{pmatrix}
\}
\end{equation}

$$
 \SO_k(\bar{F_v})=\{g\in \SL_k (\bar{F_v})\ |\ {}^tgw_kg=w_k\}.
$$
Similarly, for an algebraic closure of $F$, we have the groups $\GL_{2k}(\bar{F})$, $\Sp_{2k}(\bar{F})$, $\SO_k(\bar{F})$, etc.  We will oftentimes denote $\Sp_{2k}$, or $\SO_k$ by $H_\ell$, where $\ell$ is the
number of variables of the corresponding anti-symmetric, or
symmetric form. When we define subgroups of one the groups above, which are defined over $F$, e.g. parabolic subgroups, we will use, for short, the notation  $H_\ell$ for anyone of the groups $H_\ell(\bar{F_v})$, or $H_\ell(\bar{F})$, and similarly for  $\GL_k$ etc. For example, here is such a definition of a subgroup of $H_\ell$. Let $j\leq [\frac{\ell}{2}]$. Then we denote by $\widehat{\GL_j}$ the subgroup of $H_\ell$ consisting of the elements
\begin{equation}\label{1.1'}
\hat{a}=\begin{pmatrix}a\\&I_{\ell-2j}\\&&a^*\end{pmatrix},\quad a\in \GL_j,
\end{equation}
where, for $a\in \GL_j$, 
\begin{equation}\label{1.1''}
a^*=w_j{}^ta^{-1}w_j.
\end{equation}
We don't mention $\ell$ in this definition. This will always be clear from the context. As usual, we denote by $H_\ell(F_v)$  the subgroup of matrices in $H_\ell(\bar{F_v})$ with coordinates in $F_v$, and similarly for $H_\ell(F)$. Now, we can form the adele groups $H_\ell(\BA)$.\\
For each place $v$, we fix a maximal compact subgroup $K_{H_\ell(F_v)}$ of $H_\ell(F_v)$. When $v$ is finite, we will take it to be $H_\ell(\mathcal{O}_v)$, where $\mathcal{O}_v$ is the ring of integers inside $F_v$. We denote
\begin{equation}\label{1.1*}
K_{H_\ell(\BA)}=\prod_v K_{H_\ell(F_v)}\subset H_\ell(\BA).
\end{equation}

We will consider also double covers of symplectic groups over the local fields $F_v$. We will denote these groups by $\Sp^{(2)}_{2k}(F_v)$. In case $F_v=\BC$, $\Sp_{2k}^{(2)}(\BC)=\Sp_{2k}(\BC)\times 1$. Otherwise, we realize these double covers as $\Sp^{(2)}_{2k}(F_v)\times \{\pm 1\}$, using the normalized Ranga Rao cocycle, corresponding to the standard Siegel parabolic subgroup (\cite{Rao93}). See, for example, \cite{JS07}, Sec. 2.1. Although $\Sp^{(2)}_{2k}(F_v)$ is not the group of $F_v$ - points of an algebraic group defined over $F_v$, we will apply the language of algebraic groups by saying that the $F_v$ - points of $\Sp_{2k}^{(2)}$ is the group $\Sp_{2k}^{(2)}(F_v)$. 
We know that unipotent subgroups $U_v$ of $\Sp_{2k}(F_v)$ split in $\Sp^{(2)}_{2k}(F_v)$, and if $U_v$ consists of upper triangular unipotent matrices, then the Ranga Rao cocycle is trivial on $U_v\times U_v$. Thus, $U_v\times 1$ is a subgroup of $\Sp^{(2)}_{2k}(F_v)$.  We will identify $U_v$ and $U_v\times 1$. Also, we know that there is a finite set of places $S_0=S_{0,F}$ containing the Archimedean places, such that, for $v\notin S_0$,   there is a unique embedding 
$\zeta_v: K_{\Sp_{2k}(F_v)}\rightarrow \Sp^{(2)}_{2k}(F_v)$ of the form 
\begin{equation}\label{1.1.0}
\zeta_v(r)=(r,\lambda_v(r)). 
\end{equation}
See \cite{MVW87}, p. 43. See also \cite{Sw90}, Prop. 1.6.5, where the function $\lambda_v$ is described. We will denote  $K_{\Sp^{(2)}_{2k}(F_v)}=\zeta_v(K_{\Sp_{2k}(F_v)})$. In \cite{Sw90}, Prop. 1.6.6, Sweet extends the function $\lambda_v$ to $\Sp_{2k}(F_v)$. In particular, he shows that $\lambda_v$ is trivial on $Q_k(F_v)$ ($Q_k$ is the Siegel parabolic subgroup of $\Sp_{2k}$). In the places $v\in S_0$, we will take $K_{\Sp^{(2)}_{2k}(F_v)}$ to be the inverse image of $K_{\Sp_{2k}(F_v)}$ inside $\Sp^{(2)}_{2k}(F_v)$. Also, for $v\in S_0$, we define (as Sweet does in \cite{Sw90}) $\lambda_v=1$, as a function on $\Sp_{2k}(F_v)$.  
Consider the restricted direct product 
$\widetilde{\Sp}_{2k}(\BA)=\prod'_v Sp^{(2)}_{2k}(F_v)$, with respect to the groups $\{K_{\Sp^{(2)}_{2k}(F_v)}\}_{v\notin S_0}$. Then the double cover $\Sp^{(2)}_{2k}(\BA)$ is the quotient of $\widetilde{\Sp}_{2k}(\BA)$ by the subgroup 
\begin{equation}\label{1.x}
C'=\{\Pi'_v(I_{2k},\mu_v)\in \widetilde{\Sp}_{2k}(\BA) \ |\ \Pi_v\mu_v=1\}.
\end{equation}
See \cite{JS07}, Sec. 2.2. Let 
$$
p=p_{2k}:\widetilde{\Sp}_{2k}(\BA)\rightarrow \Sp^{(2)}_{2k}(\BA),
$$ 
denote the quotient map. 
\begin{equation}\label{1.x.1}
p(\Pi'_v(g_v,\mu_v))=C'\Pi'_v(g_v,\mu_v).
\end{equation}
The kernel of $p$ is
\begin{equation}\label{1.1.0.a}
C_2=\{C'\Pi_v'(I_{2k},\mu_v)\in \Sp^{(2)}_{2k}(\BA)\}\cong \{\pm 1\}.
\end{equation}
Denote 
\begin{equation}\label{1.1.a}
K_{\Sp^{(2)}_{2k}(\BA)}=p(\prod_v K_{\Sp^{(2)}_{2k}(F_v)}).
\end{equation}
The group $\Sp_{2k}(F)$ embeds "diagonally" in $\Sp^{(2)}_{2k}(\BA)$ by
\begin{equation}\label{1.1.0.b}
\gamma\mapsto C'\Pi'_v(\gamma,1).
\end{equation}
We will identify $\Sp_{2k}(F)$ as a subgroup of $\Sp^{(2)}_{2k}(\BA)$, and oftentimes, we will denote the r.h.s. of \eqref{1.1.0.b} simply by $(\gamma,1)$ or even by $\gamma$. It will be convenient to denote by $C^{(1)}_2$ the trivial subgroup of $H_m(\BA)=H_m^{(1)}(\BA)$, and
\begin{equation}\label{1.1.0.c}
C_2^{(2)}=C_2\subset \Sp^{(2)}_{2k}(\BA).
\end{equation}
We don't mention $m$ or $k$ in our notation, and there won't be any confusion. Thus, in both cases we may considet the quotient $C^{(\epsilon)}H_m(F)\backslash H_m^{(\epsilon)}(\BA)$.

When we consider an irreducible, automorphic, cuspidal representation $\sigma$ of $\Sp^{(2)}_{2k}(\BA)$, we always assume that it is genuine. This means that an element $C'\cdot \Pi'_v(I_{2k},\mu_v)$ acts by multiplication by $\Pi_v\mu_v$. The represntation $\sigma$ decomposes as a restricted tensor product of local irreducible genuine representations $\sigma_v$ of $\Sp^{(2)}(F_v)$, $\otimes'_v\sigma_v$, in the sense that 
\begin{equation}\label{1.1.b}
\sigma\circ p\cong \otimes'_v\sigma_v.
\end{equation}
We will use sometimes the following  notation. Let $\Pi'_v(g_v,\mu_v)\in  \widetilde{\Sp}_{2k}(\BA)$. Let $g=\Pi'_vg_v\in \Sp_{2k}(\BA)$, and let $\bar{\mu}$ denote the sequence $(\mu_v)_v$. Note, that for almost all $v\notin S_0$, $g_v\in K_{\Sp_{2k}(F_v)}$ and $\mu_v=\lambda_v(g_v)$. We will denote
\begin{equation}\label{1.1.c}
(g,\bar{\mu})=\Pi'_v(g_v,\mu_v).
\end{equation}
Sometimes we will denote, for short,
\begin{equation}\label{1.1.d}
\sigma(p((g,\bar{\mu}))):=\sigma((g,\bar{\mu})).
\end{equation}
For an $F$- subgroup $U$ of $\Sp_{2k}$, consisting of upper unipotent matrices, we will identify $U(\BA)$ with the subgroup of elements $p((u,\bar{1}))$, where $u\in U(\BA)$, and $\bar{1}$ is the sequence with coordinate $1$, at all places $v$. We will usually denote $p((u,\bar{1}))$ by $(u,1)$, or by $u$.

In order to unify notation, let $\epsilon=1,2$. We will consider the groups $H_\ell^{(1)}=H_\ell$, and the groups $H_\ell^{(2)}$, only when $H_\ell$ is symplectic. In both cases, we will use the notation $H^{(\epsilon)}_\ell$ with the agreement that when $H_\ell$ is orthogonal then $\epsilon=1$. We will refer to the case where $\epsilon=1$ as the linear case. We will refer to the case where $\epsilon=2$ (and then $H_\ell$ is symplectic) as the metaplectic case.

We will also consider, when $F_v\neq \BC$, the double cover $\GL_n^{(2)}(F_v)$ of $\GL_n(F_v)$, obtained as the set of pairs $(g,\mu)$, $g\in \GL_n(F_v)$, $\mu=\pm 1$, with multiplication given by
\begin{equation}\label{1.1**}
(g_1,\mu_1)\cdot (g_2,\mu_2)=(g_1g_2,\mu_1\mu_2(\det(g_1)\det(g_2))),
\end{equation}
where $(\det(g_1),\det(g_2))=(\det(g_1),\det(g_2))_v$ is the local Hilbert symbol of $F_v$ of the elements $\det(g_1)$ and $\det(g_2)$. When $F_v=\BC$, we take $\GL_n^{(2)}(\BC)=\GL_n(\BC)\times 1$. Fix, for each place $v$, a maximal compact subgroup $K_{\GL_n(F_v)}$ of $\GL_n(F_v)$. When $v$ is finite, take it to be $\GL_n(\mathcal{O}_v)$. When $v$ is odd, $K_{\GL_n(F_v)}\times 1$ is a subgroup of $\GL_n^{(2)}(F_v)$. As before, we denote in this case $K_{\GL^{(2)}_n(F_v)}=K_{\GL_n(F_v)}\times 1$. We will identify this subgroup with $K_{\GL_n(F_v)}$. For the remaining even places $v$, let $K_{\GL^{(2)}_n(F_v)} $ be the inverse image of $K_{\GL_n(F_v)}$ inside of $\GL_n^{(2)}(F_v)$. 

We define the group $\GL^{(2)}_n(\BA)$ similarly. Its elements are pairs $(g,\mu)$, where $g\in \GL_n(\BA)$, and $\mu=\pm 1$, and the multiplication rule is similar to \eqref{1.1**}. 
\begin{equation}\label{1.1***}
(g_1,\mu_1)\cdot (g_2,\mu_2)=(g_1g_2,\mu_1\mu_2(\det(g_1)\det(g_2))),
\end{equation}
where $(\det(g_1),\det(g_2)$ is the global Hilbert symbol. Consider the restricted direct product $\widetilde{\GL}_n(\BA)=\prod'_v\GL_n^{(2)}(F_v)$, with respect to the groups $K_{\GL^{(2)}_n(F_v)}$, $v$ odd. Then $\GL_n^{(2)}(\BA)$ is the image of the homomorphism from $\widetilde{\GL}_n(\BA)$ given by 
\begin{equation}\label{1.y}
p'=p'_n:\Pi'_v(g_v,\mu_v)\mapsto (\Pi'_vg_v,\Pi_v\mu_v). 
\end{equation}
Note that the inverse image, $M_n^{(2)}(\BA)$, of the Levi subgroup $M_n(\BA)=\widehat{\GL}_n(\BA)$ of the Siegel parabolic subgroup $Q_n(\BA)$ of $\Sp_{2n}(\BA)$, is isomorphic to $\GL^{(2)}_n(\BA)$. Indeed, let $j=1,2$, $a_j\in \GL_n(\BA)$, and $\bar{\mu}_j=(\mu_{j,v})_v$, $\mu_{j,v}=\pm 1$, such that $(\hat{a}_j,\bar{\mu}_j)\in \widetilde{\Sp}_{2n}(\BA)$. Note that this implies that $\mu_{j,v}=1$, for almost all $v$ (since $\lambda_v$ is trivial on $\widehat{\GL}_n(\mathcal{O}_v)$). We have in $\widetilde{\Sp}_{2n}(\BA)$,
\begin{equation}\label{1.1.1}
(\hat{a}_1,\bar{\mu}_1)(\hat{a}_2,\bar{\mu}_2)=
(\widehat{a_1a_2},\bar{\mu}_1\bar{\mu}_2\Pi'_v(\det(a_{1,v}),\det(a_{2,v}))_v).
\end{equation}
Here, $\Pi'_v(\det(a_{1,v}),\det(a_{2,v}))_v$ denotes the sequence of local Hilbert symbols\\
$(\det(a_{1,v}),\det(a_{2,v}))_v$. Thus, the map $p((\hat{a},\bar{\mu}))\mapsto(a,\Pi_v\mu_v)$ defines an isomorphism from $M_n^{(2)}(\BA)$ onto $\GL_n^{(2)}(\BA)$. 
Similarly, $M_n^{(2)}(F_v)$ is isomorphic to $\GL^{(2)}_n(F_v)$. We will abuse notation a little and write in \eqref{1.1.1} $(\det(a_1),\det(a_2))$, instead of 
the $\Pi'_v(\det(a_{1,v}),\det(a_{2,v}))_v$, writing
\begin{equation}\label{1.1.1'}
(\hat{a}_1,\bar{\mu}_1)(\hat{a}_2,\bar{\mu}_2)=
(\widehat{a_1a_2},\bar{\mu}_1\bar{\mu}_2(\det(a_1),\det(a_2))).
\end{equation}
There will not be any confusion. We will identify
$M^{(2)}_n(F_v)$ with $\GL_n^{(2)}(F_v)$, and $M_n^{(2)}(\BA)$ with $\GL^{(2)}_n(\BA)$. For $a\in \GL_n(\BA)$, we will denote $p((\hat{a},\bar{\mu}))$ also by $(\hat{a},\Pi_v\mu_v)$, and, sometimes, also, by $\widehat{(a, \Pi_v\mu_v)}$. In particular, if $\mu_v=1$, for all $v$, we will denote this element of $M_n^{(2)}(\BA)$ by $(\hat{a},1)$.
Again, we will unify notation and denote by $\GL_n^{(\epsilon)}(F_v)$ the group $\GL_n(F_v)$, when $\epsilon=1$, and the group $\GL_n^{(2)}(\BA)$, when $\epsilon=2$. We will similarly use the notation $\GL_n^{(\epsilon)}(\BA)$.

We fix positive integers $n, m$, and a nonnegative integer $i$. We are mainly interested in the case where $i$ is positive, but the definitions and results make sense also for $i=0$. As will be seen, the case $i=0$ corresponds to the generalized doubling method of \cite{CFGK17}. We will consider soon certain Eisenstein series on $H^{(\epsilon)}_{2n(m+i)}(\BA)$. In case this group is symplectic, or metaplectic, we assume that $m$ is even. When convenient,
we will shorten our notation by putting $H=H^{(\epsilon)}_{2n(m+i)}$, and
$r_H=2n(m+i)$. We let $\delta_H=1$ when $H$ is orthogonal, and
$\delta_H=-1$, when $H$ is symplectic, or metaplectic. When $H$ is orthogonal, and $m$ is even, we will assume that $m\geq 4$. (The reason is that we will soon consider cuspidal representations of $H_m(\BA)$. The case of split $\SO_2$ is problematic, since it doesn't have nontrivial unipotent radicals.) When $H$ is orthogonal and $m$ is odd, we will allow $m=1$. As will be clear in the sequel, all arguments work here nicely. 

We denote the standard basis
of $F^{2n(m+i)}$
by
$$
 \{e_1,...,e_{n(m+i)}, e_{-n(m+i)},...,e_{-1}\}.
$$
Assume that $H$ is linear. For positive integers, $k_1,...,k_\ell$, such that
$k_1+\cdots+k_\ell\leq \frac{r_H}{2}$, let
$Q_{k_1,...,k_\ell}=Q^{(1)}_{k_1,...,k_\ell}=Q_{k_1,...,k_\ell}^H$ denote the standard
parabolic subgroup of $H$, whose Levi part is isomorphic to
$\GL_{k_1}\times\cdots\times \GL_{k_\ell}\times H'$, where $H'$ is a
classical group of the same type as $H$. Denote the corresponding
Levi part by $M_{k_1,...,k_\ell}=M_{k_1,...,k_\ell}^H$, and
unipotent radical by $U_{k_1,...,k_\ell}=U_{k_1,...,k_\ell}^H$. When
$k_1=\cdots=k_\ell=k$, we will simply denote $Q_{k^\ell}$,
$M_{k^\ell}$, $U_{k^\ell}$. In the metaplectic case, we consider the analogous subgroups $Q^{(2)}_{k_1,...,k_\ell}(F_v)$, $Q^{(2)}_{k_1,...,k_\ell}(\BA)$, obtained as the inverse image in $H(F_v)$, $H(\BA)$, respectively, of the similar parabolic subgroup of the corresponding symplectic group.

Given positive integers $k_1,...,k_\ell$, such that
$k_1+\cdots+k_\ell=k$, we denote by $P_{k_1,...,k_\ell}$ the
standard parabolic subgroup of $\GL_k$, consisting of upper
triangular block matrices, with diagonal of the form
$diag(g_1,...,g_\ell)$, where $g_i\in\GL_{k_i}$, for $1\leq i\leq
\ell$. We denote the corresponding unipotent radical by
$V_{k_1,...,k_\ell}$. We will denote $Z_n=V_{1^n}$. This is the standard maximal unipotent subgroup of $\GL_n$.

Consider the parabolic subgroup $Q_{m^{n-1}}\subset H$. The elements of $U_{m^{n-1}}$ have the form
\\
\begin{equation}\label{1.2}
u=\begin{pmatrix}I_m&x_1&&&&\star&&&&\star\\&\ddots&&&&\cdots&&&&\cdots\\&&I_m&x_{n-2}&&\star&&&&\star\\
&&&I_m&y_1&y_2&y_3&\star&&\star\\&&&&I_{[\frac{m}{2}]}&&&y'_3&&\star\\&&&&&I_{2(ni+[\frac{m+1}{2}])}&&y'_2&&\star\\
&&&&&&I_{[\frac{m}{2}]}&y'_1&&\star\\&&&&&&&I_m&x'_{n-2}\
\cdots&\star\\&&&&&&&&\ddots&\star\\&&&&&&&&&x'_1\\&&&&&&&&&I_m\end{pmatrix},
\end{equation}
and, of course, $u$ should lie in $H$. In particular, $x_i$ and $y_j$ determine $x'_i$ and $y'_j$, respectively ($1\leq i\leq n-2$, $j=1,2,3$.) In case $H$ is metaplectic, we identify $U_{m^{n-1}}$ as a subgroup of $H$, over local fields, or adeles, by the embedding $u\mapsto (u,1)$. We fix a nontrivial character
$\psi$ of $F\backslash \BA$. It defines the following character
$\psi_H=\psi_{U_{m^{n-1}}}$ of $U_{m^{n-1}}(\BA)$, trivial on
$U_{m^{n-1}}(F)$. Its value on the element $u$ of the form
\eqref{1.2}, with adele coordinates, is
\begin{equation}\label{1.3}
\psi_H(u)=\psi(tr(x_1+\cdots+x_{n-2}))\psi(tr((y_1,y_2,y_3)A_H)),
\end{equation}
where $A_H$ is the following matrix. When $m$ is even,
\begin{equation}\label{1.4}
A_H=\begin{pmatrix}I_{\frac{m}{2}}&0\\0_{(2ni+m)\times \frac{m}{2}}&0_{(2ni+m)\times \frac{m}{2}}\\0&I_{\frac{m}{2}}\end{pmatrix}.
\end{equation}
When $m$ is odd (and hence $H$ is orthogonal),
\begin{equation}\label{1.5}
A_H=\begin{pmatrix}I_{[\frac{m}{2}]}&0&0\\0&0&0\\0&1&0\\0&\frac{1}{2}&0\\0&0&0\\0&0&I_{[\frac{m}{2}]}\end{pmatrix},
\end{equation}
where the second and fifth block rows of zeroes contain each $ni+[\frac{m}{2}]$
rows. We note that $\psi_H$ corresponds to the nilpotent orbit in
the Lie algebra of $H$ corresponding to the partition
$((2n-1)^m, 1^{2ni+m})$. See \cite{MW87}, \cite{GRS03}.\\
Assume that $H$ is linear. The stabilizer of $\psi_H$ in
$M_{m^{n-1}}(\BA)$ is the group of adelic points of an algebraic group over $F$, which we denote by $D=D_{\psi_H}$. It is isomorphic to $H_m\times
H_{2ni+m}$. Note that when $m=2m'$ is even and $H$ is symplectic, this is
$\Sp_{2m'}\times \Sp_{2(ni+m')}$. When $H$ is orthogonal,
this is the group $\SO_{2m'}\times \SO_{2(ni+m')}$. Assume that $m=2m'-1$ is odd, so that $H$ is orthogonal. Then $D$ is isomorphic to
$\SO_{2m'-1}\times \SO_{2ni+2m'-1}$. See \cite{CM93}, Theorem 6.1.3. The elements of $D$ are realized
as
\begin{equation}\label{1.6}
t(g,h)=diag(g^{\Delta_{n-1}},j(g,h),(g^*)^{\Delta_{n-1}}),\ g\in
H_m, \quad h\in H_{2ni+m},
\end{equation}
where $g^{\Delta_{n-1}}=diag(g,...,g)$ ($n-1$ times), and $j(g,h)$
is as follows.\\
Assume that $m=2m'$ is even. Then $g\in H_{2m'}$ and $h\in H_{2(ni+m')}$, and
$H$ is either symplectic, or orthogonal. Write
$g=\begin{pmatrix}a&b\\c&d\end{pmatrix}$, where $a,...,d$ are
$m'\times m'$ matrices. Then
\begin{equation}\label{1.7}
j(g,h)=\begin{pmatrix}a&&b\\&h\\c&&d\end{pmatrix},\ g\in H_{2m'},
h\in H_{2(ni+m')}.
\end{equation}
Let $X^\pm_{m'}$ be the subspaces spanned by
$$
e_{\pm(2m'(n-1)+1)},e_{\pm(2m'(n-1)+2)},...,e_{\pm(2m'(n-1)+m')},
$$ 
and let $Y^\pm_{ni+m'}$ be the subspace spanned by
$$
e_{\pm(2m'(n-1)+m'+1},
e_{\pm(2m'(n-1)+m'+2},...,e_{\pm(2m'(n-1)+2m'+ni)}. 
$$
Denote
$$
X_m=X^+_{m'}+X^-_{m'}, \quad Y_{2ni+m}=Y^+_{ni+m'}+Y^-_{ni+m'}. 
$$
Then
$j(g,h)$ acts on $X_m$ according to $g$, and acts on $Y_{2ni+m}$
according to $h$.\\
Assume that $m=2m'-1$ is odd. Then $H=\SO_{2n(2m'-1+i)}$, $g\in \SO_{2m'-1}=H_m$, $h\in
H_{2ni+2m'-1}\cong\SO_{2ni+2m'-1}$. In this case, we will write the elements
of $H_{2ni+2m'-1}$ with respect to the symmetric matrix
$w'_{2ni+2m'-1}$, where
\begin{equation}\label{1.7.1}
w'_{2ni+2m'-1}=\begin{pmatrix}&&w_{ni+m'-1}\\&-1\\w_{ni+m'-1}\end{pmatrix},
\end{equation}
so that 
$$
H_{2ni+2m'-1}=\{g\in \SL_{2ni+2m'-1}\ | \
{}^tgw'_{2ni+2m'-1}g=w'_{2ni+2m'-1}\}.
$$
Write
$$
g=\begin{pmatrix}a_1&b_1&c_1\\a_2&b_2&c_2\\a_3&b_3&c_3\end{pmatrix},\
h=\begin{pmatrix}A_1&B_1&C_1\\A_2&B_2&C_2\\A_3&B_3&C_3\end{pmatrix},
$$
where the first and third block rows (resp. columns) of $g$ contain
each $m'-1$ rows (resp. columns), and similarly, the first and third
block rows (resp. columns) of $h$ contain each $ni+m'-1$ rows (resp.
columns). Then
\begin{equation}\label{1.8}
j(g,h)=\begin{pmatrix}a_1&0&\frac{1}{2}b_1&b_1&0&c_1\\0&A_1&\frac{1}{2}B_1&-B_1&C_1&0\\
a_2&A_2&\frac{1}{2}(b_2+B_2)&b_2-B_2&C_2&c_2\\\frac{1}{2}a_2&-\frac{1}{2}A_2&\frac{1}{4}(b_2-B_2)&\frac{1}{2}(b_2+B_2)&-\frac{1}{2}C_2&\frac{1}{2}c_2\\
0&A_3&\frac{1}{2}B_3&-B_3&C_3&0\\a_3&0&\frac{1}{2}b_3&b_3&0&c_3\end{pmatrix}.
\end{equation}
Let $X^\pm_{m'-1}$ be the subspaces spanned by
$$
e_{\pm((2m'-1)(n-1)+1)},e_{\pm((2m'-1)(n-1)+2)},...,e_{\pm((2m'-1)(n-1)+m'-1)},
$$
and let $Y^\pm_{ni+m'-1}$ be the subspaces spanned by
$$
e_{\pm((2m'-1)(n-1)+m')},
e_{\pm((2m'-1)(n-1)+m'+1)},...,e_{\pm((2m'-1)(n-1)+ni+2m'-2)}. 
$$
Denote
$$
X_m=X^+_{m'-1}+F(e_{n(m+i)}+\frac{1}{2}e_{-n(m+i)})+X^-_{m'-1},
$$
$$
Y_{2ni+m}=Y^+_{ni+m'-1}+F(e_{n(m+i)}-\frac{1}{2}e_{-n(m+i)})+Y^-_{ni+m'-1}.
$$
Then $j(g,h)$ acts on $X_m$ according to $g$, and it acts on
$Y_{2ni+m}$ according to $h$. More precisely, for $x^\pm\in
X^\pm_{m'-1}$, and $z\in F$,
$$
j(g,h)(x^++z(e_{n(m+i)}+\frac{1}{2}e_{-n(m+i)})+x^-)=(x')^++z'(e_{n(m+i)}+\frac{1}{2}e_{-n(m+i)})+(x')^-,
$$
where $(x')^\pm \in X^\pm_{m'-1}$, $z'\in F$, such that
$$
g\begin{pmatrix}x^+\\z\\x^-\end{pmatrix}=\begin{pmatrix}(x')^+\\z'\\(x')^-\end{pmatrix},
$$
where we identify $x^\pm$ with its column vector of $m'-1$
coordinates according to the basis above of $X^\pm_{m'-1}$.
Similarly, for $y^\pm \in Y^\pm_{ni+m'-1}$, and $z\in F$,
$$
j(g,h)(y^++z(e_{n(m+i)}-\frac{1}{2}e_{-n(m+i)})+y^-)=(y')^++z'(e_{n(m+i)}-\frac{1}{2}e_{-n(m+i)})+(y')^-,
$$
where $(y')^\pm \in Y^\pm_{ni+m'-1}$, $z'\in F$, such that
$$
h\begin{pmatrix}y^+\\z\\y^-\end{pmatrix}=\begin{pmatrix}(y')^+\\z'\\(y')^-\end{pmatrix},
$$
where we identify $y^\pm$ with its column vector of $ni+m'-1$
coordinates according to the basis above of $Y^\pm_{ni+m'-1}$.\\
Assume that $H$ is metaplectic. Then the analogue of \eqref{1.6} is given by the homomorphism (over $F_v$)
$$
t_v^{(2)}:\Sp_{2m'}^{(2)}(F_v)\times \Sp_{2ni+2m'}^{(2)}(F_v)\rightarrow \Sp_{2n(2m'+i)}^{(2)}(F_v)
$$
given by
\begin{equation}\label{1.9}
t_v^{(2)}((g,\alpha),(h,\beta))=(t(g,h),\alpha\beta (x_1(g),x_2(h))),
\end{equation}
where $t(g,h)$ is given by \eqref{1.6}, $\alpha,\beta=\pm 1$, and $x_1=x_{1,v},x_2=x_{2,v}$ are the Ranga Rao $x$-functions on $\Sp_{2m'}(F_v)$, $\Sp_{2ni+2m'}(F_v)$, respectively. See \cite{Rao93}, Lemma 5.1. They take values in $F_v^*/(F_v^*)^2$; $(x_1(g),x_2(h))=(x_{1,v}(g),x_{2,v}(h))_v$ is the Hilbert symbol. The kernel of the homorphism \eqref{1.9} is 
$$
\{((I_{2m'},\epsilon),(I_{2ni+2m'},\epsilon))\ |\epsilon=\pm 1 \}.
$$
See, also, \cite{JS07}, Sec. 2.1. The adele version of \eqref{1.9} is similar.
Then the collection $\{t_v^{(2)}\}$ defines the homomorphism
$$
\tilde{t}:\widetilde{\Sp}_{2m'}(\BA)\times \widetilde{\Sp}_{2ni+2m'}(\BA)\rightarrow \widetilde{\Sp}_{2n(2m'+i)}(\BA),
$$
such that, for each place $v$, $\tilde{t}$ on $\Sp_{2m'}^{(2)}(F_v)\times \Sp_{2ni+2m'}^{(2)}(F_v)$ is $t^{(2)}_v$. See \cite{JS07}, Sec. 2.2.
 Thus, we have the homomorphism
$$
t^{(2)}:\Sp^{(2)}_{2m'}(\BA)\times \Sp^{(2)}_{2ni+2m'}(\BA)\rightarrow \Sp^{(2)}_{2n(2m'+i)}(\BA),
$$
\begin{equation}\label{1.9.1''}
t^{(2)}(p((g,\bar{\alpha})),p((h,\bar{\beta})))=p((t(g,h),\bar{\alpha}\bar{\beta} (x_1(g),x_2(h)))).
\end{equation}
We used the notation \eqref{1.1.c}. Also, $t(g,h)$ is the element whose coordinate at $v$ is $t(g_v,h_v)$, and $(x_1(g),x_2(h))$ is the element in $\BA^*$, whose coordinate at $v$ is $(x_{1,v}(g_v),x_{2,v}(h_v))_v$.

In order to unify and ease our notations, we will re-denote, when convenient, $t(g,h)$, $j(g,h)$, in the linear case, by $t^{(1)}(g,h)$, $j^{(1)}(g,h)$. Similarly, we will re-denote, in the linear case, the stabilizer of $\psi_H$ by $D^{(1)}$, and in the metaplectic case, we will denote by $D^{(2)}(F_v)$ the image of the homorphism \eqref{1.9}. When there is no risk of confusion, we will simply denote $t(g,h)=t^{(\epsilon)}(g,h)$, $D=D^{(\epsilon)}$, $\epsilon=1,2$.\\  
\\
{\bf 2. Eisenstein series}\\

Let $\tau$ be an irreducible, automorphic, cuspidal representation
of $\GL_n(\BA)$. Assume that $\tau$ has a unitary central character $\omega_\tau$.
Let $\Delta(\tau, m+i)$ be the Speh representation of
$\GL_{n(m+i)}(\BA)$. This representation is spanned by the (multi-)
residues of Eisenstein series corresponding to the parabolic induction
from
$$
\tau|\det\cdot|^{s_1}\times
\tau|\det\cdot|^{s_2}\times\cdots\times
\tau|\det\cdot|^{s_{m+i}},
$$
at the point
$$
(\frac{m+i-1}{2},\frac{m+i-3}{2},...,-\frac{m+i-1}{2}).
$$
In general, for a positive integer $a$, divisible by $n$, $a=n\ell$, with $\ell>1$, $\Delta(\tau,\ell)$ is an irreducible, unitary, automorphic representation of $\GL_a(\BA)$, which appears with multiplicity one in the discrete, non-cuspidal part of $L^2(\GL_a(F)\backslash \GL_a(\BA))_\chi$, where $\chi=\omega_\tau^\ell$. As we run over the positive integers $n$ and $\ell>1$, such that $a=n\ell$, and $\tau$, as above, with $\chi=\omega_\tau^\ell$, the sum of all $\Delta(\tau,\ell)$ exhausts the non-cuspidal part of $L^2(\GL_a(F)\backslash \GL_a(\BA))_\chi$. This is proved in \cite{MW89}.

Consider the Eisenstein series on $H(\BA)$, $E(f_{\Delta(\tau,
m+i)\gamma^{(\epsilon)}_\psi,s})$, corresponding to a smooth, holomorphic section
$f_{\Delta(\tau,m+i)\gamma^{(\epsilon)}_\psi,s}$ of the parabolic induction
$$
\rho_{\Delta(\tau,m+i)\gamma^{(\epsilon)}_\psi,s}=\Ind_{Q^{(\epsilon)}_{n(m+i)}(\BA)}^{H(\BA)}\Delta(\tau,
m+i)\gamma^{(\epsilon)}_\psi|\det\cdot|^s.
$$
Here, $\epsilon=1,2$, according to whether $H$ is linear, or metaplectic. When $\epsilon=1$, $H$ is linear and $\gamma^{(1)}_\psi=1$. When $\epsilon=2$, $H(\BA)=\Sp^{(2)}_{2n(2m'+i)}(\BA)$, and $\gamma^{(2)}_\psi=\gamma_\psi\circ \det$ is the Weil factor attached to $\psi$, composed with the determinant. This is a character of the double cover of $\GL_{n(2m'+i)}(\BA)$.
In this case, $\Delta(\tau,m+i)\gamma^{(2)}_\psi$ is the representation of $\GL_{n(m+i)}^{(2)}(\BA)$, acting by right translations on the space of automorphic functions functions 
$(g,\mu)\mapsto \mu \gamma_\psi(\det(g))\xi(g)$, where $\xi$ is an automorphic form in the space of $\Delta(\tau,m+i)$. Recall that $M_{n(m+i)}^{(2)}(\BA)$ (the inverse image of the Levi part of $Q_{n(m+i)}(\BA)$) is isomorphic to $\GL^{(2)}_{n(m+i)}(\BA)$. See \eqref{1.1.1}. When $\epsilon=2$, we will re- denote $\Delta(\tau,m+i)\gamma^{(2)}_\psi$ by $\Delta(\tau,m+i)\gamma_\psi$.
We will denote the value at $h$ of our Eisenstein series by $E(f_{\Delta(\tau, m+i)\gamma^{(\epsilon)}_\psi,s},h)$.
Consider the Fourier coefficient of $E(f_{\Delta(\tau, m+i)\gamma^{(\epsilon)}_\psi,s})$
along $U_{m^{n-1}}$ with respect to the character $\psi_H$, and
view it as a function on $D(\BA)=D_{\psi_H}(\BA)$,\\
\\
$\mathcal{F}_\psi(E(f_{\Delta(\tau, m+i)\gamma^{(\epsilon)}_\psi,s}))(g,h)=$
\begin{equation}\label{1.9.1}
\int_{U_{m^{n-1}}(F)\backslash
U_{m^{n-1}}(\BA)}E(f_{\Delta(\tau,
m+i)\gamma^{(\epsilon)}_\psi,s},ut(g,h))\psi^{-1}_H(u)du,
\end{equation}
where $g\in H^{(\epsilon)}_m(\BA)$, $h\in H^{(\epsilon)}_{2ni+m}(\BA)$. Recall that in the metaplectic case, we identify $U_{m^{n-1}}(\BA)$ as a subgroup of $H(\BA)$. Since $D(\BA)$
stabilizes $\psi_H$, the function \eqref{1.9.1} is automorphic on
$D(\BA)$.

Let $\sigma$ be an irreducible, automorphic, cuspidal representation
of $H^{(\epsilon)}_m(\BA)$. In the metaplectic case, we always assume that $\sigma$ is genuine. One of the main objects of study in this paper is the
following kernel integral,
\begin{equation}\label{1.10}
\mathcal{E}(f_{\Delta(\tau,m+i)\gamma^{(\epsilon)}_\psi,s},\varphi_\sigma)(h)=
\int_{C_2^{(\epsilon)}H_m(F)\backslash
H^{(\epsilon)}_m(\BA)}\mathcal{F}_\psi(E(f_{\Delta(\tau,
m+i)\gamma^{(\epsilon)}_\psi,s}))(g,h)\varphi_\sigma(g)dg,
\end{equation}
where $\varphi_\sigma$ is in the space of $\sigma$, and $h\in
H^{(\epsilon)}_{2ni+m}(\BA)$. This is an automorphic function on
$H^{(\epsilon)}_{2ni+m}(\BA)$. Note that in the metaplectic case, \eqref{1.10} is equal to\\
\\
$\mathcal{E}(f_{\Delta(\tau,m+i)\gamma_\psi,s},\varphi_\sigma)(h)=$
\begin{equation}\label{1.10.1}
\int_{\Sp_{2m'}(F)\backslash \Sp_{2m'}(\BA)}\mathcal{F}_\psi(E(f_{\Delta(\tau,
	m+i)\gamma_\psi,s}))(\tilde{g},h)\varphi_\sigma(\tilde{g})dg.
\end{equation} 
Here, $\tilde{g}$ is any element in $\Sp_{2m'}^{(2)}(\BA)$, which projects to $g$ in $\Sp_{2m'}(\BA)$, and $h\in \Sp^{(2)}_{2(m'+ni)}(\BA)$. Note that the integration in \eqref{1.10.1} is over the linear group due to the cancellation of the cocycles in the integrand.
 
Our first main goal is to identify
$\mathcal{E}(f_{\Delta(\tau, m+i)\gamma^{(\epsilon)}_\psi,s}, \varphi_\sigma)$ as an
Eisenstein series on $H_{2ni+m}^{(\epsilon)}(\BA)$, parabolically induced from
$\Delta(\tau, i)\gamma^{(\epsilon)}_\psi|\det\cdot|^s\times \sigma^\iota$, where $\sigma^\iota$ is the following representation obtained from $\sigma$. Assume, first, that $H$ is linear. Then $\sigma^\iota$ is the following outer conjugation. Let $J_0=J_{0;m,n}$ be the following matrix. 
Assume that $m(n-1)$ is even. Then
\begin{equation}\label{1.10.1'}
J_0=\begin{pmatrix}&&I_{[\frac{m}{2}]}\\&I_{m-2[\frac{m}{2}]}\\-\delta_HI_{[\frac{m}{2}]}\end{pmatrix}.
\end{equation}
Assume that $m(n-1)$ is odd, i.e. $m=2m'-1$ is odd and $n$ is even. Then
\begin{equation}\label{1.10.1''}
J_0=\begin{pmatrix}&&-I_{m'-1}\\&1\\I_{m'-1}\end{pmatrix}.
\end{equation}
Define, for $b\in H_m(\BA)$,
\begin{equation}\label{1.10.1.1}
b^\iota=J_0^{-1}bJ_0,\quad \sigma^\iota(b)=\sigma(b^\iota).
\end{equation} 
In the metaplectic case, let $m=2m'$, and consider the matrix $J_0$ as above. Note that now,
$$
J_0=\begin{pmatrix}&I_{m'}\\I_{m'}\end{pmatrix},
$$
and consider, as in \eqref{1.10.1.1}, the automorphism of $H_m=\Sp_{2m'}$ given by $b\mapsto b^\iota=J_0^{-1}bJ_0$. We will soon lift this automorphism, over each local field, to $\Sp^{(2)}_{2m'}(F_v)$. We use the same notation, as in \eqref{1.10.1.1}, and similarly define $\sigma^\iota$. We will show in Sec. 3
that, for $Re(s)$ sufficiently large,
\begin{equation}\label{1.10.1.2}
\mathcal{E}(f_{\Delta(\tau,
m+i)\gamma^{(\epsilon)}_\psi,s},\varphi_\sigma)(h)=\sum_{\gamma\in Q_{ni}(F)\backslash
H_{2ni+m}(F)}\Lambda(f_{\Delta(\tau,
m+i)\gamma^{(\epsilon)}_\psi,s},\varphi_\sigma)(\gamma h),
\end{equation}
where\\
\\
$\Lambda(f_{\Delta(\tau, m+i)\gamma^{(\epsilon)}_\psi,s},\varphi_\sigma)(h)=$
\begin{equation}\label{1.10.1.3}
\int_{C_2^{(\epsilon)}\backslash H^{(\epsilon)}_m(\BA)}\varphi_\sigma(g)\int_{U'_{m(n-1)}(\BA)}
f^\psi_{\Delta(\tau, m+i)\gamma^{(\epsilon)}_\psi,s}(\delta_0ut(g,h))\psi^{-1}_H(u)dudg.
\end{equation}
Here, $U'_{m(n-1)}$ is a certain subgroup of $U_{m(n-1)}$,
$\delta_0$ is a certain element in $H(F)$ (written explicitly in \eqref{3.21.1} - \eqref{3.21.2}); the upper $\psi$ on the section denotes a composition of the section with a Fourier
coefficient on $\Delta(\tau,m+i)$. This Fourier coefficient is along $V_{ni,m^n}$, with respect to the character
$$
\begin{pmatrix}I_{ni}&\star&\star&\star&\cdots&\star\\&I_m&x_1&\star&\cdots&\star\\&&I_m&x_2&\cdots&\star\\&&&\ddots\\&&&&I_m&x_{n-1}\\
&&&&&I_m\end{pmatrix}\mapsto\psi(tr(x_1+x_2+\cdots+x_{n-1})).
$$
We prove in Sec. 3, 4 that if we assume further that $f_{\Delta(\tau,
	m+i)\gamma^{(\epsilon)}_\psi,s}$ is $K_{H(\BA)}$ - finite, then $\Lambda(f_{\Delta(\tau,
m+i)\gamma^{(\epsilon)}_\psi,s},\varphi_\sigma)$ is a smooth, meromorphic section of
\begin{equation}\label{1.10.2}
\Ind_{Q^{(\epsilon)}_n(\BA)}^{H^{(\epsilon)}_{2ni+m}(\BA)}\Delta(\tau, i)\gamma^{(\epsilon)}_\psi|\det\cdot|^s\times
\sigma^\iota. 
\end{equation}
Note the case $i=0$. Here, in \eqref{1.10.1.2} there is no summation. It reads 
$$
\mathcal{E}(f_{\Delta(\tau,
	m)\gamma^{(\epsilon)}_\psi,s},\varphi_\sigma)(h)=\Lambda(f_{\Delta(\tau,
	m)\gamma^{(\epsilon)}_\psi,s},\varphi_\sigma)(h)=
$$
$$
\int_{C_2^{(\epsilon)}\backslash H^{(\epsilon)}_m(\BA)}\varphi_\sigma(g)\int_{U'_{m(n-1)}(\BA)}
f^\psi_{\Delta(\tau, m)\gamma^{(\epsilon)}_\psi,s}(\delta_0ut(g,h))\psi^{-1}_H(u)dudg,
$$
and \eqref{1.10.2} says that $\Lambda(f_{\Delta(\tau,
	m)\gamma^{(\epsilon)}_\psi,s},\varphi_\sigma)$ is smooth, meromorphic and takes values in $\sigma^\iota$. 

Let us write the precise form of $\sigma^\iota$ in \eqref{1.10.2}, in the metaplectic case. For this, we first write the lift to $\Sp_m^{(2)}(\BA)$ of the automorphism $b\mapsto b^\iota$ of $\Sp_m(\BA)$, $m=2m'$. Let
$$
\bar{u}_0=\begin{pmatrix}I_m\\I_m&I_m\end{pmatrix}\in \Sp_{4m'}(F).
$$
We denote by $c=c_v$ the normalized Ranga Rao cocycle of $\Sp^{(2)}_{2m'}(F_v)$, without mentioning the rank of the corresponding symplectic group. This will always be clear from the context. For the definition of $c$, see \cite{Rao93}, Theorem 5.3. Define, for each local field $F_v$, for $(b,\epsilon)\in \Sp_m^{(2)}(F_v)$,
\begin{equation}\label{1.10.2.1}
(b,\epsilon)^\iota=(b^\iota,\epsilon c(\bar{u}_0,j(b^\iota,b))(x(b^\iota),x(b)).
\end{equation}
Here, $x$ is the Ranga Rao $x$-function on $\Sp_m(F_v)$. Note that when we write 
$$
b=\begin{pmatrix}b_1&b_2\\b_3&b_4\end{pmatrix},
$$
with $b_i$ being $m'\times m'$ blocks, then
$$
j(b^\iota,b)=\begin{pmatrix}b_4&&&b_3\\&b_1&b_2\\&b_3&b_4\\b_2&&&b_1\end{pmatrix}.
$$
It is an exercise to check that \eqref{1.10.2.1} defines an automorphism (of order two), and that, for $v\notin S_0$, we have, for each $r\in \Sp_{2m'}(\mathcal{O}_v)$,
\begin{equation}\label{1.10.2.2}
\zeta_v(r)^\iota=\zeta_v(r^\iota).
\end{equation}
See \eqref{1.1.0}. This defines the automorphism of $\widetilde{\Sp}_{2m'}(\BA)$
\begin{equation}\label{1.10.2.3}
(\Pi'_v(g_v,\epsilon_v))^\iota=\Pi'_v((g_v,\epsilon_v)^\iota).
\end{equation}
Since it is trivial on $C'$, it defines an automorphisn of $\Sp^{(2)}(\BA)$ by
\begin{equation}\label{1.10.2.4}
(C'\Pi'_v(g_v,\epsilon_v))^\iota=C'\Pi'_v((g_v,\epsilon_v)^\iota).
\end{equation}
We leave it to the reader to show that, for $\gamma\in \Sp_{2m'}(F)$, embedded in $\Sp^{(2)}_{2m'}(\BA)$, as in \eqref{1.1.0.b},
\begin{equation}\label{1.10.2.5}
(C'\Pi'_v(\gamma,1))^\iota=C'\Pi'_v(\gamma^\iota,1),
\end{equation}
which in our shorthand notation, we also rewrite as
\begin{equation}\label{1.10.2.6}
(\gamma,1)^\iota=(\gamma^\iota,1).
\end{equation}
 We define, for $g\in \Sp^{(2)}_{2m'}(\BA)$
\begin{equation}\label{1.10.2.7}
\sigma^\iota(g))=\sigma(g^\iota). 
\end{equation}
The inducing representation in \eqref{1.10.2} in the metaplectic case is 
\begin{equation}\label{1.10.3}
p((\begin{pmatrix}a\\&b\\&&a^*\end{pmatrix},\bar{\alpha}))\mapsto 
\gamma_\psi(\det(a))(\det(a),x(b))|\det(a)|^s\Delta(\tau,i)(a)\otimes \sigma^\iota(p(b,\bar{\alpha})),
\end{equation}
for $a\in \GL_{ni}(\BA)$, $b\in \Sp_{2m'}(\BA)$. Recall that $\bar{\alpha}$ denotes a sequence indexed by the places $v$ with coordinates $\pm 1$.

Our identity can be formulated in terms of normalized
Eisenstein series. Let $S$ be a finite set of places of $F$,
containing the infinite ones, outside which $\sigma$, $\tau$
and $\psi$ are unramified. Assume that our section is decomposable,
unramified outside $S$, and normalized in a way which we don't
specify now (see \eqref{4.2.1*}). Let us multiply our given Eisenstein series on $H(\BA)$
by its normalizing factor outside $S$, $d_\tau^{H,S}(s)$. We list them for convenience.
\begin{equation}\label{1.10.4}
d_\tau^{\Sp_{2n(2j+1)},S}(s)=L^S(\tau,s+j+1)\prod_{k=1}^{j+1}L^S(\tau,\wedge^2,2s+2k-1)\prod_{k=1}^jL^S(\tau,sym^2,2s+2k);
\end{equation}
\begin{equation}\label{1.10.5}
d_\tau^{\Sp_{4nj},S}(s)=L^S(\tau,s+j+\frac{1}{2})\prod_{k=1}^{j}L^S(\tau,\wedge^2,2s+2k)L^S(\tau,sym^2,2s+2k-1);
\end{equation}
\begin{equation}\label{1.10.6}
d_\tau^{\Sp^{(2)}_{2n(2j+1)},S}(s)=\prod_{k=1}^{j}L^S(\tau,\wedge^2,2s+2k)\prod_{k=1}^{j+1}L^S(\tau,sym^2,2s+2k-1);
\end{equation}
\begin{equation}\label{1.10.7}
d_\tau^{\Sp^{(2)}_{4nj},S}(s)=\prod_{k=1}^{j}L^S(\tau,\wedge^2,2s+2k-1)L^S(\tau,sym^2,2s+2k);
\end{equation}
\begin{equation}\label{1.10.8}
d_\tau^{\SO_{2n(2j+1)},S}(s)=\prod_{k=1}^{j+1}L^S(\tau,\wedge^2,2s+2k-1)\prod_{k=1}^{j}L^S(\tau,sym^2,2s+2k);
\end{equation}
\begin{equation}\label{1.10.9}
d_\tau^{\SO_{4nj},S}(s)=\prod_{k=1}^jL^S(\tau,\wedge^2,2s+2k)L^S(\tau,sym^2,2s+2k-1).
\end{equation}
Denote
\begin{equation}\label{1.10.9*}
E_S^*(f_{\Delta(\tau,
m+i)\gamma_\psi^{(\epsilon)},s},h)=d_\tau^{H,S}(s)E(f_{\Delta(\tau, m+i)\gamma_\psi^{(\epsilon)},s},h).
\end{equation}
This is a normalized (outside $S$) Eisenstein series.\\ 
Let $d_{\sigma,\tau}^{H_{2ni+m}^{(\epsilon),S}}(s)$ be the normalizing factor, outside $S$, corresponding to an Eisenstein series attached to the global induced representation \eqref{1.10.2}. Explicitly,
\begin{equation}\label{1.10.9.1}
d_{\sigma,\tau}^{\Sp_{2n(2j+1)+m},S}(s)=L^S(\sigma\times\tau,s+j+1)\prod_{k=1}^{j+1}L^S(\tau,\wedge^2,2s+2k-1)\prod_{k=1}^jL^S(\tau,sym^2,2s+2k);
\end{equation} 
\begin{equation}\label{1.10.9.2}
d_{\sigma,\tau}^{\Sp_{4nj+m},S}(s)=L^S(\sigma\times\tau,s+j+\frac{1}{2})\prod_{k=1}^jL^S(\tau,\wedge^2,2s+2k)L^S(\tau,sym^2,2s+2k-1);
\end{equation}
\begin{equation}\label{1.10.9.3}
d_{\sigma,\tau}^{\Sp^{(2)}_{2n(2j+1)+m},S}(s)=L_\psi^S(\sigma\times\tau,s+j+1)\prod_{k=1}^jL^S(\tau,\wedge^2,2s+2k)\prod_{k=1}^{j+1}L^S(\tau,sym^2,2s+2k-1);
\end{equation}
\begin{equation}\label{1.10.9.4}
d_{\sigma,\tau}^{\Sp^{(2)}_{4nj+m},S}(s)=L_\psi^S(\sigma\times\tau,s+j+\frac{1}{2})\prod_{k=1}^jL^S(\tau,\wedge^2,2s+2k-1)L^S(\tau,sym^2,2s+2k);
\end{equation}
\begin{equation}\label{1.10.9.5}
d_{\sigma,\tau}^{\SO_{2n(2j+1)+2m'},S}(s)=L^S(\sigma\times\tau,s+j+1)\prod_{k=1}^{j+1}L^S(\tau,\wedge^2,2s+2k-1)\prod_{k=1}^jL^S(\tau,sym^2,2s+2k);
\end{equation}
\begin{equation}\label{1.10.9.6}
d_{\sigma,\tau}^{\SO_{4nj+2m'},S}(s)=L^S(\sigma\times\tau,s+j+\frac{1}{2})\prod_{k=1}^jL^S(\tau,\wedge^2,2s+2k)L^S(\tau,sym^2,2s+2k-1);
\end{equation}
\begin{equation}\label{1.10.9.7}
d_{\sigma,\tau}^{\SO_{2n(2j+1)+2m'-1},S}(s)=L^S(\sigma\times\tau,s+j+1)\prod_{k=1}^jL^S(\tau,\wedge^2,2s+2k)\prod_{k=1}^{j+1}L^S(\tau,sym^2,2s+2k-1);
\end{equation}
\begin{equation}\label{1.10.9.8}
d_{\sigma,\tau}^{\SO_{4nj+2m'-1},S}(s)=L^S(\sigma\times\tau,s+j+\frac{1}{2})\prod_{k=1}^jL^S(\tau,\wedge^2,2s+2k-1)L^S(\tau,sym^2,2s+2k).
\end{equation}
Recall that in the symplectic, or metaplectic cases $m$ is even. Also, in the metaplectic case, $\sigma$ is a genuine representation of $\Sp^{(2)}_m(\BA)$, and the $L$-function of $\sigma$ twisted by $\tau$ depends on $\psi$. This why we denote it  $L_\psi^S(\sigma\times \tau,s)$. We will denote in general  $L_{\epsilon,\psi}^S(\sigma\times\tau,s)$, so that when $\epsilon=1$, this is just  $L^S(\sigma\times\tau,s)$, and when $\epsilon=2$, this is  $L_\psi^S(\sigma\times \tau,s)$. Let
$$
\lambda_S(f_{\Delta(\tau,
m+i)\gamma_\psi^{(\epsilon)},s},\varphi_\sigma)=\frac{d_\tau^{H,S}(s)}{d_{\sigma,\tau}^{H_{2ni+m}^{(\epsilon),S}}(s)}\Lambda(f_{\Delta(\tau,
m+i)\gamma_\psi^{(\epsilon)},s},\varphi_\sigma).
$$
We prove in Sec. 4 that, for $f_{\Delta(\tau,
	m+i)\gamma_\psi^{(\epsilon)},s}$ $K_{H(\BA)}$-finite (decomposable and normalized) and $\varphi_\sigma$ decomposable, the section  $\lambda_S(f_{\Delta(\tau,
m+i)\gamma_\psi^{(\epsilon)},s},\varphi_\sigma)$ is decomposable, and it is normalized,
so that its component at $v$ outside $S$ is such that its value at 1
is a pre-chosen unramified vector in $\Delta(\tau_v,i)\gamma_\psi^{(\epsilon)}\otimes\sigma_v$.
Consider then \eqref{1.10} in a normalized form
\begin{equation}\label{1.10.10}
\mathcal{E}^*_S(f_{\Delta(\tau, m+i)\gamma_\psi^{(\epsilon)},s},\varphi_\sigma)(h)=
\int_{C_2^{(\epsilon)}H_m(F)\backslash
H^{(\epsilon)}_m(\BA)}\mathcal{F}_\psi(E^*_S(f_{\Delta(\tau,
m+i)\gamma_\psi^{(\epsilon)},s}))(g,h)\varphi_\sigma(g)dg.
\end{equation}
Then \eqref{1.10.1.2} says that for $Re(s)$ sufficiently large,\\
\\
$\mathcal{E}^*_S(f_{\Delta(\tau, m+i)\gamma_\psi^{(\epsilon)},s},\varphi_\sigma)(h)=$
\begin{equation}\label{1.10.11}
d_{\sigma,\tau}^{H_{2ni+m}^{(\epsilon),S}}(s)\sum_{h'\in
Q_n(F)\backslash H_{2ni+m}(F)}\lambda_S(f_{\Delta(\tau,
m+i)\gamma_\psi^{(\epsilon)},s},\varphi_\sigma)(h'h).
\end{equation}
The right hand side of \eqref{1.10.11} is the normalized Eisenstein
series (outside $S$)\\
$E^*_S(\lambda_S(f_{\Delta(\tau, m+i)\gamma_\psi^{(\epsilon)},s},\varphi_\sigma))$ on
$H^{(\epsilon)}_{2ni+m}(\BA)$, corresponding to the normalized section
$\lambda_S(f_{\Delta(\tau, m+i)\gamma_\psi^{(\epsilon)},s},\varphi_\sigma)$ of
$\Ind_{Q^{(\epsilon)}_n(\BA)}^{H^{(\epsilon)}_{2ni+m}(\BA)}\Delta(\tau,i)\gamma_\psi^{(\epsilon)}|\det\cdot|^s\times
\sigma^\iota$. Our first main theorem then is the identity
\begin{thm}\label{thm 1.1}
Let $f_{\Delta(\tau,m+i)\gamma_\psi^{(\epsilon)},s}$ be a smooth, holomorphic, $K_{H(\BA)}$-finite section of $\rho_{\Delta(\tau,
	m+i)\gamma_\psi^{(\epsilon)},s}$, and let $\varphi_\sigma$ be a cusp form in the space of $\sigma$. Then we have the identity (of meromorphic functions)
	
$$
\mathcal{E}^*_S(f_{\Delta(\tau,
m+i)\gamma_\psi^{(\epsilon)},s},\varphi_\sigma)=E^*_S(\lambda_S(f_{\Delta(\tau,
m+i)\gamma_\psi^{(\epsilon)},s},\varphi_\sigma)).
$$
The left hand side is given by the kernel integral \eqref{1.10},
normalized by $d_\tau^{H,S}(s)$. The right hand side is the
normalized Eisenstein series on $H^{(\epsilon)}_{2ni+m}(\BA)$ corresponding to
the section $\lambda_S(f_{\Delta(\tau, m+i)\gamma_\psi^{(\epsilon)},s},\varphi_\sigma)$ of
$\Ind_{Q^{(\epsilon)}_n(\BA)}^{H^{(\epsilon)}_{2ni+m}(\BA)}\Delta(\tau,i)\gamma_\psi^{(\epsilon)}|\det\cdot|^s\times
\sigma^\iota$. This section is spherical and normalized outside $S$.
\end{thm}
The proof of this theorem will occupy the next three sections. 
As we remarked in the introduction, we allow the case $n=1$, where $\tau$ is simply a character of
$F^*\backslash \BA^*$, $\Delta(\tau, m+i)=\tau\circ
\det_{\GL_{m+i}}$ and $\rho_{\Delta(\tau,
m+i)\gamma_\psi^{(\epsilon)},s}=\Ind_{Q^{(\epsilon)}_{(m+i)}(\BA)}^{H^{(\epsilon)}_{2(m+i)}(\BA)}\tau(\det\cdot)\gamma_\psi^{(\epsilon)}
|\det\cdot|^s$. Since now $U_{m^{n-1}}$ is trivial, the
function \eqref{1.9.1} is the restriction of the Eisenstein series
$E(f_{(\tau\circ \det)\gamma_\psi^{(\epsilon)},s},\cdot)$ to $H^{(\epsilon)}_m(\BA)\times
H^{(\epsilon)}_{2i+m}(\BA)$. The Eisenstein
series obtained from \eqref{1.10} corresponds to a section of the
parabolic induction from $(\tau\circ \det_{\GL_i})\gamma_\psi^{(\epsilon)} |\det|^s\times \sigma$ to
$H^{(\epsilon)}_{2i+m}(\BA)$. This section is the analytic continuation of the
following integral, which converges absolutely for $\Re(s)$
sufficiently large,
\begin{equation}\label{1.11}
\Lambda(f_{(\tau\circ\det)\gamma_\psi^{(\epsilon)},s},\varphi_\sigma)(h)=\int_{H_m(\BA)}f_{(\tau\circ\det)\gamma_\psi^{(\epsilon)},s}(t(g,h))\varphi_\sigma(g)dg.
\end{equation}
Note that when $n=1$, $t(g,h)=j(g,h)$. The integral \eqref{1.11} was considered in
\cite{GPSR97}, Sec. 1, for (arbitrary) orthogonal groups and $i=1$ and generalized by Moeglin in \cite{M97} for any $i$ and $H$ linear.

Our second main theorem states roughly that the "descent" of the Eisenstein series parabolically induced from $\Delta(\tau,i+1)|\det\cdot |^s$ is an Eisenstein series parabolically induced from $\Delta(\tau,i)|\det\cdot |^s$. More precisely, let $H$ be one of the groups $\SO_{2n(i+1)}$, $\SO_{2n(i+1)+1}$, $\Sp_{2n(i+1)}$, $\Sp^{(2)}_{2n(i+1)}$. Consider the Eisenstein series on $H(\BA)$, $E(f_{\Delta(\tau,i+1)\gamma_\psi^{(\epsilon)},s})$. Let $j_0=0$, when $H$ is odd orthogonal, and $j_0=1$, when $H$ is even orthogonal, symplectic, or metaplectic. Consider the Fourier coefficient along $U_{1^{n-j_0}}$ with respect to the character (on adele points) 
\begin{equation}\label{1.12}
\psi_{n-j_0}(\begin{pmatrix}z&x&y\\&I_{2ni+1+j_0}&x'\\&&z^*\end{pmatrix})=\psi(\sum_{r=1}^{n-j_0-1}z_{r,r+1})\psi(x_{n-j_0}\cdot e_0).
\end{equation}
Here, $z\in Z_{n-j_0}(\BA)$. (Recall that $Z_k=V_{1^k}$ denotes the standard maximal unipotent subgroup of $\GL_k$); $x_{n-j_0}$ is the last row of $x$ and $e_0$ is the following column vector in $F^{2ni+1+j_0}$.
When $H$ is even orthogonal,
$$
e_0=\begin{pmatrix}0_{ni}\\1\\ \frac{1}{2}\\0_{ni}\end{pmatrix}.
$$
When $H$ is odd orthogonal,
$$
e_0=\begin{pmatrix}0_{ni}\\1\\0_{ni}\end{pmatrix}.
$$
When $H$ is symplectic, or metaplectic,
$$
e_0=\begin{pmatrix}1\\0_{2ni+1}\end{pmatrix}.
$$
Assume that $H$ is linear. The character $\psi_{n-j_0}$ is stabilized by the adele points of the following subgroup of elements:
\begin{equation}\label{1.12.1}
t^{(1)}(h)=\begin{pmatrix}I_{n-j_0}\\&h\\&&I_{n-j_0}\end{pmatrix}\in H_{2n(i+1)+1-j_0},\ he_0=e_0.
\end{equation}
In case $H$ is orthogonal, this is isomorphic to $H_{2ni+j_0}$. In case $H$ is symplectic, we get the semi-direct product of $H_{2ni}=\Sp_{2ni}$ and the Heisenberg group in $2ni+1$ variables $\mathcal{H}_{2ni+1}$. It is realized as the subgroup pf the following elements,
\begin{equation}\label{1.12.2}
t^{(1)}(\begin{pmatrix}1\\&g\\&&1\end{pmatrix})t^{(1)}(\begin{pmatrix}1&x&c\\&I_{2ni}&x'\\&&1\end{pmatrix})\in \Sp_{2n(i+1)}.
\end{equation}
It will be convenient to abuse the notation a little, and re-denote the first factor in \eqref{1.12.2} by $t^{(1)}(g)$. 
The isomorphism with $\mathcal{H}_{2ni+1}$ is given by 
\begin{equation}\label{1.12.2.1}
(x;c)\mapsto t^{(1)}(\begin{pmatrix}1&x&c\\&I_{2ni}&x'\\&&1\end{pmatrix}):=t^{(1)}((x,c)).
\end{equation}
We take $\mathcal{H}_{2ni+1}$ as the group of pairs $(x,c)$, as in \eqref{1.12.2.1} with multiplication defined by
$$
(x,e)\cdot (y,z)=(x+y,e+z+x\begin{pmatrix}&w_{ni}\\-w_{ni}\end{pmatrix}{}^ty).
$$ 
When $H$ is metaplectic, 
in place of \eqref{1.12.2}, we will consider the following subgroup, which is isomorphic to the semi-direct product of $H^{(2)}_{2ni}=\Sp^{(2)}_{2ni}$ and the Heisenberg group $\mathcal{H}_{2ni+1}$. It is realized as follows. Over a local field $F_v$, with $g\in \Sp_{2ni}(F_v)$, $\mu=\pm 1$, $(x,c)\in \mathcal{H}_{2ni+1}(F_v)$,
\begin{equation}\label{1.12.3}
t^{(2)}((g,\mu))t^{(2)}(\begin{pmatrix}1&x&c\\&I_{2ni}&x'\\&&1\end{pmatrix},1):=(t^{(1)}(\begin{pmatrix}1&x&c\\&g&gx'\\&&1\end{pmatrix}),\mu)\in \Sp^{(2)}_{2n(i+1)}(F_v).
\end{equation} 
Over $\BA$, for $(g,\bar{\mu}\in \widetilde{\Sp}_{2ni}(\BA)$, $(x,c)\in \mathcal{H}_{2ni+1}(\BA)$,
\begin{equation}\label{1.12.3'}
t^{(2)}(p((g,\bar{\mu})))t^{(2)}(\begin{pmatrix}1&x&c\\&I_{2ni}&x'\\&&1\end{pmatrix},1):=p((t^{(1)}(\begin{pmatrix}1&x&c\\&g&gx'\\&&1\end{pmatrix}),\bar{\mu}))\in \Sp^{(2)}_{2n(i+1)}(\BA).
\end{equation} 

Assume that $H$ is symplectic. We have the following projection $\beta$ from $U_{1^n}$ onto the Heisenberg group $\mathcal{H}_{2ni+1}$. It is given by
\begin{equation}\label{1.12.4}
\beta(\begin{pmatrix}z&x&y\\&I_{2ni}&x'\\&&z^*\end{pmatrix})=(x_n,;y_{n,1}),
\end{equation}
where $z\in Z_n$, and $x_n$ is the $n$-th row of $x$. Note that  $U_{1^n}=U_{1^{n-1}}\rtimes t^{(1)}(\mathcal{H}_{2ni+1})$. In this case, we extend the character $\psi_{n-1}$ from $U_{1^{n-1}}(\BA)$ to $U_{1^n}(\BA)$ by making it trivial on $t^{(1)}((\mathcal{H}_{2ni+1}(\BA))$. We keep denoting the extension by $\psi_{n-1}$.
As before, it will be convenient to use the notation $t$ for either one of the embeddings $t^{(1)}$, or $t^{(2)}$ when there is no confusion.
We remark that the last embedding $t$ coincides with the previous embedding (\eqref{1.6}, or \eqref{1.9}) only in the case of \eqref{1.6} for orthogonal groups, with $m=1$, so that there, we have $H_1=\SO_1$ is the trivial group, and we get the embedding $t(1,h)$ of $\SO_1\times \SO_{2ni+1}$ inside $\SO_{2n(i+1)}$. 

 Let $H$ be orthogonal.  Define, for $h\in H_{2ni+j_0}(\BA)$,
\begin{equation}\label{1.13}
\mathcal{D}_{\psi,ni}(E(f_{\Delta(\tau,i+1),s}))(h)=\int_{U_{1^{n-j_0}}(F)\backslash U_{1^{n-j_0}}(\BA)}E(f_{\Delta(\tau,i+1),s},ut(h))\psi_{n-j_0}^{-1}(u)du.
\end{equation}
This is an automorphic function on $H_{2ni+j_0}(\BA)$. It is defined by applying a Bessel coefficient to $E(f_{\Delta(\tau,i+1),s})$. The nilpotent orbit corresponding to this Fourier coefficient is associated to the partition $(2n-1,1^{2ni+1})$, when $H$ is even orthogonal, and to the partition $(2n+1,1^{2ni})$, when $H$ is odd orthogonal. For the definition of these Bessel coefficients, see, for example, \cite{GRS11}, Sec. 3.1, where they are called Gelfand-Graev coefficients (see \cite{K85}). These coefficients are also referred to as Bessel coefficients (see \cite{NPS73}). We have chosen the second option, as it seems that this is commonly used nowadays, for example in the Gan-Gross-Prasad conjectures.

Assume that $H$ is symplectic, or metaplectic. In this case, we consider the Fourier-Jacobi coefficient on $E(f_{\Delta(\tau,i+1)\gamma_\psi^{(\epsilon)},s})$, corresponding to the partition\\ 
$(2n,1^{2ni})$. See \cite{GRS11}, Sec. 3.2, for the definition of Fourier-jacobi coefficients. Thus, for $H=\Sp_{2n(i+1)}$, we define, for $\tilde{h}\in \Sp_{2ni}^{(2)}(\BA)$, projecting to $h\in Sp_{2ni}(\BA)$,
\\
\\
$\mathcal{D}^\phi_{\psi,ni}(E(f_{\Delta(\tau,i+1),s}))(\tilde{h})$
\begin{equation}\label{1.14}
=\int_{U_{1^n}(F)\backslash U_{1^n}(\BA)}E(f_{\Delta(\tau,i+1),s},ut(h))\psi_{n-1}^{-1}(u)\theta^\phi_{\psi^{-1}}(\beta(u)\tilde{h})du,
\end{equation}
where $\theta^\phi_{\psi^{-1}}$ is the theta series corresponding to $\phi\in \mathcal{S}(\BA^{ni})$. For $H=\Sp^{(2)}_{2n(i+1)}$, we define, for $h\in \Sp_{2ni}(\BA)$, which is the projection of $\tilde{h}\in Sp^{(2)}_{2ni}(\BA)$,
\\
\\
$\mathcal{D}^\phi_{\psi,ni}(E(f_{\Delta(\tau,i+1)\gamma_\psi,s}))(h)$
\begin{equation}\label{1.15}
=\int_{U_{1^n}(F)\backslash U_{1^n}(\BA)}E(f_{\Delta(\tau,i+1)\gamma_\psi,s},ut(\tilde{h}))\psi_{n-1}^{-1}(u)\theta^\phi_{\psi^{-1}}(\beta(u)\tilde{h})du.
\end{equation}
(In the metaplectic case, we identify $U_{1^n}$ withe $U_{1^n}\times 1$.) We regard \\ 
$\mathcal{D}^\phi_{\psi,ni}(E(f_{\Delta(\tau,i+1)\gamma^{(\epsilon)}_\psi,s}))$
as an automorphic function on $\Sp^{(2)}_{2ni}(\BA)$, or $\Sp_{2ni}(\BA)$, respectively. Recall that, for $\phi\in \mathcal{S}(\BA^{ni})$, $(x,c)\in \mathcal{H}_{2ni+1}(\BA)$, and $p((h,\bar{\mu}))\in \Sp^{(2)}_{2ni}(\BA)$,
\begin{equation}\label{1.15.1}
\theta^\phi_{\psi^{-1}}((x,c)p((h,\bar{\mu})))=\sum_{e\in F^{ni}}\omega_{\psi^{-1}}((x,c)p((h,\bar{\mu}))\phi(e),
\end{equation}
where $\omega_{\psi^{-1}}$ is the Weil representation of $\mathcal{H}_{2ni+1}(\BA)\rtimes \Sp^{(2)}_{2ni}(\BA)$. We will need the following formulas, which we copy from \cite{GRS11}, Sec. 1.2. For $(x;c)$, as above, with $c\in \BA$, write $x=(x_1,x_2)$, where $x_1,x_2\in \BA^{ni}$. Then, for $y\in \BA^{ni}$,
\begin{equation}\label{1.15.2}
\begin{aligned}
\omega_{\psi^{-1}}(0,0;c)\phi(y)&=\psi^{-1}(c)\phi(y)\\
\omega_{\psi^{-1}}(x_1,0;0)\phi(y)&=\phi(y+x_1)\\
\omega_{\psi^{-1}}(0,x_2;0)\phi(y)&=\psi^{-1}(2yw_{ni}{}^tx_2)\phi(y).
\end{aligned}
\end{equation}
For $g\in \Sp_{2ni}(\BA)$, $z\in M_{ni}(\BA)$, such that $w_{ni}z$ is symmetric, and $\epsilon=\pm 1$, using the conventions explained right after \eqref{1.1.1'},
\begin{equation}\label{1.15.3}
\begin{aligned}
\omega_{\psi^{-1}}((\begin{pmatrix}g\\&g^*\end{pmatrix},\epsilon))\phi(y)&=\epsilon \gamma_{\psi^{-1}}(\det g)|\det g|^{\frac{1}{2}}\phi(y\cdot g)\\
\omega_{\psi^{-1}}((\begin{pmatrix}I_{ni}&z\\&I_{ni}\end{pmatrix},\epsilon))\phi(y)&=\epsilon \psi^{-1}(yzw_{ni}{}^ty).
\end{aligned}
\end{equation}
 
When $H$ is even orthogonal, the fact that \eqref{1.13} is an Eisenstein series on $\SO_{2ni+1}(\BA)$, parabolically induced from $\Delta(\tau,i)|\det\cdot|^s$ is a special case of \eqref{1.10.1.2} and Theorem \ref{thm 1.1}. We simply take there $m=1$. We write this down in Sec. 5. The remaining work on our second main result is summarized in the following two theorems and appears in Sec. 6-8.

	\begin{thm}\label{thm 1.3}
		Assume that $H$ is odd orthogonal. Then $\mathcal{D}_{\psi,ni}(E(f_{\Delta(\tau,i+1),s}))$ is a sum of two Eisenstein series on  $SO_{2ni}(\BA)$. The corresponding sections\\ 
		$\Lambda^\pm(f_{\Delta(\tau,i+1),s})$ are given by explicit unipotent adelic integrations of $f_{\Delta(\tau,i+1),s}$, similar to \eqref{1.10.1.3} (without the $dg$-integration). They both define 
		smooth meromorphic sections of $\Ind_{Q_{ni}(\BA)}^{\SO_{2ni}(\BA)}\Delta(\tau,i)
		|\det\cdot|^s$.	Moreover, $\mathcal{D}_{\psi,ni}(E^*_S(f_{\Delta(\tau,i+1),s}))$ is the sum of the two normalized (outside $S$) Eisenstein series on
		$\SO_{2ni}(\BA)$, corresponding to the above two sections.
		\end{thm}
	
		\begin{thm}\label{thm 1.4}
		Assume that $H$ is symplectic (resp. metaplectic). Then\\ $\mathcal{D}^\phi_{\psi,ni}(E(f_{\Delta(\tau,i+1)\gamma_\psi^{(\epsilon)},s}))$ is an Eisenstein series on  $Sp^{(2)}_{2ni}(\BA)$ (resp. $\Sp_{2ni}(\BA)$). The corresponding section $\Lambda(f_{\Delta(\tau,i+1)\gamma_\psi^{(\epsilon)},s},\phi)$ is given by an explicit unipotent adelic integration of $f_{\Delta(\tau,i+1)\gamma_\psi^{(\epsilon)},s}$. It defines a smooth meromorphic section of\\ $\Ind_{Q^{(2)}_{ni}(\BA)}^{\Sp^{(2)}_{2ni}(\BA)}\Delta(\tau,i)\gamma_{\psi^{-1}}
		|\det\cdot|^s$ (resp. $\Ind_{Q_{ni}(\BA)}^{\Sp_{2ni}(\BA)}\Delta(\tau,i)
		|\det\cdot|^s$ ).	Moreover,\\ 
		$\mathcal{D}_{\psi,ni}^\phi(E^*_S(f_{\Delta(\tau,i+1)\gamma_\psi^{(\epsilon)},s}))$ is the normalized (outside $S$) Eisenstein series on
		$\Sp^{(2)}_{2ni}(\BA)$ (resp. $\Sp_{2ni}(\BA)$), corresponding to the section $\Lambda(f_{\Delta(\tau,i+1)\gamma_\psi^{(\epsilon)},s},\phi)$.
	\end{thm}

As we remarked in the introduction, we allow the case $n=1$, where $\tau$ is a character of
$F^*\backslash \BA^*$, $\Delta(\tau, i+1)=\tau\circ
\det_{\GL_{i+1}}$. Note that when $H=\SO_{2i+2}$, \eqref{1.13} is simply the restriction from $\SO_{2i+2}(\BA)$ to $\SO_{2i+1}(\BA)$ and then $\mathcal{D}_{\psi,i}(E(f_{\tau\circ
	\det_{\GL_{i+1}},s}))$
is an Eisenstein series on $\SO_{2i+1}(\BA)$, corresponding to a section (explicit) of the parabolic induction from $\tau\circ
\det_{\GL_i}|\det\cdot|^s$. When $n=1$, Theorem \ref{thm 1.4} for $H$ symplectic is a special case of a theorem of Ikeda. See \cite{I94}.

In the sequel, we will need the following proposition. Recall again the notion of
Fourier coefficients corresponding to nilpotent orbits.
\begin{prop}\label{prop 1.2}
Let $\mathcal{O}$ be a nilpotent orbit in $\mathfrak{gl}_{nj}$,
corresponding to a partition $\underline{P}$ of $nj$. Assume
that $\Delta(\tau, j)$ admits a nontrivial Fourier
coefficient with respect to $\mathcal{O}$. Then
$$
\underline{P}\leq (n^j).
$$
Moreover, let $\mathcal{O}(\Delta(\tau,j))$ denote the set of
maximal nilpotent orbits, whose corresponding Fourier coefficients
are supported by $\Delta(\tau, j)$. Then
$$
\mathcal{O}(\Delta(\tau, j))=(n^j).
$$
\end{prop}
This is Prop. 5.3 in \cite{G06}. See \cite{JL13} for a detailed
proof.

\section{Analysis of the Fourier coefficient $\mathcal{F}_\psi(E(f_{\Delta(\tau,
m+i)\gamma_\psi^{(\epsilon)},s},\cdot))(g,h)$}

In this section, we unfold the Fourier coefficient
$\mathcal{F}_\psi(E(f_{\Delta(\tau, m+i)\gamma_\psi^{(\epsilon)},s},\cdot))(g,h)$. We
will analyze the contributions to the Fourier coefficient of the
various double cosets in $Q_{n(m+i)}(F)\backslash
H(F)/D(F)U_{m^{n-1}}(F)$, and we will show that except the open
double coset all others contribute zero.

Assume that $\Re(s)$ is sufficiently large. Then, for $x\in H(\BA)$, our Eisenstein series is given by the following absolutely convergent series
$$
E(f_{\Delta(\tau, m+i)\gamma_\psi^{(\epsilon)},s},x)=\sum_{\gamma\in
Q_{n(m+i)}(F)\backslash H(F)}f_{\Delta(\tau, m+i)\gamma_\psi^{(\epsilon)},s}(\gamma x).
$$
Recall that in the metaplectic case, we identify $\Sp_{2n(m+i)}(F)$ as a subgroup of $\Sp_{2n(2m'+i)}^{(2)}(\BA)$. See \eqref{1.1.0.b}. We
factor the summation modulo $Q_{m(n-1)}(F)$ from the right.
\begin{equation}\label {2.1}
E(f_{\Delta(\tau, m+i)\gamma_\psi^{(\epsilon)},s},x)=\sum_{\alpha\in
Q_{n(m+i)}(F)\backslash H(F)/Q_{m(n-1)}(F)}\sum_\gamma f_{\Delta(\tau,
m+i)\gamma_\psi^{(\epsilon)},s}(\alpha\gamma x),
\end{equation}
where the second summation is over $\gamma\in (Q_{m(n-1)}(F)\cap
\alpha^{-1}Q_{n(m+i)}(F)\alpha)\backslash Q_{m(n-1)}(F)$. The
representatives $\alpha$ in \eqref{2.1} are described in
\cite{GRS11}, Sec. 4.2. They are parameterized by integers $0\leq
r\leq m(n-1)$,
\begin{equation}\label{2.2}
\alpha_r=\begin{pmatrix}I_r\\&\alpha'_r\\&&I_r\end{pmatrix},
\end{equation}
where
$$
\alpha'_r=\begin{pmatrix}0&I_{m+ni}&0&0\\0&0&0&I_{m(n-1)-r}\\
\delta_H
I_{m(n-1)-r}&0&0&0\\0&0&I_{m+ni}&0\end{pmatrix}\omega_H^{m(n-1)-r}.
$$
Here, $\omega_H=I$, unless $H$ is orthogonal, where we need its presence to guarrantee that the determinant of $\alpha_r$ is $1$. We will choose $\omega_H$ (when $H$ is orthogonal), as follows. When $m$ is even,
$$
\omega_H=diag (I_{(m+i)n-1},w_2,I_{(m+i)n-1}).
$$
Recall that $w_2=\begin{pmatrix}&1\\1\end{pmatrix}$. When $m$ is odd,
$$
\omega_H=diag (I_{(m+i)n-1},\begin{pmatrix}&2\\
\frac{1}{2}\end{pmatrix},I_{(m+i)n-1}).
$$
Note that both matrices $w_2$ and $\begin{pmatrix}&2\\ \frac{1}{2}\end{pmatrix}$ are of order $2$, and lie in $\RO_2(F)\smallsetminus \SO_2(F)$. Thus, an even power of $\omega_H$ is the identity, and an odd power of $\omega_H$ is $\omega_H$.

Since $H$ will always be clear from the context, we will write
$\omega_0$ instead of $\omega_H$. We will use the same notation
$\omega_0$ for the similar element in a group of the same type as  $H$,
but in a different number of variables, that is when we replace
$I_{(m+i)n-1}$ by an identity matrix of a different
size. There won't be any confusion. The elements of $Q_{m(n-1)}(F)\cap
\alpha_r^{-1}Q_{n(m+i)}(F)\alpha_r$, which we will denote, for
short, by $Q^{(r)}(F)$, have the form
\begin{equation}\label{2.3}
\begin{pmatrix}a_1&a_2&y_1&y_2&z_1&z_2\\0&a_4&0&y_4&0&z_4\\
&&c&v&y'_4&y'_2\\&&&c^*&0&y'_1\\&&&&a_4^*&a_2'\\&&&&&a_1^*\end{pmatrix}^{\omega_0^{m(n-1)-r}}\in
H(F),
\end{equation}
where $a_1\in \GL_r(F)$, $a_4\in \GL_{m(n-1)-r}(F)$, $c\in \GL_{m+ni}(F)$. As the notation suggests, $Q^{(r)}$ is an $F$- algebraic subgroup of $Q_{m(n-1)}$.
The element \eqref{2.3} is conjugated by $\alpha_r$ to the element
\begin{equation}\label{2.4}
\begin{pmatrix}a_1&y_1&z_1&a_2&y_2&z_2\\&c&y'_4&0&v&y'_2\\&&a_4^*&0&0&a_2'\\
&&&a_4&y_4&z'_1\\&&&&c^*&y'_1\\&&&&&a_1^*\end{pmatrix}.
\end{equation}
(For the definition of $c^*$, $a_4^*$, $a_1^*$, recall  \eqref{1.1''}.) Apply the Fourier coefficient \eqref{1.9}. Then by \eqref{2.1}\\
\\
$\mathcal{F}_\psi(E(f_{\Delta(\tau, m+i)\gamma_\psi^{(\epsilon)},s}))(x)=$
\begin{equation}\label{2.5}
\sum_{r=0}^{m(n-1)}\int_{U_{m^{n-1}}(F)\backslash
U_{m^{n-1}}(\BA)}\sum_{\gamma\in Q^{(r)}(F)\backslash Q_{m(n-1)}(F)}
f_{\Delta(\tau, m+i)\gamma_\psi^{(\epsilon)},s}(\alpha_r\gamma ux)\psi^{-1}_H(u)du.
\end{equation}
Note that, according to our notation (see \eqref{1.x.1}), when $H$ is metaplectic, \eqref {2.5} has the following form. \\
\\
$\mathcal{F}_\psi(E(f_{\Delta(\tau, m+i)\gamma_\psi,s}))(x)=$
\begin{equation}\label{2.5.1}
\sum_{r=0}^{2m'(n-1)}\int\sum_{\gamma\in Q^{(r)}(F)\backslash Q_{2m'(n-1)}(F)}
f_{\Delta(\tau, 2m'+i)\gamma_\psi,s}((p((\Pi'_v(\alpha_r\gamma u_v,1))x))\psi^{-1}_H(u)du,
\end{equation}
where $x\in \Sp^{(2)}_{2n(2m'+i)}(\BA)$, and the $du$-integration is over 
$U_{(2m')^{n-1}}(F)\backslash
U_{(2m')^{n-1}}(\BA)$. Note that in \eqref{2.5.1}, 
$$
p(\Pi'_v(\alpha_r,1)(\gamma,1)(u_v,1))=p(\Pi'_v(\alpha_r\gamma u_v,1)).
$$ 
Denote the summand corresponding to $r$ in \eqref{2.5} by
$\mathcal{F}_{\psi,r}(f_{\Delta(\tau, m+i)\gamma_\psi^{(\epsilon)},s})(x)$.
\begin{thm}\label{thm 2.1}
For all $0<r\leq m(n-1)$, and all smooth holomorphic
sections $f_{\Delta(\tau, m+i)\gamma_\psi^{(\epsilon)},s}$,
$$
\mathcal{F}_{\psi,r}(f_{\Delta(\tau, m+i)\gamma_\psi^{(\epsilon)},s})=0.
$$
\end{thm}
The proof will occupy the rest of this section. It is similar in
nature to \cite{GJS15}, Sec. 3. We continue the unfolding process of \eqref{2.5}, assuming that
$\Re(s)$ is sufficiently large. We will factorize the summation in
$\gamma$ in \eqref{2.5} in several steps. The first one is modulo
$Q_{m^{n-1}}(F)$ from the right. The summation in $\gamma$ in
\eqref{2.5} becomes, for fixed $r$ and fixed $u$,
\begin{equation}\label{2.6}
\sum_{\gamma\in Q^{(r)}(F)\backslash
Q_{m(n-1)}(F)/Q_{m^{n-1}}(F)}\sum_\eta  f_{\Delta(\tau,
m+i)\gamma_\psi^{(\epsilon)},s}(\alpha_r\gamma\eta ux).
\end{equation}
The summation in $\eta$ is over $(Q_{m^{n-1}}(F)\cap
\gamma^{-1}Q^{(r)}(F)\gamma)\backslash Q_{m^{n-1}}(F)$. Looking at \eqref{2.3}, we may take the elements $\gamma=\hat{a}$ as
representatives for\\ 
$Q^{(r)}(F)\backslash Q_{m(n-1)}(F)/Q_{m^{n-1}}(F)$,
where $a$ is a representative of \\
$P_{r,m(n-1)-r}(F)\backslash
\GL_{m(n-1)}(F)/P_{m^{n-1}}(F)$. (See \eqref{1.1'} for the definition of $\hat{a}$.) We will choose these representatives
as follows. See (3.19) in \cite{GJS15}. Let $r_1,...,r_{n-1}$ and
$\ell_1,...,\ell_{n-1}$ be non-negative integers, such that
\begin{equation}\label{2.7}
\sum_{i=1}^{n-1}r_i=r,\ \sum_{i=1}^{n-1}\ell_i=m(n-1)-r; \
r_i+\ell_i=m,\ i=1,...,n-1.
\end{equation}
Put $\bar{r}=(r_1,...,r_{n-1})$. Let $a_{\bar{r}}$ be the
following permutation matrix in $\GL_{m(n-1)}(F)$. It has the form
\begin{equation}\label{2.8}
a_{\bar{r}}=\begin{pmatrix}w'_1&w'_2&\cdots&w'_{n-1}\\w''_1&w''_2&\cdots
&w''_{n-1}\end{pmatrix},
\end{equation}
where, for $1\leq i\leq n-1$, each matrix $w'_i$ is of size $r\times
m$, and each matrix $w''_i$ is of size $(m(n-1)-r)\times m$.
The matrices $w'_i$ have the form
\begin{equation}\label{2.9}
w'_i=(w'_{i,1},0_{r\times \ell_i}),\
w'_{i,1}=\begin{pmatrix}0_{(r_1+\cdots +r_{i-1})\times
r_i}\\I_{r_i}\\0_{(r_{i+1}+\cdots +r_{n-1})\times
r_i}\end{pmatrix},\ 1\leq i\leq n-1.
\end{equation}
Similarly,
\begin{equation}\label{2.10}
w''_i=(0_{(m(n-1)-r)\times r_i}, w''_{i,2}),\
w''_{i,2}=\begin{pmatrix}0_{(\ell_1+\cdots +\ell_{i-1})\times
\ell_i}\\I_{\ell_i}\\0_{(\ell_{i+1}+\cdots +\ell_{n-1})\times
\ell_i}\end{pmatrix},\ 1\leq i\leq n-1.
\end{equation}
For example, for $n=4$,
$$
a_{\bar{r}}=\begin{pmatrix}I_{r_1}&0&0&0&0&0\\0&0&I_{r_2}&0&0&0&\\
0&0&0&0&I_{r_3}&0\\0&I_{\ell_1}&0&0&0&0\\0&0&0&I_{\ell_2}&0&0&\\0&0&0&0&0&I_{\ell_3}\end{pmatrix}.
$$
We will take $\gamma=\hat{a}_{\bar{r}}$ in \eqref{2.6}. Note again, that since we deal with rational elements, we don't need to worry about the cocycle in the metaplectic case. We now
describe the subgroup $Q_{m^{n-1}}(F)\cap
\hat{a}_{\bar{r}}^{-1}Q^{(r)}(F)\hat{a}_{\bar{r}}$. Its
elements have the form
\begin{equation}\label{2.11}
\theta=\begin{pmatrix}A&B&C\\&h&B'\\&&A^*\end{pmatrix}\in H \ , A\in
\GL_{m(n-1)}(F),
\end{equation}
and we specify the forms of each block. First, $h\in
Q_{ni+m}(F)^{\omega_0^{m(n-1)-r}}$, that is $h$ lies in the
$\omega_0^{m(n-1)-r}$-conjugate of the $F$- points of the standard parabolic subgroup
of $H_{2(ni+m)}$, whose Levi part is isomorphic to $\GL_{ni+m}$.
Next,
$$
A=\begin{pmatrix}g_1&x_{1,2}&\cdots
&&x_{1,n-1}\\&g_2&\cdots&&x_{2,n-1}\\&&\ddots\\&&&g_{n-2}&x_{n-2,n-1}\\&&&&g_{n-1}\end{pmatrix},
$$
where each block is of size $m\times m$, such that, for $1\leq
i,j\leq n-1$, $g_i\in P_{r_i,\ell_i}(F)$, and
\begin{equation}\label {2.12}
x_{i,j}=\begin{pmatrix}x_{i,j}^{(1)}&x_{i,j}^{(2)}\\0_{\ell_i\times
r_j}&x_{i,j}^{(4)}\end{pmatrix}.
\end{equation}
Note that $x_{i,j}^{(1)}$ is of size $r_i\times r_j$,
$x_{i,j}^{(2)}$ is of size $r_i\times \ell_j$, and $x_{i,j}^{(4)}$
is of size $\ell_i\times \ell_j$.
$$
B=\begin{pmatrix}x_1\\x_2\\ \vdots\\ x_{n-1}\end{pmatrix},
$$
where each block is of size $m\times 2(ni+m)$, such that, for
$1\leq i\leq n-1$,
\begin{equation}\label{2.13}
x_i=\begin{pmatrix}x_i^{(1)}&x_i^{(2)}\\0_{\ell_i\times
(ni+m)}&x_i^{(4)}\end{pmatrix}\omega_0^{m(n-1)-r}.
\end{equation}
Note that $x_i^{(1)}, x_i^{(2)}$ are of size $r_i\times (ni+m)$,
and $x_i^{(4)}$ is of size $\ell_i\times (ni+m)$. Finally, write
$C$ as a $m(n-1)\times m(n-1)$ block matrix, with all blocks
$c_{i,j}$ of size $m\times m$, $1\leq i,j\leq n-1$. Then
\begin{equation}\label {2.14}
c_{i,j}=\begin{pmatrix}c_{i,j}^{(1)}&c_{i,j}^{(2)}\\0_{\ell_i\times
\ell_{n-j}}&c_{i,j}^{(4)}\end{pmatrix}.
\end{equation}

For the element \eqref{2.11}, $\theta\in Q_{m^{n-1}}(F)\cap
\hat{a}_{\bar{r}}^{-1}Q^{(r)}(F)\hat{a}_{\bar{r}}$, let us
describe $\hat{a}_{\bar{r}}\theta \hat{a}_{\bar{r}}^{-1}$.
For this, write, for $1\leq i\leq n-1$,
$$
g_i=\begin{pmatrix}g_i^{(1)}&g_i^{(2)}\\0_{\ell_i\times
r_i}&g_i^{(4)}\end{pmatrix}.
$$
Write $\hat{a}_{\bar{r}}\theta \hat{a}_{\bar{r}}^{-1}$ in
the form \eqref{2.3}, with the same notation. Then, for $t=1,2,4$,
\begin{equation}\label{2.14.1}
a_t=\begin{pmatrix}g^{(t)}_1&x^{(t)}_{1,2}&\cdots
&&x^{(t)}_{1,n-1}\\&g^{(t)}_2&\cdots&&x^{(t)}_{2,n-1}\\&&\ddots\\&&&g^{(t)}_{n-2}&x^{(t)}_{n-2,n-1}\\&&&&g^{(t)}_{n-1}\end{pmatrix};
\end{equation}
\begin{equation}\label{2.14.2}
y_t=\begin{pmatrix}x_1^{(t)}\\x_2^{(t)}\\
\cdots\\x_{n-1}^{(t)}\end{pmatrix};
\end{equation}
\begin{equation}\label{2.14.3}
z_t=\begin{pmatrix}c_{1,1}^{(t)}&\cdots&c_{1,n-1}^{(t)}\\&\cdots\\c_{n-1,1}^{(t)}&\cdots&c_{n-1,n-1}^{(t)}\end{pmatrix}.
\end{equation}
Note the form of \eqref{2.4} when applied to
$\hat{a}_{\bar{r}}\theta \hat{a}_{\bar{r}}^{-1}$. We will use this repeatedly in the proof of Prop. \ref{prop 2.3} to produce inner integrals to 
(suumands of ) $\mathcal{F}_{\psi,r}(f_{\Delta(\tau, m+i)\gamma_\psi^{(\epsilon)},s})$, which are equal to zero, due to the fact that these
are integrals of nontrivial characters along compact groups.

Let us factor the summation in $\eta$ in \eqref{2.6} modulo
$U_{m^{n-1}}(F)$ from the right. From the
description \eqref{2.12}-\eqref{2.14} of $Q_{m^{n-1}}(F)\cap
\hat{a}_{\bar{r}}^{-1}Q^{(r)}(F)\hat{a}_{\bar{r}}$, we may
take representatives of $(Q_{m^{n-1}}(F)\cap
\hat{a}_{\bar{r}}^{-1}Q^{(r)}(F)\hat{a}_{\bar{r}})\backslash Q_{m^{n-1}}(F)/U_{m^{n-1}}(F)$ of
the form
\begin{equation}\label{2.15}
h_{\bar{\eta};\gamma}=diag(\eta_1,...,\eta_{n-1},\gamma,\eta^*_{n-1},...,\eta^*_1),
\end{equation}
where $\eta_i$ is any coset representative in $ P_{r_i,\ell_i}(F)\backslash \GL_{m}(F)$, for $1\leq
i\leq n-1$, and $\gamma$ is any coset representative in
$Q_{ni+m}(F)^{\omega_0^{m(n-1)-r}}\backslash
H_{2(ni+m)}(F)$.\\
Denote
\begin{equation}\label{2.15.1}
\bar{\eta}=(\eta_1,...,\eta_{n-1}).
\end{equation}
Put
$Q^{(\bar{r};\bar{\eta};\gamma)}(F)=h_{\bar{\eta};\gamma}^{-1}(Q_{m^{n-1}}(F)\cap
\hat{a}_{\bar{r}}^{-1}Q^{(r)}(F)\hat{a}_{\bar{r}})h_{\bar{\eta};\gamma}$.
The summation \eqref{2.6} becomes
\begin{equation}\label{2.16}
\sum_{\bar{r}}\sum_\gamma \sum_{\bar{\eta} }\sum_{u'\in
(Q^{(\bar{r};\bar{\eta};\gamma)}(F)\cap U_{m^{n-1}}(F))\backslash
U_{m^{n-1}}(F)} f_{\Delta(\tau,
m+i)\gamma_\psi^{(\epsilon)},s}(\alpha_r\hat{a}_{\bar{r}}h_{\bar{\eta};\gamma}u'
ux).
\end{equation}
Substitute \eqref{2.16} in \eqref{2.5} to get, for $0\leq r\leq
m(n-1)$,\\
\\
$\mathcal{F}_{\psi,r}(f_{\Delta(\tau, m+i)\gamma_\psi^{(\epsilon)},s})(x)=$
\begin{equation}\label{2.17}
\sum_{\bar{r}}\sum_{\gamma}\sum_{\bar{\eta}}
\int_{(Q^{(\bar{r};\bar{\eta};\gamma)}(F)\cap U_{m^{n-1}}(F))\backslash
U_{m^{n-1}}(\BA)} f_{\Delta(\tau,
m+i)\gamma_\psi^{(\epsilon)},s}(\alpha_r\hat{a}_{\bar{r}}h_{\bar{\eta};\gamma}
ux)\psi^{-1}_H(u)du.
\end{equation}
Factor the summation in $\gamma$ modulo $j(H_m(F)\times
H_{2ni+m}(F))$ from the right. Thus the sum over $\gamma \in
Q_{ni+m}(F)^{\omega_0^{m(n-1)-r}}\backslash
H_{2(ni+m)}(F)/j(H_m(F)\times H_{2ni+m}(F))$ is followed by a sum over $(g',h')$ in
\begin{equation}\label{2.17.1}
j^{-1}(\gamma^{-1}Q_{ni+m}(F)^{\omega_0^{m(n-1)-r}}\gamma \cap
j(H_m(F)\times H_{2ni+m}(F)))\backslash H_m(F)\times H_{2ni+m}(F).
\end{equation}
In
\eqref{2.17}, $h_{\bar{\eta};\gamma}$ is now replaced by
$h_{\bar{\eta};\gamma j(g',h')}$, and $Q^{(\bar{r};\bar{\eta};\gamma)}(F)$ is replaced by\\
$Q^{(\bar{r};\bar{\eta};\gamma j(g',h'))}(F)$. Now, change variable in
$\bar{\eta}$, $\eta_i\mapsto \eta_ig'$, so that $\bar{\eta}\mapsto
\bar{\eta}(g')^{\Delta_{n-1}}$. Note that
$h_{(g')^{\Delta_{n-1}},j(g',h')}=t(g',h')$. Since $t(g',h')$
normalizes $U_{m^{n-1}}(F)$, $U_{m^{n-1}}(\BA)$ and preserves
$\psi_H$ (the same in the metaplectic case), \eqref{2.17} becomes\\
\\
$\mathcal{F}_{\psi,r}(f_{\Delta(\tau, m+i)\gamma_\psi^{(\epsilon)},s})(x)=$
\begin{equation}\label{2.18}
\sum_{\bar{r}}\sum_{\gamma}\sum_{(g',h')}\sum_{\bar{\eta}}
\int f_{\Delta(\tau,
m+i)\gamma_\psi^{(\epsilon)},s}(\alpha_r\hat{a}_{\bar{r}}h_{\bar{\eta};\gamma}
ut(g',h')x)\psi^{-1}_H(u)du,
\end{equation}
where $\gamma$ is summed over
$Q_{ni+m}(F)^{\omega_0^{m(n-1)-r}}\backslash
H_{2(ni+m)}(F)/j(H_m(F)\times H_{2ni+m}(F))$, 
$\bar{\eta}=(\eta_1,...,\eta_{n-1})$ is summed over
$\prod_{i=1}^{n-1}P_{r_i,\ell_i}(F)\backslash \GL_m(F)$, $(g',h')$ is summed over \eqref{2.17.1}. Finally, $u$ is integrated along
$(Q^{(\bar{r};\bar{\eta};\gamma)}(F)\cap U_{m^{n-1}}(F))\backslash
	U_{m^{n-1}}(\BA)$.

We will now write a set of representatives of
$$
Q_{ni+m}(F)^{\omega_0^{m(n-1)-r}}\backslash
H_{2(ni+m)}(F)/j(H_m(F)\times H_{2ni+m}(F)).
$$
Note that $j(H_m(F)\times H_{2ni+m}(F))$ is invariant under conjugation by
$\omega_0$. Thus, it is enough to find a set of representatives of
$Q_{ni+m}(F)\backslash H_{2(ni+m)}(F)/j(H_m(F)\times H_{2ni+m}(F))$. Let
us realize $Q_{ni+m}(F)\backslash H_{2(ni+m)}(F)$ as the variety
$\mathcal{Z}$ of all (maximal) $ni+m$-dimensional isotropic
subspaces of the space in which $H_{2(ni+m)}(F)$ acts. We consider the 
action of $j(H_m(F)\times H_{2ni+m}(F))$ on $\mathcal{Z}$. Since the elements of 
$j(H_m(F)\times H_{2ni+m}(F))$ preserve $X_m$, it follows that for a given 
orbit $\mathcal{O}$ of this action,  $\dim(Z\cap
X_m)=d_{\mathcal{O}}$ is independent of $Z$, for any $Z\in
\mathcal{O}$. (See right after \eqref{1.7} for the definition of
$X_m$.) The following lemma is proved exactly as Lemma 3.1 in
\cite{GJS15}.
\begin{lem}\label{lem 2.2}
The number $d_{\mathcal{O}}$ is the only invariant of the orbit
$\mathcal{O}$.
\end{lem}
Using this lemma, we can list the following representatives,
parametrized by $e=d_{\mathcal{O}}$, $0\leq e\leq [\frac{m}{2}]$.
\begin{equation}\label{2.19}
\gamma_e=diag(I_e,\begin{pmatrix}I_{[\frac{m}{2}]-e}\\0&I_{[\frac{m}{2}]-e}\\0&0&I_{2(ni+(m-2[\frac{m}{2}])+e)}\\
I_{[\frac{m}{2}]-e}&0&0&I_{[\frac{m}{2}]-e}\\0&-\delta_HI_{[\frac{m}{2}]-e}&0&0&I_{[\frac{m}{2}]-e}\end{pmatrix},I_e).
\end{equation}
We may now take in \eqref{2.18}, $\gamma=\gamma_e$. Note that
$\gamma_e$ commutes with $\omega_0$.
\begin{prop}\label{prop 2.3}
Let $\bar{r}=(r_1,...,r_{n-1})$ be as in \eqref{2.7} and $e$ as
above. If the summand corresponding to $\bar{r}, \gamma_e$, in
\eqref{2.18}, is nonzero, then
$$
r_1\leq r_2\leq\cdots\leq r_{n-1}\leq e\leq [\frac{m}{2}].
$$
\end{prop}
\begin{proof}
Given $\bar{r}$ and $e$, the corresponding summand in \eqref{2.18} is
$$
\sum_{(g',h')}\sum_{\bar{\eta}}
\int f_{\Delta(\tau,
m+i)\gamma_\psi^{(\epsilon)},s}(\alpha_r\hat{a}_{\bar{r}}h_{\bar{\eta};\gamma_e}
ut(g',h')x)\psi^{-1}_H(u)du.
$$
Assume that it is nonzero, for some $x$. Then there are $(g',h')$ and $\bar{\eta}$, such that
\begin{equation}\label{2.19.1}
\int_{(Q^{(\bar{r};\bar{\eta};\gamma_e)}(F)\cap U_{m^{n-1}}(F))\backslash
	U_{m^{n-1}}(\BA)} f_{\Delta(\tau,
m+i)\gamma_\psi^{(\epsilon)},s}(\alpha_r\hat{a}_{\bar{r}}h_{\bar{\eta};\gamma_e}
ut(g',h')x)\psi^{-1}_H(u)du \neq 0.
\end{equation}
Replacing the right- $t(g',h')x$ translate of $f_{\Delta(\tau,
m+i)\gamma_\psi^{(\epsilon)},s}$ by $f_{\Delta(\tau,
m+i)\gamma_\psi^{(\epsilon)},s}$, we may assume that in \eqref{2.19.1}, $t(g',h')x=I$. Denote $U'=Q^{(\bar{r};\bar{\eta};\gamma_e)}\cap U_{m^{n-1}}$. Factoring the $du$-integration in \eqref{2.19.1}, modulo
$U'(\BA)$ from the left, it follows that there is a $u\in U_{m^{n-1}}(\BA)$, such that 
\begin{equation}\label{2.19.1.1}
\int_{U'(F))\backslash
	U'(\BA)} f_{\Delta(\tau,
m+i)\gamma_\psi^{(\epsilon)},s}(\alpha_r\hat{a}_{\bar{r}}h_{\bar{\eta};\gamma_e}
u'u)\psi^{-1}_H(u')du' \neq 0.
\end{equation}
Again, we may assume that in \eqref{2.19.1.1}, $u=I$. Thus, assume that
\begin{equation}\label{2.19.1.1'}
\int_{U'(F))\backslash
	U'(\BA)} f_{\Delta(\tau,
m+i)\gamma_\psi^{(\epsilon)},s}(\alpha_r\hat{a}_{\bar{r}}h_{\bar{\eta};\gamma_e}
u')\psi^{-1}_H(u')du' \neq 0.
\end{equation}

We are going to find an $F$- subgroup $V$ of $U'=Q^{(\bar{r};\bar{\eta};\gamma_e)}\cap U_{m^{n-1}}$, such that its conjugation by
$\alpha_r\hat{a}_{\bar{r}}h_{\bar{\eta};\gamma_e}$ is contained in $U_{(m+i)n}$ (as $F$-algebraic groups). This will imply that, for all $v\in V(\BA)$ and all $f_{\Delta(\tau,m+i)\gamma_\psi^{(\epsilon)},s}$,
\begin{equation}\label{2.19.2}
f_{\Delta(\tau,
m+i)\gamma_\psi^{(\epsilon)},s}(\alpha_r\hat{a}_{\bar{r}}h_{\bar{\eta};\gamma}
v)=f_{\Delta(\tau,
m+i)\gamma_\psi^{(\epsilon)},s}(\alpha_r\hat{a}_{\bar{r}}h_{\bar{\eta};\gamma}
).
\end{equation}
This will enable us to get an inner integration, along $V(F)\backslash V(\BA)$, inside \eqref{2.19.1.1'}, such that, due to \eqref{2.19.2},
$$
\int_{V(F)\backslash V(\BA)}f_{\Delta(\tau,
m+i)\gamma_\psi^{(\epsilon)},s}(\alpha_r\hat{a}_{\bar{r}}h_{\bar{\eta};\gamma_e}
vu')\psi^{-1}_H(vu')dv=
$$
\begin{equation}\label{2.19.3}
f_{\Delta(\tau,
m+i)\gamma_\psi^{(\epsilon)},s}(\alpha_r\hat{a}_{\bar{r}}h_{\bar{\eta};\gamma_e}
u')\psi^{-1}_H(u')\int_{V(F)\backslash V(\BA)}\psi^{-1}_H(v)dv.
\end{equation}
Since, by \eqref{2.19.1}, the $dv$-integration of the character $\psi^{-1}_H$ must be nonzero, we must have that $\psi_H$ is trivial on $V(\BA)$. This will imply the conditions on $\bar{r}$ that we want. We now define the subgroup $V$. $V$ consists of all matrices  $v$ of the form \eqref{2.11}, with
$B=0$, $C=0$, $h=I_{2(ni+m)}$, $A$ with
$g_1=\cdots=g_{n-1}=I_m$, and, for $1\leq i<j\leq n-1$,
$$
x_{i,j}=\eta_i^{-1}\begin{pmatrix}0&y_{i,j}^{(2)}\\0&0\end{pmatrix}\eta_j,
$$
where the last matrix has  block division
of sizes as in \eqref{2.12}. It follows that $v\in
Q^{(\bar{r};\bar{\eta};\gamma_e)}\cap U_{m^{n-1}}$. By
\eqref{2.14.1},
\begin{equation}\label{2.20}
\hat{a}_{\bar{r}}h_{\bar{\eta};\gamma_e}vh_{\bar{\eta};\gamma_e}^{-1}\hat{a}_{\bar{r}}^{-1}=
\begin{pmatrix}I_r&y\\0&I_{m(n-1)-r}\end{pmatrix}^\wedge,
\end{equation}
(this is an element of the form \eqref{2.3}); the matrix $y$ is the
following
$$
y=\begin{pmatrix}0&y_{1,2}^{(2)}&\cdots
&y_{1,n-1}^{(2)}\\&0&&y_{2,n-1}^{(2)}\\&&\ddots\\&&&y_{n-2,n-1}^{(2)}\\&&&0\end{pmatrix}.
$$
By \eqref{2.4}, the element on the r.h.s. of \eqref{2.20} is conjugated
by $\alpha_r$ into $U_{(m+i)n}$. Thus, \eqref{2.19.2} is satisfied, for all $v\in V(\BA)$. Also, in the notation above, for $v\in V(\BA)$, 
\begin{equation}\label{2.22}
\psi_H(v)=\prod_{i=2}^{n-1}\psi(tr(\begin{pmatrix}0&y_{i-1,i}^{(2)}\\0&0\end{pmatrix}\eta_i\eta_{i-1}^{-1})).
\end{equation}
In order to use \eqref{2.19.3}, let us explain how do we get the integration along $V(F)\backslash V(\BA)$ as an inner integration of
\eqref{2.19.1.1}. for this, we note that we have a semidirect product decomposition of $F$- unipotent subgroups $U'=U'_1\rtimes U_2''$,
where $U'_1$ is the subgroup of elements of the form $h_{\bar{\eta},\gamma_e}^{-1}\theta h_{\bar{\eta},\gamma_e}$, where $\theta$ is of the form \eqref{2.11}, with $h=I_{2(ni+m)}$, $g_1=\cdots=g_{n-1}=I_m$, $B$, $C$, as in  \eqref{2.13}, \eqref{2.14}, but in \eqref{2.12}, we require that $x_{i,j}^{(1)}=0$ and $x_{i,j}^{(4)}=0$. The subgroup $U'_2$ consists of the elements form $h_{\bar{\eta},\gamma_e}^{-1}\theta h_{\bar{\eta},\gamma_e}$, where $\theta$ is of the form \eqref{2.11}, with $h=I_{2(ni+m)}$, $g_1=\cdots=g_{n-1}=I_m$, $B=0$, $C=0$, and in \eqref{2.12}, we require that $x_{i,j}^{(2)}=0$. Now, we can write the integral \eqref{2.19.1.1} (with $u=I$) as
\begin{equation}\label{2.20.1}
\int_{U_2'(F))\backslash
	U_2'(\BA)}\int_{U'_1(F)\backslash U'_1(\BA)} f_{\Delta(\tau,
m+i)\gamma_\psi^{(\epsilon)},s}(\alpha_r\hat{a}_{\bar{r}}h_{\bar{\eta};\gamma_e}
u_1'u'_2)\psi^{-1}_H(u_1'u'_2)du_1'du'_2.
\end{equation}
As in \eqref{2.19.1}, there is $u'_2\in U'_2(\BA)$, such that the inner $du'_1$-integration of \eqref{2.20.1} is nonzero, and, again, we may assume that $u'_2=I$, so that
\begin{equation}\label{2.20.2}
\int_{U'_1(F)\backslash U'_1(\BA)} f_{\Delta(\tau,
m+i)\gamma_\psi^{(\epsilon)},s}(\alpha_r\hat{a}_{\bar{r}}h_{\bar{\eta};\gamma_e}
u_1')\psi^{-1}_H(u_1')du_1' \neq 0.
\end{equation}
We have $U'_1=U''\ltimes V$, where $U''=U'_1\cap U_{m(n-1)}$. As in \eqref{2.20.1}, we can write the integral \eqref{2.20.2} as
\begin{equation}\label{2.20.3}
\int_{V(F))\backslash
	V(\BA)}\int_{U''(F)\backslash U''(\BA)} f_{\Delta(\tau,
m+i)\gamma_\psi^{(\epsilon)},s}(\alpha_r\hat{a}_{\bar{r}}h_{\bar{\eta};\gamma_e}
u''v)\psi^{-1}_H(u''v)du''dv.
\end{equation}
By \eqref{2.19.2}, the integral \eqref{2.20.3} is equal to
\begin{equation}\label{2.20.4}
\int_{V(F))\backslash
	V(\BA)}\int_{U''(F)\backslash U''(\BA)} f_{\Delta(\tau,
m+i)\gamma_\psi^{(\epsilon)},s}(\alpha_r\hat{a}_{\bar{r}}h_{\bar{\eta};\gamma_e}
v^{-1}u''v)\psi^{-1}_H(u''v)du''dv.
\end{equation}
We leave it to the reader to see that in the inner $du''$-integration of \eqref{2.20.4}, we can change variable $u''\mapsto vu''v^{-1}$.
This can be checked by direct multiplication of upper unipotent matrices. We get that
\begin{equation}\label{2.20.5}
\int_{V(F))\backslash
	V(\BA)}\int_{U''(F)\backslash U''(\BA)} f_{\Delta(\tau,
m+i)\gamma_\psi^{(\epsilon)},s}(\alpha_r\hat{a}_{\bar{r}}h_{\bar{\eta};\gamma_e}
u'')\psi^{-1}_H(vu'')du''dv \neq 0.
\end{equation}
The integral in \eqref{2.20.5} is equal to
\begin{equation}\label{2.20.6}
(\int_{U''(F)\backslash U''(\BA)} f_{\Delta(\tau,
m+i)\gamma_\psi^{(\epsilon)},s}(\alpha_r\hat{a}_{\bar{r}}h_{\bar{\eta};\gamma_e}
u'')\psi^{-1}_H(u'')du'')\cdot (\int_{V(F))\backslash V(\BA)}\psi^{-1}_H(v)dv).
\end{equation}
This explains \eqref{2.19.3}. We conclude that \eqref{2.22} is identically $1$. This forces 
$\eta_i\eta_{i-1}^{-1}$ to have the following form, for $2\leq i\leq n-1$,
\begin{equation}\label{2.23}
\eta_i\eta_{i-1}^{-1}=\begin{pmatrix}\mu_{i-1}^{(1)}&\mu_{i-1}^{(2)}\\0_{\ell_i\times
r_{i-1}}&\mu_{i-1}^{(4)}\end{pmatrix}.
\end{equation}
Note that $\mu_{i-1}^{(1)}\in M_{r_i\times r_{i-1}}(F)$. Since
$\eta_i\eta_{i-1}^{-1}\in \GL_m(F)$, the rank of the matrix on the r.h.s. of \eqref{2.23} is $m$, and hence, we must have that $r_{i-1}\leq
r_i$, for $2\leq i\leq n-1$. Note, again, that the $\eta_i$ are any coset representatives of $P_{r_i,\ell_i}(F)\backslash \GL_m(F)$. The form of the matrix on the r.h.s. of \eqref{2.23} remains the same if we replace $\eta_i$ and $\eta_{i-1}$ by other representatives.

It remains to show that $r_{n-1}\leq
e$. The proof is similar. We find an $F$-subgroup $V'\subset U'= Q^{(\bar{r};\bar{\eta};\gamma)}\cap
U_{m^{n-1}}$, such that \eqref{2.19.2} is satisfied, for all $v\in V'(\BA)$, and so that we can find the  inner integration, along $V'(F)\backslash V'(\BA)$, inside \eqref{2.19.1.1}, such that we get the analog of \eqref{2.19.3}, with $V'$ instead of $V$, and then, since the integral of $\psi^{-1}_H$ along $V'(F)\backslash V'(\BA)$ is nontrivial (by \eqref{2.19.1.1}, it follows that $\psi_H$ must be trivial on $V'(\BA)$. Since the details are similar to those related to $V$, we will only define $V'$, compute the restriction of $\psi_H$ to $V'(\BA)$, and then infer the inequality $r_{n-1}\leq e$.

The group $V'$ consists of all matrices  $v$ of 
the form \eqref{2.11}, with
$A=I_{m(n-1)}$, $h=I_{2(ni+m)}$, $B$ with
$$
x_i=\eta_i^{-1}\begin{pmatrix}0&y_i^{(2)}\\0&0\end{pmatrix}\omega_0^{m(n-1)-r}\gamma_e,\
1\leq i\leq n-1,
$$
where the last matrix has block division of sizes as in
\eqref{2.13}. The matrix $C$ is with
$$
c_{i,j}=\eta_i^{-1}\begin{pmatrix}0&\zeta_{i,j}^{(2)}\\0&0\end{pmatrix}(\eta_{n-j}^*)^{-1},
$$
where the last matrix has block division of sizes as in
\eqref{2.14}. Then $v\in Q^{(\bar{r};\bar{\eta};\gamma_e)}\cap
U_{m^{n-1}}$, and by \eqref{2.14.2}, \eqref{2.14.3},
\begin{equation}\label{2.24}
\hat{a}_{\bar{r}}h_{\bar{\eta};\gamma_e}vh_{\bar{\eta};\gamma_e}^{-1}\hat{a}_{\bar{r}}^{-1}=
\begin{pmatrix}I_r&0&0&y&0&\zeta\\&I_{m(n-1)-r}&0&0&0&0\\&&I_{ni+m}&0&0&y'\\&&&I_{ni+m}&0&0\\&&&&I_{m(n-1)-r}&0\\
&&&&&I_r\end{pmatrix}^{\omega_0^{m(n-1)-r}},
\end{equation}
(this is an element of the form \eqref{2.3}); the matrix $y$ is the
following
$$
y=\begin{pmatrix}y_1^{(2)}\\y_2^{(2)}\\\cdots\\
y_{n-1}^{(2)}\end{pmatrix},
$$
and $\zeta$ is the matrix of all blocks $\zeta_{i,j}^{(2)}$. By
\eqref{2.4}, the element on the r.h.s. of \eqref{2.24} is conjugated by
$\alpha_r$ into $U_{(m+i)n}(\BA)$. Now we note that, for $v\in V'(\BA)$, with the above notation,
(since $\omega_0\gamma_e=\gamma_e\omega_0$ and $\omega_0A_H=A_H$)
\begin{equation}\label{2.25}
\psi_H(v)=\psi(tr(\begin{pmatrix}0&y_{n-1}^{(2)}\\0&0\end{pmatrix}\gamma_e
A_H\eta_{n-1}^{-1})).
\end{equation}
Since \eqref{2.25} must be identically $1$, we conclude that
\begin{equation}\label{2.25.1}
\gamma_e
A_H\eta_{n-1}^{-1}=\begin{pmatrix}\nu^{(1)}&\nu^{(2)}\\0_{(ni+m)\times
r_{n-1}}&\nu^{(4)}\end{pmatrix}.
\end{equation}
Note that $\nu^{(1)}\in M_{(ni+m)\times r_{n-1}}(F)$. It follows that $\eta_{n-1}^{-1}$ has the following form,
that
\begin{equation}\label{2.26}
\eta_{n-1}^{-1}=\begin{pmatrix}\mu_{n-1}^{(1)}&\mu_{n-1}^{(2)}\\0_{(m-e)\times
r_{n-1}}&\mu_{n-1}^{(4)}\end{pmatrix}.
\end{equation}
Since $\eta_{n-1}^{-1}\in \GL_m(F)$, we must have $r_{n-1}\leq
e$. This proves the proposition.
\end{proof}
Carrying out the matrix multiplication in \eqref{2.25.1}, let us express the $\nu^{(i)}$ in terms of the $\mu^{(i)}_{n-1}$ in \eqref{2.26}.
We get that
\begin{equation}\label{2.26.1}
\nu^{(1)}=\begin{pmatrix}\mu_{n-1}^{(1)}\\0_{(ni+m-e)\times
r_{n-1}}\end{pmatrix}.
\end{equation}
It also follows that for $m$ even,
\begin{equation}\label{2.26.2}
\nu^{(4)}=\begin{pmatrix}0_{(ni+e)\times
\ell_{n-1}}\\\mu_{n-1}^{(4)}\end{pmatrix},
\end{equation}
and for $m$ odd, if we write
$$
\mu_{n-1}^{(4)}=\begin{pmatrix}d_1\\d_2\\d_3\end{pmatrix},\ d_1\in
M_{([\frac{m}{2}]-e)\times \ell_{n-1}}(F),\ d_2\in M_{1\times \ell_{n-1}}(F),\
d_3\in M_{[\frac{m}{2}]\times \ell_{n-1}}(F),
$$
then
\begin{equation}\label{2.26.3}
\nu^{(4)}=\begin{pmatrix}\frac{1}{2}d_2\\0_{(ni+e)\times
\ell_{n-1}}\\d_1\\d_3\end{pmatrix}.
\end{equation}

We continue the line of proof of Prop. \ref{prop 2.3}. Assume then that the integral \eqref{2.19.1.1'} is nonzero,
for one of the remaining representatives $\alpha_r\hat{a}_{\bar{r}}h_{\bar{\eta},\gamma_e}$, satisfying the inequalities of Prop. \ref{prop 2.3}, and $\bar{\eta}$ as in \eqref{2.23}-\eqref{2.26}. Our next goal is to show that this is impossible if $r_{n-1}>0$. This will complete the proof of Theorem \ref{thm 2.1}. As a preparation for the proof, we rewrite the integral \eqref{2.19.1.1'} by carrying out in its integrand the conjugation by $\alpha_r\hat{a}_{\bar{r}}h_{\bar{\eta},\gamma_e}$. Then the integral \eqref{2.19.1.1'} is equal to
\begin{equation}\label{2.27}
\int_{E_{\bar{r}}(F)\backslash E_{\bar{r}}(\BA)}f_{\Delta(\tau,
m+i)\gamma_\psi^{(\epsilon)},s}(u\alpha_r\hat{a}_{\bar{r}}h_{\bar{\eta};\gamma_e}
)\psi^{-1}_{E_{\bar{r}};\bar{\eta},e}(u)du,
\end{equation}
where $E_{\bar{r}}=\alpha_r\hat{a}_{\bar{r}}h_{\bar{\eta},\gamma_e}U'(\alpha_r\hat{a}_{\bar{r}}h_{\bar{\eta},\gamma_e})^{-1}$, and 
$\psi_{E_{\bar{r}};\bar{\eta},e}$ is the character of $E_{\bar{r}}(\BA)$ obtained by composing the restriction of $\psi_H$ to $U'(\BA)$ with the conjugation by $\alpha_r\hat{a}_{\bar{r}}h_{\bar{\eta},\gamma_e}$. The form of the matrices in $E_{\bar{r}}$ can be read from \eqref{2.3}-\eqref{2.4}, \eqref{2.11}-\eqref{2.14.3}. Since these matrices are upper unipotent, we don't need to worry in the metaplectic case, about the cocycle. See \cite{Rao93}, Cor. 5.5.

The subgroup $E_{\bar{r}}$ is a semi-direct product of two unipotent
subgroups. One is an $F$- subgroup $U^{\bar{r}}_{(m+i)n}$ of $U_{(m+i)n}$, and the other is
the following unipotent subgroup $V_{\bar{r}}^\wedge\subset
\GL_{(m+i)n}^\wedge$, which normalizes $U^{\bar{r}}_{(m+i)n}$. Its elements have the form
\begin{equation}\label{2.28}
\hat{v}=\begin{pmatrix}I_{r_1}&x_{1,2}&\cdots&x_{1,n-1}&b_1&\ast
&&\cdots&\ast\\&I_{r_2}&\cdots&x_{2,n-1}&b_2&\ast&&\cdots&\ast\\&&\ddots\\&&&I_{r_{n-1}}&b_{n-1}&\ast&&\cdots&\ast\\
&&&&I_{ni+m}&c_{n-1}&&\cdots&c_1\\&&&&&I_{\ell_{n-1}}&y_{n-2,n-1}&\cdots&y_{1,n-2}\\&&&&&&\ddots\\&&&&&&&&y_{1,2}\\
&&&&&&&&I_{\ell_1}\end{pmatrix}^\wedge.
\end{equation}
Note that
$V_{\bar{r}}$ is the standard unipotent radical $V_{r_1,...,r_{n-1},m+ni,\ell_{n-1},...,\ell_1}$ (defined right before \eqref{1.2}). 
The character $\psi_{E_{\bar{r}};\bar{\eta},e}$ is trivial on\\
$E_{\bar{r}}(\BA)\cap U_{(m+i)n}(\BA)$, and takes the following
value on $\hat{v}$ in \eqref{2.28}, with adele coordinates,
\begin{equation}\label{2.29}
\prod_{i=1}^{n-2}\psi(tr(\mu_i^{(1)}x_{i,i+1}))\psi(tr(\tilde{\mu}^{(4)}_iy_{i,i+1}))\psi(tr(\nu^{(1)}b_{n-1}))\psi(tr(\tilde{\nu}^{(4)}c_{n-1})),
\end{equation}
where $\mu_i,\nu$ are as in \eqref{2.23}, \eqref{2.25.1}, and we
define
$$
\tilde{\mu}^{(4)}_i=-w_{\ell_i}{}^t\mu_i^{(4)}w_{\ell_{i+1}},\
\tilde{\nu}^{(4)}=-w_{\ell_{n-1}}{}^t\nu^{(4)}w_{ni+m}.
$$
It follows that in \eqref{2.27}, for $u=u^{\bar{r}}\hat{v}$, with $u^{\bar{r}}\in U^{\bar{r}}_{(m+i)n}(\BA)$, $v\in V_{\bar{r}}(\BA)$,
\begin{equation}\label{2.29.1}
f_{\Delta(\tau,
m+i)\gamma_\psi^{(\epsilon)},s}(u^{\bar{r}}\hat{v}\alpha_r\hat{a}_{\bar{r}}h_{\bar{\eta};\gamma_e}
)\psi^{-1}_{E_{\bar{r}};\bar{\eta},e}(u^{\bar{r}}\hat{v})=f_{\Delta(\tau,
m+i)\gamma_\psi^{(\epsilon)},s}(\hat{v}\alpha_r\hat{a}_{\bar{r}}h_{\bar{\eta};\gamma_e}
)\psi^{-1}_{E_{\bar{r}};\bar{\eta},e}(\hat{v}).
\end{equation}
Thus, up to the measure of $U^{\bar{r}}_{(m+i)n}(F)\backslash U^{\bar{r}}_{(m+i)n}(\BA)$, the integral \eqref{2.27} is equal to
\begin{equation}\label{2.29.2}
\int_{V_{\bar{r}}(F)\backslash V_{\bar{r}}(\BA)}f_{\Delta(\tau,
m+i)\gamma_\psi^{(\epsilon)},s}(\hat{v}\alpha_r\hat{a}_{\bar{r}}h_{\bar{\eta};\gamma_e}
)\psi^{-1}_{E_{\bar{r}};\bar{\eta},e}(\hat{v})dv.
\end{equation}
Let us view \eqref{2.29} as a character of $V_{\bar{r}}(\BA)$ and
denote it by $\psi_{V_{\bar{r}};\bar{\eta},e}$. For an automorphic
form $\xi$ in the space of $\Delta(\tau,m+i)$, consider the
following Fourier coefficient
\begin{equation}\label{2.30}
\xi^{\psi,\bar{r},\bar{\eta},e}(a)=\int_{V_{\bar{r}}(F)\backslash
V_{\bar{r}}(\BA)}\xi(va)\psi^{-1}_{V_{\bar{r}};\bar{\eta},e}(v)dv.
\end{equation}
Then \eqref{2.27}, and hence, \eqref{2.29.2}, is the application of the Fourier coefficient
\eqref{2.30} to the automorphic form on $\GL_{(m+i)n}(\BA)$, given by 
$$
b\mapsto \delta^{-\frac{1}{2}}(\hat{b})|\det(b)|^{-s}(\gamma_\psi^{(\epsilon)}(\det(b)))^{-1}f_{\Delta(\tau,
m+i)\gamma_\psi^{(\epsilon)},s}(\hat{b}\alpha_r\hat{a}_{\bar{r}}h_{\bar{\eta};\gamma_e}
),
$$
evaluated at $b=I_{(m+i)n}$.  We denote this Fourier coefficient, evaluated at
$I_{(m+i)n}$, by
$f^{\psi,\bar{r},\bar{\eta},e}_{\Delta(\tau,m+i)\gamma_\psi^{(\epsilon)},s}(\alpha_r\hat{a}_{\bar{r}}h_{\bar{\eta};\gamma_e})$. 
Summarizing, we showed that, 
for any one of the representatives $\alpha_r\hat{a}_{\bar{r}}h_{\bar{\eta},\gamma_e}$, satisfying the inequalities of Prop. \ref{prop 2.3}, and $\bar{\eta}$, as in \eqref{2.23}-\eqref{2.26}, the integral \eqref{2.19.1.1'} is equal, up to the measure of $U^{\bar{r}}_{(m+i)n}(F)\backslash U^{\bar{r}}_{(m+i)n}(\BA)$, to $f^{\psi,\bar{r},\bar{\eta},e}_{\Delta(\tau,m+i)\gamma_\psi^{(\epsilon)},s}(\alpha_r\hat{a}_{\bar{r}}h_{\bar{\eta};\gamma_e})$. The next proposition says that if $r_{n-1}>0$, then $f^{\psi,\bar{r},\bar{\eta},e}_{\Delta(\tau,m+i)\gamma_\psi^{(\epsilon)},s}(\alpha_r\hat{a}_{\bar{r}}h_{\bar{\eta};\gamma_e})=0$, for all $f^{\psi,\bar{r},\bar{\eta},e}_{\Delta(\tau,m+i)\gamma_\psi^{(\epsilon)},s}$.

\begin{prop}\label{prop 2.4}
Assume that $r_{n-1}>0$. Then the Fourier coefficient \eqref{2.30}
is identically zero on $\Delta(\tau,m+i)$.
\end{prop}
\begin{proof}
We note that $rank(\mu_i^{(1)})=r_i$,
$rank(\mu^{(4)}_i)=\ell_{i+1}$, for $1\leq i\leq n-2$. Also,
$rank(\nu^{(1)})=rank(\mu_{n-1}^{(1)})=r_{n-1}$,
$rank(\nu^{(4)})=rank(\mu_{n-1}^{(4)})=m-e$. These follow from the
last proof. When we conjugate the character $\psi_{V_{\bar{r}};\bar{\eta},e}$, given by
\eqref{2.29}, by an element
$diag(a_1,...,a_{n-1},b,\alpha_{n-1},...,\alpha_1)\in
\GL_{(m+i)n}(F)$, where $a_i\in \GL_{r_i}(F)$, $\alpha_i\in
\GL_{\ell_i}(F)$, for $1\leq i\leq n-1$, and $b\in \GL_{ni+m}(F)$,
to a character of the same form (as \eqref{2.29}), with $\mu_i^{(1)}$, $\tilde{\mu}_i^{(4)}$ replaced by $a_{i+1}^{-1}\mu_i^{(1)}a_i$, 
$\alpha_i^{-1}\tilde{\mu}_i^{(4)}\alpha_{i+1}$, respectively, for $1\leq i\leq n-2$, and $\nu^{(1)}$, $\tilde{\nu}^{(4)}$ are replaced by 
$b^{-1}\nu^{(1)}a_{n-1}$, $\alpha_{n-1}^{-1}\tilde{\nu}^{(4)}b$, respectively.
Due to the form of $\nu^{(1)}, \nu^{(4)}$,
\eqref{2.26.1}-\eqref{2.26.3}, it is an exercise in linear algebra (bringing matrices to row echelon forms) 
to show that one can find $b$ and $a_i,\alpha_i$, as above, such that
$$
a_{i+1}^{-1}\mu_i^{(1)}a_i=R_i=\begin{pmatrix}I_{r_i}\\0_{(r_{i+1}-r_i)\times
r_i}\end{pmatrix},\  1\leq i\leq n-2;
$$
$$
b^{-1}\nu^{(1)}a_{n-1}=R_{n-1}=\begin{pmatrix}I_{r_{n-1}}\\0_{(ni+m-r_{n-1})\times
r_{n-1}}\end{pmatrix};
$$
$$
\alpha_i^{-1}\tilde{\mu}_i^{(4)}\alpha_{i+1}=L_i=\begin{pmatrix}I_{\ell_{i+1}}\\0_{(\ell_i-\ell_{i+1})\times
\ell_{i+1}}\end{pmatrix},\  1\leq i\leq n-2;
$$
$$
\alpha_{n-1}^{-1}\tilde{\nu}^{(4)}b=S=\begin{pmatrix}I_{m-e}&0_{(m-e)\times
(ni+e)}\\0_{(e-r_{n-1})\times (m-e)}&0_{(e-r_{n-1})\times
(ni+e)}\end{pmatrix}.
$$
After such a conjugation, the character  $\psi_{V_{\bar{r}};\bar{\eta},e}$
becomes the character
\begin{equation}\label{2.32}
\prod_{i=1}^{n-2}\psi(tr(R_ix_{i,i+1}))\psi(tr(L_iy_{i,i+1}))\psi(tr(R_{n-1}b_{n-1}))\psi(tr(Sc_{n-1})).
\end{equation}
The character \eqref{2.32} corresponds to the nilpotent orbit (in
$\mathfrak{gl}_{(m+i)n}(\bar{F})$) of the matrix
\begin{equation}\label{2.33}
A=\begin{pmatrix}0_{r_1}\\R_1&0_{r_2}\\&&\ddots\\&&&0_{r_{n-2}}\\&&&R_{n-2}&
0_{r_{n-1}}\\&&&&R_{n-1}&0_{ni+m}\\&&&&&S&0_{\ell_{n-1}}\\&&&&&&L_{n-1}&0_{\ell_{n-2}}\\
&&&&&&&&\ddots\\&&&&&&&&L_2&0_{\ell_1}\end{pmatrix}.
\end{equation}
Assume that the Fourier coefficient \eqref{2.30}
is nontrivial on $\Delta(\tau,m+i)$. The orbit of $A$ corresponds to a partition of the form
$((2n-1)^{r_1},...)$. By Prop. \ref{prop 1.2}, we must have $r_1=0$,
unless $n=1$. This case is impossible, since we assume that $r_{n-1}>0$, and hence $n\geq 2$. Note that when $n=1$, the sum \eqref{2.7} is empty, and also , $\Delta(\tau,m+i)=|\det_{\GL_{n(m+i)}(\BA)}|$.
Assume that $n>1$. Consider the matrix \eqref{2.33}, with $r_1=0$.
Then, similarly, its orbit corresponds to a partition of the form
$((2n-2)^{r_2},...)$. By Prop. \ref{prop 1.2}, we must have
$r_2=0$, unless $n=2$, in which case, $r_{n-1}=r_1=0$, while we assume that $r_{n-1}>0$.
If $n>i-1$, and we proved that $r_{i-1}=0$, then we get
that the orbit of $A$ corresponds to a partition of the form
$((2n-i)^{r_i},...)$. By Prop. \ref{prop 1.2}, we must have $r_i=0$,
unless $n=i$, in which case, $r_{n-1}=r_{i-1}=0$, while we assume that $r_{n-1}>0$.
By induction, we get that $r_{n-1}=0$, contradicting our assumption. This proves the proposition.
\end{proof}
The last proposition implies that if
$\mathcal{F}_{\psi,r}(f_{\Delta(\tau, m+i)\gamma_\psi^{(\epsilon)},s})$ is nonzero, then
$r_1=\cdots=r_{n-1}=0$ (due to Prop. \ref{prop 2.3}). By
\eqref{2.7}, $r=0$. This completes the proof of Theorem \ref{thm
2.1}.

\section{The contribution of the open orbit to $\mathcal{E}(f_{\Delta(\tau,
		m+i)\gamma_\psi^{(\epsilon)},s},\varphi_\sigma)$}

Theorem \ref{thm 2.1}, \eqref{2.5} and \eqref{2.18} imply that\\
\\
$\mathcal{F}_\psi(E(f_{\Delta(\tau,
m+i)\gamma_\psi^{(\epsilon)},s}))(x)=\mathcal{F}_{\psi,0}(E(f_{\Delta(\tau,
m+i)\gamma_\psi^{(\epsilon)},s}))(x)=$
\begin{equation}\label{3.1}
\sum_{e=0}^{[\frac{m}{2}]}\sum_{(g',h')}\sum_{\bar{\eta}}
\int_{U_{m^{n-1}}^{\bar{0},\bar{\eta},e}(\BA)\backslash
U_{m^{n-1}}(\BA)} f^{\psi,\bar{0},\bar{\eta},e}_{\Delta(\tau,
m+i)\gamma_\psi^{(\epsilon)},s}(\alpha_0\hat{a}_{\bar{0}}h_{\bar{\eta};\gamma_e}
ut(g',h')x)\psi^{-1}_H(u)du.
\end{equation}
Let us describe the ingredients of \eqref{3.1} in the present case
(i.e. $r=0$). Recall that the summation on
$\bar{\eta}=(\eta_1,...,\eta_{n-1})$ runs over $\prod_{i=1}^{n-1}
P_{r_i,\ell_i}(F)\backslash \GL_m(F)$. Since $r_i=0$, and $r_i+\ell_i=m$, for $1\leq
i\leq n-1$, $P_{r_i,\ell_i}$ is $\GL_m$, and we may take $\eta_i=I_m$, for all $i$, so that there
is no further summation on $\eta$ in \eqref{3.1}. Let us denote, in this case, the
element $h_{\bar{\eta},\gamma_e}$  by $h_{\gamma_e}$, and denote 
$f^{\psi,\bar{0},\bar{\eta},e}_{\Delta(\tau,
m+i)\gamma_\psi^{(\epsilon)},s}=f^{\psi,e}_{\Delta(\tau, m+i)\gamma_\psi^{(\epsilon)},s}$. Note that
$V_{\bar{0}}=V_{ni+m,m^{n-1}}$, and the character
$\psi_{V_{\bar{0}};\bar{\eta},e}$ (defined in \eqref{2.29}, and viewed as a character of $V_{\bar{0}}(\BA)$) is now the following character
$\psi'_e$ of $V_{ni+m,m^{n-1}}(\BA)$ (read from \eqref{2.26.1}-\eqref{2.26.3} and \eqref{2.29}),
\begin{equation}\label{3.3}
\psi'_e(\begin{pmatrix}I_{ni+m}&c_{n-1}&&\cdots&c_1\\&I_m&y_{n-2,n-1}&\cdots&y_{1,n-2}\\&&\ddots\\&&&&y_{1,2}\\
&&&&I_m\end{pmatrix})=\prod_{i=1}^{n-2}\psi^{-1}(tr(y_{i,i+1}))\psi^{-1}(tr(\nu
c_{n-1})),
\end{equation}
where in case $m$ is even,
\begin{equation}\label{3.3.1}
\nu=\begin{pmatrix}I_{m-e}&0_{(m-e)\times (ni+e)}\\0_{e\times
(m-e)}&0_{e\times (ni+e)}\end{pmatrix},
\end{equation}
and in case $m=2m'-1$ is odd,
\begin{equation}\label{3.3.2}
\nu=\begin{pmatrix}I_{m'-1}&0&0_{(m'-1)\times
(ni+e)}&0\\0&0&0&\frac{1}{2}\\0&I_{m'-1-e}&0&0\\0&0&0_{e\times
(ni+e)}&0\end{pmatrix}.
\end{equation}
Note also that $a_{\bar{0}}=I_{m(n-1)}$, and that
\begin{equation}\label{3.2}
\alpha_0=\begin{pmatrix}0&I_{ni+m}&0&0\\0&0&0&I_{m(n-1)}\\
\delta_H
I_{m(n-1)}&0&0&0\\0&0&I_{ni+m}&0\end{pmatrix}\omega_0^{m(n-1)}.
\end{equation}
Note that $\omega_0^{m(n-1)}=I$, when $m$ is even. When $m$ is odd, i.e. $H_m$ is odd orthogonal, $\omega_0^{m(n-1)}=\omega_0^{n-1}$. Next, the subgroup $U_{m^{n-1}}^{\bar{0},\bar{\eta},e}$, which we
denote, for short, $U_{m^{n-1}}^e$, consists of the elements
\begin{equation}\label{3.2.1}
\begin{pmatrix}v&y&0\\&I_{2(ni+m)}&y'\\&&v^*\end{pmatrix}^{\omega_0^{m(n-1)}}\in
U_{m^{n-1}},
\end{equation}
where $v\in V_{m^{n-1}}$ and $y$ has the form
$$
y=(0_{m(n-1)\times e},y_1,0_{m(n-1)\times
(ni+(m-2[\frac{m}{2}])+e)},y_2,\tilde{y}_1,y_3),
$$
where $y_1$ is of size $m(n-1)\times 2([\frac{m}{2}]-e)$, $y_2$ is of
size $m(n-1)\times (ni+(m-2[\frac{m}{2}])+e)$, $y_3$ is of size $m(n-1)\times
e$, and $ \tilde{y}_1=y_1\cdot diag(I_{[\frac{m}{2}]-e},-\delta_H
I_{[\frac{m}{2}]-e})$. Finally, the summation in $(g',h')$ in \eqref{3.1}
runs over
$j^{-1}(\gamma_e^{-1}Q_{ni+m}(F)^{\omega_0^{m(n-1)}}\gamma_e \cap
j(H_m(F)\times H_{2ni+m}(F)))\backslash H_m(F)\times H_{2ni+m}(F)$. See
\eqref{2.18}. The elements of
$\gamma_e^{-1}Q_{ni+m}^{\omega_0^{m(n-1)}}\gamma_e \cap
j(H_m\times H_{2ni+m})$ have the following form.
\begin{equation}\label{3.3.1'}
(\begin{pmatrix}a&x&y\\&{}^eb&x'\\&&a^*\end{pmatrix},\begin{pmatrix}A&X&Y\\&b&X'\\&&A^*\end{pmatrix}^{\hat{\beta}_e})\in
H_m\times H_{2ni+m},
\end{equation}
where $a\in \GL_e$, $A\in \GL_{ni+e}$, and, for $b\in H_{m-2e}$,
\begin{equation}\label{3.3.1''}
{}^eb=J^{-1}_e b J_e,
\end{equation}
and $J_e=J_{e;m,n}$ is defined as follows. If $m(n-1)$ is even (so that $\omega_0^{m(n-1)}=I$), then
\begin{equation}\label{3.3.1*} 
J_e=\begin{pmatrix}&&I_{[\frac{m}{2}]-e}\\&I_{m-2[\frac{m}{2}]}\\-\delta_H
I_{[\frac{m}{2}]-e}\end{pmatrix};
\end{equation}
and if $m(n-1)$ is odd, so that $m=2m'-1$ is odd and $n$ is even, then
\begin{equation}\label{3.3.1**}
J_e=\begin{pmatrix}&&-I_{m'-1-e}\\&1\\I_{m'-1-e}\end{pmatrix}.
\end{equation}
Also,
\begin{equation}\label{3.3.1'''}
\beta_e=\begin{pmatrix}&I_{[\frac{m}{2}]-e}\\I_{ni+e}\end{pmatrix}.
\end{equation}
Note that $b\mapsto {}^e b$ is an outer conjugation of $H_{2ni+m}$. Now \eqref{3.1} becomes\\
\\
$\mathcal{F}_\psi(E(f_{\Delta(\tau,
m+i)\gamma_\psi^{(\epsilon)},s}))(x)=\sum_{e=0}^{[\frac{m}{2}]}\sum_{h'\in
Q_{ni+e}(F)\backslash H_{2ni+m}(F)}\sum_{g'\in Q^0_e(F)\backslash H_m(F)}$
\begin{equation}\label{3.4}
\int_{U_{m^{n-1}}^e(\BA)\backslash U_{m^{n-1}}(\BA)}
f^{\psi'_e}_{\Delta(\tau, m+i)\gamma_\psi^{(\epsilon)},s}(\alpha_0h_{\gamma_e}
ut(g',\hat{\beta}_eh')x)\psi^{-1}_H(u)du.
\end{equation}
Here, $Q_{ni+e}=Q_{ni+e}^{H_{2ni+m}}$, and $Q^0_e$ is the subgroup of
$Q_e^{H_m}$ consisting of elements as the first matrix in
\eqref{3.3.1}, with $b=I_{m-2e}$. The Fourier coefficient
$f^{\psi'_e}_{\Delta(\tau, m+i)\gamma_\psi^{(\epsilon)},s}$ is obtained by 
applying the $\psi'_e$ - coefficient \eqref{3.3} to the automorphic form on $\GL_{n(m+i)}(\BA)$ (and then we evaluate at
the identity), 
$$
a\mapsto \delta^{-\frac{1}{2}}(\hat{a})|\det(a)|^{-s}(\gamma_\psi^{(\epsilon)}(\det(a)))^{-1}f_{\Delta(\tau, m+i)\gamma_\psi^{(\epsilon)},s}(\hat{a}\alpha_0h_{\gamma_e} ut(g',\hat{\beta}_eh')x)
$$  
As in \eqref{2.30}, consider the Fourier coefficient applied to an
automorphic form $\xi$ in the space of $\Delta(\tau,m+i)$,
\begin{equation}\label{3.5}
\xi^{\psi'_e}(a)=\int_{V_{ni+m,m^{n-1}}(F)\backslash
V_{ni+m,m^{n-1}}(\BA)}\xi(va)\psi'_e(v^{-1})dv,\quad a\in \GL_{n(m+i)}(\BA).
\end{equation}
It will be convenient to conjugate the character $\psi'_e$ as
follows. Define the following element $w_0\in \GL_{(m+i)n}(F)$.
Assume that $m$ is even. Then
\begin{equation}\label{3.5.1}
w_0=diag(\begin{pmatrix}&I_{ni}\\I_m\end{pmatrix},I_{m(n-1)}).
\end{equation}
Assume that $m=2m'-1$ is odd. Then
\begin{equation}\label{3.5.2}
w_0=diag(\begin{pmatrix}0&0&I_{ni}&0\\I_{m'-1}&0&0&0\\0&0&0&\frac{1}{2}\\0&I_{m'-1}&0&0\end{pmatrix},I_{(2m'-1)(n-1)}).
\end{equation}
We will use the following conjugate of $\psi'_e$:
$\psi_e(v)=\psi'_e(w_0^{-1}vw_0)$. Then
\begin{equation}\label{3.6}
\xi^{\psi_e}(w_0a)=\int_{V_{ni+m,m^{n-1}}(F)\backslash
V_{ni+m,m^{n-1}}(\BA)}\xi(vw_0a)\psi^{-1}_e(v)dv=\xi^{\psi'_e}(a).
\end{equation}
Explicitly, as can be read from \eqref{3.3}-\eqref{3.3.2}, the character $\psi_e$ is given by
\begin{equation}\label{3.7}
\psi_e(\begin{pmatrix}I_{ni}&0&y_1&&&\cdots&y_{n-1}\\&I_m&x_1&&&\cdots&\ast\\&&&\ddots\\
&&&&&I_m&x_{n-1}\\&&&&&&I_m\end{pmatrix})=\psi^{-1}(tr(a_ex_1))\prod_{i=2}^{n-1}\psi^{-1}(tr(x_i)),
\end{equation}
where $a_e=\begin{pmatrix}I_{m-e}&0\\0&0_{e\times
e}\end{pmatrix}$.
\begin{prop}\label{prop 3.1}
For any automorphic form $\xi$ in the space of
$\Delta(\tau,m+i)$, and any $y\in M_{ni\times m}(\BA)$,
$\xi^{\psi_e}$ is invariant under left multiplication by\\
$n_y=diag(\begin{pmatrix}I_{ni}&y\\&I_m\end{pmatrix},I_{m(n-1)})$.
\end{prop}
\begin{proof}
Consider, for fixed $a$, the following function on $M_{ni\times
m}(F)\backslash M_{ni\times m}(\BA)$,
$$
\phi_\xi(y)=\xi^{\psi_e}(n_ya).
$$
Note that, for $y\in M_{ni\times m}(F)$, $n_y$ normalizes the
rational and adele points of $V_{m+ni,m^{n-1}}$, and it preserves
$\psi_e$. Consider the Fourier expansion of $\phi_\xi$ along the
compact Abelian group $M_{ni\times m}(F)\backslash M_{ni\times
m}(\BA)$. We claim that the Fourier coefficients of $\phi_\xi$
with respect to all nontrivial characters are zero. Indeed, such a
Fourier coefficient is, in fact, a Fourier coefficient of $\xi$
along the unipotent radical $V_{ni,m^n}$ with respect to a
character of the form
\begin{equation}\label{3.8}
\psi_{A;e}(\begin{pmatrix}I_{ni}&y&y_1&&&\cdots&y_{n-1}\\&I_m&x_1&&&\cdots&\ast\\&&&\ddots\\
&&&&&I_m&x_{n-1}\\&&&&&&I_m\end{pmatrix})=
\end{equation}
$$
\psi^{-1}(tr(Ay))\psi^{-1}(tr(a_ex_1))\prod_{i=2}^{n-1}\psi^{-1}(tr(x_i)),
$$
where $A\in M_{m\times ni}(F)$ is nonzero. The character \eqref{3.8} corresponds to the nilpotent orbit in $\mathfrak{gl}_{n(m+i)}(\bar{F})$
of the matrix
\begin{equation}\label{3.8.1}
\begin{pmatrix}0_{ni}\\A&0_m\\0&a_e&0_m\\&&I_m&0_m\\&&&&&\ddots\\&&&&&&I_m&0_m\end{pmatrix},
\end{equation}
and the orbit of this matrix corresponds to a partition of the form
$(n+1,...)$. By Prop. \ref{prop 1.2}, the above Fourier coefficient must be trivial. Thus,
$\phi_\xi$ is equal to its constant term along $M_{ni\times
m}(F)\backslash M_{ni\times m}(\BA)$. This proves the
proposition.
\end{proof}
Consider the character \eqref{3.8} $\psi_{0;e}$ (i.e. $A=0$). Denote
\begin{equation}\label{3.9}
\xi^{\psi_{0;e}}(a)=\int_{V_{ni,m^n}(F)\backslash
V_{ni,m^n}(\BA)}\xi(va)\psi^{-1}_{0;e}(v)dv.
\end{equation}
Then the last proposition and \eqref{3.6} show that, for any
automorphic form $\xi$ in the space of $\Delta(\tau,m+i)$,
\begin{equation}\label{3.10}
\xi^{\psi'_e}(a)=\xi^{\psi_{0;e}}(w_0a).
\end{equation}
Denote 
\begin{equation}\label{3.10.1}
\tilde{\alpha}_0=\hat{w}_0\alpha_0. 
\end{equation}
By \eqref{3.10} and
Prop.
\ref{prop 3.1}, we can rewrite \eqref{3.4} as\\
\\
$\mathcal{F}_\psi(E(f_{\Delta(\tau, m+i)\gamma_\psi^{(\epsilon)},s}))(x)=$
\begin{equation}\label{3.11}
\sum_{e=0}^{[\frac{m}{2}]}\sum_{h'}\sum_{g'}
\int_{U_{m^{n-1}}^e(\BA)\backslash U_{m^{n-1}}(\BA)}
f^{\psi_{0;e}}_{\Delta(\tau,
m+i)\gamma_\psi^{(\epsilon)},s}(\tilde{\alpha}_0h_{\gamma_e}
ut(g',\hat{\beta}_eh')x)\psi^{-1}_H(u)du.
\end{equation}
Here, $h'$ runs over $Q_{ni+e}(F)\backslash H_{2ni+m}(F)$, and $g'$ runs
over $Q^0_e(F)\backslash H_m(F)$. (The subgroup $U^e_{m^{n-1}}$ is defined in \eqref{3.2.1}.)

At this point, we go back to our
kernel integral \eqref{1.10}, and bring into our considerations, for
the first time, the cuspidal representation $\sigma$ of
$H_m(\BA)$. Using \eqref{3.11}, we have, for $Re(s)$ sufficiently large,
\begin{equation}\label{3.12}
\mathcal{E}(f_{\Delta(\tau,
m+i)\gamma_\psi^{(\epsilon)},s},\varphi_\sigma)(h)=\sum_{e=0}^{[\frac{m}{2}]}\sum_{h'\in
Q_{ni+e}(F)\backslash H_{2ni+m}(F)}\Lambda_e(f_{\Delta(\tau,
m+i)\gamma_\psi^{(\epsilon)},s},\varphi_\sigma)(h'h),
\end{equation}
where\\
\\
$\Lambda_e(f_{\Delta(\tau,
	m+i)\gamma_\psi^{(\epsilon)},s},\varphi_\sigma)(h)=$
\begin{equation}\label{3.13}
\int_{C_2^{(\epsilon)}Q^0_e(F)\backslash
	H^{(\epsilon)}_m(\BA)}\varphi_\sigma(g)\int_{U_{m^{n-1}}^e(\BA)\backslash
	U_{m^{n-1}}(\BA)} f^{\psi_{0;e}}_{\Delta(\tau,
	m+i)\gamma_\psi^{(\epsilon)},s}(\tilde{\alpha}_0h_{\gamma_e}
ut(g,\hat{\beta}_eh))\psi^{-1}_H(u)dudg.
\end{equation}
Here, we collapsed to $C_2^{(\epsilon)}Q^0_e(F)\backslash
	H^{(\epsilon)}_m(\BA)$ the $dg$-integration along $C_2^{(\epsilon)}H_m(F)\backslash H^{(\epsilon)}_m(\BA)$, which appears in the definition 
of $\mathcal{E}(f_{\Delta(\tau,
m+i)\gamma_\psi^{(\epsilon)},s},\varphi_\sigma)$, followed by the summation of $g'$ over $Q^0_e(F)\backslash H_m(F)$ in \eqref{3.11}.

Denote by $M^0_e$ the reductive part of $Q^0_e$
($M^0_e=\widehat{\GL}_e$). The unipotent radical of $Q^0_e$ is
$U_e=U^{H_m}_e$. We factor the $dg$-integration in \eqref{3.13}
modulo $U_e(\BA)$. The subgroup
$\hat{w}_0\alpha_0h_{\gamma_e}t(U_e,1)h_{\gamma_e}^{-1}\alpha_0^{-1}\hat{w}_0^{-1}$
lies inside $Q_{(m+i)n}$. The elements of the projection of this
subgroup on the Levi part of $Q_{(m+i)n}$ have the following form
\begin{equation}\label{3.14}
diag(\begin{pmatrix}I_{ni}\\&I_e&w&0\\&&I_{m-2e}&0\\&&&I_e\end{pmatrix},
\begin{pmatrix}I_e&w&c\\&I_{m-2e}&w'\\&&I_e\end{pmatrix}^{\Delta_{n-1}})^\wedge\in
\widehat{\GL}_{(m+i)n},
\end{equation}
(Recall from \eqref{1.8} our notation for a diagonally repreated matrix). Denote
$$
y=\begin{pmatrix}I_e&w&c\\&I_{m-2e}&w'\\&&I_e\end{pmatrix},\quad c(y)=\begin{pmatrix}I_e&w&0\\&I_{m-2e}&0\\&&I_e\end{pmatrix}.
$$
The matrix $y$ is a typical element  in
$U_e$. Denote the element \eqref{3.14} by $\widehat{z(y)}$. Thus,
\begin{equation}\label{3.14.1}
z(y)=diag(I_{ni},c(y),y^{\Delta_{n-1}}).
\end{equation}
Now one checks that\\
\\
$\Lambda_e(f_{\Delta(\tau, m+i)\gamma_\psi^{(\epsilon)},s},\varphi_\sigma)(h)=$
\begin{equation}\label{3.15}
\int_{C_2^{(\epsilon)}M^0_e(F)U_e(\BA)\backslash H^{(\epsilon)}_m(\BA)}\int_{U_e(F)\backslash
U_e(\BA)}\varphi_\sigma(yg)\int f^{\psi_{0;e}}_{\Delta(\tau,
m+i)\gamma_\psi^{(\epsilon)},s}(\widehat{z(y)}\tilde{\alpha}_0h_{\gamma_e}
ut(g,\hat{\beta}_eh))
\end{equation}
$$
\psi^{-1}_H(u)dudydg,
$$
where the $du$-integration is along $U_{m^{n-1}}^e(\BA)\backslash 	U_{m^{n-1}}(\BA)$. 
\begin{prop}\label{prop 3.2}
For any $y\in U_e(\BA)$,
$$
f^{\psi_{0;e}}_{\Delta(\tau,
m+i)\gamma_\psi^{(\epsilon)},s}(\widehat{z(y)} x)=f^{\psi_{0;e}}_{\Delta(\tau, m+i)\gamma_\psi^{(\epsilon)},s}(x),
$$
for all $x\in H_\BA$ and all $f_{\Delta(\tau, m+i)\gamma_\psi^{(\epsilon)},s}$.
\end{prop}
\begin{proof}
This follows from the following extra-invariance property of the
Fourier coefficient \eqref{3.9}, namely let $g_1, g_2\in
\SL_m(\BA)$, such that 
\begin{equation}\label{3.15.1}
a_eg_1=g_2a_e.
\end{equation}
 (The matrix $a_e$ is defined in \eqref{3.7}.) 
Note that $diag(I_{ni},g_1,g_2^{\Delta_{n-1}})$ fixes the character $\psi_{0;e}$. 
Then, for any automorphic form $\xi$ in the space of
$\Delta(\tau,m+i)$,
\begin{equation}\label{3.16}
\xi^{\psi_{0;e}}(diag(I_{ni},g_1,g_2^{\Delta_{n-1}})a)=\xi^{\psi_{0;e}}(a).
\end{equation}
Note that \eqref{3.15.1} is satisfied for $g_2=y$, and $g_1=c(y)$, for all $y\in U_e(\BA)$.

We will sketch the proof of \eqref{3.16} for $g_1,g_2$ in certain standard unipotent
radicals. The proof works similarly for $g_1,g_2$ (satisfying \eqref{3.15.1}) in any root
subgroups (and these generate the subgroup of $\SL_m(\BA)\times \SL_m(\BA)$ consisting of the elements satisfying \eqref{3.15.1}). Note first that the Fourier coefficient
$\xi^{\psi_{0;e}}$ factors through the constant term of $\xi$ along
the unipotent radical $V_{ni,mn}$. Consider the restriction of such
constant terms to elements of the form $diag(I_{ni},b)$, $b\in
\GL_{mn}(\BA)$, as automorphic forms on $\GL_{mn}(\BA)$. These
form the elements of $|\det\cdot|^{\frac{1-ni}{2}}\Delta(\tau, m)$.
See Lemma 4.1 in \cite{JL13}. Let $\psi_{a_e}$ be the following
character of $V_{m^n}(\BA)$,
\begin{equation}\label{3.16.1}
\psi_{a_e}(\begin{pmatrix}I_m&x_1&&&\cdots&\ast\\&I_m&&&\cdots\\&&\ddots\\
&&&&I_m&x_{n-1}\\&&&&&I_m\end{pmatrix})=\psi^{-1}(tr(a_ex_1))\prod_{i=2}^{n-1}\psi^{-1}(tr(x_i).
\end{equation}
Consider, for $\eta$ in the space of $\Delta(\tau, m)$, the
Fourier coefficient $\eta^{\psi_{a_e}}$ of $\eta$ with respect to
$\psi_{a_e}$, along $V_{m^n}$, 
\begin{equation}\label{3.16.2}
\eta^{\psi_{a_e}}(a)=\int_{V_{m^n}(F)\backslash V_{m^n}(\BA)}\eta(va)\psi^{-1}_{a_e}(v)dv.
\end{equation}
Then we have to prove that, for
$g_1,g_2$, as in \eqref{3.16}, and $a\in \GL_{mn}(\BA)$,
\begin{equation}\label{3.16.2*}
\eta^{\psi_{a_e}}(diag(g_1,g_2^{\Delta_{n-1}})a)=\eta^{\psi_{a_e}}(a).
\end{equation}
The proof is similar to the proof of Prop. 3 in \cite{FG15}. The
basic idea of the proof is the same one used in the proof of Prop.
\ref{prop 3.1}, that is when we take a Fourier coefficient of
$u\mapsto \eta^{\psi_{a_e}}(ua)$, along an abelian unipotent subgroup
$X$ of the stabilizer above of $\psi_{a_e}$, with respect to a
nontrivial character, it follows that if this Fourier coefficient is
nontrivial, then we get a resulting nontrivial Fourier coefficient
on $\Delta(\tau,m)$, corresponding to a partition, which is
strictly larger than $(n^m)$. By Prop. \ref{prop 1.2}, we will get a
contradiction. Hence only the trivial character of $X_\BA$ appears
in the Fourier expansion of $\eta^{\psi_{a_e}}$ along $X$, that is
$\eta^{\psi_{a_e}}$ is left-$X_\BA$ invariant. 

Here is an example
sketching how the proof goes. Take above $g_1=I_m$ and
$g_2=\nu_e(z)=\begin{pmatrix}I_{m-e}&z\\&I_e\end{pmatrix}$.
Consider a Fourier coefficient
\begin{equation}\label{3.17}
\int_{M_{(m-e)\times e}(F)\backslash M_{(m-e)\times
e}(\BA)}\eta^{\psi_{a_e}}(diag(I_m,\nu_e(z)^{\Delta_{n-1}}))\psi^{-1}(tr(Bz))dz,
\end{equation}
where $0\neq B\in M_{e\times(m-e)}(F)$. Substitute \eqref{3.16.2},
and \eqref{3.17} becomes
\begin{equation}\label{3.18}
\int_{M_{(m-e)\times e}(F)\backslash M_{(m-e)\times
e}(\BA)}\int_{V_{m^n}(F)\backslash
V_{m^n}(\BA)}\eta(diag(I_m,\nu_e(z)^{\Delta_{n-1}})va)
\end{equation}
$$
\psi^{-1}_{a_e}(v)\psi^{-1}(tr(Bz))dvdz.
$$
Now we carry out the process of exchanging roots,(\cite{GRS11}, Sec.
7.1), as follows. Write in the matrix \eqref{3.16.1}, for $2\leq
j\leq n-1$,
$$
x_j=\begin{pmatrix}x_{j,1}&x_{j,2}\\x_{j,3}&x_{j,4}\end{pmatrix},
$$
where $x_{j,1}$ is of size $(m-e)\times (m-e)$. Now it is easy
to check that one can exchange (in the sense of Sec. 7.1 in
\cite{GRS11}) $x_{2,3}$ with
$diag(I_m,\nu_e(z_1),I_{(n-2)m})$. Denote by $V^1$ the
subgroup of $V_{m^n}$ obtained by replacing in \eqref{3.16.1} 
$x_{2,3}$ by zero. We conclude that if \eqref{3.17} is
nontrivial on $\Delta(\tau,m)$,  then the following Fourier coefficient
is not trivial on $\Delta(\tau,m)$ (Cor. 7.1 in \cite{GRS11}),
\begin{equation}\label{3.19.1}
\int_{M^2_{(m-e)\times e}(F)\backslash M^2_{(m-e)\times
e}(\BA)}\int_{V^1(F)\backslash
V^1(\BA)}\eta(diag(I_m,\nu_e(z_1),\nu^{\Delta_{n-2}}(z))v'a)
\end{equation}
$$
\psi^{-1}_{a_e}(v')\psi^{-1}(tr(Bz))dv'dz_1dz.
$$
Here, $M^2_{(m-e)\times e}=M_{(m-e)\times e}\times M_{(m-e)\times e}$. 
Next, we exchange roots in \eqref{3.19.1}, $x_{3,3}$ with $diag(I_{2m},\nu_e(z_2),I_{(n-3)m})$, and get the analogue of \eqref{3.19.1}, namely that the following Fourier coefficient is not identically zero on $\Delta(\tau,m)$,
\begin{equation}\label{3.19.2}
\int_{M^3_{(m-e)\times e}(F)\backslash M^3_{(m-e)\times
		e}(\BA)}\int_{V^2(F)\backslash
	V^2(\BA)}\eta(diag(I_{2m},\nu_e(z_1),\nu_e(z_2)\nu^{\Delta_{n-3}}(z))v'a)
\end{equation}
$$
\psi^{-1}_{a_e}(v')\psi^{-1}(tr(Bz))dv'dz_1dz_2dz,
$$
where $V^2$ is the subgroup of $V_{m^n}$ obtained by replacing in \eqref{3.16.1} 
$x_{2,3}$ and $x_{3,3}$ by zero, and $M^3_{(m-e)\times e}=M^2_{(m-e)\times e}\times M_{(m-e)\times e}$. We continue in this way, exchanging $x_{j,3}$ with
$diag(I_{jm},\nu_e(z_j),I_{(n-j-1)m})$, step by step, up to $j=n-1$, and finally we get that the following Fourier coefficient is not identically zero on $\Delta(\tau,m)$,
\begin{equation}\label{3.19}
\int_{M^{n-1}_{(m-e)\times e}(F)\backslash M^{n-1}_{(m-e)\times
		e}(\BA)}\int_{V'(F)\backslash
	V'(\BA)}\eta(diag(I_m,\nu_e(z_1),...,\nu_e(z_{n-1}))v'a)
\end{equation}
$$
\psi^{-1}_{a_e}(v')\psi^{-1}(tr(Bz_{n-1}))dv'dz_1\cdots dz_{n-1},
$$
where $V'$ is the subgroup of $V_{m^n}$ obtained by replacing in \eqref{3.16.1} 
$x_{j,3}$ by zero, for $2\leq j\leq n-1$, and $M^{n-1}_{(m-e)\times e}$ is the cartesian product of $n-1$ copies of $M_{(m-e)\times e}$. Note that the product of $V'$ and the subgroup of all\\
$diag(I_m,\nu_e(z_1),...,\nu_e(z_{n-1}))$ is the unipotent radical $V_{m,m-e,m^{n-2},e}$.
Now one checks that the Fourier coefficient \eqref{3.19} on
$\Delta(\tau,m)$ corresponds to a partition of $mn$, of the form
$(n+1,...)$, and, by Prop. \ref{prop 1.2}, \eqref{3.19} is
identically zero, for all nonzero $B$. A similar argument works for
proving the invariance of $\eta^{\psi_{a_e}}$ under left multiplication by the adele
points of elements of the form above with
\begin{equation}\label{3.19.1*}
g_1=\begin{pmatrix}I_e&a&0\\&I_{m-2e}&0\\&&I_e\end{pmatrix},\
g_2=\begin{pmatrix}I_e&a&0\\&I_{m-2e}&0\\&&I_e\end{pmatrix}.
\end{equation}
Note that each element $z(y)$ in \eqref{3.14.1} is a product of an element of the form $diag (I_m,\nu_e(z)^{\Delta_{n-1}})$
and an element $diag (g_1,g_2^{\Delta_{n-1}})$, where $g_1, g_2$ are as in \eqref{3.19.1*}. 
As we mentioned and explained before, \eqref{3.16} holds for all $g_1, g_2\in
\SL_m(\BA)$, satisfying \eqref{3.15.1}.
\end{proof}

\begin{cor}\label{cor 3.3}
For all $1\leq e\leq [\frac{m}{2}]$,
$$
\Lambda_e(f_{\Delta(\tau, m+i)\gamma_\psi^{(\epsilon)},s},\varphi_\sigma)=0.
$$
\end{cor}
\begin{proof}
Note that when $H_m$ is even orthogonal, we assume that $m\geq 4$, so that $U_e$ is a nontrivial unipotent radical of $H_m$, when $e\geq 1$.\\ 
By Prop. \ref{prop 3.2} and \eqref{3.15},\\
\\
$\Lambda_e(f_{\Delta(\tau, m+i)\gamma_\psi^{(\epsilon)},s},\varphi_\sigma)(h)=$
$$
\int_{C_2^{(\epsilon)}M^0_e(F)U_e(\BA)\backslash
H^{(\epsilon)}_m(\BA)}\varphi_\sigma^{U_e}(g)\int f^{\psi_{0;e}}_{\Delta(\tau,
m+i)\gamma_\psi^{(\epsilon)},s}(\tilde{\alpha}_0h_{\gamma_e}
ut(g,\hat{\beta}_eh))\psi^{-1}_H(u)dudg,
$$
where the $du$-integration is along $U_{m^{n-1}}^e(\BA)\backslash
	U_{m^{n-1}}(\BA)$, and $\varphi_\sigma^{U_e}$ denotes the constant term of
$\varphi_\sigma$ along the unipotent radical $U_e=U_e^{H_m}$. Since
$\sigma$ is cuspidal, this constant term is zero identically on
$\sigma$.
\end{proof}
Now, we can rewrite \eqref{3.12} as
\begin{equation}\label{3.20}
\mathcal{E}(f_{\Delta(\tau,
m+i)\gamma_\psi^{(\epsilon)},s},\varphi_\sigma)(h)=\sum_{h'\in Q_{ni}(F)\backslash
H_{m+2ni}(F)}\Lambda(f_{\Delta(\tau, m+i)\gamma_\psi^{(\epsilon)},s},\varphi_\sigma)(h'h),
\end{equation}
where\\
\\
$\Lambda(f_{\Delta(\tau,
m+i)\gamma_\psi^{(\epsilon)},s},\varphi_\sigma)(h)=\Lambda_0(f_{\Delta(\tau,
m+i)\gamma_\psi^{(\epsilon)},s},\varphi_\sigma)(h)=$
\begin{equation}\label{3.21}
\int_{C_2^{(\epsilon)}\backslash H^{(\epsilon)}_m(\BA)}\varphi_\sigma(g)\int_{U'_{m(n-1)}(\BA)}
f^\psi_{\Delta(\tau, m+i)\gamma_\psi^{(\epsilon)},s}(\delta_0ut(g,h))\psi^{-1}_H(u)dudg.
\end{equation}
Here,
$\delta_0=\tilde{\alpha}_0h_{\gamma_0}t(I_m,\hat{\beta}_0)$,
$f^\psi_{\Delta(\tau, m+i)\gamma_\psi^{(\epsilon)},s}=f^{\psi_{0;0}}_{\Delta(\tau,
m+i)\gamma_\psi^{(\epsilon)},s}$, and $U'_{m(n-1)}$ is the following subgroup of
$U_{m(n-1)}$, realizing the quotient $U^0_{m^{n-1}}\backslash
U_{m^{n-1}}$,
\begin{equation}\label{3.21.1}
U'_{m(n-1)}=\{u_{x;y}=\begin{pmatrix}I_{m(n-1)}&x&0&y\\&I_{m+ni}&0&0\\&&I_{m+ni}&x'\\&&&I_{m(n-1)}\end{pmatrix}^{\omega_0^{m(n-1)}}\in
H\}.
\end{equation}
Note that 
\begin{equation}\label{3.21.1'}
\delta_0=\epsilon^0h_{\gamma^0}, 
\end{equation}
with
\begin{equation}\label{3.21.2}
\epsilon^0=\begin{pmatrix}0&W_0&0&0\\0&0&0&I_{m(n-1)}\\\delta_HI_{m(n-1)}&0&0&0\\0&0&W_0^*&0\end{pmatrix}\omega_0^{m(n-1)},
\end{equation}
where, for $m=2m'$ even,
$$
W_0=\begin{pmatrix}0&I_{ni}&0\\I_{m'}&0&0\\0&0&I_{m'}\end{pmatrix},
$$
and, for $m=2m'-1$ odd,
$$
W_0=\begin{pmatrix}0&I_{ni}&0&0\\I_{m'-1}&0&0&0\\0&0&0&\frac{1}{2}\\0&0&I_{m'-1}&0\end{pmatrix};
$$
\begin{equation}\label{3.21.2}
h_{\gamma^0}=diag(I_{m(n-1)},\gamma^0,I_{m(n-1)}),
\end{equation}
with
$$
\gamma^0=\begin{pmatrix}I_{[\frac{m}{2}]}\\&I_{ni}\\&&I_{[\frac{m}{2}]}\\&&&I_{2(m-2[\frac{m}{2}])}\\I_{[\frac{m}{2}]}&0&0&&I_{[\frac{m}{2}]}\\
0&0&0&&&I_{ni}\\0&0&-\delta_HI_{[\frac{m}{2}]}&&&&I_{[\frac{m}{2}]}\end{pmatrix}.
$$
Note again that $\omega_0^{m(n-1)}=I$, unless $m$ is odd and $n$ is
even, where $\omega_0^{m(n-1)}=\omega_0$, and in this case
$\omega_0$ commutes with $h_{\gamma^0}$.
\begin{prop}\label{prop 3.4}
Let $q\in Q_{ni}^{H_{m+2ni}}(\BA)$ be of the form
$$
q=\begin{pmatrix}a&x&y\\&b&x'\\&&a^*\end{pmatrix},
$$
where $a\in \GL_{ni}(\BA)$, $b\in H_m(\BA)$. Assume that $H$ is linear. Then\\
\\
$\Lambda(f_{\Delta(\tau, m+i),s},\varphi_\sigma)(qh)=$
\begin{equation}\label{3.22}
|\det(a)|^{-m(n-1)}\int_{H_m(\BA)}\varphi_\sigma(b^\iota
g)\int_{U'_{m(n-1)}(\BA)} f^\psi_{\Delta(\tau,
m+i),s}(\hat{a}\delta_0ut(g,h))\psi^{-1}_H(u)dudg,
\end{equation}
where $b^\iota=J_0^{-1}bJ_0$ (see \eqref{1.10.1.1}, \eqref{3.3.1''}; Here,
$\hat{a}=diag(a, I_{2mn},a^*)$.)
\end{prop}
\begin{proof}
Write in \eqref{3.21}, with $qh$ instead of $h$,
$(g,qh)=(b^\iota,q)((b^\iota)^{-1}g,h)$ (see \eqref{3.3.1'}-\eqref{3.3.1'''}). Change
variable, in \eqref{3.21}, $g\mapsto b^\iota g$. We get\\
\\
$\Lambda(f_{\Delta(\tau, m+i)\gamma_\psi^{(\epsilon)},s},\varphi_\sigma)(qh)=$
$$
\int_{H_m(\BA)}\varphi_\sigma(b^\iota
g)\int_{U'_{m(n-1)}(\BA)} f^\psi_{\Delta(\tau,
m+i)\gamma_\psi^{(\epsilon)},s}(\delta_0ut(b^\iota,q)t(g,h))\psi^{-1}_H(u)dudg.
$$
Now, it is straightforward to check that the inner $du$-integral is equal to
$$
|\det(a)|^{-m(n-1)}\int_{U'_{m(n-1)}(\BA)} f^\psi_{\Delta(\tau,
m+i)\gamma_\psi^{(\epsilon)},s}((\diag(a,(b^\delta)^{\Delta_n})^\wedge\delta_0ut(g,h))\psi^{-1}_H(u)du.
$$
The factor $|\det(a)|^{-m(n-1)}$ comes from a change of variables
in $u$. The element $b^\delta$ is defined as follows. If $m=2m'$ is even, then
$$
b^\delta=\begin{pmatrix}&-I_{m'}\\I_{m'}\end{pmatrix}b\begin{pmatrix}&I_{m'}\\-I_{m'}\end{pmatrix},
$$
and if $m$ is odd, then $b^\delta=b^\iota$. The main part of the calculation is to check that
$$
\delta_0t(b^\iota,q)\delta_0^{-1}=diag(a,(b^\delta)^{\Delta_n},((b^\delta)^*)^{\Delta_n},a^*)u,
$$
where $u\in U_{(m+i)n}(\BA)$. By \eqref{3.16},
$$
f^\psi_{\Delta(\tau,
m+i)\gamma_\psi^{(\epsilon)},s}(\diag(a,(b^\delta)^{\Delta_n})^\wedge\alpha_0ut(g,h))=f^\psi_{\Delta(\tau,
m+i)\gamma_\psi^{(\epsilon)},s}(\hat{a}\alpha_0ut(g,h)).
$$
This proves the proposition.
\end{proof}
We record the following invariance property shown in the proof (with the same notations)
\begin{multline}\label{3'.22}
\int_{U'_{m(n-1)}(\BA)} f^\psi_{\Delta(\tau,
m+i)\gamma_\psi^{(\epsilon)},s}(\delta_0ut(b^\iota,q)t(g,h))\psi^{-1}_H(u)du\\
=|\det(a)|^{-m(n-1)}\int_{U'_{m(n-1)}(\BA)} f^\psi_{\Delta(\tau,
m+i)\gamma_\psi^{(\epsilon)},s}(\hat{a}\delta_0ut(g,h))\psi^{-1}_H(u)du.
\end{multline}
Let us write the version of the last proposition in the metaplectic case. Let $q$ be the matrix as in the proposition, and $h\in \Sp^{(2)}_{m+2ni}(\BA)$. Here $m=2m'$ is even. Then the analogue of \eqref{3.22} is\\
\\
$\Lambda(f_{\Delta(\tau, m+i)\gamma_\psi,s},\varphi_\sigma)(p((q,\bar{\alpha}))h)=(\det(a),x(b))|\det(a)|^{-m(n-1)}$
\begin{equation}\label{3.22.1}
\int_{C_2\backslash \Sp_m^{(2)}(\BA)}\varphi_\sigma(p((b,\bar{\alpha}))^\iota
g)\int f^\psi_{\Delta(\tau,
	m+i)\gamma_\psi,s}((\hat{a},1)\delta_0ut^{(2)}(g,h))\psi^{-1}_H(u)dudg,
\end{equation}
Here, $u$ is integrated over $U'_{m(n-1)}(\BA)$; $p((b,\bar{\alpha}))^\iota$ is given by \eqref{1.10.2.4}, ($p$ is given in \eqref{1.x.1}).
Recall that $(\hat{a},1)$ denotes the element $p(\Pi'_v(\hat{a}_v,1))$. See right after \eqref{1.1.1'}.
The main point in showing \eqref{3.22.1} is to carry out precisely in the metaplectic group the following conjugation at each place $v$,
$$
(\delta_0,1)(t(b_v^\iota,q_v),(x(b_v^\iota),\det(a_v)x(b_v)))(\delta_0,1)^{-1},
$$
and show that it is equal to
$$
(\delta_0t(b_v^\iota,q_v)\delta_0^{-1},c_v(\bar{u}_0,j(b_v^\iota,b_v))(x(b_v^\iota),x(b_v))(\det(a_v),x(b_v))).
$$ 
Let $\sigma^\iota$ be the cuspidal representation of $H^{(\epsilon)}_m(\BA)$
realized in the space of functions $b\mapsto
\varphi_\sigma(b^\iota)$, as $\varphi_\sigma$ varies in the space of
$\sigma$. 

For the next proposition, we give a formal proof. We will prove it later for $K_{H(\BA)})$-finite sections starting in Prop. \ref{prop 3.6}. We will complete the proof in the next section.
\begin{prop}\label{prop 3.5}
The function on $H^{(\epsilon)}_{m+2ni}(\BA)$, $\Lambda(f_{\Delta(\tau,
m+i)\gamma_\psi^{(\epsilon)},s},\varphi_\sigma)(h)$, defined for $\Re(s)$ sufficiently
large by the integral \eqref{3.21}, lies in the space of
\begin{equation}\label{3.22.1*}
\rho_{\Delta(\tau,i)\gamma_\psi^{(\epsilon)},\sigma^\iota,s}=\Ind_{Q^{(\epsilon)}_{ni}(\BA)}^{H^{(\epsilon)}_{m+2ni}(\BA)}\Delta(\tau,i)\gamma_\psi^{(\epsilon)}
|\det\cdot|^s \times \sigma^\iota.
\end{equation}
\end{prop}
\begin{proof}
Consider the r.h.s. of
\eqref{3.22}, or \eqref{3.22.1}, for $\Re(s)$, sufficiently large, so that the integral
converges absolutely. It is smooth in $h$ and of moderate growth. Fix $h$ in 
$K_{H^{(\epsilon)}_{m+2ni}(\BA)}$ (see \eqref{1.1*}, \eqref{1.1.a}), and then consider the integral as a function of
$(a,b)\in \GL^{(\epsilon)}_{ni}(\BA)\times H^{(\epsilon)}_m(\BA)$. It is left -
$\GL_{ni}(F)\times H_m(F)$ invariant. For any fixed $h_0\in H(\BA)$, we claim that, in the linear case (resp. the metaplectic case) the following function
on $\GL_{(m+i)n}(\BA)$ (resp. $\GL^{(2)}_{(m+i)n}(\BA)$), $A\mapsto f^\psi_{\Delta(\tau,
m+i),s}(\hat{A}h_0)$ (resp. $(A,\mu)\mapsto f^\psi_{\Delta(\tau,
m+i)\gamma_\psi,s}((\hat{A},\mu)h_0)$)
factors through the constant term along the
unipotent radical $V_{ni,mn}$ (of the parabolic subgroup
$P_{ni,mn}$ of $\GL_{(m+i)n}$). Indeed, recall that $f^\psi_{\Delta(\tau,
m+i)\gamma_\psi^{(\epsilon)},s}$ is the composition of the section $f_{\Delta(\tau,
m+i)\gamma_\psi^{(\epsilon)},s}$, with the Fourier coefficient, applied to the elements of $\Delta(\tau,m+i)$, along the unipotent radical
$V_{ni,m^n}$, with respect to the character $\psi_{0;0}$ defined in \eqref{3.8} (with $A=0$, $e=0$, $a_e=I_m$). The character $\psi_{0;0}$ is trivial on $V_{ni,mn}(\BA)$, and hence 	our claim.	Again, this constant term is
applied to $\Delta(\tau,m+i)$. Hence the following function on
$\GL_{ni}(\BA)$ (resp. $\GL_{ni}^{(2)}(\BA)$), $a\mapsto f^\psi_{\Delta(\tau, m+i),s}(\hat{a}h_0)$
(resp. $(a,\mu)\mapsto f^\psi_{\Delta(\tau, m+i)\gamma_\psi,s}((\hat{a},\mu)h_0)$) 
lies in $|\det(a)|^{s+\frac{(m+i)n-\delta_H}{2}}$ times a
function obtained as $\phi(diag(a,I_{mn}))$ (resp. $\phi((diag(a,I_{mn}),\mu))$ )
where $\phi$ lies in the constant term of $\Delta(\tau,m+i)\gamma_\psi^{(\epsilon)}$ along
$V_{ni,mn}$.  In fact, the function on $\GL^{(\epsilon)}_{ni}(\BA)\times \GL^{(\epsilon)}_{mn}(\BA)$, $\phi(diag(a_1,a_2))$
(resp.  $((a_1,\mu_1),(a_2,\mu_2))\mapsto \phi((diag(a_1,a_2),\mu_1\mu_2(\det(a_1)\det(a_2)))$) lies in $(|\det(a_1)||\det(a_2)|)^{s+\frac{(m+i)n-\delta_H}{2}}$ times
\begin{equation}\label{3.22.00}
\delta_{P_{ni,mn}}^{\frac{1}{2}}(\Delta(\tau,i)\gamma_\psi^{(\epsilon)}|\det\cdot |^{-\frac{m}{2}}\otimes \Delta(\tau,m)\gamma_\psi^{(\epsilon)}|\det\cdot|^{\frac{i}{2}}).
\end{equation}
This follows from \cite{JL13}, Lemma 4.1. Altogether, taking into account the factor
$|\det\cdot|^{-m(n-1)}$ in front of \eqref{3.22}, we get that, as a function on $\GL^{(\epsilon)}_{ni}(\BA)$, we get an element of
the space of
\begin{equation}\label{3.22*}
(\delta_{Q_{ni}^{H_{m+2ni}}}^{\frac{1}{2}})_{|\GL_{ni}(\BA)}|\det\cdot|^s\Delta(\tau,i)\gamma_\psi^{(\epsilon)}.
\end{equation}
Now, we see, formally, that the r.h.s. of \eqref{3.22} defines an element of
\begin{equation}\label{3.22**}
\delta_{Q_{ni}^{H_{m+2ni}}}^{\frac{1}{2}}\Delta(\tau,i)\gamma_\psi^{(\epsilon)}|\det\cdot|^s\otimes
\sigma^\iota. 
\end{equation}
Indeed, in the integrand of the r.h.s. of \eqref{3.22} (resp. \eqref{3.22.1}), we see the product\\
 $\varphi_\sigma(b^\iota
g) f^\psi_{\Delta(\tau,m+i),s}(\hat{a}\delta_0ut(g,h))$ (resp.\\
$\varphi_\sigma(p((b,\bar{\alpha}))^\iota
g)f^\psi_{\Delta(\tau,m+i)\gamma_\psi,s}((\hat{a},1)\delta_0ut^{(2)}(g,h))$).
The first factor, as a function of $b$ (resp. $p((b,\bar{\alpha}))$), is, for fixed $g$, a cusp form in the space of
$\sigma^\iota$. The second factor, as a function of $a$, is, as we just explained, for fixed $u,g,h$, an element of \eqref{3.22*}. Thus, formally, the integral on the r.h.s. of \eqref{3.22} (resp. \eqref{3.22.1}) is an element of \eqref{3.22**}. Finally, since the function $q\mapsto \Lambda(f_{\Delta(\tau,
m+i)\gamma_\psi^{(\epsilon)},s},\varphi_\sigma)(qh)$, for $q\in Q_{ni}^{(\epsilon),H_{m+2ni}}(\BA)$ is left
$U_{ni}^{H_{m+2ni}}(\BA)$ - invariant, we get that the function $\Lambda(f_{\Delta(\tau,
m+i)\gamma_\psi^{(\epsilon)},s},\varphi_\sigma)$ lies in the induced representation $\rho_{\Delta(\tau,i)\gamma_\psi^{(\epsilon)},\sigma^\iota,s}$.
\end{proof}
\begin{prop}\label{prop 3.6}
Assume that $H$ is linear. Assumptions and notations being as in Prop. \ref{prop 3.5}, assume further that $f_{\Delta(\tau, m+i),s}$ is $K_{H(\BA)}$ finite. Fix  $\kappa\in K_{H_{m+2ni}(\BA)}$. Consider  the following function on 
$\GL_{ni}(\BA)\times H_m(\BA)$,
\begin{equation}\label{3.23}
(a,b)\mapsto\Lambda(f_{\Delta(\tau, m+i),s},\varphi_\sigma)(\begin{pmatrix}a\\&b\\&&a^*\end{pmatrix}\kappa),
\end{equation}
Then the function above can be expressed as a finite sum of terms of the form
\begin{multline}\label{3.24}
|\det(a)|^{-m(n-1)}\xi_\sigma(b^\iota)\cdot \\
\cdot \int_{H_m(\BA)}< \sigma(g)\varphi_\sigma,\xi'_\sigma>\int_{U'_{m(n-1)}(\BA)} 
\tilde{f}^\psi_{\Delta(\tau,m+i),s}(\hat{a}\delta_0ut(g,\kappa))\psi^{-1}_H(u)dudg,
\end{multline}
where $\xi_\sigma$, $\xi'_\sigma$ lie in the space of $\sigma$ and $\tilde{f}_{\Delta(\tau,m+i),s}$ is a $K_{H(\BA)}$-finite and holomorphic section of $\rho_{\Delta(\tau,i),\sigma^\iota, s}$. Also, $< \sigma(g)\varphi_\sigma,\xi'_\sigma>$ denotes the $L^2$-inner product between the two cusp forms $ \sigma(g)\varphi_\sigma$ and $\xi'_\sigma$.
In particular, for $a$ and $h$ fixed, the function of $b$ given by \eqref{3.23}, lies in the space of $\sigma^\iota$.
\end{prop}
\begin{proof}
By \eqref{3.22}, it is enough to consider the function
\begin{multline}\label{3.25}
\Lambda'(f_{\Delta(\tau, m+i),s},\varphi_\sigma)(a,b)=\\
\int_{H_m(\BA)}\varphi_\sigma(b^\iota
g)\int_{U'_{m(n-1)}(\BA)} f^\psi_{\Delta(\tau,
m+i),s}(\hat{a}\delta_0ut(g,\kappa))\psi^{-1}_H(u)dudg.
\end{multline}
Let  $k\in K_{H_m(\BA)}$ . Change variable in \eqref{3.25}, $g\mapsto k^\iota g$, and integrate over $K_{H_m(\BA)}$. We may take the measure of $H_m(\BA)$ to be one. Then \eqref{3.25} becomes
\begin{multline}\label{3.26}
\Lambda'(f_{\Delta(\tau, m+i),s},\varphi_\sigma)(a,b)=\\
\int_{H_m(\BA)}\int_{K_{H_m(\BA)}}\varphi_\sigma(b^\iota k^\iota
g)\int_{U'_{m(n-1)}(\BA)} f^\psi_{\Delta(\tau,
m+i),s}(\hat{a}\delta_0ut(g,\tilde{k}\kappa))\psi^{-1}_H(u)dudkdg
\end{multline}
where $\tilde{k}=diag(I_{ni},k^{-1},I_{ni})$. We used \eqref{3'.22}. 
Since $f_{\Delta(\tau, m+i),s}$ is $K_{H(\BA)}$ - finite, we may write as a finite sum,
\begin{equation}\label{3.26.1}
f^\psi_{\Delta(\tau,m+i),s}(\hat{a}\delta_0ut(g,\tilde{k}\kappa))=\sum c(k)\tilde{f}^\psi_{\Delta(\tau,m+i),s}(\hat{a}\delta_0ut(g,\kappa)),
\end{equation}
where the functions $c$ are matrix coefficients of finite dimensional representations of $K_{H_m(\BA)}$. Substituting this in \eqref{3.26}, it is enough to consider
\begin{equation}\label{3.27}
\int_{H_m(\BA)}(\int_{K_{H_m(\BA)}}c(k^\iota)\varphi_\sigma(b^\iota k
g)dk)\int_{U'_{m(n-1)}(\BA)} \tilde{f}^\psi_{\Delta(\tau,
m+i),s}(\hat{a}\delta_0ut(g,\kappa))\psi^{-1}_H(u)dudg.
\end{equation}
Now, write as a finite sum,
\begin{equation}\label{3.26.2}
\int_{K_{H_m(\BA)}}c(k^\iota)\sigma(k)(\sigma(g)\varphi_\sigma)dk=\sum \beta(\sigma(g)\varphi_\sigma,\xi_\sigma,c)\xi_\sigma,
\end{equation}
where
\begin{equation}\label{3.26.3}
\beta(\sigma(g)\varphi_\sigma,\xi_\sigma,c)=\int_{K_{H_m(\BA)}}c(k^\iota)<\sigma(kg)\varphi_\sigma,\xi_\sigma>dk=
<\sigma(g)\varphi_\sigma,\xi'_\sigma>,
\end{equation}
with
\begin{equation}\label{3.26.4}
\xi'_\sigma=\int_{K_{H_m(\BA)}}\bar{c(k^\iota)}\sigma(k^{-1})\xi_\sigma dk.
\end{equation}
Thus, \eqref{3.27}, and hence \eqref{3.26} is equal to a finite sum of the form we want, that is
$$
\sum \xi_\sigma(b^\iota) \int_{H_m(\BA)} <\sigma(g)\varphi_\sigma,\xi'_\sigma>\int_{U'_{m(n-1)}(\BA)} \tilde{f}^\psi_{\Delta(\tau,
m+i),s}(\hat{a}\delta_0ut(g,\kappa))\psi^{-1}_H(u)dudg.
$$
\end{proof}
As in \eqref{3.22.1}, the analog in the metaplectic case of Prop. \ref{prop 3.6} holds and is proved in the same way. Thus, assume that $f_{\Delta(\tau, m+i)\gamma_\psi,s}$ is $K_{H(\BA)}$ - finite. Fix  $\kappa\in K_{H^{(2)}_{m+2ni}(\BA)}$. Then 
$$
\Lambda(f_{\Delta(\tau, m+i)\gamma_\psi,s},\varphi_\sigma)(p((\begin{pmatrix}a\\&b\\&&a^*\end{pmatrix},\bar{\alpha}))\kappa).
$$
is equal to a finite sum of terms of the form
\begin{multline}\label{3.24.1}
(\det(a),x(b))|\det(a)|^{-m(n-1)}\xi_\sigma(p((b,\bar{\alpha}))^\iota)\cdot \\
\cdot \int_{C_2\backslash H^{(2)}_m(\BA)}< \sigma(g)\varphi_\sigma,\xi'_\sigma>\int_{U'_{m(n-1)}(\BA)} 
\tilde{f}^\psi_{\Delta(\tau,m+i)\gamma_\psi,s}((\hat{a},1)\delta_0ut(g,\kappa))\psi^{-1}_H(u)dudg,
\end{multline}
with similar notations, as in the last proposition.

\begin{rmk}\label{rmk 3.6}
In the proof of the last proposition, we  needed only that $f_{\Delta(\tau,m+i)\gamma_\psi,s}$ be $t(I\times K_{H_{m+2ni}^{(\epsilon)}(\BA)})$-finite.
\end{rmk}
\begin{rmk}\label{rmk 3.6.1}
Note the case $i=0$ of \eqref{3.24}, \eqref{3.24.1}. Here, $\Lambda(f_{\Delta(\tau, m+i)\gamma_\psi,s},\varphi_\sigma)$ lies in the space of $\sigma^\iota$. In this case, $\Lambda(f_{\Delta(\tau, m+i)\gamma_\psi,s},\varphi_\sigma)(b)$ ($b\in H_m^{(\epsilon)}(\BA)$) is equal to a finite sum of terms of the form
\begin{multline}\label{3.24.1*}
\xi_\sigma(b^\iota)\cdot \\
\cdot \int_{C^{(\epsilon)}_2\backslash H^{(\epsilon)}_m(\BA)}< \sigma(g)\varphi_\sigma,\xi'_\sigma>\int_{U'_{m(n-1)}(\BA)} 
\tilde{f}^\psi_{\Delta(\tau,m)\gamma^{(\epsilon)}_\psi,s}(\delta_0ut(g,1))\psi^{-1}_H(u)dudg.
\end{multline}
\end{rmk}

In order to complete the proof (rigorously) the proof of Proposition \ref{prop 3.5} (for $K_{H(\BA)}$ - finite sections), it remains to show that the integrals on the r.h.s. of \eqref{3.24}, (resp. \eqref{3.24.1}), as functions of $a$ (resp. $(a,\pm 1)$), lie in the space of the representation \eqref{3.22*}. This will follow from Cor. \ref{cor 4.5} in the next section. 

\begin{prop}\label{prop 3.7}
Let $f_{\Delta(\tau,m+i)\gamma_\psi^{(\epsilon)},s}$ be a smooth holomorphic section of  $\rho_{\Delta(\tau,
m+i)\gamma_\psi^{(\epsilon)},s}$. Then the function on $H^{(\epsilon)}_{m+2ni}(\BA)$, $\Lambda(f_{\Delta(\tau,
m+i)\gamma_\psi^{(\epsilon)},s},\varphi_\sigma)(h)$, defined for $\Re(s)$ sufficiently
large by the integral \eqref{3.21}, admits an analytic continuation
to a meromorphic function of $s$ in the whole plane. 
\end{prop}
\begin{proof}
We give the proof for $H=H_{2mn}=\Sp_{2nm}$. This is just for simplicity of notations. The proof in the other cases is entirely the same, except for additional notations, taking care of double covers, or ensuring that matrices have determinant one. We saw in Prop. \ref{prop 3.4} that $\Lambda(f_{\Delta(\tau,
m+i),s},\varphi_\sigma)$ is left $U^{H_{m+2ni}}_{ni}(\BA)$-invariant, and we have the identities \eqref{3.22}, \eqref{3.22*}. Now, \eqref{3.20} exhibits, for $Re(s)$, sufficiently large, $\mathcal{E}(f_{\Delta(\tau,
m+i),s},\varphi_\sigma)(h)$ as an "Eisenstein summation", corresponding to $\Lambda(f_{\Delta(\tau,
m+i),s},\varphi_\sigma)$. Hence the usual calculation of the constant term of the series in the r.h.s. of \eqref{3.20}, along the unipotent radical $U^{H_{m+2ni}}_{ni}$, works. See \cite{MW95}, II.1.7. We now recall this briefly, and to shorten notation, we will denote in what follows, $Q_{ni}=Q^{H_{m+2ni}}_{ni}$, $U_{ni}=U^{H_{m+2ni}}_{ni}$, etc. The calculation is given by
\begin{multline}\label{3.28}
\mathcal{E}^{U_{ni}}(f_{\Delta(\tau,
m+i),s},\varphi_\sigma)(h)=\\
\sum_w\sum_{\gamma\in M^w_{ni}(F)\backslash M_{ni}(F)}\int_{U_{ni}^w(F)\backslash U_{ni}(\BA)}\Lambda(f_{\Delta(\tau,
m+i),s},\varphi_\sigma)(wu \hat{\gamma} h)du,
\end{multline}
where the outer sum is over $Q_{ni}(F)\backslash H_{m+2ni}(F)/Q_{ni}(F)$, \\
$M_{ni}^w=M_{ni}(F)\cap w^{-1}M_{ni}(F)w$, $U_{ni}^w(F)=U_{ni}(F)\cap w^{-1}Q_{ni}(F)w$. Here is a set of representatives $w$ of the double cosets $Q_{ni}(F)\backslash H_{m+2ni}(F)/Q_{ni}(F)$. It is taken from \cite{GRS11}, Sec. 4.2. They are parametrized by pairs of integers $(r,r')$, such that $0\leq r\leq r'\leq ni$. The corresponding representative is
$$
w_{r,r'}=\begin{pmatrix}I_r\\&0&0&I_{r'-r}&0&0&0&0&\\&0&0&0&0&0&-I_{ni-r'}&0\\&I_{r'-r}&0&0&0&0&0&0&\\&0&0&0&I_{m-2(r'-r}&0&0&0&\\
&0&0&0&0&0&0&I_{r'-r}\\&0&I_{ni-r'}&0&0&0&0&0&\\&0&0&0&0&I_{r'-r}&0&0&\\&&&&&&&&I_r\end{pmatrix}.
$$ 
The subgroups $M_{ni}^{w_{r,r'}}$, $U_{ni}^{w_{r,r'}}$ are described in \cite{GRS11}, Sec. 4.3. Factor the $du$-integration in \eqref{3.28} through $U_{ni}^{w_{r,r'}}(\BA)$. We check that the projection of $w_{r,r'}U_{ni}^{w_{r,r'}}w_{r,r'}$ on the Levi subgroup $M_{ni}$ is the direct product of the unipotent radicals $\widehat{V}_{r,r'-r,ni-r'}$ and $U_{r'-r}^{H_m}$ inside $\widehat{\GL}_{ni} \times H_m$. It follows that  the $du$-integration in \eqref{3.28} is equal to
$$
\int_{U_{ni}^{w_{r,r'}}(\BA)\backslash U_{ni}(\BA)}[\Lambda(f_{\Delta(\tau,
m+i),s},\varphi_\sigma)]^{V_{r,r'-r,ni-r'},U_{r'-r}^{H_m}}(w_{r',r}u \gamma h)du,
$$
where \\
\\
$[\Lambda(f_{\Delta(\tau,m+i)\gamma_\psi^{(\epsilon)},s},\varphi_\sigma)]^{V_{r,r'-r,ni-r'},U_{r'-r}^{H_m}}(h')=$
\begin{equation}\label{3.29}
\int \int\Lambda(f_{\Delta(\tau,m+i),s},\varphi_\sigma)
(\begin{pmatrix}v\\&u'\\&&v^*\end{pmatrix}h')dvdu'.
\end{equation}
Here, the $dv$-integration is over $V_{r,r'-r,ni-r'}(F)\backslash V_{r,r'-r,ni-r'}(\BA)$, and the $du'$-integration is over $U_{r'-r}^{H_m}(F)\backslash U_{r'-r}^{H_m}(\BA)$. If $U^{H_m}_{r'-r}$ is trivial, omit it from the super index in the l.h.s. of \eqref{3.29}. By \eqref{3.22}, \eqref{3.29} is equal to 
\begin{equation}\label{3.30}
\int_{H_m(\BA)}\varphi_\sigma^{(U_{r'-r}^{H_m})^\iota}(g)\int_{U'_{m(n-1)}(\BA)} f^{\psi,V_{r,r'-r,ni-r'}}_{\Delta(\tau,
m+i),s}(\delta_0ut(g,h'))\psi^{-1}_H(u)dudg,
\end{equation}
where $\varphi_\sigma^{(U_{r'-r}^{H_m})^\iota}$ is the constant term of $\varphi_\sigma$ along the unipotent radical $(U_{r'-r}^{H_m})^\iota$, and 
\begin{equation}\label{3.31}
f^{\psi,V_{r,r'-r,ni-r'}}_{\Delta(\tau,m+i),s}(x)=\int_{V_{r,r'-r,ni-r'}(F)\backslash V_{r,r'-r,ni-r'}(\BA)}f^\psi_{\Delta(\tau,m+i),s}(\hat{v}x)dv.
\end{equation}
Since $\sigma$ is cuspidal, we get that only $r'=r$ contribute to the constant term \eqref{3.28}. Thus, let $r'=r$. Next, by \eqref{3.31}, \eqref{3.22*}, $f^{\psi,V_{r,ni-r}}_{\Delta(\tau,m+i),s}$ is the composition of $f^\psi_{\Delta(\tau,m+i),s}$, and the constant term along $\Delta(\tau,i)$, along $V_{r,ni-r}$. Since $\tau$ is cuspidal, these constant terms are zero, unless $r=nj$, where $j$ is an integer, such that $0\leq j\leq i$. In this case, we know from \cite{JL13}, Lemma 4.1, that, for $x\in H_{m+2ni}(\BA)$, the function on $\GL_{ni}(\BA)$, given by $a\mapsto f^{\psi,V_{nj,n(i-j)}}_{\Delta(\tau,m+i),s}(\hat{a}x)$ lies in $|\det\cdot|^{s+\frac{1}{2}{n(m+i)+1}}$ times
\begin{equation}\label{3.32}
\Ind_{P_{nj,n(i-j)}(\BA)}^{\GL_{ni}(\BA)}|\det\cdot|^{-\frac{i-j}{2}}\Delta(\tau,j)\times |\det\cdot|^{\frac{j}{2}}\Delta(\tau,i-j).
\end{equation}
Denote $\epsilon_j=w_{nj,nj}$. Then $M_{ni}^{\epsilon_j}$ is the subgroup of matrices
$$
M_{ni}^{\epsilon_j}=\{\begin{pmatrix}a\\&b\\&&a^*\end{pmatrix}\  | \  a\in P_{nj,n(i-j)},\ b\in H_m. \},
$$
and
$$
U_{ni}^{\epsilon_j}=\{\begin{pmatrix}I_{nj}&0&y&z_1&z_2\\&I_{n(i-j)}&0&0&z'_1\\&&I_m&0&y'\\&&&I_{n(i-j)}&0\\&&&&I_{nj}\end{pmatrix}\in H_{m+2ni}\}.
$$
Note that $M_{ni}^{\epsilon_i}=M_{ni}$, and $U_{ni}^{\epsilon_i}=U_{ni}$
Define, for $Re(s)$ sufficiently large, 
\begin{equation}\label{3.33}
M_{\epsilon_j}(\Lambda(f_{\Delta(\tau,m+i),s},\varphi_\sigma))(h')=\int_{U_{ni}^{\epsilon_j}(\BA)\backslash U_{ni}(\BA)}[\Lambda(f_{\Delta(\tau,
m+i),s},\varphi_\sigma)]^{V_{nj,n(i-j)}}(\epsilon_ju h')du.
\end{equation}
In particular, 
$$
M_{\epsilon_i}(\Lambda(f_{\Delta(\tau,m+i),s},\varphi_\sigma))=\Lambda(f_{\Delta(\tau,m+i),s},\varphi_\sigma).
$$
Then \eqref{3.28} becomes
\begin{equation}\label{3.34}
\mathcal{E}^{U_{ni}}(f_{\Delta(\tau,
m+i),s},\varphi_\sigma)(h)=\\
\sum_{j=0}^i\sum_{\gamma\in M^{\epsilon_j}_{ni}(F)\backslash M_{ni}(F)}M_{\epsilon_j}(\Lambda(f_{\Delta(\tau,m+i),s},\varphi_\sigma))(\hat{\gamma}h),
\end{equation}
Let $t\in \BA^*$. Then using \eqref{3.33} and \eqref{3.32}, we have
\begin{equation}\label{3.35}
M_{\epsilon_j}(\Lambda(f_{\Delta(\tau,m+i),s},\varphi_\sigma)(\hat{tI}_{ni}h')=|t|^{(2j-i)ns+k_{i,j}}\omega^{\nu_{i,j}}_\tau(t)M_{\epsilon_j}(\Lambda(f_{\Delta(\tau,m+i),s},\varphi_\sigma))(h'),
\end{equation}
where $k_{i,j}$ are certain integers that can be easily computed, $\nu_{i,j}=0$, when $0<j<i$, $\nu_{i,0}=-1$, and $\nu_{i,i}=1$.
Denote, for $Re(s)$ sufficiently large, and $0\leq j\leq i-1$,
$$
E_j(M_{\epsilon_j}(\Lambda(f_{\Delta(\tau,m+i),s},\varphi_\sigma)),h)=\sum_{\gamma\in M^{\epsilon_j}_{ni}(F)\backslash M_{ni}(F)}M_{\epsilon_j}(\Lambda(f_{\Delta(\tau,m+i),s},\varphi_\sigma))(\hat{\gamma}h).
$$
Then, by \eqref{3.34}, \eqref{3.35}, for $Re(s)$, sufficiently large, and for $t\in \BA^*$,
\begin{multline}\label{3.36}
\mathcal{E}^{U_{ni}}(f_{\Delta(\tau,m+i),s},\varphi_\sigma)(\hat{tI}_{ni}h)=|t|^{-ins+k_{i,i}}\omega_\tau(t)\Lambda(f_{\Delta(\tau,m+i),s},\varphi_\sigma))(h)+\\
+\sum_{j=0}^{i-1}|t|^{(2j-i)ns+k_{i,j}}\omega^{\nu_{i,j}}_\tau(t)E_j(M_{\epsilon_j}(\Lambda(f_{\Delta(\tau,m+i),s},\varphi_\sigma)),h).
\end{multline}
Choose $t_0,...,t_i\in \BA^*$, such that $|t_0|,...,|t_i|$ are $i+1$ different positive numbers. Substitute $t=t_\ell$, $0\leq \ell\leq i$ in \eqref{3.36}, and view the resulting $i+1$ equalities as a linear system of $i+1$ equations in the unknowns $\Lambda(f_{\Delta(\tau,m+i),s},\varphi_\sigma))(h)$, $M_{\epsilon_j}(\Lambda(f_{\Delta(\tau,m+i),s},\varphi_\sigma)),h)$, $0\leq j\leq i-1$. Given a right half plane, we may choose an open domain in it, where the determinant of the coefficient matrix is nonzero everywhere, and there, use Cramer's rule to express each unknown, and, in particular, $\Lambda(f_{\Delta(\tau,m+i),s},\varphi_\sigma))(h)$, as a linear combination of $\mathcal{E}^{U^{H_{m+2ni}}_{ni}}(f_{\Delta(\tau,m+i),s},\varphi_\sigma)(\hat{t_\ell I}_{ni}h)$, $0\leq \ell\leq i$, with coefficients which are linear combinations of exponentials. Since  $\mathcal{E}(f_{\Delta(\tau,m+i),s},\varphi_\sigma)$ is meromorphic in the complex plane, we get the meromorphic continuation of $\Lambda(f_{\Delta(\tau,m+i),s},\varphi_\sigma))$ to the whole plane.

\end{proof}

\section{The section $\Lambda(f_{\Delta(\tau, m+i)\gamma_\psi^{(\epsilon)},s},\varphi_\sigma)$ and its relation to the doubling
integrals for $\sigma\times \tau$}

Our main goal in this section is to relate $\Lambda(f_{\Delta(\tau, m+i)\gamma_\psi^{(\epsilon)},s},\varphi_\sigma)$ as a
meromorphic function to the generalized doubling integrals for $\sigma\times \tau$ of \cite{CFGK17}. Prop. \ref{prop 3.6} tells us that 
$\Lambda(f_{\Delta(\tau, m+i)\gamma_\psi^{(\epsilon)},s},\varphi_\sigma)$, evaluated at the identity, is a linear combination of the integrals \eqref{3.24}, or \eqref{3.24.1} (with $a=I_{ni}$, and $b=I_m$). These integrals resemble the integrals of the generalized doubling method (after the unfolding process), except that the unipotent inner $du$-integration is on a larger group $U'_{m(n-1)}(\BA)$. See \cite{CFGK17}, Theorem 1. This will be made precise soon in Prop. \ref{prop 4.1}. Our first goal is to consider the inner $du$-integral just mentioned and show how to "shrink" $U'_{m(n-1)}$, so that we can express the $du$-integration as a finite linear combination of similar integrals, but along a unipotent subgroup of $U'_{m(n-1)}$ which can be identitfied with the one which appears in the integrals of the generalized doubling method for $\sigma\times \tau$. Thus, let us consider, for $Re(s)$ sufficiently large, for $a\in \GL^{(\epsilon)}_{ni}(\BA)$, and for $g\in H_m^{(\epsilon)}(\BA)$,
\begin{equation}\label{4'.0}
\ell_\psi(f_{\Delta(\tau,
m+i)\gamma_\psi^{(\epsilon)},s})(a,g)=\int_{U'_{m(n-1)}(\BA)} f^\psi_{\Delta(\tau,
m+i)\gamma_\psi^{(\epsilon)},s}(\hat{a}\delta_0ut(g,1))\psi^{-1}_H(u)du.
\end{equation}
In the metaplectic case, we use the identification of $\GL^{(2)}_{ni}(\BA)$ and $M_{ni}^{(2)}(\BA)$, detailed right after \eqref{1.1.1'}. Recall, that in this case, for $a=(a',\mu)\in \GL_{ni}^{(2)}(\BA)$, we denote $\hat{a}=(\hat{a'},\mu)$ (which is short for $p(\Pi'_v(\hat{a'}_v,\mu_v))$, where $\mu_v=1$, for almost all $v$, and $\Pi_v\mu_v=\mu$.)  We keep the presence of $a$, since we also want to complete the proof of Prop. \ref{prop 3.5}, and show that, as a function of $a$, the last integral lies in \\ 
$\delta_{Q_{ni}^{H_{m+2ni}}}^{\frac{1}{2}}(\hat{a})|\det(a)|^{s+m(n-1)}\Delta(\tau,i)\gamma_\psi^{(\epsilon)}$. By \eqref{3'.22} and \eqref{3.22.1}, \eqref{4'.0} is equal, in the linear case, to
\begin{multline}\label{4'.1}
\ell_\psi(f_{\Delta(\tau,
m+i),s})(a,g)=\\
\int_{U'_{m(n-1)}(\BA)} f^\psi_{\Delta(\tau,
m+i),s}(\hat{a}\delta_0ut(I,diag(I_{ni},\iota^{-1}(g^{-1}),I_{ni})))\psi^{-1}_H(u)du,
\end{multline}
where $\iota^{-1}(g^{-1})=J_0g^{-1}J_0^{-1}$. See \eqref{1.10.1'}. Note that when $m$ is even, $\iota^{-1}(g^{-1}) =(g^{-1})^\iota=(g^\iota)^{-1}$. In the metaplectic case, for $a\in \GL^{(2)}_{ni}(\BA)$, $(g,\bar{\alpha})\in \widetilde{\Sp}_m(\BA)$, 
\begin{multline}\label{4'.2}
\ell_\psi(f_{\Delta(\tau,
m+i)\gamma_\psi,s})(a,p((g,\bar{\alpha}))=\\
\int_{U'_{m(n-1)}(\BA)} f^\psi_{\Delta(\tau,
m+i)\gamma_\psi,s}(\hat{a}\delta_0ut(1,p(\Pi'_v(diag(I_{ni},g_v^\iota,I_{ni}),\alpha_v\mu(g_v))^{-1})\psi^{-1}_H(u)du,
\end{multline}
where $\mu(g_v)=c_v(\bar{u}_0,j(g_v,g_v^\iota))(x(g_v),x(g_v^\iota))$. See \eqref{1.x.1}, \eqref{1.10.2.1}. 

It will be convenient to consider $\ell_\psi(f_{\Delta(\tau,
m+i)\gamma_\psi^{(\epsilon)},s})$ via a product of analogous local integrals over the places of $F$. We now explain how this is done. It follows from \eqref{3.8} and  \eqref{3.22.00} that for a fixed $x\in H(\BA)$, the function on $\GL^{(\epsilon)}_{(m+i)n}(\BA)$ given by $d\mapsto f^\psi_{\Delta(\tau,
m+i)\gamma_\psi^{(\epsilon)},s}(\hat{d}x)$  lies in 
\begin{equation}\label{4'.3}
|\det\cdot|^{s+\frac{(m+i)n-\delta_H}{2}}\cdot \Ind_{P^{(\epsilon)}_{ni,m}(\BA)}^{\GL_{(m+i)n}^{(\epsilon)}(\BA)}(\Delta(\tau,i)\gamma_\psi^{(\epsilon)}|\det\cdot|^{-\frac{m}{2}}\times
W_{\psi_{V_{m^n}}}(\Delta(\tau,m))\gamma_\psi^{(\epsilon)}|\det\cdot|^{\frac{i}{2}}),
\end{equation}
where $W_{\psi_{V_{m^n}}}$ denotes the model of $\Delta(\tau,m)$ corresponding to the Fourier coefficient along $V_{m^n}$ with respect to the following character of $V_{m^n}(\BA)$,
\begin{equation}\label{4.0}
\psi^{-1}_{V_{m^n}}(\begin{pmatrix}I_m&x_1&\ast&\ast\\&\ddots&\ddots\\&&\ddots&x_{n-1}\\&&&I_m\end{pmatrix})= \psi^{-1}(tr(x_1)+\cdots+tr(x_{n-1})).
\end{equation}
The model above is in the the space of functions on $\GL_{mn}(\BA)$ given by
\begin{equation}\label{4'.4}
g\mapsto \int_{V_{m^n}(F)\backslash V_{m^n}(\BA)}\xi(ug)\psi_{V_{m^n}}(u)du,
\end{equation}
for $\xi$ in the space of $\Delta(\tau,m)$. This model is referred to in \cite{CFGK17}, Sec. 2.2, as the (global) Whittaker- Speh- Shalika model of $\Delta(\tau,m)$. This model is unique locally at all places. This means the following. Write $\Delta(\tau,m)$ as a restricted tensor product of local representations, denoting the corresponding local representation at a place $v$ by $\Delta(\tau_v,m)$, and we let it act on a vector space $E_v$. Then, up to scalars multiples, there is a unique (continuous)  linear functional $\beta_v$ on $E_v$, satisfying, for all $e\in E_v$, and all $u\in V_{m^n}(F_v)$,
\begin{equation}\label{4'.5}
\beta_v(\Delta(\tau_v,m)(u)e)=\psi_{V_{m^n},v}^{-1}(u)\beta_v(e),
\end{equation}
where $\psi_{V_{m^n},v}$ is the character of $V_{m^n}(F_v)$ analogous to \eqref{4.0}, with $\psi_v$ instead of $\psi$. See \cite{CFK18}, Theorem 3. The corresponding model of $\Delta(\tau_v,m)$  is the space of functions on $\GL_{mn}(F_v)$ given by $g\mapsto \beta_v(\Delta(\tau_v,m)(g)e)$, for $e\in E_v$. This is the (local) Whittaker, Speh, Shalika model of $\Delta(\tau_v,m)$ with respect to $\psi^{-1}_{V_{m^n},v}$. We will call it, for short, the $\psi^{-1}_{V_{m^n},v}$-model of $\Delta(\tau_v,m)$, and denote it by $W_{\psi^{-1}_{V_{m^n},v}}(\Delta(\tau_v,m))$. Thus, we could take $E_v=W_{\psi^{-1}_{V_{m^n},v}}(\Delta(\tau_v,m))$. We do the same for $\Delta(\tau,i)$. In particular, we realize $\Delta(\tau_v,i)$, at each place $v$, in $W_{\psi^{-1}_{V_{i^n},v}}(\Delta(\tau_v,i))$. In the metaplectic case, we consider the representation $\Delta(\tau_v,m)\gamma_{\psi_v}$ of $\GL^{(2)}_{nm}(F_v)$ acting by right translations in the space of functions on $\GL_{nm}(F_v)$,
\begin{equation}\label{4'.5.1}
\tilde{W}((g,\mu))=\mu \gamma_{\psi_v}(\det(g))W(g),
\end{equation}
where $W\in W_{\psi^{-1}_{V_{m^n},v}}(\Delta(\tau_v,m))$. Note the action $\rho((g_0,\mu_0))$ on $\tilde{W}$ by right translation by $(g_0,\mu_0)$,
\begin{equation}\label{4'.5.1}
\rho((g_0,\mu_0))(\tilde{W})=\mu_0\gamma_{\psi_v}(\det(g_0))\widetilde{\rho(g_0)(W)}.
\end{equation}

For each place $v$, let
\begin{equation}\label{4'.6}
\rho_{\Delta(\tau_v,i;m)\gamma_{\psi_v}^{(\epsilon)},s}=\Ind_{Q^{(\epsilon)}_{ni,mn}(F_v)}^{H(F_v)}(\Delta(\tau_v,i)\gamma_{\psi_v}^{(\epsilon)}|\det\cdot|^{s-\frac{m}{2}}\times
\Delta(\tau_v,m)\gamma_{\psi_v}^{(\epsilon)}|\det\cdot|^{s+\frac{i}{2}}).
\end{equation}
Consider a section $f_{\Delta(\tau_v,i;m)\gamma_{\psi_v}^{(\epsilon)},s}$ of $\rho_{\Delta(\tau_v,i;m)\gamma_{\psi_v}^{(\epsilon)},s}$. In the linear case, we view it as a
function on $H(F_v)\times \GL_{ni}(F_v)\times \GL_{mn}(F_v)$, such
that for a fixed element in $H(F_v)$, the function in the two other
variables lies in the tensor product of the two models above of
$\Delta(\tau_v,i)$ and $\Delta(\tau_v,m)$. It will be convenient to
simplify notation and to re-denote $f_{\Delta(\tau_v,i;m),s}(y)=f_{\Delta(\tau_v,i;m),s}(y;I_{ni},I_{mn})$. In the metaplectic case, the representation $\Delta(\tau_v,i)\gamma_{\psi_v}|\det\cdot|^{s-\frac{m}{2}}\times
\Delta(\tau_v,m)\gamma_{\psi_v}|\det\cdot|^{s+\frac{i}{2}}$ acts on the Levi part by\\
\\
$(diag(a_1,a_2)^\wedge,1)\mapsto$
\begin{equation}\label{4.2.1}
 \gamma_{\psi_v}(\det(a_1)\det(a_2))|\det(a_1)|^{s-\frac{m}{2}}|\det(a_2)|^{s+\frac{i}{2}}\Delta(\tau_v,i)(a_1)\otimes\Delta(\tau_v,m)(a_2).
\end{equation}
Here, $a_1\in \GL_{ni}(F_v)$, $a_2\in \GL_{mn}(F_v)$. We view $f_{\Delta(\tau_v,i;m)\gamma_{\psi_v},s}$ as a function on $\Sp^{(2)}_{2n(m+i)}(F_v)\times \GL^{(2)}_{ni}(F_v)\times\GL^{(2)}_{mn}(F_v)$, where we use the homomorphism
$$
\GL^{(2)}_{ni}(F_v)\times\GL^{(2)}_{mn}(F_v)\rightarrow \GL^{(2)}_{n(m+i)}(F_v)
$$
given by
$$
((a_1,\alpha),(a_2,\beta))\mapsto (\begin{pmatrix}a_1\\&a_2\end{pmatrix},\alpha\beta(\det(a_1),\det(a_2))).
$$ 

Let us  fix a finite set of places $S$, containing the Archimedean places,
outside which $\tau$, and $\psi$ are unramified, $f_{\Delta(\tau,m+i)\gamma_\psi^{(\epsilon)},s}$, and hence $f^\psi_{\Delta(\tau,m+i)\gamma_\psi^{(\epsilon)},s}$ (in \eqref{4'.2}) is right $K_{H(F_v)}$-invariant, for all $v\notin S$. For all such $v$, fix in the space of \eqref{4'.6} the spherical vector $f^0_{\Delta(\tau_v,i;m)\gamma_{\psi_v}^{(\epsilon)},s}$, normalized, such that, in the linear case, for $a\in \GL_{ni}(F_v)$, $b\in \GL_{mn}(F_v)$, 
\begin{equation}\label{4.2.1*} f^0_{\Delta(\tau_v,i;m),s}(1;a,b)=W^0_{\Delta(\tau_v,i),\psi_v}(a)W^0_{\Delta(\tau_v,m),\psi_v}(b)
\end{equation} 
where $W^0_{\Delta(\tau_v,i),\psi_v}$ is the unique unramified function in the $\psi^{-1}_{V_{i^n},v}$- model above of $\Delta(\tau_v,i)$, such that its value at $I_{ni}$ is $1$, and similarly for  $W^0_{\Delta(\tau_v,m),\psi_v}$. 
In the metaplectic case, $ f^0_{\Delta(\tau_v,i;m)\gamma_{\psi_v},s}(1;a,b)=\tilde{W}^0_{\Delta(\tau_v,i),\psi_v}(a)\tilde{W}^0_{\Delta(\tau_v,m),\psi_v}(b)$, where $a\in \GL^{(2)}_{ni}(F_v)$, $b\in \GL_{mn}(F_v)$. See \eqref{4'.5.1}.\\
Now we can interpret \eqref{4'.2} via a product over the places of $F$. Assume that $f_{\Delta(\tau,m+i)\gamma_\psi^{(\epsilon)},s}$ is decomposable. Then 
$f^\psi_{\Delta(\tau,m+i)\gamma_\psi^{(\epsilon)},s}$ is decomposable. Assume that it corresponds to the tensor product of the 
local sections $f_{\Delta(\tau_v,i; m)\gamma^{(\epsilon)}_{\psi_v},s}$ above. Thus, let us  fix an isomorphism
\begin{equation}\label{4.2.1.a}
p_{\tau,i}:\otimes_v'\Delta(\tau_v,i)\rightarrow \Delta(\tau,i).
\end{equation}
 Then, in the linear case, for each $x\in H(\BA)$, the automorphic form on 
$\GL_{ni}(\BA)$ given by 
$$
a\mapsto |\det(a)|^{-\frac{(m+i)n}{2}-s}f^\psi_{\Delta(\tau, m+i),s}(\hat{a} x)
$$
is the image under $p_{\tau,i}$ of $\otimes'_vW_v$, where $W_v$ is the function in the $\psi^{-1}_{V_{i^n},v}$-model of $\Delta(\tau_v,i)$ given by $a_v\mapsto  |\det(a_v)|^{-\frac{(m+i)n}{2}-s}f_{\Delta(\tau_v,i; m),s}(\hat{a}_vx_v)$. Here, $x_v$ is the coordinate at the place $v$ of $x$. 
In the metaplectic case, an automorphic form $\tilde{\varphi}$ in the space of $\Delta(\tau,i)\gamma_\psi$ has the form
$$
\tilde{\varphi}((a,\mu))=\mu\gamma_\psi(\det(a))\varphi(a),
$$
where $\varphi$ is an automorphic form in the space of $\Delta(\tau,i)$. Now, define
\begin{equation}\label{4.2.1.b}
\tilde{p}_{\tau,i}:\otimes'_v\Delta(\tau_v,i)\gamma_{\psi_v}\rightarrow \Delta(\tau,i)\gamma_\psi,
\end{equation}
such that if $\varphi$ above is decomposable with $\varphi=p_{\tau,i}(\otimes'_vW_v)$, where\\
 $W_v\in W_{\psi^{-1}_{V_{i^n},v}}(\Delta(\tau_v,i))$, then
$$
\tilde{p}_{\tau,i}(\otimes'_v\tilde{W}_v)=\tilde{\varphi}.
$$
Now, the automorphic form on 
$\GL_{ni}^{(2)}(\BA)$ given by\\
 $(a,\alpha)\mapsto  |\det(a)|^{-\frac{(m+i)n}{2}-s} f^\psi_{\Delta(\tau, m+i)\gamma_\psi,s}((\hat{a},\alpha) x)$ is the image under $\tilde{p}_{\tau,i}$ of $\otimes'_v\tilde{W}_v$, where $\tilde{W}_v$ is the function in the $\psi^{-1}_{V_{i^n},v}$-model of $\Delta(\tau_v,i)\gamma_{\psi_v}$ given by $(a_v,\mu)\mapsto |\det(a_v)|^{-\frac{(m+i)n}{2}-s}f_{\Delta(\tau_v,i; m)\gamma_{\psi_v},s}((\hat{a}_v,\mu)x_v)$.) Thus, \eqref{4'.2}, viewed as the automorphic form $a\mapsto \ell_\psi(f_{\Delta(\tau,m+i)\gamma_\psi^{(\epsilon)},s})(a,g)$, for each fixed $g$, is the image under $p_{\tau,i}$ of $\otimes'_v(a_v\mapsto \ell_{\psi_v}(f_{\Delta(\tau_v,i;m)\gamma_{\psi_v}^{(\epsilon)},s})(a_v,g_v))$, where, in the linear case,
\begin{multline}\label{4.3}
\ell_{\psi_v}(f_{\Delta(\tau_v,i;m),s})(a_v,g_v)=\\
\int_{U'_{m(n-1)}(F_v)} f_{\Delta(\tau_v,i;m),s}(\hat{a}_v\delta_0ut(I,diag(I_{ni},\iota^{-1}(g_v^{-1}),I_{ni})))\psi^{-1}_{H,v}(u)du.
\end{multline}
Here $\iota^{-1}(g_v^{-1})=J_0g_v^{-1}J_0^{-1}$ (see \eqref{1.10.1.1}), and $\psi_{H,v}$ is the component of $\psi_H$ at $v$. \\
In the metaplectic case, as in \eqref{4'.2}, for $a_v\in \GL_{ni}^{(2)}(F_v)$, $\alpha_v=\pm 1$,
\begin{multline}\label{4.3.1}
\ell_{\psi_v}(f_{\Delta(\tau_v,i;m)\gamma_{\psi_v},s})(a_v,(g_v,\alpha_v))=\\
\int_{U'_{m(n-1)}(F_v)} f^\psi_{\Delta(\tau_v,i;m)\gamma_{\psi_v},s}(\hat{a}_v\delta_0ut(1,(diag(I_{ni},g_v^\iota,I_{ni}),\alpha_v\mu(g_v))^{-1})\psi^{-1}_{H,v}(u)du.
\end{multline}

Denote, for $g_v\in H_m^{(\epsilon)}(F_v)$,
\begin{multline}\label{4.3.2}
\tilde{g}_v=diag(I_{ni},\iota^{-1}(g_v^{-1}),I_{ni})\in H_{2ni+m}(F_v),\ \textit{if}\  \epsilon=1;\\
\tilde{g}_v= (\diag(I_{ni}, b_v^\iota,I_{ni}),\alpha_v\mu(b_v))^{-1}\in H_{2ni+m}^{(2)}(F_v),\ \textit{if}\  \epsilon=2,\ \textit{and}\ g_v=(b_v,\alpha_v).
\end{multline}

Let us write the elements of $U'_{m(n-1)}(F_v)$ in the form $u_{x;y}$ as
in \eqref{3.21.1}, with $x=(x_1,x_2,x_3)$, where $x_1$ has $[\frac{m}{2}]$
columns, $x_2$ has $ni$ columns, and $x_3$ has $m-[\frac{m}{2}]$ columns. Let
$$
U''_{m(n-1)}(F_v)=\{u_{x;y}\in U'_{m(n-1)}(F_v) |\ x_2=0\}.
$$
\begin{prop}\label{prop 4.1}
Let $v$ be a place of $F$. There is a smooth, holomorphic section $f'_{\Delta(\tau_v,i;m)\gamma_{\psi_v}^{(\epsilon)},s}$ of $\rho_{\Delta(\tau_v,i;m)\gamma_{\psi_v}^{(\epsilon)},s}$ (see \eqref{4'.6}), which
depends on (the smoothness of) $f_{\Delta(\tau_v,i;m)\gamma_{\psi_v}^{(\epsilon)},s}$,
such that for all $a_v\in \GL^{(\epsilon)}_{ni}(F_v)$ and all $g_v\in H^{(\epsilon)}_m(F_v)$,
\begin{equation}\label{4.4}
\ell_{\psi_v}(f_{\Delta(\tau_v,i;m)\gamma_{\psi_v}^{(\epsilon)},s})(a_v,g_v)=\int_{U''_{m(n-1)}(F_v)}
f'_{\Delta(\tau_v,i;m)\gamma_{\psi_v}^{(\epsilon)},s}(\hat{a}_v\delta_0ut(1,\tilde{g}_v))\psi^{-1}_{H,v}(u)du.
\end{equation}
If $f_{\Delta(\tau_v,i;m)\gamma_{\psi_v}^{(\epsilon)},s}$ is spherical, \\
then $f'_{\Delta(\tau_v,i;m)\gamma_{\psi_v}^{(\epsilon)},s}=f_{\Delta(\tau_v,i;m)\gamma_{\psi_v}^{(\epsilon)},s}$. In general, (when  $f_{\Delta(\tau_v,i;m)\gamma_{\psi_v}^{(\epsilon)},s}$ is not spherical), $f'_{\Delta(\tau_v,i;m)\gamma_{\psi_v}^{(\epsilon)},s}$ is obtained
from $f_{\Delta(\tau_v,i;m)\gamma_{\psi_v}^{(\epsilon)},s}$ by a finite sequence of
convolutions, along certain unipotent subgroups, against certain Schwartz-Bruhat functions  (described in the
proof).
\end{prop}
\begin{proof}
Assume that $v$ is a finite place. We assume, for simplicity, that $\psi_v$ is normalized, so that $\psi_v$ is trivial on the ring of integers $\mathcal{O}_v$ of $F_v$, and is non trivial on the inverse of the maximal ideal $\mathcal{P}_v$.  We may assume that $\psi_v$ is normalized, for all $v\notin S$. Denote, for $x_2\in
M_{m(n-1)\times ni}(F_v)$,\\
\\
$\ell'_{\psi_v}(f_{\Delta(\tau_v,i;m)\gamma_{\psi_v}^{(\epsilon)},s})(a_v,g_v;x_2)=$
\begin{equation}\label{4.5}
\int_{U''_{m(n-1)}(F_v)} f_{\Delta(\tau_v,i;m)\gamma_{\psi_v}^{(\epsilon)},s}(\hat{a}_v\delta_0u''u_{0,x_2,0;0}t(1,\tilde{g}_v))\psi^{-1}_{H,v}(u'')du''.
\end{equation}
Then
\begin{equation}\label{4.5.0}
\ell_{\psi_v}(f_{\Delta(\tau_v,i;m)\gamma_{\psi_v}^{(\epsilon)},s})(a_v,g_v)=\int_{ M_{m(n-1)\times ni}(F_v)}\ell'_{\psi_v}(f_{\Delta(\tau_v,i;m)\gamma_{\psi_v}^{(\epsilon)},s})(a_v,g_v;x_2)dx_2.
\end{equation}
We will show that \eqref{4.5} has support in $x_2$ in a compact set $\Omega_v$, independent of $a_v,g_v$. In case $v$ is outside $S$,
$\Omega_v=M_{m(n-1)\times ni}(\mathcal{O}_v)$. 
Let $z\in M_{ni\times [\frac{m}{2}]}(\mathcal{P}_v^{k})$. Take $k$ to be
so large that the right translation by the matrix
$$
c_z=\begin{pmatrix}I_{m(n-1)}\\&I_{[\frac{m}{2}]}\\&z&I_{ni}\end{pmatrix}^\wedge\in
H(F_v),
$$
fixes $f_{\Delta(\tau_v,i;m)\gamma_{\psi_v}^{(\epsilon)},s}$. (Note, that in the metaplectic case, we identify $c_z$ and $(c_z,1)$.) For $v$ outside $S$, we
take $k=0$. Note that $c_z$ commutes with $\omega_0$, $t(1,\tilde{g}_v)$, $h_{\gamma^0}$, and
$h_{\gamma^0}$ commutes with the elements $u_{x;y}$. We have
\begin{multline}\label{4.5.1}
\ell'_{\psi_v}(f_{\Delta(\tau_v,i;m)\gamma_{\psi_v}^{(\epsilon)},s})(a_v,g_v;x_2)=\ell'_{\psi_v}(\rho(c_z)f_{\Delta(\tau_v,i;m)\gamma_{\psi_v}^{(\epsilon)},s})(a_v,g_v;x_2)=\\
\int_{U''_{m(n-1)}(F_v)} f_{\Delta(\tau_v,i;m)\gamma_{\psi_v}^{(\epsilon)},s}(\hat{a}_v\delta_0u''u_{0,x_2,0;0}c_zt(1,\tilde{g}_v))\psi^{-1}_{H,v}(u'')du''.
\end{multline}
Here, we abbreviated $\rho_{\Delta(\tau_v,i;m)\gamma_{\psi_v}^{(\epsilon)},s}(c_z)=\rho(c_z)$. This is simply the right translation by $c_z$.
Conjugate inside \eqref{4.5.1}
$$
c_z^{-1}u_{x_1,0,x_3;y}u_{0,x_2,0;0}c_z=u_{x_1+x_2z,0,x_3;y}u_{0,x_2,0;0}.
$$
(Note that in the metaplectic case,
$$
(c_z,1)^{-1}(u_{x_1,0,x_3;y},1)(u_{0,x_2,0;0},1)(c_z,1)=(u_{x_1+x_2z,0,x_3;y},1)(u_{0,x_2,0;0},1).)
$$
Now, change variable $x_1\mapsto x_1-x_2z$. Note that $\psi^{-1}_{H,v}(u'')$ changes to\\
$\psi^{-1}_{H,v}(u'')\psi_v(tr(x_2^{(n-1,1)}z))$, where we write
\begin{equation}\label{4.5.1*}
x_2=\begin{pmatrix}x_2^{(1)}\\ \vdots\\x_2^{(n-1)}\end{pmatrix},\
x_2^{(j)}=\begin{pmatrix}x_2^{(j,1)}\\x_2^{(j,2)}\end{pmatrix};\
x_2^{(j,1)}\in M_{[\frac{m}{2}]\times ni}(F_v),\ x_2^{(j,2)}\in
M_{(m-[\frac{m}{2}])\times ni}(F_v).
\end{equation}
Recall from \eqref{3.21.1'} that $\delta_0=\epsilon^0h_{\gamma^0}$. Thus \eqref{4.5.1} becomes
\begin{multline}\label{4.5.2}
\ell'_{\psi_v}(f_{\Delta(\tau_v,i;m)\gamma_{\psi_v}^{(\epsilon)},s})(a_v,g_v;x_2)=\psi_v(tr(x_2^{(n-1,1)}z))\cdot\\
\cdot\int_{U''_{m(n-1)}(F_v)} f_{\Delta(\tau_v,i;m)\gamma_{\psi_v}^{(\epsilon)},s}(\hat{a}_v\epsilon^0c_zh_{\gamma^0}0u''u_{0,x_2,0;0}t(1,\tilde{g}_v))\psi^{-1}_{H,v}(u'')du''.
\end{multline}
Next, conjugation by $\epsilon^0$ takes $c_z$ into the
unipotent radical $U_{Q_{ni,mn}}(F_v)$, and hence
$$
\ell'_{\psi_v}(f_{\Delta(\tau_v,i;m)\gamma_{\psi_v}^{(\epsilon)},s})(a_v,g_v;x_2)=\psi_v(tr(x_2^{(n-1,1)}z))\ell'_{\psi_v}(f_{\Delta(\tau_v,i;m)\gamma_{\psi_v}^{(\epsilon)},s})(a_v,g_v;x_2).
$$
This implies that $x_2^{(n-1,1)}\in M_{[\frac{m}{2}]\times
ni}(\mathcal{P}_v^{-k})$, or else \eqref{4.5} is zero. In case our
section is spherical, this means that $x_2^{(n-1,1)}$ has
integral coordinates. (We assume that $\psi_v$ is normalized for $v\notin S$.) 
In general (that is when our section is not necessarily spherical), denote by $x_2(n-1,1)$ be the element obtained from
$x_2$ by replacing (in \eqref{4.5.1*}) $x_2^{(n-1,1)}$ by zero. Write $M_{[\frac{m}{2}]\times
ni}(\mathcal{P}_v^{-k})$ as a finite union of cosets modulo $M_{[\frac{m}{2}]\times
ni}(\mathcal{P}_v^k)$. Choose coset representatives $\alpha_1,...,\alpha_N$. Let $\alpha'_j=u_{0,\alpha_j^0,0;0}$, where $\alpha_j^0$ is the matrix \eqref{4.5.1*} with $x_2^{(1)}=\cdots x_2^{(n-2)}=0$, $x_2^{(n-1,2)}=0$, $x_2^{(n-1,1)}=\alpha_j$. Then 
\begin{multline}\label{4.6}
\ell'_{\psi_v}(f_{\Delta(\tau_v,i;m)\gamma_{\psi_v}^{(\epsilon)},s})(a_v,g_v;x_2)=\\
\sum_{j=1}^N
\phi^j(x_2^{(n-1,1)})\ell'_{\psi_v}(\rho(\alpha'_j)f_{\Delta(\tau_v,i;m)\gamma_{\psi_v}^{(\epsilon)},s})(a_v,g_v;x_2(n-1,1)),
\end{multline}
where $\phi^j$ is the  characteristic function of $\alpha_j+M_{[\frac{m}{2}]\times ni}(\mathcal{P}_v^k)$.
Again, when  $f_{\Delta(\tau_v,i;m)\gamma_{\psi_v}^{(\epsilon)},s}$ is spherical, $N=1$, $\phi^1$ is the characteristic function of $M_{[\frac{m}{2}]\times ni}(\mathcal{O}_v)$ and
$\alpha'_1=I$ (i.e. $\alpha_1=0$). We may now continue and examine
$\ell'_{\psi_v}(f^j_{\Delta(\tau_v,i;m)\gamma_{\psi_v}^{(\epsilon)},s})(\hat{a},x_2(n-1,1))$. Thus, we may consider \eqref{4.5} with
$x_2(n-1,1)$ instead of $x_2$.

Let $z\in M_{[\frac{m}{2}]\times ni}(\mathcal{P}^k_v)$, with $k$ large,
so that the right translation by the following element fixes 
$f_{\Delta(\tau_v,i;m)\gamma_{\psi_v}^{(\epsilon)},s}$,
$$
e_z=diag(I_{m(n-1)},\begin{pmatrix}I_{[\frac{m}{2}]}&0&0&z&0\\&I_{ni}&0&0&z'\\
&&I_{2(m-[\frac{m}{2}])}&0&0\\&&&I_{ni}&0\\&&&&I_{[\frac{m}{2}]}\end{pmatrix},I_{m(n-1)}).
$$
(Note, again, that in the metaplectic case, we identify the element $e_z$ with $(e_z,1)$.) In the spherical case, we take $k=0$. Note that
$e_z$ commutes with $\omega_0$, $t(I,\tilde{g}_v)$, and that conjugation by $\delta_0$ takes
$e_{z}$ into the unipotent radical $U_{Q_{ni,mn}}(F_v)$. Next,
$$
e_z^{-1}u_{x;y}e_z=
\begin{pmatrix}I_{m(n-1)}&x&v(x_1z,x_2z')&y\\&I_{m+ni}&0&v(x_1z,x_2z')'\\&&I_{m+ni}&x'\\&&&I_{m(n-1)}\end{pmatrix}^{\omega_0^{m(n-1)}},
$$
where $v(x_1z,x_2z')=(0_{m(n-1)\times [\frac{m}{2}]},x_1z,x_2z')$. Denote
the last matrix by\\
 $u_{x,v(x_1z,x_2z');y}$. Now we repeat the previous argument, as in \eqref{4.5.1}, \eqref{4.5.2}, with $e_z$ instead of $c_z$.
 Thus, we get that\\
\\
$\ell'_{\psi_v}(f_{\Delta(\tau_v,i;m)\gamma_{\psi_v}^{(\epsilon)},s})(a_v,g_v;x_2)=\ell'_{\psi_v}(\rho(e_z)f_{\Delta(\tau_v,i;m)\gamma_{\psi_v}^{(\epsilon)},s})(a_v,g_v;x_2)=$
$$
\int_{U''_{m(n-1)}(F_v)} f_{\Delta(\tau_v,i;m)\gamma_{\psi_v}^{(\epsilon)},s}(\hat{a}_v\delta_0u_{0,v(x_1z,x_2z');0}u''u_{0,x_2,0;0}t(I,\tilde{g}_v)\psi^{-1}_{H,v}(u'')du'',
$$
where $u''=u_{x_1,0,x_3;y}$, and $x_2=x_2(n-1,1)$. Performing the conjugation of $u_{0,v(x_1z,x_2z'}$ by $\delta_0$ in the last integral, and changing variable in $u''$ (which does not affect $\psi_{H,v}(u'')$), we get that the last integral is equal to\\
$\psi(tr(x_2^{(n-1,3)}z'))\ell'_{\psi_v}(f_{\Delta(\tau_v,i;m)\gamma_{\psi_v}^{(\epsilon)},s})(a_v,g_v;x_2)$, where
$x_2^{(n-1,3)}=x_2^{(n-1,2)}$, in case $m$ is even. In case $m$ is odd,
$x_2^{(n-1,3)}$ is obtained from $x_2^{(n-1,2)}$ by deleting its
first row. We conclude that
$$
\ell'_{\psi_v}(f_{\Delta(\tau_v,i;m)\gamma_{\psi_v}^{(\epsilon)},s})(a_v,g_v;x_2)=\psi(tr(x_2^{(n-1,3)}z'))\ell'_{\psi_v}(f_{\Delta(\tau_v,i;m)\gamma_{\psi_v}^{(\epsilon)},s})(a_v,g_v;x_2),
$$
for all $z\in M_{[\frac{m}{2}]\times ni}(\mathcal{P}^k_v)$.  This implies that if
$\ell'_{\psi_v}(f_{\Delta(\tau_v,i;m)\gamma_{\psi_v}^{(\epsilon)},s})(a_v,g_v,x_2(n-1,1))$ is nonzero, then
$x_2^{(n-1,3)}\in M_{[\frac{m}{2}]\times
ni}(\mathcal{P}_v^{-k})$. When $m=2m'-1$ is odd, we still need to show that we
have a similar compact support in the first row of $x_2^{(n-1,2)}$.
For this, we use right translations by the
$\omega_0^{m(n-1)}$-conjugate of
$$
\begin{pmatrix}I_{m(n-1)}\\&I_{ni}&0&z\\&&I_{m'-1}&0\\&&&1\end{pmatrix}^\wedge,
$$
for column vectors $z\in (\mathcal{P}^k_v)^{ni}$ and $k$ large. The
argument is similar. We showed, so far, that
$\ell'_{\psi_v}(f_{\Delta(\tau_v,i;m)\gamma_{\psi_v}^{(\epsilon)},s})(a_v,g_v;x_2)$ is
supported such that $x_2^{(n-1)}\in M_{m\times
ni}(\mathcal{P}_v^{-k})$, with $k=0$ in the spherical case, and then
(for $x_2^{(n-1)}\in M_{m\times ni}(\mathcal{O}_v)$)
$$
\ell'_{\psi_v}(f_{\Delta(\tau_v,i;m)\gamma_{\psi_v}^{(\epsilon)},s})(a_v,g_v;x_2)=\ell'_{\psi_v}(f_{\Delta(\tau_v,i;m)\gamma_{\psi_v}^{(\epsilon)},s})(a_v,g_v;x_2(n-1)),
$$
where $x_2(n-1)$ is obtained from $x_2$ by replacing $x_2^{(n-1)}$
by zero. For $v\in S$ and finite, we conclude that there is an
expression similar to \eqref{4.6} with $x_2(n-1)$ instead of
$x_2(n-1,1)$. The Schwartz -Bruhat  functions are now on $M_{m\times
ni}(F_v)$, evaluated on $x_2^{(n-1)}$. Thus, we may
continue and consider \eqref{4.5} with $x_2(n-1)$ instead of $x_2$.
The argument is similar, using right translations by
$$
\begin{pmatrix}I_{m(n-2)}\\&I_m\\&&I_{[\frac{m}{2}]}\\&z&&I_{ni}\end{pmatrix}^\wedge,
$$
for $z\in M_{ni\times m}(\mathcal{P}_v^k)$, and $k$ large so that right translations by these elements fix $f_{\Delta(\tau_v,i;m)\gamma_{\psi_v}^{(\epsilon)},s}$. We
conclude that in the support, $x_2^{(n-2)}\in 
M_{m\times ni}(\mathcal{P}_v^{-k})$. In general, let
$x_2(n-j+1)$ be the matrix obtained from $x_2$ (in \eqref{4.5.1*}) by replacing
$x_2^{(n-1)},x_2^{(n-2)},...,x_2^{(n-j+1)}$ by zero. Then we show
that \\
$\ell'_{\psi_v}(f_{\Delta(\tau_v,i;m)\gamma_{\psi_v}^{(\epsilon)},s})(a_v,g_v;x_2(n-j+1))$ is supported such that $x_2^{(n-j)}\in
 M_{m\times ni}(\mathcal{P}_v^{-k})$, for $k$
large, indepently of $a_v,g_v$. For this, we apply a similar argument by using right translations by the following elements fixing
$f_{\Delta(\tau_v,i;m)\gamma_{\psi_v}^{(\epsilon)},s}$, for all $z\in M_{ni\times
m}(\mathcal{P}_v^k)$,
$$
\begin{pmatrix}I_{m(n-j)}\\&I_{m}\\&&I_{m(j-2)}\\&&&I_{[\frac{m}{2}]}\\&z&&&I_{ni}\end{pmatrix}^\wedge.
$$
All these considerations are the same in the metaplectic case, as the cocycle plays no part in them. We do this for $j=3,4,...,n$. Altogether, we proved that the support of 
 $\ell'_{\psi_v}(f_{\Delta(\tau_v,i;m)\gamma_{\psi_v}^{(\epsilon)},s})(a_v,g_v;x_2)$ is compact, as a function of $x_2^{(n-1)}$, $x_2^{(n-2)}$,...,$x_2^{(1)}$, and hence the support, denote it by $\Omega$, is compact , as a function of $x_2$. We saw that this compact support is independent of $a_v,g_v$, and
that in the spherical case, the support is $M_{m(n-1)\times ni}(\mathcal{O}_v)$. Moreover, repeating the analog of \eqref{4.6}, at each step, we finally get an expression of the form
\begin{equation}\label{4.6.1}
\ell'_{\psi_v}(f_{\Delta(\tau_v,i;m)\gamma_{\psi_v}^{(\epsilon)},s})(a_v,g_v;x_2)=
\sum_{j=1}^{N'}
\chi^j(x_2)\ell'_{\psi_v}(\rho(\zeta'_j)f_{\Delta(\tau_v,i;m)\gamma_{\psi_v}^{(\epsilon)},s})(a_v,g_v;0),
\end{equation}
where the $\zeta'_j =u_{0,\zeta_j,0}$, with $\zeta_j\in \Omega$, and $\chi^j$ is the characteristic function of a small neighborhood of $\zeta_j$. When  $f_{\Delta(\tau_v,i;m)\gamma_{\psi_v}^{(\epsilon)},s}$ is spherical, $N'=1$, $\chi^1$ is the characteristic function of $M_{m(n-1)\times ni}(\mathcal{O}_v)$ and $\zeta'_1=I$ (i.e. $\zeta_1=0$). Now, we get our proposition (in case $v$ is non-Archimedean). We use \eqref{4.6.1} to produce the desired section $f'_{\Delta(\tau_v,i;m)\gamma_{\psi_v}^{(\epsilon)},s}$. Define
\begin{equation}\label{4.6.2}
f'_{\Delta(\tau_v,i;m)\gamma_{\psi_v}^{(\epsilon)},s}=
\int_{M_{m(n-1)\times ni}(F_v)}(\sum_{j=1}^{N'}\chi^j(x_2)\rho(\zeta'_j)f_{\Delta(\tau_v,i;m)\gamma_{\psi_v}^{(\epsilon)},s})dx_2.
\end{equation}
Then, by \eqref{4.6.1}, \eqref{4.5.0}, we get \eqref{4.4}.

Finally, when $v$ is Archimedean, the proof is the same, except that instead
of right translations by the unipotent elements above, with $z$
close to zero, we use convolutions against Schwartz functions on the
unipotent subgroups above, and apply at each step the
Dixmier-Malliavin Lemma \cite{DM78}. We show how to adapt the first step
of the proof in case $v$ is Archimedean. By the Dixmier-Malliavin
Lemma, $f_{\Delta(\tau_v,i;m)\gamma_{\psi_v}^{(\epsilon)},s}$ is a finite sum of
sections of the form $\phi'\star \varphi_{\Delta(\tau_v,i;m)\gamma_{\psi_v}^{(\epsilon)},s}$, where $\phi'(c_z)=\phi(z)$, 
$\phi\in \mathcal{S}(M_{ni\times[\frac{m}{2}]}(F_v))$, and
$$
\phi'\star \varphi_{\Delta(\tau_v,i;m)\gamma_{\psi_v}^{(\epsilon)},s}=\int_{M_{ni\times[\frac{m}{2}]}(F_v)}\phi(z)\rho(c_z)\varphi_{\Delta(\tau_v,i;m)\gamma_{\psi_v}^{(\epsilon)},s}dz.
$$
Recall that we abbreviate $\rho_{\Delta(\tau_v,i;m)\gamma_{\psi_v}^{(\epsilon)},s}(c_z)=\rho(c_z)$. Thus, we may assume that 
$f_{\Delta(\tau_v,i;m)\gamma_{\psi_v}^{(\epsilon)},s}=\phi'\star \varphi_{\Delta(\tau_v,i;m)\gamma_{\psi_v}^{(\epsilon)},s}$.
Then the calculations that we did in this case show that
\begin{equation}\label{4.6.3}
\ell'_{\psi_v}(\phi'\star \varphi_{\Delta(\tau_v,i;m)\gamma_{\psi_v}^{(\epsilon)},s})(a_v,g_v;x_2)=\hat{\phi}(x_2^{(n-1,1)})\ell'_{\psi_v}(\varphi_{\Delta(\tau_v,i;m)\gamma_{\psi_v}^{(\epsilon)},s})(a_v,g_v;x_2),
\end{equation}
where $\hat{\phi}$ is the Fourier transform
$$
\hat{\phi}(x_2^{(n-1,1)})=\int_{M_{ni\times[\frac{m}{2}]}(F_v)}\phi(z)\psi(tr(x_2^{(n-1,1)}z))dz.
$$
(The analogy with the similar step \eqref{4.5.1}, \eqref{4.5.2} in the non-Archimedean case is that there we take $\phi$ to be the characteristic function of the small neighborhood $M_{ni\times [\frac{m}{2}]}(\mathcal{P}_v^k)$, divided by the measure of the neighborhood. Its convolution against $f_{\Delta(\tau_v,i;m)\gamma_{\psi_v}^{(\epsilon)},s}$ is simply $f_{\Delta(\tau_v,i;m)\gamma_{\psi_v}^{(\epsilon)},s}$, and its Fourier transform is the characteristic function of $M_{[\frac{m}{2}]\times ni}(\mathcal{P}_v^{-k})$.)\\
Now, integrate \eqref{4.6.3} in the variable $x_2^{(n-1,1)}$. This is part of the full $dx_2$-integration in the integral defining 
$\ell_{\psi_v}(\phi'\star \varphi_{\Delta(\tau_v,i;m)\gamma_{\psi_v}^{(\epsilon)},s})(a_v,g_v)$ (see \eqref{4.5.0}). We get that
\begin{multline}\label{4.6.4}
\int_{M_{[\frac{m}{2}]\times ni}(F_v)}\ell'_{\psi_v}(f_{\Delta(\tau_v,i;m)\gamma_{\psi_v}^{(\epsilon)},s})(a_v,g_v;x_2)dx_2^{(n-1,1)}=\\
\ell'_{\psi_v}(\hat{\phi}\star\varphi_{\Delta(\tau_v,i;m)\gamma_{\psi_v}^{(\epsilon)},s})(a_v,g_v;x_2(n-1,1)),
\end{multline}
where we keep denoting by $\hat{\phi}\star\varphi_{\Delta(\tau_v,i;m)\gamma_{\psi_v}^{(\epsilon)},s}$ the convolution of $\hat{\phi}$ against \\
$\varphi_{\Delta(\tau_v,i;m)\gamma_{\psi_v}^{(\epsilon)},s}$. Thus, we may assume that 
$f_{\Delta(\tau_v,i;m)\gamma_{\psi_v}^{(\epsilon)},s}=\hat{\phi}\star \varphi_{\Delta(\tau_v,i;m)\gamma_{\psi_v}^{(\epsilon)},s}$, and consider next $\ell'_{\psi_v}(f_{\Delta(\tau_v,i;m)\gamma_{\psi_v}^{(\epsilon)},s})(a_v,g_v;x_2(n-1,1))$, that is we "got rid" of the 
$dx_2^{(n-1,1)}$ - integration by taking $\hat{\phi}\star \varphi_{\Delta(\tau_v,i;m)\gamma_{\psi_v}^{(\epsilon)},s}$ instead of 
$f_{\Delta(\tau_v,i;m)\gamma_{\psi_v}^{(\epsilon)},s}$. This completes the analog of the first step we did in the non-Archimedean case.
The proof carries over step by step, in complete analogy with steps in the non-Archimedean
case. 
\end{proof}

\begin{prop}\label{prop 4.2}
In the notation of the last proposition, for each $g_v\in H_m^{(\epsilon)}(F_v)$, the function $a_v\mapsto \ell_{\psi_v}(f_{\Delta(\tau_v,i;m)\gamma_{\psi_v}^{(\epsilon)},s})(a_v,g_v)$, on $\GL^{(\epsilon)}_{ni}(F_v)$,  lies in 
$$
\delta_{Q_{ni}^{H_{m+2ni}}}^{\frac{1}{2}}|\det\cdot|^{s+m(n-1)}(\gamma_{\psi_v}^{(\epsilon)}\circ\det)W_{\psi^{-1}_{V_{i^n}}}(\Delta(\tau_v,i)).
$$
\end{prop}
\begin{proof}
By Prop. \ref{prop 4.1}, it is enough to consider the following integral, as a function of $a_v$, (for $Re(s)$ large)
$$
\ell'_{\psi_v}(f'_{\Delta(\tau_v,i;m)\gamma_{\psi_v}^{(\epsilon)},s})(a_v,g_v)=\int_{U''_{m(n-1)}(F_v)}
f'_{\Delta(\tau_v,i;m)\gamma_{\psi_v}^{(\epsilon)},s}(\hat{a}_v\delta_0ut(1,\tilde{g}_v))\psi^{-1}_{H,v}(u)du.
$$
Similar calculations to \eqref{3'.22} and \eqref{3.22.1} show that the last integral is equal to
\begin{equation}\label{4'.7}
\int_{U''_{m(n-1)}(F_v)}
f'_{\Delta(\tau_v,i;m)\gamma_{\psi_v}^{(\epsilon)},s}(\hat{a}_v\delta_0ut(g_v,1))\psi^{-1}_{H,v}(u)du.
\end{equation}
Recall that $\delta_0=\epsilon^0h_{\gamma_0}$. Now, we carry the following conjugation inside \eqref{4'.7}.
Write $u\in U''_{m(n-1)}(F_v)$ in the form $u_{x_1,0,x_3;y}$, as we did in the proof of Prop. \ref{prop 4.1}. Put $x=(x_1,x_3)$, 
viewed as a $m(n-1)\times m$ matrix. We have, in the linear case,
\begin{equation}\label{4.8}
\epsilon^0h_{\gamma^0}ut(g,I_{m+2ni})(\epsilon^0)^{-1}=diag(I_{ni},\epsilon_0h_{\gamma'_0}u'_{x;y}t'(g,I_m)\epsilon_0^{-1},I_{ni}),
\end{equation}
where $\epsilon_0$ is obtained from $\epsilon^0$ by deleting in
$W_0$, in \eqref{3.21.2}, the row and column of $I_{ni}$, and similarly in
$W_0^*$. For example, when $m$ is even,
$$
\epsilon_0=\begin{pmatrix}0&I_m&0&0\\0&0&0&I_{m(n-1)}\\\delta_HI_{m(n-1)}&0&0&0\\0&0&I_m&0\end{pmatrix}.
$$
Also,
$$
u'_{x;y}=\begin{pmatrix}I_{m(n-1)}&x&0&y\\&I_m&0&0\\&&I_m&x'\\&&&I_{m(n-1)}\end{pmatrix}^{\omega_0^{m(n-1)}};
$$
$$
h_{\gamma'_0}=diag(I_{m(n-1)},\gamma'_0,I_{m(n-1)}),
$$
$$
\gamma'_0=\begin{pmatrix}I_{[\frac{m}{2}]}\\&I_{[\frac{m}{2}]}\\&&I_{2(m-2[\frac{m}{2}])}\\I_{[\frac{m}{2}]}&0&&I_{[\frac{m}{2}]}\\0&-\delta_HI_{[\frac{m}{2}]}&&&I_{[\frac{m}{2}]}\end{pmatrix}.
$$
Finally, we explicate $t'(g,I_m)$. Assume that $m=2m'$ is even, and
write $g=\begin{pmatrix}a&b\\c&d\end{pmatrix}$, where the blocks are
$m'\times m'$ matrices. Then
$$
t'(g,I_m)=diag(g^{\Delta_{n-1}},\begin{pmatrix}a&0&b\\0&I_m&0\\c&0&d\end{pmatrix},(g^*)^{\Delta_{n-1}});
$$
Assume that $m=2m'-1$ is odd. Write
$g=\begin{pmatrix}a_1&b_1&c_1\\a_2&b_2&c_2\\a_3&b_3&c_3\end{pmatrix}$,
where the first and third block rows (resp. columns) contain $m'-1$
rows (resp. columns).Then
$$
t'(g,I_m)=diag(g^{\Delta_{n-1}},\begin{pmatrix}a_1&0&\frac{1}{2}b_1&b_1&0&c_1\\0&I_{m'-1}&0&0&0&0\\
a_2&0&\frac{1}{2}(b_2+1)&b_2-1&0&c_2\\\frac{1}{2}a_2&0&\frac{1}{4}(b_2-1)&\frac{1}{2}(b_2+1)&0&\frac{1}{2}c_2\\
0&0&0&0&I_{m'-1}&0\\a_3&0&\frac{1}{2}b_3&b_3&0&c_3\end{pmatrix},(g^*)^{\Delta_{n-1}}).
$$
In the metaplectic case, using \cite{Rao93}, Cor. 5.6, it is possible to show that we have the following analog of \eqref{4.8},
\begin{multline}\label{4.8.1}
(\epsilon^0,1)(h_{\gamma^0},c_v(\epsilon^0,h_{\gamma^0}))(u,1)(t(g,I_{m+2ni}),1)(\epsilon^0,1)^{-1}=\\
=(diag(I_{ni},\epsilon_0h_{\gamma'_0}u'_{x;y}t'(g,I_m)\epsilon_0^{-1},I_{ni}),\alpha(u'_{x,y},g)),
\end{multline}
where $\alpha(u'_{x,y},g)=\pm 1$ is obtained through the multiplication in $\Sp^{(2)}_{2mn}(F_v)$,
\begin{multline}\label{4.8.2}
(\epsilon_0,1)(h_{\gamma'_0},c_v(\epsilon_0,h_{\gamma^0}))(u'_{x;y},1)(t'(g,I_m),1)(\epsilon_0,1)^{-1}=\\
=(\epsilon_0h_{\gamma'_0}u'_{x;y}t'(g,I_m)\epsilon_0^{-1},\alpha(u'_{x,y},g)).
\end{multline}
(We used the same notation for the local Ranga Rao cocycle $c_v$, one for $\Sp_{2n(m+i)}(F_v)$, and one for $\Sp_{2mn}(F_v)$.)

Put $\delta_0'=\epsilon_0h_{\gamma_0'}$. (In the metaplectic case, we take $(\delta'_0,1)$, and re-denote it, sometimes, as $\delta'_0$, again.) Let $U^0_{m(n-1)}(F_v)$
denote the subgroup of $H_{2mn}(F_v)$ consisting of the elements $u'_{x;y}$.
Consider the character $\psi_{H_{2nm,v}}$ of $U_{m^{n-1}}^{H_{2nm}}(F_v)$
defined similarly to $\psi_{H,v}$, only that in \eqref{1.3}, we take
the following $2m\times m$ matrix  $A_{H_{2nm}}$ instead of $A_H$. When $m=2m'$ is even,
$$
A_{H_{2nm}}=\begin{pmatrix}I_{m'}&0\\0_{m\times m'}&0_{m\times
m'}\\0&I_{m'}\end{pmatrix}.
$$
When $m=2m'-1$ is odd (and hence $H_{2nm}=\SO_{2n(2m'-1)}$),
$$
A_{H_{2nm}}=\begin{pmatrix}I_{m'-1}&0&0\\0&0&0\\0&1&0\\0&\frac{1}{2}&0\\0&0&0\\0&0&I_{m'-1}\end{pmatrix},
$$
where the second and fifth block rows of zeroes contain each $m'-1$
rows. Note that the stabilizer of $\psi_{H_{2nm,v}}$ inside
$M^{H_{2nm}}_{m^{n-1}}(F_v)$ is isomorphic to $H_m(F_v)\times
H_m(F_v)$. This is the subgroup of all elements
$$
t'(g,h)=diag(g^{\Delta_{n-1}},j'(g,h),(g^*)^{\Delta_{n-1}}),
$$
where $g,h\in H_m(F_v)\times H_m(F_v)$, $t'(g,I_m)$ is
described above, and $t'(I_m,h)$ is given as follows. When
$m=2m'$ is even,
$$
t'(I_m,h)=diag(I_{m'},h,I_{m'});
$$
In the metaplectic case, we lift $t'$ to a homomorphism from $\Sp^{(2)}_{2m'}(\BA)\times \Sp^{(2)}_{2m'}(\BA)$ to
$\Sp^{(2)}_{4nm'}(\BA)$, as in \eqref{1.9}.

When $m=2m'-1$ is odd, write
$$
h=\begin{pmatrix}A_1&B_1&C_1\\A_2&B_2&C_2\\A_3&B_3&C_3\end{pmatrix},
$$
where the corner blocks are matrices of size $(m'-1)\times (m'-1)$
Then, as in \eqref{1.8},
$$
t'(I_m,h)=\begin{pmatrix}I_{m-1}&0&0&0&0&0\\0&A_1&\frac{1}{2}B_1&-B_1&C_1&0\\
0&A_2&\frac{1}{2}(1+B_2)&1-B_2&C_2&0\\0&-\frac{1}{2}A_2&\frac{1}{4}(1-B_2)&\frac{1}{2}(1+B_2)&-\frac{1}{2}C_2&0\\
0&A_3&\frac{1}{2}B_3&-B_3&C_3&0\\0&0&0&0&0&I_{m-1}\end{pmatrix}.
$$
Let, in the linear case, 
$$
f''_{\Delta(\tau_v,i;m),s}=\rho(\diag(I_{ni},\epsilon^{-1}_0,I_{ni})\epsilon^0)f'_{\Delta(\tau_v,i;m),s},
$$
and in the metaplectic case, let
$$
f''_{\Delta(\tau_v,i;m)\gamma_{\psi_v},s}=\rho(((\diag(I_{ni},\epsilon_0,I_{ni}),1))^{-1}(\epsilon^0,1))f'_{\Delta(\tau_v,i;m)\gamma_{\psi_v},s}.
$$
Assume that $H$ is linear. Then, by \eqref{4'.7}, \eqref{4.8}, 
\begin{multline}\label{4'.9}
\ell'_{\psi_v}(f'_{\Delta(\tau_v,i;m),s})(a_v,g_v)=\\
\int_{U^0_{m(n-1)}(F_v)}f''_{\Delta(\tau_v,i;m),s}(diag(a_v,\delta'_0ut'(g_v,I_m),a^*_v))\psi^{-1}_{H_{2mn},v}(u)du.
\end{multline}
By transitivity of induction, we identify $\rho_{\Delta(\tau_v,i;m),s}$ (see \eqref{4'.6}) and\\
 $\Ind_{Q_{ni}(F_v)}^{H(F_v)}(\Delta(\tau_v,i)|\det\cdot|^{s-\frac{m}{2}}\times \rho_{\Delta(\tau_v,m),s+\frac{i}{2}})$, where
\begin{equation}\label{4'.10}
\rho_{\Delta(\tau_v,m),s+\frac{i}{2}}=\Ind_{Q_{mn}(F_v)}^{H_{2nm}(F_v)}
\Delta(\tau_v,m)|\det\cdot|^{s+\frac{i}{2}}.
\end{equation}
Thus, we can write the integrand in \eqref{4'.9} as a finite sum
\begin{equation}\label{4'.11}
|\det(a_v)|^{s+m(n-1)+\frac{m+ni+\delta_H}{2}}\sum W_v(a_v)\varphi_{\Delta(\tau_v,m),s+\frac{i}{2}}(\delta'_0ut'(g_v,I_m)),
\end{equation}
where $W_v\in W_{\psi^{-1}_{V_{i^n}}}(\Delta(\tau_v,i))$, and $\varphi_{\Delta(\tau_v,m),s+\frac{i}{2}}$ are smooth, holomorphic sections of $\rho_{\Delta(\tau_v,m),s+\frac{i}{2}}$. Substituting \eqref{4'.11} in \eqref{4'.9}, we get
\begin{multline}\label{4'.12}
\ell'_{\psi_v}(f'_{\Delta(\tau_v,i;m),s})(a_v,g_v)=)|\det(a)|^{s+m(n-1)+\frac{m+ni-\delta_H}{2}}\\
\cdot \sum W_v(a_v)\int_{U^0_{m(n-1)}(F_v)}\varphi_{\Delta(\tau_v,m),s+\frac{i}{2}}(\delta'_0ut'(g_v,I_m))\psi^{-1}_{H_{2mn},v}(u)du.
\end{multline}
Similarly, in the metaplectic case, using \eqref{4'.7}, \eqref{4.8.1}, \eqref{4.8.2}, we get an expression as a finite sum, with $a_v\in \GL_{ni}(F_v)$, $\mu_v=\pm 1$, $g_v\in \Sp^{(2)}_{2m'}(F_v)$,
\begin{multline}\label{4'.12.1*}
\ell'_{\psi_v}(f'_{\Delta(\tau_v,i;m)\gamma_{\psi_v},s})((a_v,\mu_v),g_v)=\mu_v\gamma_{\psi_v}(\det(a_v))|\det(a_v)|^{s+m(n-1)+\frac{m+ni-\delta_H}{2}}\\
\cdot \sum W_v(a_v)\int_{U^0_{m(n-1)}(F_v)}\varphi_{\Delta(\tau_v,m)\gamma_{\psi_v},s+\frac{i}{2}}(\delta'_0ut'(g_v,1))\psi^{-1}_{H_{2mn},v}(u)du.
\end{multline}
Here, $W_v\in W_{\psi^{-1}_{V_{i^n}}}(\Delta(\tau_v,i))$, and $\varphi_{\Delta(\tau_v,m)\gamma_{\psi_v},s+\frac{i}{2}}$ are smooth, holomorphic sections of $\rho_{\Delta(\tau_v,m)\gamma_{\psi_v},s+\frac{i}{2}}$.
This completes the proof of the proposition.

\end{proof}

\begin{rmk}\label{remk 4.3}
The integrals 
$$
\int_{U^0_{m(n-1)}(F_v)}\varphi_{\Delta(\tau_v,m)\gamma_{\psi_v}^{(\epsilon)},s+\frac{i}{2}}(\delta'_0ut'(g_v,I_m))\psi^{-1}_{H_{2mn},v}(u)du,
$$
which appear in \eqref{4'.12}, \eqref{4'.12.1} are the inner unipotent integrals, which appear in the local generalized doubling integrals of \cite{CFGK17} (Theorem 1 and (3.1) in loc. cit.).
\end{rmk}

In the following corollary, in the metaplectic case ($\epsilon=2$), for $a_v=(a'_v,\mu_v)\in \GL^{(2)}_{ni}(F_v)$, we denote $\gamma_{\psi_v}(\det(a_v))=\mu_v\gamma_{\psi_v}(\det(a'_v))$, $|\det(a_v)|=|\det(a'_v)|$, and $W_v^k(a_v)=W_v^k(a'_v)$.

\begin{cor}\label{cor 4.4}
In the notation of Prop. \ref{prop 4.1}, given the section $f_{\Delta(\tau_v,i;m)\gamma_{\psi_v}^{(\epsilon)},s}$, there are smooth, holomorphic sections $\varphi^k_{\Delta(\tau_v,m)\gamma_{\psi_v}^{(\epsilon)},s+\frac{i}{2}}$, and functions $W_v^k\in W_{\psi^{-1}_{V_{i^n}}}(\Delta(\tau_v,i))$,  $1\leq k\leq N$, such that, for all $a_v\in \GL^{(\epsilon)}_{ni}(F_v)$, $g_v\in H_m^{(\epsilon)}(F_v)$,
\begin{multline}\label{4'.12.1}
\ell_{\psi_v}(f_{\Delta(\tau_v,i;m)\gamma_{\psi_v}^{(\epsilon)},s})(a_v,g_v)=\gamma_{\psi_v}^{(\epsilon)}(\det(a))|\det(a)|^{s+m(n-1)+\frac{m+ni-\delta_H}{2}}\\
\cdot \sum_{k=1}^N W_v^k(a_v)\int_{U^0_{m(n-1)}(F_v)}\varphi^k_{\Delta(\tau_v,m)\gamma_{\psi_v}^{(\epsilon)},s+\frac{i}{2}}(\delta'_0ut'(g_v,I_m))\psi^{-1}_{H_{2mn},v}(u)du.
\end{multline}
Assume that $f_{\Delta(\tau_v,i;m)\gamma_{\psi_v}^{(\epsilon)},s}=f^0_{\Delta(\tau_v,i;
m)\gamma_{\psi_v}^{(\epsilon)},s}$ (the unramified normalized section defined right after \eqref{4.2.1}). Then \eqref{4'.12.1} takes the form
\begin{multline}\label{4'.16}
\ell_{\psi_v}(f^0_{\Delta(\tau_v,i;m)\gamma_{\psi_v}^{(\epsilon)},s})(a_v,g_v)=\gamma_{\psi_v}^{(\epsilon)}(\det(a))|\det(a)|^{s+m(n-1)+\frac{m+ni+\delta_H}{2}}\\
\cdot W^0_{\Delta(\tau_v,i),\psi_v}(a_v)
 \int_{U^0_{m(n-1)}(F_v)}f^0_{\Delta(\tau_v,m)\gamma_{\psi_v}^{(\epsilon)},s+\frac{i}{2}}(\delta'_0ut'(g_v,I_m))\psi^{-1}_{H_{2mn},v}(u)du,
\end{multline}
where $f^0_{\Delta(\tau_v, m)\gamma_{\psi_v}^{(\epsilon)},s+\frac{i}{2}}$ is the unramified and normalized section of $\rho_{\Delta(\tau_v,m)\gamma_{\psi_v}^{(\epsilon)},s+\frac{i}{2}}$.
\end{cor}
\begin{proof}
The expression \eqref{4'.12.1} follows from Prop. \ref{prop 4.1}, \eqref{4'.12} and\eqref{4'.12.1*}. As for the unramified case, we get, in the linear case, by Prop. \ref{prop 4.1} and \eqref{4'.9},
\begin{multline}\label{4'.13}
\ell_{\psi_v}(f^0_{\Delta(\tau_v,i;m),s})(a_v,g_v)=\\
\int_{U^0_{m(n-1)}(F_v)}f^0_{\Delta(\tau_v,i;m),s}(diag(a_v,\delta'_0ut'(g_v,I_m),a^*_v))\psi^{-1}_{H_{2mn},v}(u)du.
\end{multline}
We have
$$
f^0_{\Delta(\tau_v, m),s+\frac{i}{2}}(h)=f^0_{\Delta(\tau_v,i;
m),s}(diag(I_{ni},h,I_{ni})).
$$
In this case, \eqref{4'.11} becomes
\begin{multline}\label{4'.15}
f^0_{\Delta(\tau_v,i;m),s}(diag(a_v,\delta'_0ut'(g_v,I_m),a^*_v))=|\det(a)|^{s+m(n-1)+\frac{m+ni+\delta_H}{2}}\\
W^0_{\Delta(\tau_v,i),\psi_v}(a_v)f^0_{\Delta(\tau_v, m),s+\frac{i}{2}}(\delta'_0ut'(g_v,I_m)).
\end{multline} 
Now, substitute this in \eqref{4'.13} to get \eqref{4'.16}. The metaplectic case is similar.
\end{proof}

Let us go back to the global section $f_{\Delta(\tau,m+i)\gamma_\psi^{(\epsilon)},s}$ of $\rho_{\Delta(\tau,m+i)\gamma_\psi^{(\epsilon)},s}$ and to $\ell_\psi(f_{\Delta(\tau, m+i)\gamma_\psi^{(\epsilon)},s})(a,g)$ in \eqref{4'.0}. Recall that we fixed  a finite set of places $S$, containing the Archimedean places, outside which $\tau$ and $\psi$ are unramified, $f_{\Delta(\tau,m+i)\gamma_\psi^{(\epsilon)},s}$, and hence $f^\psi_{\Delta(\tau,m+i)\gamma_\psi^{(\epsilon)},s}$ (in \eqref{4'.2}) is right $K_{H(F_v)}$-invariant, for all $v\notin S$. (See right after \eqref{4.2.1}.) Assume that $f_{\Delta(\tau,m+i)\gamma_\psi^{(\epsilon)},s}$ is decomposable, so that $f^\psi_{\Delta(\tau,m+i)\gamma_\psi^{(\epsilon)},s}$ corresponds to the tensor product of the local sections $f_{\Delta(\tau_v,i; m)\gamma^{(\epsilon)}_{\psi_v},s}$, such that, for $v\notin S$, $f_{\Delta(\tau_v,i; m)\gamma^{(\epsilon)}_{\psi_v},s}=f^0_{\Delta(\tau_v,i; m)\gamma^{(\epsilon)}_{\psi_v},s}$. Recall, also, the isomorphisms \eqref{4.2.1.a}, \eqref{4.2.1.b} that we fixed $p_{\tau,i}:\otimes_v'\Delta(\tau_v,i)\rightarrow \Delta(\tau,i)$, $\tilde{p}_{\tau,i}:\otimes'_v\Delta(\tau_v,i)\gamma_{\psi_v}\rightarrow \Delta(\tau,i)\gamma_\psi$, and similarly $p_{\tau,m}$, $\tilde{p}_{\tau,m}$. It will be convenient to denote $p^{(\epsilon)}_{\tau,m}=p_{\tau,m}$, when $\epsilon=1$, and $\tilde{p}_{\tau,m}$, when $\epsilon=2$. Then Cor. \ref{cor 4.4} applied to all places $v\in S$ and all places $v\notin S$ implies

\begin{cor} \label{cor 4.5}
There are automorphic forms $\alpha^j$ in the space of $\Delta(\tau,i)\gamma_\psi^{(\epsilon)}$ and smooth, holomorphic sections $\varphi^j_{\Delta(\tau,m)\gamma_\psi^{(\epsilon)},s+\frac{i}{2}}$ of $\rho_{\Delta(\tau,m)\gamma_\psi^{(\epsilon)},s+\frac{i}{2}}$, $1\leq j\leq N'$, such that, for all $a\in \GL^{(\epsilon)}_{ni}(\BA)$, $g\in H_m^{(\epsilon)}(\BA)$,
\begin{multline}\label{4'.17}
\ell_\psi(f_{\Delta(\tau,m+i)\gamma_\psi^{(\epsilon)},s})(a,g)=\gamma_\psi^{(\epsilon)}(\det(a))|\det(a)|^{s+m(n-1)+\frac{m+ni-\delta_H}{2}}
\cdot \\\sum_{j=1}^{N'} \alpha^j(a)
\int_{U^0_{m(n-1)}(\BA)}\varphi^{j,\psi}_{\Delta(\tau,m)\gamma_\psi^{(\epsilon)},s+\frac{i}{2}}(\delta'_0ut'(g,1))\psi^{-1}_{H_{2mn}}(u)du,
\end{multline}
where
\begin{equation}\label{4'.17.0}
\varphi^{j,\psi}_{\Delta(\tau,m)\gamma_\psi^{(\epsilon)},s+\frac{i}{2}}(h)=\int_{V_{m^n}(F)\backslash V_{m^n}(\BA)}\varphi^j_{\Delta(\tau,m)\gamma_\psi^{(\epsilon)},s+\frac{i}{2}}(\hat{v}h)\psi_{V_{m^n}}(v)dv.
\end{equation}
In particular, for each $g\in H_m^{(\epsilon)}(\BA)$, the function $a\mapsto \ell_\psi(f_{\Delta(\tau,m+i)\gamma_\psi^{(\epsilon)},s})(a,g)$ is an automorphic form in the space of 
\begin{equation}\label{4'.17.1}
|\det\cdot|^{s+m(n-1)+\frac{m+ni-\delta_H}{2}}\Delta(\tau,i)\gamma_\psi^{(\epsilon)}.
\end{equation}
Moreover, for each place $v\in S$ and each $j\leq N'$, there are $W^j_v\in W_{\psi_{V_{i^n},v}}(\Delta(\tau_v,i))$ and smooth, holomorphic sections $\varphi^j_{\Delta(\tau_v,m)\gamma_{\psi_v}^{(\epsilon)},s+\frac{i}{2}}$ of $\rho _{\Delta(\tau_v,m)\gamma_{\psi_v}^{(\epsilon)},s+\frac{i}{2}}$, such that 
\begin{multline}\label{4'.18}
\alpha^j=p_{\tau,i}((\otimes_{v\in S}W^j_v)\otimes (\otimes_{v\notin S}W^0_{\Delta(\tau_v,i),\psi_v})), \ (\textit{linear case})\\
\alpha^j=\tilde{p}_{\tau,i}((\otimes_{v\in S}\tilde{W}^j_v)\otimes (\otimes_{v\notin S}\tilde{W}^0_{\Delta(\tau_v,i),\psi_v})), \ (\textit{metaplectic case})\\
\varphi^j_{\Delta(\tau,m)\gamma_\psi^{(\epsilon)},s+\frac{i}{2}}=p_{\tau,m}^{(\epsilon)}\circ((\otimes_{v\in S}\varphi^j_{\Delta(\tau_v,m)\gamma_{\psi_v}^{(\epsilon)},s+\frac{i}{2}})\otimes(\otimes_{v\notin S}f^0_{\Delta(\tau_v,m)\gamma_{\psi_v}^{(\epsilon)},s+\frac{i}{2}})).
\end{multline}
Finally, the automorphic form, in the space of \eqref{4'.17.1}, $a\mapsto \ell_\psi(f_{\Delta(\tau,m+i)\gamma_\psi^{(\epsilon)},s})(a,g)$ is decomposable, in the sense that
it is equal to the application of $p^{(\epsilon)}_{\tau,i}$ to $\otimes'_v (a_v\mapsto \ell_{\psi_v}(f_{\Delta(\tau_v,i;m)\gamma_{\psi_v}^{(\epsilon)},s})(a_v,g_v))$, that is
\begin{equation}\label{4'.18.1}
\ell_\psi(f_{\Delta(\tau,m+i),s})(\cdot,g)=p^{(\epsilon)}_{\tau,i}(\otimes' \ell_{\psi_v}(f_{\Delta(\tau_v,i;m)\gamma_{\psi_v}^{(\epsilon)},s})(\cdot,g_v)).
\end{equation}
\end{cor} 

Now, we are going to complete the proof of Prop. \ref{prop 3.5}, for a $K_{H(\BA)}$-finite section $f_{\Delta(\tau,m+i)\gamma_\psi^{(\epsilon)},s}$, and hence complete the proof of Theorem A. We assume that it is decomposable as detailed in the paragraph before the last corollary. Recall that we proved one part of it in Prop. \ref{prop 3.6}, and we commented after its proof, right before Prop. \ref{prop 3.7}, on what remains to be proved. We do this now, and, moreover, we will obtain a nice explicit expression of $\Lambda(f_{\Delta(\tau, m+i)\gamma_\psi^{(\epsilon)},s},\varphi_\sigma)$ on $Q_{ni}^{H_{2ni+m}^{(\epsilon)}}(\BA)$.

 Let us assume that $\sigma$ is unramified outside $S$, and fix, for each $v\notin S$, an unramified vector $\varphi_{\sigma_v}^0$ in a space, $V_{\sigma_v}$ of $\sigma_v$. For each place $v$, let $\hat{\sigma}_v$ denote the contragredient of $\sigma_v$, acting on the smooth dual $\hat{V}_{\sigma_v}$ of $V_{\sigma_v}$. Denote by $<\ ,\ >$ the natural pairing on $V_{\sigma_v}\times \hat{V}_{\sigma_v}$.
 For each $v\notin S$, choose the unramified vector $\varphi^0_{\hat{\sigma}_v}\in \hat{V}_{\sigma_v}$, such that $<\varphi^0_{\sigma_v},\varphi^0_{\hat{\sigma}_v}>=1$.
    Also, fix an isomorphism $p_\sigma: \otimes'_vV_{\sigma_v}\rightarrow V_\sigma$, where $V_\sigma$ is the space of cusp forms on $H_m^{(\epsilon)}(\BA)$ on which $\sigma$ acts.  The complex conjugate of $V_\sigma$ is isomorphic to $\otimes'_v\hat{\sigma}_v$ acting in $\otimes'_v\hat{V}_{\sigma_v}$. Choose the isomorphism $p_{\bar{\sigma}}:\otimes'_v\hat{V}_{\sigma_v}\rightarrow \bar{V}_\sigma$, such that, for all  cusp forms $\varphi_\sigma$, $\xi_\sigma$ in $V_\sigma$, with $\varphi_\sigma=p_\sigma(\otimes'_v\varphi_{\sigma_v})$, $\bar{\xi}_\sigma=p_{\bar{\sigma}}(\otimes'_v\xi_{\hat{\sigma}_v})$, we have
$$
<\varphi_\sigma,\xi_\sigma>=\int_{H_m(F)\backslash H_m(\BA)}\varphi_\sigma(g)\bar{\xi}_\sigma(g)dg=\prod_v<\varphi_{\sigma_v},\xi_{\hat{\sigma}_v}>.
$$
In the next theorem, assume that the cusp form $\varphi_\sigma$ is decomposable and corresponds under $p_\sigma$ to $\otimes'_v\varphi_{\sigma_v}$, where, for all $v\notin S$, $\varphi_{\sigma_v}=\varphi^0_{\sigma_v}$. We write the following theorem in the linear case. The metaplectic case is similar.

\begin{thm}\label{thm 4.6}
Given the $K_{H(\BA)}$-finite, holomorphic, decomposable section above, $f_{\Delta(\tau,m+i),s}$, there are
cusp forms $\xi^j_\sigma$ , $\varphi^j_\sigma$ in $V_\sigma$, automorphic forms $\alpha^j$ in the space of $\Delta(\tau,i)$, and smooth, holomorphic sections\\
 $\varphi^j_{\Delta(\tau,m),s+\frac{i}{2}}$ of $\rho_{\Delta(\tau,m),s+\frac{i}{2}}$, 
$1\leq j\leq \ell'$, such that, for $Re(s)$ sufficiently large, and for all $a\in \GL_{ni}(\BA)$, $b\in H_m(\BA)$,
\begin{multline}\label{4'.19}
\delta_{Q_{ni}^{H_{m+2ni}}}^{-\frac{1}{2}}(\hat{a})\Lambda(f_{\Delta(\tau, m+i),s},\varphi_\sigma)(\begin{pmatrix}a\\&b\\&&a^*\end{pmatrix})=|\det(a)|^s
\sum_{j=1}^\ell \alpha^j(a)\xi^j_\sigma(b^\iota)\\
\cdot \int_{H_m(\BA)}< \sigma(g)\varphi_\sigma,\varphi^j_\sigma>\int_{U^0_{m(n-1)}(\BA)} 
\varphi^{j,\psi}_{\Delta(\tau,m),s+\frac{i}{2}}(\delta'_0ut'(g,I_m))\psi^{-1}_{H_{2mn}}(u)dudg.
\end{multline}
In particular, the automorphic form on $\GL_{ni}(\BA)\times H_m(\BA)$,
$$
(a,b)\mapsto \Lambda(f_{\Delta(\tau, m+i),s},\varphi_\sigma)(\begin{pmatrix}a\\&b\\&&a^*\end{pmatrix})
$$ 
takes values in 
$\delta_{Q_{ni}^{H_{m+2ni}}}^{\frac{1}{2}}\Delta(\tau,i)|\det\cdot|^s\otimes
\sigma^\iota$. \\
Moreover, for each place $v\in S$, and each $j\leq \ell$, there are $\xi_{\sigma^j_v}\in V_{\sigma_v}$, $\varphi^j_{\hat{\sigma}_v}\in \hat{V}_{\sigma_v}$, $W^j_v\in W_{\psi_{V_{i^n},v}}(\Delta(\tau_v,i))$ and smooth, holomorphic sections $\varphi^j_{\Delta(\tau_v,m),s+\frac{i}{2}}$ of $\rho _{\Delta(\tau_v,m),s+\frac{i}{2}}$, such that we have the relations \eqref{4'.18}, and
\begin{equation}\label{4'.20}
\xi^j_\sigma=p_\sigma((\otimes_{v\in S}\xi^j_{\sigma_v})\otimes (\otimes_{v\notin S}\varphi^0_{\sigma_v})),\ 
\bar{\varphi}^j_\sigma=p_\sigma((\otimes_{v\in S}\varphi^j_{\hat{\sigma_v}})\otimes (\otimes_{v\notin S}\varphi^0_{\hat{\sigma}_v})).
\end{equation}

\end{thm}

\begin{proof}
By Prop. \ref{prop 3.6}, we can write as a finite sum
\begin{multline}\label{4'.21}
|\det(a)|^{m(n-1)}\Lambda(f_{\Delta(\tau, m+i),s},\varphi_\sigma)(\begin{pmatrix}a\\&b\\&&a^*\end{pmatrix})=\\
\sum \xi_\sigma(b^\iota) \int_{H_m(\BA)} <\sigma(g)\varphi_\sigma,\xi'_\sigma>\int_{U'_{m(n-1)}(\BA)} \tilde{f}^\psi_{\Delta(\tau,
m+i),s}(\hat{a}\delta_0ut(g,I))\psi^{-1}_H(u)dudg\\
=\sum \xi_\sigma(b^\iota)\int_{H_m(\BA)} <\sigma(g)\varphi_\sigma,\xi'_\sigma>\ell_\psi(\tilde{f}_{\Delta(\tau,m+i),s})(a,g)dg,
\end{multline}
where $\xi_\sigma$, $\xi'_\sigma$ lie in the space of $\sigma$ and $\tilde{f}_{\Delta(\tau,m+i),s}$ is a $K_{H(\BA)}$-finite and holomorphic section of $\rho_{\Delta(\tau,m+i), s}$. When we review the proof of Prop. \ref{prop 3.6}, we see that $\tilde{f}_{\Delta(\tau,m+i),s}$ remains unramified outside $S$, as well as the cusp forms $\xi_\sigma$ in \eqref{3.26.2} and the $\xi'_\sigma$ in \eqref{3.26.3} (defined in \eqref{3.26.4}) (all are unramified outside $S$). Write each of $\xi_\sigma$, $\xi'_\sigma$  as a sum of decomposable vectors, which are, respectively, unramified outside $S$, 
$$
\xi_\sigma=\sum_{j_1=1}^{\ell_1}\xi^{j_1}_\sigma,\ \ \xi'_\sigma=\sum_{j_2=1}^{\ell_2}\varphi^{j_2}_\sigma,\
$$
so that, for each of the $\xi^{j_1}_\sigma$, $\varphi^{j_2}_\sigma$ the corresponding vector for $v\notin S$ is $\varphi^0_{\sigma_v}$.
Similarly, write $\tilde{f}_{\Delta(\tau,m+i),s}$ as a sum of decomposable $K_{H(\BA)}$-finite, holomorphic sections, all unramified outside $S$,
$$
\tilde{f}_{\Delta(\tau,m+i),s}=\sum_{j_3=1}^{\ell_3}f^{j_3}_{\Delta(\tau,m+i),s},
$$
so that, for each of the $f^{j_3}_{\Delta(\tau,m+i),s}$, when we write the tensor product $\otimes'_vf^{j_3}_{\Delta(\tau_v,i;m),s}$ corresponding to $(f^{j_3})^\psi_{\Delta(\tau,m+i),s}$, then, for all $v\notin S$, $f^{j_3}_{\Delta(\tau_v,i;m),s}=f^0_{\Delta(\tau_v,i;m),s}$.
Thus
\begin{multline}\label{4'.22}
|\det(a)|^{m(n-1)}\Lambda(f_{\Delta(\tau, m+i),s},\varphi_\sigma)(\begin{pmatrix}a\\&b\\&&a^*\end{pmatrix})=\\
=\sum_{j_1,j_2,j_3} \xi^{j_1}_\sigma(b^\iota)\int_{H_m(\BA)} <\sigma(g)\varphi_\sigma,\varphi^{j_2}_\sigma>\ell_\psi(f^{j_3}_{\Delta(\tau,m+i),s})(a,g)dg.
\end{multline}
Now, we use Cor. \ref{cor 4.5} for $\ell_\psi(f^{j_3}_{\Delta(\tau,m+i),s})(a,g)$, writing it as in \eqref{4'.17}. Substitute this in \eqref{4'.22}, rename the indices, and we get an expression of the form \eqref{4'.19}, such that \eqref{4'.18}, \eqref{4'.20} are satisfied.

\end{proof}

In the metaplectic case, we have (see \eqref{3.24.1}), with analogous notations and properties,  
\begin{multline}\label{4'.23}
\Lambda(f_{\Delta(\tau, m+i),s},\varphi_\sigma)(p((\begin{pmatrix}a\\&b\\&&a^*\end{pmatrix},\bar{\alpha})))=\\ \delta_{Q_{ni}^{H_{m+2ni}}}^{\frac{1}{2}}(\hat{a})(\det(a),x(b))\gamma_\psi(\det(a))|\det(a)|^s 
\sum_{j=1}^\ell \alpha^j(a)\xi^j_\sigma(p((b,\bar{\alpha}))^\iota)
\cdot \\
\cdot \int_{C_2\backslash H^{(2)}_m(\BA)}< \sigma(g)\varphi_\sigma,\varphi^j_\sigma>\int_{U^0_{m(n-1)}(\BA)} 
\varphi^j_{\Delta(\tau,m)\gamma_\psi,s+\frac{i}{2}}(\delta'_0ut'(g,I_m))\psi^{-1}_{H_{2mn}}(u)dudg.
\end{multline}

At this point, we have completed the proof of Theorem A. Recall that we started with the kernel integral \eqref{1.10},
$$
\mathcal{E}(f_{\Delta(\tau,m+i)\gamma^{(\epsilon)}_\psi,s},\varphi_\sigma)(h)=
\int_{C^{(\epsilon)}H_m(F)\backslash
H^{(\epsilon)}_m(\BA)}\mathcal{F}_\psi(E(f_{\Delta(\tau,
m+i)\gamma^{(\epsilon)}_\psi,s}))(g,h)\varphi_\sigma(g)dg.
$$
We proved in \eqref{3.20} that, for $Re(s)$ sufficiently large, and $h\in H_{m+2ni}^{(\epsilon)}(\BA)$,
$$
\mathcal{E}(f_{\Delta(\tau,
m+i)\gamma_\psi^{(\epsilon)},s},\varphi_\sigma)(h)=\sum_{h'\in Q_{ni}(F)\backslash
H_{m+2ni}(F)}\Lambda(f_{\Delta(\tau, m+i)\gamma_\psi^{(\epsilon)},s},\varphi_\sigma)(h'h).
$$
Then in Prop. \ref{prop 3.7} and in Theorem \ref{thm 4.6}, we proved that $\Lambda(f_{\Delta(\tau, m+i)\gamma_\psi^{(\epsilon)},s},\varphi_\sigma)$ admits a meromorphic continuation to the complex plane, and when $f_{\Delta(\tau, m+i)\gamma_\psi^{(\epsilon)},s}$ is $K_{H(\BA)}$-finite, $\Lambda(f_{\Delta(\tau, m+i)\gamma_\psi^{(\epsilon)},s},\varphi_\sigma)$ defines a meromorphic ($K_{H^{(\epsilon)}_{m+2ni}(\BA)}$-finite) section of  $\rho_{\Delta(\tau,i)\gamma_\psi^{(\epsilon)},\sigma^\iota,s}$.
Thus, $\mathcal{E}(f_{\Delta(\tau,m+i)\gamma_\psi^{(\epsilon)},s},\varphi_\sigma)$ is an Eisenstein series on $H_{m+2ni}^{(\epsilon)}(\BA)$, corresponding to the section $\Lambda(f_{\Delta(\tau, m+i)\gamma_\psi^{(\epsilon)},s},\varphi_\sigma)$.

Note that the integrals \eqref{4'.19} (in the metaplectic case, \eqref{4'.23}) are the generalized doubling integrals  (\cite{CFGK17}, Theorem 1, \cite{CFGK16}, Theorem 2) obtained after the unfolding process, for $H_m^{(\epsilon)}(\BA)\times \GL_n(\BA)$, representing $L_{\epsilon,\psi}^S(\sigma\times,\tau,s+\frac{i+1}{2})$. (Recall that in the metaplectic case, the corresponding partial $L$-function depends on the choice of the additive character $\psi$, and that in this case, we denote it by $L_\psi^S(\sigma\times\tau,s+\frac{i+1}{2})$. Our notation is unified, such that, in case $\epsilon=1$, this is $L^S(\sigma\times\tau,s+\frac{i+1}{2})$, and the case $\epsilon=2$ corresponds to the metaplectic case.) In particular, we know that the integrals \eqref{4'.19} , \eqref{4'.22} have an analytic continuation to meromorphic functions in the complex plane. This gives another proof of Prop. \ref{prop 3.7}. 
For later use, denote the global generalized doubling integrals as in \eqref{4'.19}, \eqref{4'.23} by
\begin{multline}\label{4'.23*}
\mathcal{Z}(\varphi_\sigma,,\varphi^j_\sigma,\varphi^j_{\Delta(\tau,m)\gamma_\psi^{(\epsilon)},s+\frac{i}{2}})=\\
 \int_{C^{(\epsilon)}_2\backslash H^{(\epsilon)}_m(\BA)}< \sigma(g)\varphi_\sigma,\varphi^j_\sigma>\int_{U^0_{m(n-1)}(\BA)} 
\varphi^j_{\Delta(\tau,m)\gamma^{(\epsilon)}_\psi,s+\frac{i}{2}}(\delta'_0ut'(g,I_m))\psi^{-1}_{H_{2mn}}(u)dudg.
\end{multline}
Here, $Re(s)$ is sufficiently large.

\begin{thm}\label{thm 4.6.1}
Keep the assumptions and notations of Theorem \ref{thm 4.6}. Then \\
$\Lambda(f_{\Delta(\tau, m+i)\gamma_\psi^{(\epsilon)},s},\varphi_\sigma)$ is decomposable. We have 
\begin{equation}\label{4'.23.-1}
\Lambda(f_{\Delta(\tau, m+i)\gamma_\psi,s},\varphi_\sigma)=(p^{(\epsilon)}_{\tau,i}\otimes (\iota\circ p_\sigma))(\otimes'_v \Lambda_v(f_{\Delta(\tau_v,i;m)\gamma^{(\epsilon)}_{\psi_v},s},\varphi_{\sigma_v})),
\end{equation}
where, for each place $v$,  
$\Lambda_v(f_{\Delta(\tau_v,i;m)\gamma_{\psi_v}^{(\epsilon)},s},\varphi_{\sigma_v})$ is the smooth, meromorphic section of 
$$
\rho_{\Delta(\tau_v,i)\gamma_{\psi_v}^{(\epsilon)},\sigma^\iota,s}=\Ind_{(Q_{ni}^{H_{m+2ni}})^{(\epsilon)}(F_v)}^{H_{m+2ni}^{(\epsilon)}(F_v)}\Delta(\tau_v,i)\gamma_{\psi_v}^{(\epsilon)}|\det\cdot|^s\otimes \sigma_v^\iota,
$$
given, for $Re(s)$ sufficiently large, and $h_v\in H^{(\epsilon)}_{m+2ni}(F_v)$, by
\begin{multline}\label{4'.23.0}
\Lambda_v(f_{\Delta(\tau_v,i;m)\gamma_{\psi_v}^{(\epsilon)},s},\varphi_{\sigma_v})(h)=\\
|det_{\GL_{ni}}\cdot |^{-m(n-1)}\int_{C_2^{(\epsilon)}\backslash H^{(\epsilon)}_m(F_v)}(\sigma_v(g_v)(\varphi_{\sigma_v})\otimes \ell_{\psi_v}(\rho (t(1,h_v))f_{\Delta(\tau_v,i;m)\gamma_{\psi_v}^{(\epsilon)},s})(\cdot,g_v))dg_v.
\end{multline}
Let $v\notin S$ (so that $\sigma_v$, $\tau_v$ are unramified). Then
\begin{multline}\label{4'.23.01}
\Lambda_v(f^0_{\Delta(\tau_v,i;m)\gamma_{\psi_v}^{(\epsilon)},s},\varphi^0_{\sigma_v})(1)=
(\gamma_{\psi_v}^{(\epsilon)}\circ \det)|det_{\GL_{ni}}\cdot |^{s+\frac{m+ni-\delta_H}{2}}(W^0_{\Delta(\tau_v,i),\psi_v}\otimes \varphi^0_{\sigma_v})\cdot \\ \cdot \int_{C_2^{(\epsilon)}\backslash H^{(\epsilon)}_m(F_v)}<\sigma_v(g_v)(\varphi^0_{\sigma_v}),\varphi^0_{\hat{\sigma}_v}>\int_{U^0_{m(n-1)}(F_v)}f^0_{\Delta(\tau_v,m)\gamma_{\psi_v}^{(\epsilon)},s+\frac{i}{2}}(\delta'_0ut'(g_v,1))\psi^{-1}_{H_{2mn},v}dudg_v.
\end{multline}

\end{thm}
\begin{proof}
In the linear case, by \eqref{3.22}, for $Re(s)$ large enough, $a\in \GL_{ni}(\BA)$, $b\in H_m(\BA)$, $h\in H_{m+2ni}(\BA)$,
\begin{multline}\label{4'.23.1}
\Lambda(f_{\Delta(\tau, m+i),s},\varphi_\sigma)(\begin{pmatrix}a\\&b\\&&a^*\end{pmatrix}h)=\\
=|\det(a)|^{-m(n-1)}\int_{H_m(\BA)} \sigma(g)\varphi_\sigma(b^\iota)\ell_\psi(\rho(t(1,h))f_{\Delta(\tau,m+i),s})(a,g)dg.
\end{multline}
The proof of Theorem \ref{thm 4.6} shows that, for fixed $h$, the l.h.s. of \eqref{4'.23.1}, as an automorphic form on $\GL_{ni}(\BA)\times H_m(\BA)$,
lies in the space of 
$\delta_{Q_{ni}^{H_{m+2ni}}}^{\frac{1}{2}}\Delta(\tau,i)|\det\cdot|^s\otimes \sigma^\iota$.
By the last part of Cor. \ref{cor 4.5}, this automorphic form is the image under $p_{\tau,i}\otimes (\iota\circ p_\sigma)$ of 
\begin{equation}\label{4'.23.3}
\otimes'_v(|det_{\GL_{ni}(F_v)}\cdot |^{-m(n-1)}\int_{H_m(F_v)}(\sigma_v(g_v)(\varphi_{\sigma_v})\otimes \ell_{\psi_v}(\rho(t(1,h_v))f_{\Delta(\tau_v,i;m),s})(\cdot,g_v))dg_v).
\end{equation}
In the metaplectic case, for $h\in H^{(2)}_{m+2ni}(\BA)$, we consider the automorphic form on $\GL_{ni}^{(2)}(\BA)\times \Sp^{(2)}_{2m'}(\BA)$, in the space of 
$\delta_{Q_{ni}^{H_{m+2ni}}}^{\frac{1}{2}}\Delta(\tau,i)\gamma_\psi |\det\cdot|^s\otimes
\sigma$,
\begin{equation}\label{4'.23.5}
((a,\mu),p((b,\bar{\alpha})))\mapsto \mu (\det(a),x(b))\Lambda(f_{\Delta(\tau, m+i)\gamma_\psi,s},\varphi_\sigma)(p((\begin{pmatrix}a\\&b\\&&a^*\end{pmatrix},\bar{\alpha})) h).
\end{equation} 
Then it is the image under $\tilde{p}_{\tau,i}\otimes (\iota\circ p_\sigma)$ of 
\begin{equation}\label{4'.23.6}
\otimes'_v(|det_{\GL_{ni}(F_v)}\cdot |^{-m(n-1)}\int_{C_2\backslash H^{(2)}_m(F_v)}(\sigma_v(g_v)(\varphi_{\sigma_v})\otimes \ell_{\psi_v}(\rho(t(1,h_v))f_{\Delta(\tau_v,i;m)\gamma_\psi,s})(\cdot,g_v))dg_v).
\end{equation}
Denote the factor at the place $v$, in \eqref{4'.23.3}, or in \eqref{4'.23.6}, by $\Lambda_v(\rho(h_v)f_{\Delta(\tau_v,i;m)\gamma_{\psi_v}^{(\epsilon)},s},\varphi_{\sigma_v})$. This is a section of $\rho_{\Delta(\tau_v,i)\gamma_{\psi_v}^{(\epsilon)},\sigma^\iota,s}$.  Let us explain this. For simplicity of notation, we do this in the linear case. The proof in the metaplectic case is completely similar. Let $x_v\in \GL_{ni}(F_v)$, $y_v\in H_m(F_v)$, $u_v\in U_{ni}^{H_{m+2nui}}(F_v)$. We have 
\begin{multline}\label{4'.23.7}
\Lambda(f_{\Delta(\tau, m+i),s},\varphi_\sigma)(\begin{pmatrix}x_v\\&y_v\\&&x_v^*\end{pmatrix}u_vh_v)=\\
=\Lambda(\rho(t(1,\begin{pmatrix}x_v\\&y_v\\&&x^*_v\end{pmatrix})u_vh_v)f_{\Delta(\tau, m+i),s},\varphi_\sigma)(I_{m+2ni}).
\end{multline}
Apply \eqref{4'.23.1} to each side of \eqref{4'.23.7} . Then we get that, for all $r_v\in \GL_{ni}(F_v)$,
\begin{multline}\label{4'.23.8}
\int_{H_m(F_v)}\sigma_v(g_v)(\varphi_{\sigma_v}) \ell_{\psi_v}(\rho(t(1,\begin{pmatrix}x_v\\&y_v\\&&x^*_v\end{pmatrix}u_vh_v))f_{\Delta(\tau_v,i;m),s})(r_v,g_v)dg_v=\\
|\det(x_v)|^{-m(n-1)}\int_{H_m(F_v)}\sigma_v(y^\iota_v)(\sigma_v(g_v)(\varphi_{\sigma_v})) \ell_{\psi_v}(\rho(h_v)f_{\Delta(\tau_v,i;m),s})(r_vx_v,g_v)dg_v.
\end{multline}
We may assume now that $h_v$ is the identity element. As in the beginnibg of the proof of Theorem \ref{thm 4.6}, we can follow the proof of Prop. \ref{prop 3.6} at the place $v$, and get that the r.h.s. of \eqref{4'.23.8} is equal to a finite sum\\
\\
$|\det(x_v)|^{-m(n-1)}\cdot$
\begin{equation}\label{4'.23.9}
\cdot\sum \sigma_v(y_v^\iota)\xi_{\sigma_v}\int_{H_m(F_v)} <\sigma_v(g_v)\varphi_{\sigma_v},\xi'_{\hat{\sigma}_v}>\ell_{\psi_v}(\tilde{f}_{\Delta(\tau_v,i;m),s})(r_vx_v,g_v)dg_v,
\end{equation}
where $\xi_{\sigma_v}$ lies in the space of $\sigma_v$, $\xi'_{\hat{\sigma}_v}$ lise in the space of $\hat{\sigma}_v$ and $\tilde{f}_{\Delta(\tau_v,i;m),s}$ is a $K_{H(F_v)}$-finite and holomorphic section of $\rho_{\Delta(\tau_v,i;m),s}$.
By Cor. \ref{cor 4.4}, for each summand in \eqref{4'.23.9}, there are smooth, holomorphic sections $\varphi_{\Delta(\tau_v,m),s+\frac{i}{2}}$, and functions $W_v\in W_{\psi^{-1}_{V_{i^n}}}(\Delta(\tau_v,i))$, such that \eqref{4'.12.1} holds (for the term $\ell_{\psi_v}(\tilde{f}_{\Delta(\tau_v,i;m),s})$). Thus, \eqref{4'.23.9} is equal to a sum of the following form (for all $r_v, x_v\in \GL_{ni}(F_v)$, $y_v\in H_m^{(\epsilon)}(F_v)$), 
\begin{multline}\label{4'.23.10}
|\det(x_v)|^{s+\frac{m+ni-\delta_H}{2}}|\det(r_v))|^{s+m(n-1)+\frac{m+ni-\delta_H}{2}}\cdot \sum W_v(r_vx_v)\sigma_v(y_v^\iota)\xi_{\sigma_v}\\
\int<\sigma_v(g_v)\varphi_{\sigma_v},\xi'_{\hat{\sigma}_v}>\int_{U^0_{m(n-1)}(F_v)}\varphi_{\Delta(\tau_v,m),s+\frac{i}{2}}(\delta'_0ut'(g_v,I_m))\psi^{-1}_{H_{2mn},v}(u)dudg_v,
\end{multline} 
where $g_v$ is integrated over $H_m(F_v)$. Thus,
\begin{multline}\label{4'.23.11}
\Lambda_v(\rho(t(1,\begin{pmatrix}x_v\\&y_v\\&&x^*_v\end{pmatrix}u_v))f_{\Delta(\tau_v,i;m),s},\varphi_{\sigma_v})=\\
\sum\rho(x_v)(|\det\cdot|^{s+\frac{m+ni-\delta_H}{2}}W_v)\otimes \sigma_v(y_v^\iota)\xi_{\sigma_v}\\
\int<\sigma_v(g_v)\varphi_{\sigma_v},\xi'_{\hat{\sigma}_v}>\int_{U^0_{m(n-1)}(F_v)}\varphi_{\Delta(\tau_v,m),s+\frac{i}{2}}(\delta'_0ut'(g_v,I_m))\psi^{-1}_{H_{2mn},v}(u)dudg_v.
\end{multline} 
Recall that $\rho(x_v)$ denotes a right translation by $x_v$.
This shows that, for $Re(s)$ sufficiently large, $\Lambda_v(\rho(h_v)f_{\Delta(\tau_v,i;m),s},\varphi_{\sigma_v})$ defines a smooth (holomorphic in the domain of absolute convergence) section of $\rho_{\Delta(\tau_v,i),\sigma_v^\iota,s}$. 

Note that the integrals which appear in \eqref{4'.23.11} are local generalized doubling integrals. In particular, they admit a meromorphic continuation to the whole plane, and hence $\Lambda_v(\rho(h_v)f_{\Delta(\tau_v,i;m),s},\varphi_{\sigma_v})$
is meromorphic.

Finally, let $v\notin S$, and again, just for simplicity of notation, we stay with the linear case. The proof in the metaplectic case is completely similar. By \eqref{4'.23.8} (we assume that the measure of $K_{H_m(F_v)}$ is one),
\begin{multline}\label{4'.23.12}
\Lambda_v(f^0_{\Delta(\tau_v,i;m),s},\varphi^0_{\sigma_v})=
\int_{K_{H_m(F_v)}}\Lambda_v(\rho(t(1,\begin{pmatrix}I_{ni}\\&k\\&&I_{ni}\end{pmatrix}))f^0_{\Delta(\tau_v,i;m),s},\varphi^0_{\sigma_v})=\\
=\int_{H_m(F_v)}\int_{K_{H_m(F_v)}}\sigma_v(k^\iota_v)(\sigma_v(g_v)(\varphi^0_{\sigma_v})) \ell_{\psi_v}(f^0_{\Delta(\tau_v,i;m),s})(\cdot,g_v)dkdg_v= \\
(\int_{H_m(F_v)}<\sigma_v(g_v)(\varphi^0_{\sigma_v}),\varphi^0_{\hat{\sigma}_v}> \ell_{\psi_v}(f^0_{\Delta(\tau_v,i;m),s})(\cdot,g_v)dg_v)\cdot \varphi^0_{\sigma_v}.
\end{multline}
By \eqref{4'.16},
\begin{multline}\label{4'.23.13}
\Lambda_v(f^0_{\Delta(\tau_v,i;m),s},\varphi^0_{\sigma_v})=|\det\cdot|^{s+\frac{m+ni-\delta_H}{2}}
( W^0_{\Delta(\tau_v,i),\psi_v}\otimes \varphi^0_{\sigma_v})\cdot \\
\int<\sigma_v(g_v)(\varphi^0_{\sigma_v}),\varphi^0_{\hat{\sigma}_v}>\int_{U^0_{m(n-1)}(F_v)}f^0_{\Delta(\tau_v,m),s+\frac{i}{2}}(\delta'_0ut'(g_v,I_m))\psi^{-1}_{H_{2mn},v}(u)dudg_v.
\end{multline}
This completes the proof of the theorem.

\end{proof}

Now, we proceed to the proof of Theorem \ref{thm 1.1}. 
As we noted in the last proof, the integral \eqref{4'.23.01} is a typical local integral of the generalized doubling method. It continues to a meromorphic function in the whole complex plane. We denote it, as the local version of \eqref{4'.23*} at the place $v$, by $\mathcal{Z}_v(\varphi_{\sigma_v},\xi'_{\hat{\sigma}_v},\varphi_{\Delta(\tau_v,m)\gamma_{\psi_v}^{(\epsilon)},s+\frac{i}{2}})$. For $Re(s)$ sufficiently large, 
\begin{multline}\label{4'.24}
\mathcal{Z}_v(\varphi_{\sigma_v},\xi'_{\hat{\sigma}_v},\varphi_{\Delta(\tau_v,m)\gamma_{\psi_v}^{(\epsilon)},s+\frac{i}{2}})=\\
\int<\sigma_v(g)(\varphi_{\sigma_v}),\xi'_{\hat{\sigma}_v}>\int_{U^0_{m(n-1)}(F_v)} 
\varphi_{\Delta(\tau_v,m)\gamma_{\psi_v}^{(\epsilon)}
,s+\frac{i}{2}}(\delta'_0ut'(g,I_m))\psi^{-1}_{H_{2mn},v}(u)dudg,
\end{multline}
where $g$ is integrated over $C_2^{(\epsilon)}\backslash H_m(F_v)$. 
This is the local integral at the place $v$ of the integral of the form \eqref{4'.23*}. Thus, let $\varphi_\sigma$ $\xi'_\sigma$ be decomposable cusp forms in the space of $V_\sigma$ of $\sigma$, such that $\varphi_\sigma=p_\sigma(\otimes'_v \varphi_{\sigma_v})$, $\bar{\xi}'_\sigma=p_{\bar{\sigma}}(\otimes'_v\xi'_{\hat{\sigma}_v})$. See right before Theorem \ref{thm 4.6}. Let $\varphi_{\Delta(\tau,m)\gamma_\psi^{(\epsilon)},s+\frac{i}{2}}$ be a smooth, holomorphic section of 
$\rho_{\Delta(\tau,m)\gamma_\psi^{(\epsilon)},s+\frac{i}{2}}$. Then the integral \eqref{4'.23*}, defining $\mathcal{Z}(\varphi_\sigma,\xi'_\sigma,\varphi_{\Delta(\tau,m)\gamma_\psi^{(\epsilon)},s+\frac{i}{2}})$, which converges absolutely, for $Re(s)$ sufficiently large, admits an analytic continuation to a meromorphic function to the whole complex plane. It is Eulerian and we have (equality of meromorphic functions)
\begin{equation}\label{4'.25}
\mathcal{Z}(\varphi_\sigma,\xi'_\sigma,\varphi_{\Delta(\tau,m)\gamma_\psi^{(\epsilon)},s+\frac{i}{2}})=\prod_v 
\mathcal{Z}_v(\varphi_{\sigma_v},\xi'_{\hat{\sigma}_v},\varphi_{\Delta(\tau_v,m)\gamma_{\psi_v}^{(\epsilon)},s+\frac{i}{2}}).
\end{equation}
Note that by \eqref{4'.23.13} (and its analog in the metaplectic case),
\begin{multline}\label{4'.23.14}
\Lambda_v(f^0_{\Delta(\tau_v,i;m),s},\varphi^0_{\sigma_v})=(\gamma_{\psi_v}^{(\epsilon)}\circ\det\cdot)|\det\cdot|^{s+\frac{m+ni-\delta_H}{2}}
( W^0_{\Delta(\tau_v,i),\psi_v}\otimes \varphi^0_{\sigma_v})\cdot \\
\cdot \mathcal{Z}_v(\varphi^0_{\sigma_v},\varphi^0_{\hat{\sigma}_v},\varphi^0_{\Delta(\tau_v,m)\gamma_{\psi_v}^{(\epsilon)},s+\frac{i}{2}}).
\end{multline}
In  \cite{CFGK17}, Sec. 3, the unramified computation of these local integrals is carried out. (The analog in the metaplectic case is a special case of the generalized doubling integrals for covering groups. See \cite{CFGK16}. It can also be done specifically along the lines of Sec. 3 in \cite{CFGK17}.) We write the unramified computation explicitly in the next proposition, as we will use it soon.

\begin{prop}\label{prop 4.7}
Let $v$ be a place outside $S$. Then
$$
\mathcal{Z}_v(f^0_{\Delta(\tau_v,m)
\gamma_{\psi_v}^{(\epsilon)},s+\frac{i}{2}},\varphi^0_{\sigma_v},\varphi^0_{\hat{\sigma}_v})=\frac{L_{\epsilon,\psi_v}(\sigma_v\times
\tau_v,s+\frac{i+1}{2})}{D^{H,m}_{\tau_v}(s)}.
$$
When $m=2m'$, $H=\Sp_{2n(2m'+i)}$ ,
$$
D^{H,m}_{\tau_v}(s)=L(\tau_v,s+m'+\frac{i+1}{2})\prod_{k=1}^{m'}L(\tau_v,\wedge^2,2s+2k+i)L(\tau_v,sym^2,2s+2k+i-1).
$$
When $m=2m'$, $H=\Sp^{(2)}_{2n(2m'+i)}$,
$$
D^{H,m}_{\tau_v}(s)=\prod_{k=1}^{m'}L(\tau_v,\wedge^2,2s+2k+i-1)L(\tau_v,sym^2,2s+2k+i)
$$
When $m=2m'$, $H=\SO_{2n(2m'+i)}$,
$$
D^{H,m}_{\tau_v}(s)=\prod_{k=1}^{m'}L(\tau_v,\wedge^2,2s+2k+i)L(\tau_v,sym^2,2s+2k+i-1).
$$
When $m=2m'-1$, $H=\SO_{2n(2m'-1+i)}$,
$$
D^{H,m}_{\tau_v}(s)=\prod_{k=1}^{m'}L(\tau_v,\wedge^2,2s+2k+i-1)\prod_{k=1}^{m'-1}L(\tau_v,sym^2,2s+2k+i).
$$
\end{prop}

Recall the normalizing factor, outside of $S$, $d_\tau^{H,S}$, of
$E(f_{\Delta(\tau,m+i)\gamma_\psi^{(\epsilon)},s},\cdot)$ on $H(\BA)$. See
\eqref{1.10.4}-\eqref{1.10.9}. Let $d^H_{\tau_v}$ be its local
analog at $v$ outside $S$. Note that in Prop. \ref{prop 4.7},
$D^{H,m}_{\tau_v}(s)=d_{\tau_v}^{H^{(\epsilon)}_{2nm}}(s+\frac{i}{2})$. Recall, also (\eqref{1.10.9.1} - \eqref{1.10.9.8}) the normalizing factor, outside of $S$, $d_{\sigma,\tau}^{H_{2ni+m}^{(\epsilon),S}}(s)$, corresponding to an Eisenstein series attached to the global induced representation
$$
\rho_{\Delta(\tau,i)\gamma_\psi^{(\epsilon)},\sigma^\iota,s}=\Ind_{Q^{(\epsilon)}_{ni}(\BA)}^{H^{(\epsilon)}_{2ni+m}(\BA)}\Delta(\tau,i)\gamma_\psi^{(\epsilon)}|\det\cdot|^s\otimes \sigma^\iota.
$$
Similarly, let $d_{\sigma_v,\tau_v}^{H_{2ni+m}^{(\epsilon)}}(s)$ be the local analog at $v$ outside $S$.
\begin{prop}\label{prop 4.4}
In the notation above, for $v$ outside $S$,
\begin{equation}\label{4.10}
d^H_{\tau_v}(s)L_{\epsilon,\psi_v}(\sigma_v\times
\tau_v,s+\frac{i+1}{2})=d_{\tau_v}^{H^{(\epsilon)}_{2nm}}(s+\frac{i}{2})d_{\sigma_v,\tau_v}^{H_{2ni+m}^{(\epsilon)}}(s).
\end{equation}
\end{prop}
\begin{proof}
The proof is by a straight forward verification. Let us carry this out in one example. Take $H=\Sp_{2n(m+i)}$, when $i=2i'+1$ is odd. Here, $m=2m'$ is even (as always in the symplectic, or metaplectic cases). Then, as in \eqref{1.10.4},\\
\\
$d_{\tau_v}^{\Sp_{2n(m+i)}}(s)=$ 
\begin{equation}\label{4.10.1}
L(\tau_v,s+m'+i'+1)\prod_{k=1}^{m'+i'+1}L(\tau_v,\wedge^2,2s+2k-1)\prod_{k=1}^{m'+i'}L(\tau_v,sym^2,2s+2k).
\end{equation}
We have, as in \eqref{1.10.5},\\
\\
$D_{\tau_v}^{H,m}(s)=d_{\tau_v}^{\Sp_{4nm'}}(s+\frac{i}{2})=$
\begin{equation}\label{4.10.2}
L(\tau_v,s+m'+i'+1)\prod_{k=i'+2}^{m'+i'+1}L(\tau_v,\wedge^2,2s+2k-1)\prod_{k=i'+1}^{m'+i'}L(\tau_v,sym^2,2s+2k).
\end{equation}
Thus, by \eqref{4.10.1}, \eqref{4.10.2}, 
\begin{equation}\label{4.10.3}
d_{\tau_v}^{\Sp_{2n(m+i)}}(s)=d_{\tau_v}^{\Sp_{4nm'}}(s+\frac{i}{2})\prod_{k=1}^{i'+1}L(\tau_v,\wedge^2,2s+2k-1)\prod_{k=1}^{i'}L(\tau_v,sym^2,2s+2k).
\end{equation}
We note that
\begin{equation}\label{4.10.4}
d_{\sigma_v,\tau_v}^{\Sp_{2ni+m}}(s)=L(\sigma_v\times\tau_v,s+i'+1)\prod_{k=1}^{i'+1}L(\tau_v,\wedge^2,2s+2k-1)\prod_{k=1}^{i'}L(\tau_v,sym^2,2s+2k).
\end{equation}
Now, \eqref{4.10.3}, \eqref{4.10.4} imply that
$$
d_{\tau_v}^{\Sp_{2n(m+i)}}(s)L(\sigma_v\times\tau_v,s+i'+1)=d_{\tau_v}^{\Sp_{4nm'}}(s+\frac{i}{2})d_{\sigma_v,\tau_v}^{\Sp_{2ni+m}}(s),
$$
which is \eqref{4.10} in this case. All other cases are checked similarly.
\end{proof}

Denote in Prop. \ref{prop 4.7},
$$
f^{0,*}_{\Delta(\tau_v,i;m)\gamma^{(\epsilon)}_{\psi_v},s}=d^H_{\tau_v}(s)f^0_{\Delta(\tau_v,i; m)\gamma^{(\epsilon)}_{\psi_v},s}.
$$
Then \eqref{4'.23.14}, Prop. \ref{prop 4.7} and \eqref{4.10} imply
\begin{equation}\label{4.11}
\Lambda_v(f^{0,*}_{\Delta(\tau_v,i;m),s},\varphi^0_{\sigma_v})=d_{\sigma_v,\tau_v}^{H_{2ni+m}^{(\epsilon)}}(s)(\gamma_{\psi_v}^{(\epsilon)}\circ\det\cdot)\delta^{\frac{1}{2}}_{Q_{ni}^{H_{m+2ni}}}|\det\cdot |^s
( W^0_{\Delta(\tau_v,i),\psi_v}\otimes \varphi^0_{\sigma_v}).
\end{equation}
The factor $d_{\sigma_v,\tau_v}^{H_{2ni+m}^{(\epsilon)}}(s)$ is the normalizing factor, at $v$, corresponding to an Eisenstein series
associated to a section of $\rho_{\Delta(\tau,i)\gamma_\psi^{(\epsilon)},\sigma^\iota,s}$, say a
decomposable section, such that at a place $v$, outside $S$, the
local section is unramified and normalized, such that its value at the identity element is a tensor product of two unramified
vectors in the spaces of $\Delta(\tau_v,i)\gamma_{\psi_v}^{(\epsilon)}$ and $ \sigma_v$, which we fix outside
$S$, to be $(\gamma_{\psi_v}^{(\epsilon)}\circ\det )W^0_{\Delta(\tau_v,i),\psi_v}$ and $\varphi^0_{\sigma_v}$.
See \eqref{1.10.9.1} - \eqref{1.10.9.8}. Thus, by \eqref{4.11} $(d_{\sigma_v,\tau_v}^{H_{2ni+m}^{(\epsilon)}}(s))^{-1}\Lambda_v(f^{0,*}_{\Delta(\tau_v,i;m),s},\varphi^0_{\sigma_v})$ is this normalized section, for $v\notin S$. Denote
$$
\lambda_S(f_{\Delta(\tau,
m+i)\gamma^{(\epsilon)}_\psi,s},\varphi_\sigma)=\frac{d_\tau^{H,S}(s)}{d_{\sigma,\tau}^{H_{2ni+m}^{(\epsilon)},S}(s)}\Lambda(f_{\Delta(\tau,
m+i)\gamma^{(\epsilon)}_\psi,s},\varphi_\sigma).
$$
Consider the Eisenstein series on $H^{(\epsilon)}_{2ni+m}(\BA)$, corresponding
to this section. Denote it by $E(\lambda_S(f_{\Delta(\tau,
m+i)\gamma^{(\epsilon)}_\psi,s},\varphi_\sigma)$, and in normalized form,
$$
E_S^*(\lambda_S(f_{\Delta(\tau,
m+i)\gamma^{(\epsilon)}_\psi,s},\varphi_\sigma)=d_{\sigma,\tau}^{H_{2ni+m}^{(\epsilon)},S}(s)E(\lambda_S(f_{\Delta(\tau,
m+i)\gamma^{(\epsilon)}_\psi,s},\varphi_\sigma).
$$
Now we get the identity of Theorem \ref{thm 1.1},
\begin{equation}\label{4.12}
\mathcal{E}^*_S(f_{\Delta(\tau,
m+i)\gamma^{(\epsilon)}_\psi,s},\varphi_\sigma)=E^*_S(\lambda_S(f_{\Delta(\tau,
m+i)\gamma^{(\epsilon)}_\psi,s},\varphi_\sigma)),
\end{equation}
where the left hand side of \eqref{4.12} is our kernel integral
\eqref{1.10}, with the normalized Eisenstein series on $H(\BA)$,
$E^*_S(f_{\Delta(\tau,m+i)\gamma^{(\epsilon)}_\psi,s})=d_\tau^{H,S}(s)E(f_{\Delta(\tau,m+i)\gamma^{(\epsilon)}_\psi,s})$. See \eqref{1.10.9*}.

\section{Application of Bessel coefficients to $E(f_{\Delta(\tau,i+1),s})$ on $\SO_{2n(1+i)}(\BA)$: descent to $\SO_{2ni+1}(\BA)$}

Let $H=\SO_{2n(i+1)}$. The previous theorems are valid here when we take $m=1$. Thus, in this case, $H_1=\SO_1$ is the trivial group, and $H_{m+2ni}$ is $\SO_{2ni+1}$. Let us write the group $U_{m^{n-1}}$ and the character $\psi_H$ in this case. The elements of $U_{1^{n-1}}$ have the form
\begin{equation}\label{5.1}
u=\begin{pmatrix}z&x&y\\&I_{2ni+2}&x'\\&&z^*\end{pmatrix}\in H,\ z\in Z_{n-1}
\end{equation}
where $Z_k$ denotes the upper maximal unipotent subgroup of $\GL_k$. The character $\psi_H$ of $U_{1^{n-1}}(\BA)$, which we re-denote by $\psi_{n-1}$ is given on $u$, with adele coordinates, as follows. Let $\psi_{Z_{n-1}}$ be the standard Whittaker character of $Z_{n-1}(\BA)$, corresponding to $\psi$, that is
$$
\psi_{Z_{n-1}}(z)=\psi(z_{1,2}+z_{2,3}+\cdots+z_{n-2,n-1}).
$$
Then
\begin{equation}\label{5.2}
\psi_{n-1}(u)=\psi_{Z_{n-1}}(z)\psi(x_{n-1,ni+1}+\frac{1}{2}x_{n-1,ni+2}).
\end{equation}
The character $\psi_{n-1}$ is stabilized by $\SO_{2ni+1}(\BA)$ realized as the subgroup of elements
$diag(I_{n-1},h,I_{n-1})$, with $h\in \SO_{2ni+2}(\BA)$ satisfying
$$
h\begin{pmatrix}0_{ni}\\1\\ \frac{1}{2}\\0_{ni}\end{pmatrix}=\begin{pmatrix}0_{ni}\\1\\ \frac{1}{2}\\0_{ni}\end{pmatrix}.
$$
Let $j$ denote the isomorphism from $\SO_{2ni+1}$ to this stabilizer, given by \eqref{1.8}. Since $H_1$ is the trivial group, we shorten our notation and re-denote $j(1,h)=j(h)$. Denote
$$
t(h)=\begin{pmatrix}I_{n-1}\\&j(h)\\&&I_{n-1}\end{pmatrix},\ h\in \SO_{2ni+1}.
$$ 
We considered the Eisenstein series $E(f_{\Delta(\tau,i+1),s})$ on $\SO_{2n(i+1)}(\BA)$, corresponding to the section $f_{\Delta(\tau,i+1),s}$, and we took in \eqref{1.9.1} its Fourier coefficient along $U_{1^{n-1}}$, with respect to $\psi_{n-1}$. We re-denote this Fourier coefficient by $\mathcal{D}_{\psi,ni}(E(f_{\Delta(\tau,i+1),s}))$, and, as before, we view it as an automorphic function on $\SO_{2ni+1}(\BA)$, realized as above. This is the Bessel coefficient used in automorphic descent. See \cite{GRS11}, Sec. 3.1 (where we refer to these coefficients as Gelfand-Graev coefficients). Thus,
\begin{equation}\label{5.3}
\mathcal{D}_{\psi,ni}(E(f_{\Delta(\tau,i+1),s}))(h)=\int_{U_{1^{n-1}}(F)\backslash U_{1^{n-1}}(\BA)}E(f_{\Delta(\tau,i+1),s},ut(h))\psi^{-1}_{n-1}(u)du.
\end{equation}
We don't need to further integrate along $H_1$ against $\sigma$, as $H_1$ is trivial. Now \eqref{3.20}, Theorem \ref{thm 4.6} and \eqref{4.12} become in this case
\begin{thm}\label{thm 5.1}
Let $f_{\Delta(\tau,i+1),s}$ be a smooth, holomorphic section of $\rho_{\Delta(\tau,i+1),s}$. Then, for $Re(s)$ sufficiently large and $h\in \SO_{2ni+1}(\BA)$,
$$
\mathcal{D}_{\psi,ni}(E(f_{\Delta(\tau,i+1),s}))(h)=\sum_{h'\in Q_{ni}\backslash
	\SO_{2ni+1}}\Lambda(f_{\Delta(\tau, i+1),s})(h'h),
$$
where
\begin{equation}\label{5.4}
\Lambda(f_{\Delta(\tau,i+1),s})(h)=
\int_{U'_{n-1}(\BA)}
f^\psi_{\Delta(\tau, i+1),s}(\delta_0ut(h))\psi_{n-1}^{-1}(u)dudg.
\end{equation}
This function  $\Lambda(f_{\Delta(\tau,i+1),s})$, defined for $\Re(s)$ sufficiently
large by the last integral, is smooth and admits an analytic continuation
to a meromorphic function of $s$ in the whole plane, and defines a smooth
meromorphic section of
$$
\rho_{\Delta(\tau,i),s}=\Ind_{Q_{ni}(\BA)}^{\SO_{2ni+1}(\BA)}\Delta(\tau,i)
|\det\cdot|^s .
$$
Thus, $\mathcal{D}_{\psi,ni}(E(f_{\Delta(\tau,i+1),s}))(h)$ is the Eisenstein series on
$\SO_{2ni+1}(\BA)$, corresponding to the section $\Lambda(f_{\Delta(\tau,i+1),s})$ of
$\rho_{\Delta(\tau,i),s}$. Moreover, when we normalize (outside $S$, as before) $E(f_{\Delta(\tau,i+1),s})$ by
$$
E^*_S(f_{\Delta(\tau,i+1),s})=d_\tau^{\SO_{2n(i+1)},S}(s)E(f_{\Delta(\tau,i+1),s}),
$$
then $\mathcal{D}_{\psi,ni}(E^*_S(f_{\Delta(\tau,i+1),s}))$ is an Eisenstein series on
$\SO_{2ni+1}(\BA)$, corresponding to $\rho_{\Delta(\tau,i),s}$, and it is normalized outside $S$.  
\end{thm}
For completeness, let us specify in this case the elements which appear in \eqref{5.4}.
$U'_{n-1}$ is the subgroup 
$$
U'_{n-1}=\{u_{x;y}=\begin{pmatrix}I_{n-1}&x&0&y\\&I_{ni+1}&0&0\\&&I_{ni+1}&x'\\&&&I_{n-1}\end{pmatrix}^{\omega_0^{n-1}}\},\ \psi_{n-1}(u_{x;y})=\psi(x_{n-1,ni+1});
$$
$$
\delta_0=\begin{pmatrix}0&I_{ni}&0&0&0&0\\0&0&\frac{1}{2}&0&0&0\\0&0&0&0&0&I_{n-1}\\I_{n-1}&0&0&0&0&0\\0&0&0&2&0&0\\0&0&0&0&I_{ni}&0\end{pmatrix}\omega_0^{n-1};
$$
Finally
$$
f^\psi_{\Delta(\tau, i+1),s}(x)=\int_{V_{ni,n}(F)\backslash V_{ni,n}(\BA)}\int_{Z_n(F)\backslash Z_n(\BA)}f_{\Delta(\tau, i+1),s}(\hat{v}\begin{pmatrix}I_{ni}\\&z\end{pmatrix}^\wedge x)\psi_{Z_n}(z)dzdv.
$$

\section{Application of Bessel coefficients to $E(f_{\Delta(\tau,i+1),s})$ on $\SO_{2n(i+1)+1}(\BA)$: descent to $\SO_{2ni}(\BA)$}

In this section, we will consider a Fourier coefficient similar to the one in the previous section applied to an Eisenstein series on $\SO_{2n(i+1)+1}(\BA)$, parabolically induced from $\Delta(\tau,i+1)$. Here, we cannot use the work of the previous sections, since there the group $H$ is even orthogonal (when $H$ is orthogonal). However, the treatment is similar, and, in fact, most of the details can be derived from \cite{GRS11}, Chapter 5 (Theorems 5.1, 5.2, 5.3). 

Denote in this section $H=\SO_{2n(i+1)+1}$ (written with respect to $w_{2n(i+1)+1}$). Consider the unipotent radical $U_{1^n}$ consisting of the elements
\begin{equation}\label{6.1}
u=\begin{pmatrix}z&x&y\\&I_{2ni+1}&x'\\&&z^*\end{pmatrix}\in H,\  z\in Z_n
\end{equation}
and the character $\psi_{ni}$ of $U_{1^n}(\BA)$ given by
\begin{equation}\label{6.2}
\psi_n(u)=\psi_{Z_n}(z)\psi(x_{n,ni+1}).
\end{equation}
This character is stabilized by the adele points of $\SO_{2ni}$ realized as the subgroup of elements $diag(I_n,h,I_n)$, with $h\in \SO_{2ni+1}$, satisfying
$$
h\begin{pmatrix}0_{ni}\\1\\0_{ni}\end{pmatrix}=\begin{pmatrix}0_{ni}\\1\\0_{ni}\end{pmatrix}.
$$
The isomorphism of this stabilizer and $\SO_{2ni}$ is given by
\begin{equation}\label{6.2.1}
j(\begin{pmatrix}a&b\\c&d\end{pmatrix})=\begin{pmatrix}a&0&b\\0&1&0\\c&0&d\end{pmatrix},
\end{equation}
where $a,b,c,d$ are $ni\times ni$ matrices and $\begin{pmatrix}a&b\\c&d\end{pmatrix}\in \SO_{2ni}$.
For $h\in \SO_{2ni}$, denote
\begin{equation}\label{6.3}
t(h)=\begin{pmatrix}I_n\\&j(h)\\&&I_n\end{pmatrix}.
\end{equation} 
Denote by $Q_{n(i+1)}$ the standard parabolic subgroup of $H$, whose Levi part, $M_{n(i+1)}$, is isomorphic to $\GL_{n(i+1)}$. We will use the following notation, similar to our previous one,
\begin{equation}\label{6.3.1}
M_{n(i+1)}=\{\hat{a}=\begin{pmatrix}a\\&1\\&&a^*\end{pmatrix}\ | a\in \GL_{n(i+1)}\}.
\end{equation}
Consider the parabolic induction
\begin{equation}\label{6.4}
\rho_{\Delta(\tau,i+1),s}=\Ind_{Q_{n(i+1)}(\BA)}^{\SO_{2n(i+1)+1}(\BA)}\Delta(\tau,i+1)|\det\cdot|^s.
\end{equation}
Let $f_{\Delta(\tau,i+1),s}$ be a smooth, holomorphic section of $\rho_{\Delta(\tau,i+1),s}$ in \eqref{6.4}. Consider the corresponding Eisenstein series $E(f_{\Delta(\tau,i+1),s})$, and apply to it the Fourier coefficient along $U_{1^n}$ in \eqref{6.1}, with respect to the character \eqref{6.2}. Denote this Fourier coefficient by $\mathcal{D}_{\psi,ni}(E(f_{\Delta(\tau,i+1),s}))$, and, as before, we view it as an automorphic function on $\SO_{2ni}(\BA)$, realized via \eqref{6.3}. This is the Bessel coefficient used in automorphic descent. See \cite{GRS11}, Sec. 3.1. Thus,
\begin{equation}\label{6.5}
\mathcal{D}_{\psi,ni}(E(f_{\Delta(\tau,i+1),s}))(h)=\int_{U_{1^{n}}(F)\backslash U_{1^{n}}(\BA)}E(f_{\Delta(\tau,i+1),s},ut(h))\psi^{-1}_n(u)du.
\end{equation}
The analog of \eqref{3.20} in this case is a little more involved. It turns out that \eqref{6.5} is a sum of two Eisenstein series on $\SO_{2ni}(\BA)$, with the second term obtained from the first by an outer conjugation. In order to write the precise identity, we need to introduce some more notations. Let $\alpha_0\in H$ be the following element
\begin{equation}\label{6.6}
\alpha_0=\begin{pmatrix}0&I_{ni}&0&0&0\\0&0&0&0&I_n\\0&0&(-1)^n&0&0\\I_n&0&0&0&0\\0&0&0&I_{ni}&0\end{pmatrix};
\end{equation}
Denote
\begin{equation}\label{6.7}
U'_n=\{u_{x_1,x_2;y}=\begin{pmatrix}I_n&x_1&x_2&0&y\\&I_{ni}&0&0&0\\&&1&0&x'_2\\&&&I_{ni}&x'_1\\&&&&I_n\end{pmatrix}\in H\};
\end{equation} 
\begin{equation}\label{6.8}
\omega_0=diag(I_{n(i+1)},-1,I_{n(i+1)})\in \RO_{2n(i+1)+1};\ \omega_0'=diag(I_{ni-1},\begin{pmatrix}&1\\1\end{pmatrix},I_{ni-1})\in \RO_{2ni}.
\end{equation}
We extend $t$ in \eqref{6.3}  to $\RO_{2ni}$, so that $t(\omega'_0)\in \RO_{2n(i+1)+1}$.
Denote $\omega_0''=\omega_0t(\omega'_0)$. Let $\psi_{V_{ni,1^n}}$ be the following character of $V_{ni,1^n}(\BA)$, similar to \eqref{3.8}  
\begin{equation}\label{6.9}
\psi_{V_{ni,1^n}}(\begin{pmatrix}I_{ni}&y\\&z\end{pmatrix})=\psi_{Z_n}(z),\ z\in Z_n(\BA),
\end{equation}
and let, for $g\in H(\BA)$,
\begin{equation}\label{6.10}
f^\psi_{\Delta(\tau,i+1),s}(g)=\int_{V_{ni,1^n}(F)\backslash V_{ni,1^n}(\BA)}f_{\Delta(\tau,i+1),s}(\hat{v}g)\psi_{V_{ni,1^n}}(v)dv.
\end{equation}
\begin{thm}\label{thm 6.1}
For $Re(s)$ sufficiently large, $h\in \SO_{2ni}(\BA)$,\\
\\
$\mathcal{D}_{\psi,ni}(E(f_{\Delta(\tau,i+1),s}))(h)=$
\begin{equation}\label{6.11}
\sum_{\gamma\in Q_{ni}(F)\backslash \SO_{2ni}(F)}\Lambda^+(f_{\Delta(\tau,i+1),s})(\gamma h)+\sum_{\gamma\in Q_{ni}(F)\backslash \SO_{2ni}(F)}\Lambda^-(f_{\Delta(\tau,i+1),s})(\gamma h^{\omega'_0}),
\end{equation}
where
$$
\Lambda^+(f_{\Delta(\tau,i+1),s})(h)=\int_{U'_n(\BA)}f^\psi_{\Delta(\tau,i+1),s}(\alpha_0ut(h))\psi_n^{-1}(u)du;
$$
$$
\Lambda^-(f_{\Delta(\tau,i+1),s})(h)=\int_{U'_n(\BA)}f^\psi_{\Delta(\tau,i+1),s}(\alpha_0ut(h)\omega_0'')\psi_n(u)du;
$$
In the sum \eqref{6.11}, $Q_{ni}=Q_{ni}^{\SO_{2ni}}$.
\end{thm}
\begin{proof}
The proof is similar to that of Theorem \ref{thm 1.1}. Many details appear in \cite{GRS11}, Chapter 5. We sketch it here for the convenience of the reader.
As in \eqref{2.1}, we factorize (for $Re(s)$ large)
\begin{equation}\label {6.12}
E(f_{\Delta(\tau, i+1),s},x)=\sum_{\alpha\in
	Q_{n(i+1)}(F)\backslash H(F)/Q_n(F)}\sum_\gamma f_{\Delta(\tau,
	i+1),s}(\alpha\gamma x),
\end{equation}
where, as before, $Q_k=Q_k^H$ is the standard parabolic subgroup of $H=\SO_{2n(i+1)+1}$, with Levi part isomorphic to $\GL_k\times \SO_{2(n(i+1)-k)+1}$ ($1\leq k\leq n(i+1)$); the second summation in \eqref{6.12} is over $\gamma\in Q_n(F)\cap
\alpha^{-1}Q_{n(i+1)}(F)\alpha\backslash Q_n(F)$. The
representatives $\alpha$ in \eqref{6.12} are described in
\cite{GRS11}, Sec. 4.2. They are parameterized by integers $0\leq
r\leq n$,
\begin{equation}\label{6.13}
\alpha_r=\begin{pmatrix}I_r\\&\alpha'_r\\&&I_r\end{pmatrix},
\end{equation}
where
$$
\alpha'_r=\begin{pmatrix}0&I_{ni}&0&0&0\\0&0&0&0&I_{n-r}\\0&0&(-1)^{n-r}&0&0\\
I_{n-r}&0&0&0&0\\0&0&0&I_{ni}&0\end{pmatrix}.
$$
Denote $Q^{(r)}=Q_n\cap
\alpha_r^{-1}Q_{n(i+1)}\alpha_r$. This is the ($F$ - algebraic) subgroup of $Q_n$, consisting of the matrices
\begin{equation}\label{6.14}
\begin{pmatrix}a&x&y_1&y_2&y_3&z_1&z_2\\0&b&0&0&y_4&0&z'_1\\&&c&u&v&y'_4&y'_3\\&&&1&u'&0&y'_2\\&&&&c^*&0&y'_1\\&&&&&b^*&x'\\&&&&&&a^*\end{pmatrix}\in H,
\end{equation}
where $a\in \GL_r$, $b\in \GL_{n-r}$, $c\in \GL_{ni}$. The element \eqref{6.14} is conjugated by $\alpha_r$ to
$$
\begin{pmatrix}a&y_1&z_1&y_2&x&y_3&z_2\\&c&y'_4&u&0&v&y'_3\\&&b^*&0&0&0&x'_2\\&&&1&0&u'&y'_2\\&&&&b&y_4&z'_1\\&&&&&c^*&y'_1\\&&&&&&a^*\end{pmatrix}^{\omega_0^{n-r}}.
$$
Thus, for $h\in \SO_{2ni}(\BA)$,\\
\\
$\mathcal{D}_{\psi,ni}(E(f_{\Delta(\tau, i+1),s}))(h)=$
\begin{equation}\label{6.15}
\sum_{r=0}^n\int_{U_{1^n}(F)\backslash
	U_{1^n}(\BA)}\sum_{\gamma\in Q^{(r)}(F)\backslash Q_n(F)}
f_{\Delta(\tau, i+1),s}(\alpha_r\gamma ut(h))\psi^{-1}_n(u)du.
\end{equation}
The analogue of Theorem \ref{thm 2.1} is valid here, namely for all $1\leq r\leq n$, the summand in \eqref{6.15}, corresponding to $r$ is (identically) zero. The proof follows similar steps as that of Theorem \ref{thm 2.1}, and it is easily read from the proof of Prop. 5.1 in \cite{GRS11}. One only needs to carry the easy translation from the language of twisted Jacquet modules to that of Fourier coefficients. Thus, \eqref{6.15} becomes\\
\\
$\mathcal{D}_{\psi,ni}(E(f_{\Delta(\tau, i+1),s}))(h)=$
\begin{equation}\label{6.16}
\int_{U_{1^n}(F)\backslash
	U_{1^n}(\BA)}\sum_{\gamma\in Q^{(0)}(F)\backslash Q_n(F)}
f_{\Delta(\tau, i+1),s}(\alpha_0\gamma ut(h))\psi^{-1}_n(u)du.
\end{equation} 	
Factor the sum in \eqref{6.16} modulo $Q_{1^n}(F)$ from the right. As before, $Q_{1^n}=Q_{1^n}^H$ denotes the standard parabolic subgroup of  of $H=\SO_{2n(i+1)+1}$, with Levi part isomorphic to $\GL_1^n\times \SO_{2ni+1}$. Note that $Q_n=Q^{(0)}Q_{1^n}$. Hence \eqref{6.16} becomes	\\
\\
$\mathcal{D}_{\psi,ni}(E(f_{\Delta(\tau, i+1),s}))(h)=$
\begin{equation}\label{6.17}
\int_{U_{1^n}(F)\backslash
	U_{1^n}(\BA)}\sum_{\eta\in Q^{(0)}(F)\cap Q_{1^n}(F)\backslash Q_{1^n}(F)}
f_{\Delta(\tau, i+1),s}(\alpha_0\eta ut(h))\psi^{-1}_n(u)du,
\end{equation} 
The subgroup $Q^{(0)}\cap Q_{1^n}$ consists of the matrices
$$
\begin{pmatrix}b&0&0&y&0\\&c&u&v&y'\\&&1&u'&0\\&&&c^*&0\\&&&&b^*\end{pmatrix}\in H,
$$
where $c\in \GL_{ni}$ and $b\in B_{\GL_n}$ - the standard Borel subgroup of $\GL_n$. Factor the sum in \eqref{6.17} modulo $U_{1^n}(F)$ from the right. We get\\
\\
$\mathcal{D}_{\psi,ni}(E(f_{\Delta(\tau, i+1),s}))(h)=$
\begin{equation}\label{6.18}
\sum_{\gamma\in Q^{\SO_{2ni+1}}_{ni}(F)\backslash \SO_{2ni+1}(F)}\int
f_{\Delta(\tau, i+1),s}(\alpha_0h_\gamma ut(h))\psi^{-1}_n(u)du,
\end{equation}  	
where the $du$-integration is over $U_{1^n}(F)\cap h^{-1}_\gamma(Q^{(0)}(F)\cap Q_{1^n}(F))h_\gamma\backslash
U_{1^n}(\BA)$, and for $\gamma\in \SO_{2ni+1}$, $h_\gamma=diag(I_n,\gamma,I_n)$.	
Now, factor the sum in \eqref{6.18} modulo $\SO_{2ni}(F)$ from the right (realized via the embedding $j$). The set\\
$Q^{\SO_{2ni+1}}_{ni}(F)\backslash \SO_{2ni+1(F)}/\SO_{2ni}(F)$ consists of three elements. See Prop. 4.4 in \cite{GRS11}. Here are representatives for the three double cosets: one is the identity $I_{2ni+1}$, another is $\gamma_0$, such that $h_{\gamma_0}=\omega_0''$, and the third is
$$
\gamma_1=diag(I_{ni-1},\begin{pmatrix}1\\1&1\\-\frac{1}{2}&-1&1\end{pmatrix},I_{ni-1}).
$$

Let us show that the contribution of $\gamma_1$ to \eqref{6.18} is zero. For this, we note that the subgroup $U^1_{1^n}=U_{1^n}\cap h^{-1}_{\gamma_1}(Q^{(0)}\cap Q_{1^n})h_{\gamma_1}$, viewed as an $F$-matrix group, consists of the matrices
\begin{equation}\label{6.19}
\begin{pmatrix}z&0&-\frac{y_1}{2}&-y_1&y_1&y_2&0\\&I_{ni-1}&0&0&0&0&y'_2\\&&1&0&0&0&y'_1\\&&&1&0&0&-y'_1\\&&&&1&0&-\frac{y'_1}{2}\\&&&&&I_{ni-1}&0\\&&&&&&z^*\end{pmatrix}\in H,\ z\in Z_n.
\end{equation}
When we conjugate the element \eqref{6.19} by $\alpha_0h_{\gamma_1}$, we get
\begin{equation}\label{6.20}
\begin{pmatrix}I_{ni-1}&0&y'_2\\&1&y'_1\\&&z^*\end{pmatrix}^\wedge,
\end{equation}
which is the subgroup $V_{ni,1^n}^\wedge$. From \eqref{6.19}, \eqref{6.20}, it follows that the contribution of $\gamma_1$ to \eqref{6.18} is
\begin{equation}\label{6.21}
\sum_{\gamma\in \SO_{2ni}(F)\cap \gamma_1^{-1}Q_{ni}(F)\gamma_1\backslash \SO_{2ni}(F)}\int_{U^1_{1^n}(\BA)\backslash U_{1^n}(\BA)}I_\psi(f_{\Delta(\tau, i+1),s})(\alpha_0h_{\gamma_1}ut(\gamma h))\psi_n^{-1}(u)du,
\end{equation} 
where
\begin{equation}\label{6.22}
I_\psi(f_{\Delta(\tau, i+1),s})(x)=\int_{V_{ni,1^n}(F)\backslash V_{ni,1^n}(\BA)}f_{\Delta(\tau, i+1),s}(\hat{v}x)\tilde{\psi}_{V_{ni,1^n}}(v)dv,
\end{equation}
where
$$
\tilde{\psi}_{V_{ni,1^n}}(\begin{pmatrix}I_{ni}&y\\&z\end{pmatrix})=\psi_{Z_n}(z)\psi(y_{ni,1}).
$$
The Fourier coefficient \eqref{6.22} is obtained by applying the Fourier coefficient with respect to $\tilde{\psi}_{V_{ni,1^n}}$ to the elements of $\Delta(\tau,i+1)$ obtained as $a\mapsto f_{\Delta(\tau, i+1),s}(\hat{a}x)$. This Fourier coefficient corresponds to the partition $(n+1,1^{ni-1})$, and by Prop. \ref{prop 1.2}, this Fourier coefficient is zero on $\Delta(\tau,i+1)$.

It remains to examine the contributions of $I_{2ni+1}$ and $\gamma_0$ to \eqref{6.18}, and it is straightforward to check that we get the two terms of \eqref{6.11}.
\end{proof}
Next, let us show that the two summands in \eqref{6.11} are Eisenstein series.
\begin{thm}\label{thm 6.2}
Both functions $\Lambda^\pm(f_{\Delta(\tau,i+1),s})$, defined for $Re(s)$ sufficiently large by the integrals in Theorem \ref{thm 6.1}, admit analytic continuations
to  meromorphic functions of $s$ in the whole plane. They define
smooth meromorphic sections of
$$
\rho_{\Delta(\tau,i),s}=\Ind_{Q_{ni}(\BA)}^{\SO_{2ni}(\BA)}\Delta(\tau,i)
|\det\cdot|^s .
$$
Thus, $\mathcal{D}_{\psi,ni}(E(f_{\Delta(\tau,i+1),s}))$ is the sum of two Eisenstein series on
$\SO_{2ni}(\BA)$; the first corresponds to the section $\Lambda^+(f_{\Delta(\tau,i+1),s})$ of
$\rho_{\Delta(\tau,i),s}$, and the second corresponds to the section $h\mapsto \Lambda^-(f_{\Delta(\tau,i+1),s})(h^{\omega'_0})$ of
$$
\rho^{\omega'_0}_{\Delta(\tau,i),s}=\Ind_{Q^{\omega'_0}_{ni}(\BA)}^{\SO_{2ni}(\BA)}\Delta(\tau,i)
|\det\cdot|^s .
$$
\end{thm}
\begin{proof}
We will show that for $Re(s)$ sufficiently large, $\Lambda^\pm$ indeed define elements of $\rho_{\Delta(\tau,i),s}$, and then the proof of meromorphic continuation is as in the proof of Prop. \ref{prop 3.7}. We will also get a different proof of the meromorphic continuation right after \eqref{6.28.9}. Since $\Lambda^-$ is defined in almost the same way as $\Lambda^+$, it suffices to show this for $\Lambda^+(f_{\Delta(\tau,i+1),s},h)$, and we may assume that $h=I_{2ni}$.

Let $q=\begin{pmatrix}a&x\\&a^*\end{pmatrix}\in Q_{ni}(\BA)\subset \SO_{2ni}(\BA)$. Note that
$$
\alpha_0t(q)\alpha_0^{-1}=\begin{pmatrix}a&0&x\\&I_{2n+1}&0\\&&a^*\end{pmatrix}
$$
Hence, for $Re(s)$ large,
\begin{equation}\label{6.23}
\Lambda^+(f_{\Delta(\tau,i+1),s},q)=|\det(a)|^{-n}\int_{U'_n(\BA)}f^\psi_{\Delta(\tau,i+1),s}(\alpha_0t(q)u)\psi_n^{-1}(u)du
\end{equation}
$$
=|\det(a)|^{-n}\int_{U'_n(\BA)}f^\psi_{\Delta(\tau,i+1),s}(\begin{pmatrix}a\\&I_n\end{pmatrix}^\wedge\alpha_0u)\psi_n^{-1}(u)du.
$$
For any fixed $x_0\in H(\BA)$, the function on $\GL_{ni}(\BA)$, $a\mapsto f^\psi_{\Delta(\tau,i+1),s}(\begin{pmatrix}a\\&I_n\end{pmatrix}^\wedge x_0)$ lies in $|\det(a)|^{s+\frac{n(i+1)}{2}}$ times a function obtained as\\ 
$\phi(diag(a,I_n))$, where $\phi$ lies in the constant term of $\Delta(\tau,i+1)$ along $V_{ni,n}$, followed by taking a Whittaker coefficient with respect to $\psi_{Z_n}$. It follows that, as a function of $a$, $|\det(a)|^{-n}f^\psi_{\Delta(\tau,i+1),s}(\begin{pmatrix}a\\&I_n\end{pmatrix}^\wedge x_0)$ lies in the space of
\begin{equation}\label{6.23.1}
|\det\cdot|^{s+\frac{ni-1}{2}}\Delta(\tau,i)=\delta^{\frac{1}{2}}_{Q_{ni}^{\SO_{2ni}}}\cdot |\det\cdot|^s\Delta(\tau,i),
\end{equation}
and hence, formally, the r.h.s. of \eqref{6.23} lies in the space of \eqref{6.23.1}. The proof that this is indeed the case follows the line of proofs of Prop. \ref{prop 4.1} and Prop. \ref{prop 4.2}. Here are some details in brief. First, we may assume that $f_{\Delta(\tau,i+1)}$ is decomposable, and hence, so is $f^\psi_{\Delta(\tau,i+1)}$. Thus, let $f_{\Delta(\tau,i+1)}$ correspond, via $p_{\tau,i}$ in \eqref{4.2.1.a}, to $\otimes'_vf_{\Delta(\tau_v,i;1)}$, where  $f_{\Delta(\tau_v,i;1),s}$ is a section
of  
\begin{equation}\label{6.23.2}
\rho_{\Delta(\tau_v,i;1),s}=\Ind_{Q_{ni,n}(F_v)}^{\SO_{2n(i+1)+1}(F_v)}(\Delta(\tau_v,i)|\det\cdot|^{s-\frac{1}{2}}\times
\tau_v|\det\cdot|^{s+\frac{i}{2}}).
\end{equation}
As before, $Q_{ni,n}$ is the standard parabolic subgroup of $H=\SO_{2n(i+1)+1}$, whose Levi part is isomorphic to $\GL_{ni}\times \GL_n$. The representation  $\Delta(\tau_v,i)$ is realized in its model with respect to
$\psi_{V_{i^n}}^{-1}$ and $\tau_v$ is 
realized in its $\psi^{-1}_{Z_n}$-Whittaker model. As in \eqref{4'.6}, we view $f_{\Delta(\tau_v,i;m),s}$ as a
function on $H(F_v)\times \GL_{ni}(F_v)\times \GL_n(F_v)$, such
that for a fixed element in $H(F_v)$, the function in the two other
variables lies in the tensor product of the two models above of
$\Delta(\tau_v,i)$ and $\tau_v$. We simplify notation and re-denote $f_{\Delta(\tau_v,i;1),s}(y)=f_{\Delta(\tau_v,i;1),s}(y;I_{ni},I_n)$.
 
We fix a finite set of places $S$, containing the Archimedean places,
outside which $\tau$ is unramified, $\psi$ is normalized (so that for $v\notin S$, $\psi_v$ is trivial on $\mathcal{O}_v$ and nontrivial on $\mathcal{P}_v^{-1}$). We assume that, for $v\notin S$,
$f_{\Delta(\tau_v,i;1),s}=f^0_{\Delta(\tau_v,i;1),s}$ is spherical and normalized as in \eqref{4.2.1*}. Now, let us consider, for $Re(s)$ sufficiently large, the following local integral at $v$, corresponding to \eqref{6.23},
\begin{equation}\label{6.23.3}
\Lambda_v^+(f_{\Delta(\tau_v,i;1),s},a_v)=|\det(a_v)|^{s+\frac{ni-1}{2}}\int_{U'_n(F_v)}f_{\Delta(\tau_v,i;1),s}(\alpha_0u;a,I_n)\psi_{n,v}^{-1}(u)du.
\end{equation}
Here, $a\in \GL_{ni}(F_v)$. Write the elements of $U'_n(F_v)$ in the form $u_{x_1,x_2;y}$ as
in \eqref{6.7}. Let
$$
U''_n(F_v)=\{u_{x_1,x_2;y}\in U'_n(F_v) |\ x_1=0\}.
$$
Then, as in Prop. \ref{prop 4.1}, with a similar proof,
there is a section $f'_{\Delta(\tau_v,i;1),s}$, which
	depends on (the smoothness of) $f_{\Delta(\tau_v,i;1),s}$, 	such that, for all $a_v\in \GL_{ni}(F_v)$ 
	\begin{equation}\label{6.26}
	\Lambda^+_v(f_{\Delta(\tau_v,i;1),s},a_v)=|\det(a_v)|^{s+\frac{ni-1}{2}}\int_{U''_n(F_v)}
	f'_{\Delta(\tau_v,i;1),s}(\alpha_0u;a_v,I_n)\psi^{-1}_{n,v}(u)du,
	\end{equation}
and if $f_{\Delta(\tau_v,i;1),s}$ is spherical, then
	$f'_{\Delta(\tau_v,i;1),s}=f_{\Delta(\tau_v,i;1),s}$. 
	 We explain, in brief, the case where $v$ is finite. We assume, for simplicity that $\psi_v$ is normalized. Write $u_{x_1,x_2;y}=v(x_1)u_n(x_2;y)=u_n(x_2;y)v(x_1)$, where $v(x_1)=u_{x_1,0;0}$, $u_n(x_2;y)=u_{0,x_2;y}$. Note that
	\begin{equation}\label{6.27}
	\alpha_0u_n(x_2;y)\alpha_0^{-1}=diag(I_{ni},\begin{pmatrix}I_n\\x'_2&1\\y&x_2&I_n\end{pmatrix},I_{ni}).
	\end{equation} 
	Let
	\begin{equation}\label{6.28}
	\Lambda'_v(f_{\Delta(\tau_v,i;1),s},a_v)(x_1)=\int_{U''_n(F_v)}
	f_{\Delta(\tau_v,i;1),s}(\alpha_0v(x_1)u;a_v,I_n)\psi^{-1}_{n,v}(u)du.
	\end{equation} 
	We will show that the last function has compact support in $M_{n\times ni}(F_v)$, and in case of a spherical section this support is $M_{n\times ni}(\mathcal{O}_v)$.
	Let 
	$$	
	e_{b,c}=diag(I_n,\begin{pmatrix}I_{ni}&b&c\\&1&b'\\&&I_{ni}\end{pmatrix},I_n)\in H(F_v),
	$$
	where $b,c$ have coordinates sufficiently close to zero, such that right translation by $e_{b,c}$ fixes $f_{\Delta(\tau_v,i;1),s}$. If the section is spherical, then we may take $b,c$ to have coordinates in $\mathcal{O}_v$.	It is straightforward to check that
	$$
	\alpha_0u_{x_1,x_2;y}e_{b,c}=u'\alpha_0u_{x_1,x_2+x_1b;y},
	$$
	where $u'$ is in the unipotent radical of $Q_{ni,n}(F_v)$. Changing variables $x_2\mapsto x_2-x_1b$ in \eqref{6.28} gives\\
	\\
	$	\Lambda'_v(f_{\Delta(\tau_v,i;1),s},a_v)(x_1)=$
	$$
=\Lambda'_v(\rho(e_{b,c})f_{\Delta(\tau_v,i;1),s})(x_1)=  \psi_v((x_1)_nb)\Lambda'_v(f_{\Delta(\tau_v,i;1),s},a_v)(x_1).
	$$
	Recall that $\rho(e_{b,c})$ denotes right translation by $e_{b,c}$; $(x_1)_n$ denotes the last row of $x_1$. We conclude that $(x_1)_n$ is supported in a compact set, which is independent of $a_v$ (of the form $(\mathcal{P}_v^{-k})^{ni}\subset F_v^{ni}$), and in case the section is spherical, then $(x_1)_n$ is supported in $\mathcal{O}_v^{ni}$. Now, we may assume that $(x_1)_n=0$, and show that the $(n-1)$-th row of $x_1$ must lie in a compact set, which is independent of $a_v$ (in order to be in the support of $\Lambda'_v(f_{\Delta(\tau_v,i;1),s},a_v)$). Then we may assume that the last two rows of $x_1$ are zero and move on to show compact support in row $(x_1)_{n-2}$, and so on. In general, assume that $(x_1)_r=0$, for $n-\ell\leq r\leq n$, $0\leq \ell\leq n-2$. Let
	$$
	c_{\ell;z}=\begin{pmatrix}I_{n-\ell-1}\\&1\\&0&I_\ell\\&z&0&I_{ni}\end{pmatrix}^\wedge\in \SO_{2n(i+1)+1}(F_v),
	$$
	with $z\in (\mathcal{P}^k_v)^{ni}$, where $k$ is sufficiently large, such that right translation by $c_{\ell,z}$ fixes $f_{\Delta(\tau_v,i;1),s}$. As before, for $x_1$ as above, we get\\
	\\
	$\Lambda'_v(f_{\Delta(\tau_v,i;1),s},a_v)(x_1)=$
	$$
	=\Lambda'_v(\rho(c_{\ell,z})f_{\Delta(\tau_v,i;1),s},a_v)(x_1)=  \psi_v((x_1)_{n-\ell-1}z)\Lambda'_v(f_{\Delta(\tau_v,i;1),s},a_v)(x_1).
	$$	
	We conclude that $(x_1)_{n-\ell-1}$ is supported in  $(\mathcal{P}_v^{-k})^{ni}$, and in case the section is spherical, it is supported in $\mathcal{O}_v^{ni}$. This proves \eqref{6.26}. Let $\tilde{f}_{\Delta(\tau_v,i;1),s}=\rho(\alpha_0)(f'_{\Delta(\tau_v,i;1),s})$. Then \eqref{6.26} and \eqref{6.27} show that	
		\begin{multline}\label{6.28.1}
	\Lambda^+_v(f_{\Delta(\tau_v,i;1),s},a_v)=\\
	|\det(a_v)|^{-n}\int
	\tilde{f}_{\Delta(\tau_v,i;1),s}(diag(a_v,\begin{pmatrix}I_n\\x'&1\\y&x&I_n\end{pmatrix},a_v^*))\psi_v^{-1}(x_n)d(x,y),
	\end{multline}
where the integration is over the opposite of the unipotent radical $U^{\SO_{2n+1}}_n(F_v)$. Also, if $f_{\Delta(\tau_v,i;1),s}$ is unramified, $\tilde{f}_{\Delta(\tau_v,i;1),s}=f_{\Delta(\tau_v,i;1),s}$. Now, let us write the function on $\GL_{ni}(F_v)\times \SO_{2n+1}(F_v)$, 
$$
(a_v,g_v)\mapsto |\det(a_v)|^{-n}\tilde{f}_{\Delta(\tau_v,i;1),s}(diag(a_v,g_v,a_v^*))
$$ 
as a finite sum of the form 
$$
|\det(a_v)|^{s+\frac{ni-1}{2}}\sum W_v(a_v)f_{\tau_v,s+\frac{i}{2}}(g_v),
$$
where $W_v\in W_{\psi^{-1}_{V_{i^n}}}(\Delta(\tau_v,i))$, and $f_{\tau_v,s+\frac{i}{2}}$ are smooth, holomorphic sections of $\rho_{\tau_v,s+\frac{i}{2}}=\Ind_{Q_n(F_v)}^{\SO_{2n+1}(F_v)}\tau_v |\det\cdot|^{s+\frac{i}{2}}$. Denote
\begin{equation}\label{6.28.4}
\alpha_0''=\begin{pmatrix}&&I_n\\&(-1)^n\\I_n\end{pmatrix}.
\end{equation}
Write in the last sum $f_{\tau_v,s+\frac{i}{2}}=\rho(\alpha_0'')(\varphi_{\tau_v,s+\frac{i}{2}})$, in each term. Then
substituting the sum in \eqref{6.28.1}, we get
\begin{multline}\label{6.28.2}
\Lambda^+_v(f_{\Delta(\tau_v,i;1),s},a_v)=\\
|\det(a_v)|^{s+\frac{ni-1}{2}}\sum W_v(a_v)\int_{U_n^{\SO_{2n+1}}(F_v)}
\varphi_{\tau_v,s+\frac{i}{2}}(\alpha_0''u)\psi_v^{-1}((-1)^nu_{n,n+1})du.
\end{multline}
We also showed that if $f_{\Delta(\tau_v,i;1),s}$ is unramified, then in \eqref{6.28.2}, there is one summand
\begin{multline}\label{6.28.3}
\Lambda^+_v(f_{\Delta(\tau_v,i;1),s},a_v)=\\
|\det(a_v)|^{s+\frac{ni-1}{2}}W_v(a_v)\int_{U_n^{\SO_{2n+1}}(F_v)}
\varphi_{\tau_v,s+\frac{i}{2}}(\alpha_0''u)\psi_v^{-1}((-1)^nu_{n,n+1})du,
\end{multline} 
where $W_v$ and $\varphi_{\tau_v,s+\frac{i}{2}}$ are unramified. Note that the integral along $U_n^{\SO_{2n+1}}(F_v)$ in \eqref{6.28.2} is a Jacquet integral applied to $\varphi_{\tau_v,s+\frac{i}{2}}$. (See \cite{J67}, p. 245. See also, \cite{S78}, p. 188.) In particular, we know that it continues to a holomorphic function in the plane. This shows that $\Lambda^+_v(f_{\Delta(\tau_v,i;1),s},a_v)$ admits an analytic continuation to a holomorphic function in the complex plane. Note, also, that \eqref{6.28.2} shows that the following function on $H(F_v)$
\begin{equation}\label{6.28.5}
h\mapsto \Lambda_v^+(\rho(h)f_{\Delta(\tau_v,i;1),s},I_{ni})=\int_{U'_n(F_v)}f_{\Delta(\tau_v,i;1),s}(\alpha_0uh)\psi_{n,v}^{-1}(u)du,
\end{equation}
initially defined for $Re(s)$ sufficiently large, and then as an analytic function in $\BC$, defines a smooth, holomorphic section of $\Ind_{Q_{ni}(F_v)}^{\SO_{2ni}(F_v)}\Delta(\tau_v,i)$.
We are now at the same point as in Cor. \ref{cor 4.4}, and, as in Cor. \ref{cor 4.5}, we conclude that there are automorphic forms $\alpha^j$ in the space of $\Delta(\tau,i)$ and smooth, holomorphic sections $\varphi^j_{\tau,s+\frac{i}{2}}$ of $\rho_{\tau,s+\frac{i}{2}}$, $1\leq j\leq N'$, such that, for all $q=\begin{pmatrix}a&x\\&a^*\end{pmatrix}\in Q_{ni}(\BA)\subset \SO_{2ni}(\BA)$, 
\begin{multline}\label{6.28.6}
\Lambda^+(f_{\Delta(\tau,i+1),s},q)=|\det(a)|^{s+\frac{ni-1}{2}}\\
\cdot \sum_{j=1}^{N'} \alpha^j(a)\int_{U^{\SO_{2n+1}}_n(\BA)}\varphi^{j,\psi}_{\tau,s+\frac{i}{2}}(\alpha''_0u))\psi^{-1}((-1)^nu_{n,n+1})du,
\end{multline}
where, for $h\in \SO_{2n+1}(\BA)$,
$$
\varphi^{j,\psi}_{\tau,s+\frac{i}{2}}(h)=\int_{Z_n(F)\backslash Z_n(\BA)}\varphi^j_{\tau,s+\frac{i}{2}}(\hat{z}h)\psi_{Z_n}(z)dz.
$$
In particular, \eqref{6.28.6}, as a function of $a\in \GL_{ni}(\BA)$, is an automorphic form in the space of $|\det\cdot|^{s+\frac{ni-1}{2}}\Delta(\tau,i)$.
Also, for each place $v\in S$ and each $j\leq N'$, there are $W^j_v\in W_{\psi_{V_{i^n},v}}(\Delta(\tau_v,i))$ and smooth, holomorphic sections $\varphi^j_{\tau_v,s+\frac{i}{2}}$ of $\rho _{\tau_v,s+\frac{i}{2}}$, such that 
	\begin{equation}\label{6.28.7}
	\alpha^j=p_{\tau,i}((\otimes_{v\in S}W^j_v)\otimes (\otimes_{v\notin S}W^0_{\Delta(\tau_v,i),\psi_v})), 
	\end{equation}
	$$
	\varphi^j_{\tau,s+\frac{i}{2}}=p_{\tau,1}\circ((\otimes_{v\in S}\varphi^j_{\tau_v,s+\frac{i}{2}})\otimes (\otimes_{v\notin S}f^0_{\tau_v,s+\frac{i}{2}})).
	$$
Finally, \eqref{6.28.6}, as an automorphic form in $a\in \GL_{ni}(\BA)$, is decomposable, in the sense that it is equal to
$$ 
[p_{\tau,i}(\otimes' \Lambda_v^+(f_{\Delta(\tau_v,i;1),s},\cdot)](a).
$$
Note that the integrals appearing in \eqref{6.28.6},
\begin{equation}\label{6.28.8}
\int_{U^{\SO_{2n+1}}_n(\BA)}\varphi^{j,\psi}_{\tau,s+\frac{i}{2}}(\alpha''_0u))\psi^{-1}((-1)^nu_{n,n+1})du,
\end{equation}
is equal to the Whittkaker coefficient of the Eisenstein series on $\SO_{2n+1}(\BA)$, $E(\varphi^{j,\psi}_{\tau,s+\frac{i}{2}},\cdot)$, corresponding to $\varphi^{j,\psi}_{\tau,s+\frac{i}{2}}$,
\begin{multline}\label{6.28.9}
\int_{U^{\SO_{2n+1}}_n(F)\backslash U_n^{\SO_{2n+1}}(\BA)}\int_{Z_n(F)\backslash Z_n(\BA)}E(\varphi^{j,\psi}_{\tau,s+\frac{i}{2}},\begin{pmatrix}&&w_n\\&(-1)^n\\w_n\end{pmatrix}u\hat{z})\cdot \\
\cdot \psi(u_{n,n+1})\psi_{Z_n}(z)dzdu.
\end{multline}
This is the starting point of the Langlands-Shahidi theory for the symmetric square $L$-function of $\tau$. In particular, we get that \eqref{6.28.8}, which converges absolutely for $Re(s)$ sufficiently large, admits a meromorphic continuation to the whole plane, and hence $\Lambda^+(f_{\Delta(\tau,i+1),s},q)$ admits a meromorphic continuation to the whole plane. This completes the proof of the theorem.
 
\end{proof}

The last thing that we want to show in this case is the comatibility with normalization. We saw that $\Lambda^+(f_{\Delta(\tau,i+1),s})$ is related to the Langlands-Shahidi integral on $\SO_{2n+1}(\BA)$, corresponding to the symmetric-square $L$-function of $\tau$. (In the previous section, we got a relation to the exterior-square $L$-function of $\tau$, which follows from Section 4, as a special case.) Of course, the same holds for $\Lambda^-(f_{\Delta(\tau,i+1),s})$, since it has almost the same structure.

\begin{thm}\label{thm 6.4}
	Let $v$ be a place outside $S$. Then, for $a_v\in \GL_{ni}(F_v)$,
	$$
	\Lambda^+_v(f^0_{\Delta(\tau_v,i;1),s},a_v)=\frac{1}{L(\tau_v,sym^2,2s+i+1)}|\det(a_v)|^{s+\frac{ni-1}{2}}W^0_{\Delta(\tau_v,i),\psi_v}(a_v).
	$$
\end{thm}
\begin{proof}
Denote, for $h\in \SO_{2n+1}(F_v)$,
$$
f^0_{\tau_v,s+\frac{i}{2}}(h)=f^0_{\Delta(\tau_v,i;1),s}(diag(I_{ni}, h,I_{ni})).
$$
This is the spherical, normalized element of
$$
\rho_{\tau_v,s+\frac{i}{2}}=\Ind_{Q_n^{\SO_{2n+1}}(F_v)}^{\SO_{2n+1}(F_v)}\tau_v|\det\cdot|^{s+\frac{i}{2}}.
$$
By \eqref{6.28.3}, for $Re(s)$ sufficiently large,
\begin{multline}\label{6.29}
\Lambda^+_v(f^0_{\Delta(\tau_v,i;1),s},a_v)=
|\det(a_v)|^{s+\frac{ni-1}{2}}W^0_{\Delta(\tau_v,i),\psi_v}(a_v)\cdot \\
\cdot\int_{U_n^{\SO_{2n+1}}(F_v)}
f^0_{\tau_v,s+\frac{i}{2}}(\alpha_0''u)\psi_v^{-1}((-1)^nu_{n,n+1})du.
\end{multline}
This is the Jacquet integral applied to $f^0_{\tau_v,s+\frac{i}{2}}$, with respect to the standard $\psi_v$-Whittaker character. Now, we get the theorem by the Casselman-Shalika formula (See \cite{CS80}, Theorem 5.4. See also \cite{S78}, Theorem 5.2.)	
\end{proof}
Let us multiply our Eisenstein series $E(f_{\Delta(\tau,i+1),s})$ on
$\SO_{2n(i+1)+1}(\BA)$ by its normalizing factor outside $S$, $d_\tau^{\SO_{2n(i+1)+1},S}(s)$. We have
\begin{equation}\label{6.30}
d_\tau^{\SO_{2n(2j+1)+1},S}(s)=\prod_{k=1}^jL^S(\tau,\wedge^2,2s+2k)\prod_{k=1}^{j+1}L^S(\tau,sym^2,2s+2k-1);
\end{equation}
\begin{equation}\label{6.31}
d_\tau^{\SO_{4nj+1},S}(s)=\prod_{k=1}^jL^S(\tau,\wedge^2,2s+2k-1)L^S(\tau,sym^2,2s+2k).
\end{equation}
One sees immediately from \eqref{6.30}, \eqref{6.31}, that (see \eqref{1.10.8}, \eqref{1.10.9})
\begin{equation}\label{6.32}
\frac{d_\tau^{\SO_{2n(i+1)+1},S}(s)}{L^S(\tau,sym^2,2s+i+1)}=d_\tau^{\SO_{2ni},S}(s).
	\end{equation}
Let
$$
E^*_S(f_{\Delta(\tau,i+1),s})=d_\tau^{\SO_{2n(i+1)+1},S}(s)E(f_{\Delta(\tau,i+1),s}).
$$
We conclude from Theorem \ref{thm 6.2}, Theorem \ref{thm 6.4} and \eqref{6.32},
\begin{thm}\label{thm 6.5}
The descent to $\SO_{2ni}(\BA)$ of the normalized Eisenstein series\\ $E^*_S(f_{\Delta(\tau,i+1),s})$, $\mathcal{D}_{\psi,n}(E^*_S(f_{\Delta(\tau,i+1),s}))$, is the sum of two normalized (outside $S$) Eisenstein series on
$\SO_{2ni}(\BA)$; the first corresponds to the section\\ $\Lambda^+(d_\tau^{\SO_{2n(i+1)+1},S}(s)f_{\Delta(\tau,i+1),s})$ and the second corresponds to the section\\ $h\mapsto \Lambda^-(d_\tau^{\SO_{2n(i+1)+1},S}(s)f_{\Delta(\tau,i+1),s})(h^{\omega'_0})$.
\end{thm}

\section{Application of Fourier-Jacobi coefficients to $E(f_{\Delta(\tau,i+1),s})$ on $\Sp_{2n(1+i)}(\BA)$: descent to $\Sp^{(2)}_{2ni}(\BA)$}

In this section and the next one, we will carry out for symplectic groups and metaplectic groups the descent, via Fourier-Jacobi coefficients, analogous to the previous two sections. Denote in this section $H=\Sp_{2n(i+1 )}$. Recall from \eqref{1.12} - \eqref{1.12.4}, that we consider the unipotent radical $U_{1^{n-1}}$, consisting of the elements
\begin{equation}\label{7.1}
u=\begin{pmatrix}z&x&y\\&I_{2ni+2}&x'\\&&z^*\end{pmatrix}\in H,\  z\in Z_{n-1}
\end{equation}
and the character $\psi_{n-1}$ of $U_{1^{n-1}}(\BA)$ given by
\begin{equation}\label{7.2}
\psi_{n-1}(u)=\psi_{Z_{n-1}}(z)\psi(x_{n-1,n}).
\end{equation}
This character is stabilized by the semi-direct product of $\Sp_{2ni}(\BA)$ and $\mathcal{H}_{2ni+1}(\BA)$, where $\Sp_{2ni}$ is realized inside $H$ by 
$$
t(h)=diag(I_n,h,I_n), h\in \Sp_{2ni},
$$ 
and the group $\mathcal{H}_{2ni+1}$ is realized inside $H$ by 
\begin{equation}\label{7.3}
t((x,e))=diag(I_{n-1},\begin{pmatrix}1&x&e\\&I_{2ni}&x'\\&&1\end{pmatrix},I_{n-1})\in H.
\end{equation}
Recall the projection $\beta$ from $U_{1^n}=U_{1^{n-1}}\rtimes t(\mathcal{H}_{2ni+1})$ onto $\mathcal{H}_{2ni+1}$, and that we extend the character \eqref{7.2} to $U_{1^n}(\BA)$ by making it trivial on $t(\mathcal{H}_{2ni+1}(\BA))$. We continue to denote this extension by $\psi_{n-1}$.\\ 
Let $\omega_{\psi^{-1}}$ be the Weil representation of $\mathcal{H}_{2ni+1}(\BA)\rtimes \Sp^{(2)}_{2ni}(\BA)$, associated to $\psi^{-1}$, i.e. the elements $(0,z)$ of the center of $\mathcal{H}_{2ni+1}(\BA)$ act by multiplication by $\psi^{-1}(z)$. We let $\omega_{\psi^{-1}}$ act on the space of Schwartz functions $\mathcal{S}(\BA^{ni})$. For $\phi\in \mathcal{S}(\BA^{ni})$, we have the corresponding theta series $\theta^\phi_{\psi^{-1}}$, viewed as a function on $\mathcal{H}_{2ni+1}(\BA)\rtimes \Sp^{(2)}_{2ni}(\BA)$. See \eqref{1.15.1} - \eqref{1.15.3}.

Consider the parabolic induction
\begin{equation}\label{7.4}
\rho_{\Delta(\tau,i+1),s}=\Ind_{Q_{n(i+1)}(\BA)}^{\Sp_{2n(i+1)}(\BA)}\Delta(\tau,i+1)|\det\cdot|^s.
\end{equation}
Let $f_{\Delta(\tau,i+1),s}$ be a smooth, holomorphic section of $\rho_{\Delta(\tau,i+1),s}$  . Consider the corresponding Eisenstein series $E(f_{\Delta(\tau,i+1),s})$, and apply to it the following Fourier-Jacobi coefficient (recall our notation \eqref{1.1.d})\\
\\
$\mathcal{D}^\phi_{\psi,ni}(E(f_{\Delta(\tau,i+1),s}))((h,\bar{\mu}))$
\begin{equation}\label{7.5}
=\int_{U_{1^n}(F)\backslash U_{1^n}(\BA)}E(f_{\Delta(\tau,i+1),s},ut(h))\psi_{n-1}^{-1}(u)\theta^\phi_{\psi^{-1}}(\beta(u)(h,\bar{\mu}))du.
\end{equation}
Here, $h\in \Sp_{2ni}(\BA)$. Let us state the identity analogous to \eqref{3.20} and Theorem \ref{thm 6.1}. We need some more notation first. Let 
\begin{equation}\label{7.6}
\alpha_0=\begin{pmatrix}0&I_{ni}&0&0\\0&0&0&I_n\\-I_n&0&0&0\\0&0&I_{ni}&0\end{pmatrix}.
\end{equation}
Denote
\begin{equation}\label{7.11}
U'_n=\{u'_{x;y}=\begin{pmatrix}I_n&x&0&y\\&I_{ni}&0&0\\&&I_{ni}&x'\\&&&I_n\end{pmatrix}\in H\}.
\end{equation}
As in \eqref{6.10}, let, for $g\in H(\BA)$,
\begin{equation}\label{7.14}
f^\psi_{\Delta(\tau,i+1),s}(g)=\int_{V_{ni,1^n}(F)\backslash V_{ni,1^n}(\BA)}f_{\Delta(\tau,i+1),s}(\hat{v}g)\psi_{V_{ni,1^n}}(v)dv,
\end{equation}
where $\psi_{V_{ni,1^n}}$ is given by \eqref{6.9}.
\begin{thm}\label{thm 7.1}
	For $Re(s)$ sufficiently large, $p(h,\bar{\mu})\in \Sp^{(2)}_{2ni}(\BA)$, 
	\begin{equation}\label{7.6.1} 
\mathcal{D}^\phi_{\psi,ni}(E(f_{\Delta(\tau,i+1),s}))((h,\bar{\mu}))=\sum_{\gamma\in Q_{ni}(F)\backslash \Sp_{2ni}(F)}\Lambda(f_{\Delta(\tau,i+1),s},\phi)((\gamma,1)(h,\bar{\mu})),
\end{equation}
where
$$
\Lambda(f_{\Delta(\tau,i+1),s},\phi)((h,\bar{\mu}))=\int_{U'_n(\BA)}\omega_{\psi^{-1}}(\beta(u)(h,\bar{\mu}))\phi(0)f^\psi_{\Delta(\tau,i+1),s}(\alpha_0 ut(h))du.
$$
In the sum \eqref{7.6.1}, $Q_{ni}=Q_{ni}^{\Sp_{2ni}}$. The function $\Lambda(f_{\Delta(\tau,i+1),s},\phi)$, defined for $Re(s)$ sufficiently large, by the last integral, admits analytic continuation
to a  meromorphic function of $s$ in the whole plane. It defines a
smooth meromorphic section of
$$
\rho_{\Delta(\tau,i)\gamma_{\psi^{-1}},s}=\Ind_{Q^{(2)}_{ni}(\BA)}^{\Sp^{(2)}_{2ni}(\BA)}\Delta(\tau,i)\gamma_{\psi^{-1}}
|\det\cdot|^s .
$$
Thus, $\mathcal{D}^\phi_{\psi,ni}(E(f_{\Delta(\tau,i+1),s}))$ is the Eisenstein series on
$\Sp^{(2)}_{2ni}(\BA)$, corresponding to the section $\Lambda(f_{\Delta(\tau,i+1),s},\phi)$ of
$\rho_{\Delta(\tau,i)\gamma_{\psi^{-1}},s}$.
\end{thm}
\begin{proof}
The proof follows the same steps as in the last sections. The details can easily be read from Chapter 6 in \cite{GRS11}. We start with unfolding the Eisenstein series in \eqref{7.5}, and factor the summation on $Q_{n(i+1)}(F)\backslash\Sp_{2n(i+1)}(F)$ modulo $Q_n(F)$ from the right. The easy analog of Prop. 6.2 in \cite{GRS11} shows that only the open double coset in $Q_{n(i+1)}(F)\backslash\Sp_{2n(i+1)}(F)/Q_n(F)$ contributes to \eqref{7.5}. The element $\alpha_0$ in \eqref{7.6} is a representative of the open double coset. (See Sec. 4.2 in \cite{GRS11} for the representatives of all double cosets.) The subgroup $Q^{(0)}=Q_n\cap \alpha_0^{-1}Q_{n(i+1)}\alpha_0$ consists of the elements
\begin{equation}\label{7.7}
\begin{pmatrix}a&0&y&0\\&b&c&y'\\&&b^*&0\\&&&a^*\end{pmatrix}\in H,
\end{equation}
where $a\in \GL_n$, $b\in \GL_{ni}$. The element \eqref{7.7} is conjugated by $\alpha_0$ to
$$
\begin{pmatrix}b&y'&0&c\\&a^*&0&0\\&&a&y\\&&&b^*\end{pmatrix}.
$$
As in \eqref{6.17}, we get (for $Re(s)$ sufficiently large)\\
\\
$\mathcal{D}^\phi_{\psi,ni}(E(f_{\Delta(\tau,i+1),s}))((h,\bar{\mu}))$
\begin{equation}\label{7.8}
=\int_{U_{1^n}(F)\backslash U_{1^n}(\BA)}\theta^\phi_{\psi^{-1}}(\beta(u)(h,\bar{\mu}))\sum_\eta f_{\Delta(\tau,i+1),s}(\alpha_0\eta ut(h))\psi_{n-1}^{-1}(u)du.
\end{equation}
Here, the summation in $\eta$ is over $(Q^{(0)}(F)\cap Q_{1^n}(F))\backslash Q_{1^n}(F)$. Let $U_{1^n}^{(0)}=U_{1^n}\cap Q^{(0)}$. The elements of $U_{1^n}^{(0)}$ have the form 
\eqref{7.7}, with $a=z\in Z_n$, $b=I_{ni}$, $c=0$. Note, that for such an element $v$ in $U_{1^n}^{(0)}$, we have 
\begin{equation}\label{7.9}
\beta(v)=((0,y_n),0),\ \alpha_0v\alpha_0^{-1}=\begin{pmatrix}I_{ni}&y'\\&z^*\end{pmatrix}^\wedge. 
\end{equation}
Factoring the summation in $\eta$ in \eqref{7.8} modulo $U_{1^n}(F)$ from the right, we get
\begin{multline}\label{7.10}
\mathcal{D}^\phi_{\psi,ni}(E(f_{\Delta(\tau,i+1),s}))((h,\bar{\mu}))
=\sum_{\gamma\in Q_{ni}(F)\backslash \Sp_{2ni}(F)}\\
\int_{U^{(0)}_{1^n}(F)\backslash U_{1^n}(\BA)}\theta^\phi_{\psi^{-1}}(\beta(u)(\gamma,1)(h,\bar{\mu}))f_{\Delta(\tau,i+1),s}(\alpha_0 ut(\gamma h))\psi_{n-1}^{-1}(u)du.
\end{multline}
Here, $Q_{ni}=Q_{ni}^{\Sp_{2ni}}$. Note that $U_{1^n}=U'_n\rtimes U_{1^n}^{(0)}$, and, for $u=u'_{x;y}\in U'_n$, written in the form \eqref{7.11}, $\beta(u)=((x_n,0),y_{n,1})$. Let us fix $\gamma$ in \eqref{7.10}. Then the corresponding summand becomes
\begin{equation}\label{7.12}
\int_{U'_n(\BA)}\int_{U^{(0)}_{1^n}(F)\backslash U^{(0)}_{1^n}(\BA)}\theta^\phi_{\psi^{-1}}(\beta(vu)(\gamma,1)(h,\bar{\mu}))f_{\Delta(\tau,i+1),s}(\alpha_0 vut(\gamma h))\psi_{n-1}^{-1}(v)dvdu.
\end{equation}
Consider in \eqref{7.12} the inner integration of the $dv$-integration, when we integrate, in the notation of \eqref{7.9}, along $Z_n$, and along $y$ with zero last row. Using \eqref{7.9}, this inner integral is  
\begin{equation}\label{7.13}
\int_{V_{ni+1,1^{n-1}}(F)\backslash V_{ni+1,1^{n-1}}(\BA)}f_{\Delta(\tau,i+1),s}(\hat{v} g)\tilde{\psi}_{V_{ni+1,1^{n-1}}}(v)dv,
\end{equation}
where $g=\hat{v}_c \alpha_0ut(\gamma h)$, 
$$
v_c=\begin{pmatrix}I_{ni}&c&0\\&1&0\\&&I_{n-1}\end{pmatrix} ,\ \tilde{\psi}_{V_{ni+1,1^{n-1}}}(\begin{pmatrix}I_{ni+1}&e\\&z\end{pmatrix})=\psi_{Z_{n-1}}(z)\psi(e_{ni+1,1}).
$$
By Prop. \ref{prop 3.1}, the integral \eqref{7.13}, as a function of $g\in H(\BA)$, is invariant under left multiplication by $\hat{v}_c$, for any $c\in \BA^{ni}$. Then \eqref{7.12} becomes
\begin{equation}\label{7.15}
\int_{U'_n(\BA)}\int_{F^{ni}\backslash \BA^{ni}}\theta^\phi_{\psi^{-1}}(((0,x),0)\beta(u)(\gamma,1)(h,\bar{\mu}))f^\psi_{\Delta(\tau,i+1),s}(\alpha_0 ut(\gamma h))dxdu.
\end{equation}
We have, for $\phi'\in\mathcal{S}(\BA^{ni})$,
$$
\theta^{\phi'}_{\psi^{-1}}(((0,x),0))=\sum_{e\in F^{ni}}\omega_{\psi^{-1}}(((0,x),0))\phi'(e)=\sum_{e\in F^{ni}}\psi^{-1}(2ew_{ni}{}^tx)\phi'(e).
$$
Carrying out the $dx$-integration in \eqref{7.15}, we get
$$\int_{U'_n(\BA)}\omega_{\psi^{-1}}(\beta(u)(\gamma,1)(h,\bar{\mu}))\phi(0)f^\psi_{\Delta(\tau,i+1),s}(\alpha_0 ut(\gamma h))du
$$
$$
=\Lambda(f_{\Delta(\tau,i+1),s},\phi)((\gamma,1)(h,\bar{\mu})).
$$
This proves \eqref{7.6.1}. The rest of the theorem is proved similarly to Theorem \ref{thm 6.2}. and we skip it. (The main steps are needed for the proof of the next proposition and are outlined there.)
\end{proof}

Now, we show compatibility with normalization. As in the last section (and the ones before), we need to consider the local integrals, analogous to \eqref{6.23.3},

\begin{equation}\label{7.16}
\Lambda_v(f_{\Delta(\tau_v,i;1),s},\phi_v)(1)=\int_{U'_n(F_v)}\phi_v(x_n)f_{\Delta(\tau_v,i;1),s}(\alpha_0u;I_{ni},I_n)\psi^{-1}_v(y_{n,1})du.
\end{equation}
Let us explain the notations. As in \eqref{6.23.3}, $f_{\Delta(\tau_v,i;1),s}$ is a section
of  
$$
\rho_{\Delta(\tau_v,i;1),s}=\Ind_{Q_{ni,n}(F_v)}^{\Sp_{2n(i+1)}(F_v)}(\Delta(\tau_v,i)|\det\cdot|^{s-\frac{1}{2}}\times
\tau_v|\det\cdot|^{s+\frac{i}{2}}).
$$
Recall that $Q_{ni,n}$ denotes the standard parabolic subgroup of $H=\Sp_{2n(i+1)}$, whose Levi part is isomorphic to $\GL_{ni}\times \GL_n$. The representation  $\Delta(\tau_v,i)$ is realized in its model with respect to
$\psi_{V_{i^n}}^{-1}$ and $\tau_v$ is 
realized in its $\psi^{-1}_{Z_n}$-Whittaker model. As before, we simplify notation and re-denote $f_{\Delta(\tau_v,i;1),s}(y)=f_{\Delta(\tau_v,i;1),s}(y;I_{ni},I_n)$. Finally, we wrote in \eqref{7.16} an element $u=u'_{x;y}\in U'_n(F_v)$ in the form \eqref{7.7}, so that $x_n$ is the last row of $x$ and $y_{n,1}$ is the $(n,1)$-coordinate of $y$.\\ 
Let $S$ be a finite set of places of $F$, containing the infinite ones, such that for a place $v$ outside $S$, $\tau_v$ is unramified, $\psi_v$ is normalized, $\phi_v=\phi^0_v$ is the characteristic function of $\mathcal{O}_v^{ni}$, and
$f_{\Delta(\tau_v,i;1),s}=f^0_{\Delta(\tau_v,i;1),s}$ is spherical and normalized.
Write the elements of $U'_n(F_v)$ in the form $u_{x_1,x_2;y}$ as
in \eqref{6.7}. Let
$$
U''_n(F_v)=\{u'_{0;y}=\begin{pmatrix}I_n&0&y\\&I_{2ni}&0\\&&I_n\end{pmatrix}\in H(F_v)\}.
$$
We have the following analog of Prop. \ref{prop 4.1}, and the proof of Theorem \ref{thm 6.2}.
\begin{prop}\label{prop 7.2}
	Let $v$ be a place of $F$. There is a section $f'_{\Delta(\tau_v,i;1),s}$, which
	depends on (the smoothness of) $f_{\Delta(\tau_v,i;1),s}$ and $\phi_v$,	such that 
	\begin{equation}\label{7.17}
	\Lambda_v(f_{\Delta(\tau_v,i;1),s},\phi_v)(1)=\int_{U''_n(F_v)}
	f'_{\Delta(\tau_v,i;1),s}(\alpha_0u)\psi^{-1}_v(u_{n,n+2ni+1})du.
	\end{equation}
	If $f_{\Delta(\tau_v,i;1),s}$ and $\phi_v$ are spherical, then
	$f'_{\Delta(\tau_v,i;1),s}=f_{\Delta(\tau_v,i;1),s}$. For $v\in S$,\\ $f'_{\Delta(\tau_v,i;1),s}$ is obtained
	from $f_{\Delta(\tau_v,i;1),s}$ by a finite sequence of
	convolutions against certain Schwartz-Bruhat functions. 
\end{prop}
\begin{proof}
	The proof is similar to that of Prop. \ref{prop 4.1}, Theorem \ref{thm 6.2}. We explain the case where $v$ is finite and indicate the main steps. We assume, for simplicity that $\psi_v$ is normalized. Write $u'_{x;y}=v(x)u_n(y)$, where $v(x)=u'_{x;0}$, $u_n(y)=u'_{0;y}$. Note that
	\begin{equation}\label{7.18}
	\alpha_0u_n(y)\alpha_0^{-1}=diag(I_{ni},\begin{pmatrix}I_n\\y&I_n\end{pmatrix},I_{ni}).
	\end{equation} 
	Let
	\begin{equation}\label{7.19}
	\Lambda'_v(f_{\Delta(\tau_v,i;1),s},\phi_v)(x)=\int_{U''_n(F_v)}\phi_v(x_n)f_{\Delta(\tau_v,i;1),s}(\alpha_0v(x)u)\psi^{-1}_v(u_{n,n+2ni+1})du.
	\end{equation} 
	We will show that the last function has compact support in $M_{n\times ni}(F_v)$, and in case of a spherical section this support is $M_{n\times ni}(\mathcal{O}_v)$. First, the support of $\Lambda'_v(f_{\Delta(\tau_v,i;1),s},\phi_v)(x)$ in $x_n$ (the last row of $x$) is compact, since $\phi_v$ is in $\mathcal{S}(F_v^{ni})$. Thus, \eqref{7.19} is a finite linear combination of integrals of the form (we keep denoting the section by $f_{\Delta(\tau_v,i;1),s}$)
		\begin{equation}\label{7.20}
	l_v(f_{\Delta(\tau_v,i;1),s})(x(n))=\int_{U''_n(F_v)}f_{\Delta(\tau_v,i;1),s}(\alpha_0v(x(n))u)\psi^{-1}_v(u_{n,n+2ni+1})du,
	\end{equation}  
	where $x(n)$ is obtained from $x$ by replacing $x_n$ with $0$.
	Assume that we have proved that the support of $l_v(f_{\Delta(\tau_v,i;1),s},\phi_v)(x(n))$ in $x_r$ (the $r$-th row of $x$), for $n-\ell\leq r\leq n$, $0\leq \ell\leq n-2$ is compact. Thus, we may assume that $x_r=0$, for $n-\ell\leq r\leq n$. Let
	$$
	c_{\ell;z}=\begin{pmatrix}I_{n-\ell-1}\\&1\\&0&I_\ell\\&z&0&I_{ni}\end{pmatrix}^\wedge\in \Sp_{2n(i+1)}(F_v),
	$$
	with $z\in (\mathcal{P}^k_v)^{ni}$, where $k$ is sufficiently large, such that right translation by $c_{\ell,z}$ fixes $f_{\Delta(\tau_v,i;1),s}$. As in the proof of Theorem \ref{thm 6.2}, we get, for $x$ as above,
	$$
	l_v(f_{\Delta(\tau_v,i;1),s})(x)=l_v(\rho(c_{\ell,z})f_{\Delta(\tau_v,i;1),s})(x)=  \psi_v(x_{n-\ell-1}z)l_v(f_{\Delta(\tau_v,i;1),s})(x).
	$$	
	We conclude that $x_{n-\ell-1}$ is supported in  $(\mathcal{P}_v^{-k})^{ni}$, and in case the section is spherical, it is supported in $\mathcal{O}_v^{ni}$.	
\end{proof}
Now we get the analog of Theorem \ref{thm 6.4}.
\begin{thm}\label{thm 7.3}
	Let $v$ be a place outside $S$. Then
	$$
	\Lambda_v(f^0_{\Delta(\tau_v,i;1),s},\phi^0_v)(1)=\frac{1}{L(\tau_v,s+\frac{i}{2}+1)L(\tau_v,\wedge^2,2s+i+1)}.
	$$
\end{thm}
\begin{proof}
	Denote, for $h\in \Sp_{2n}(F_v)$,
	$$
	f^0_{\tau_v,s+\frac{i}{2}}(h)=f^0_{\Delta(\tau_v,i;1),s}(diag(I_{ni}, h,I_{ni})).
	$$
	This is the spherical, normalized element of
	$$
	\rho_{\tau_v,s+\frac{i}{2}}=\Ind_{Q_n^{\Sp_{2n}}(F_v)}^{\Sp_{2n}(F_v)}\tau_v|\det\cdot|^{s+\frac{i}{2}}.
	$$
	Prop. \ref{prop 7.2} and \eqref{7.18} show that, for $Re(s)$ sufficiently large,
	\begin{equation}\label{7.21}
	\Lambda_v(f^0_{\Delta(\tau_v,i;1),s},\phi^0_v)(1)=\int f^0_{\tau_v,s+\frac{i}{2}}(\begin{pmatrix}I_n\\y&I_n\end{pmatrix})\psi^{-1}(y_{n,1})dy,
	\end{equation}
	where the integration is over the $n\times n$ matrices $y$ over $F_v$, such that $w_ny$ is symmetric. The integral \eqref{7.21} is the Jacquet integral applied to $f^0_{\tau_v,s+\frac{i}{2}}$, with respect to the standard $\psi_v$-Whittaker character. Now, we get the theorem by the Casselman-Shalika formula.	
\end{proof}

Let us multiply our Eisenstein series $E(f_{\Delta(\tau,i+1),s})$ on
$\Sp_{2n(i+1)}(\BA)$ by its normalizing factor outside $S$, $d_\tau^{\Sp_{2n(i+1)},S}(s)$. See \eqref{1.10.4}, \eqref{1.10.5}. By \eqref{1.10.6}, \eqref{1.10.7}, we check that
\begin{equation}\label{7.22}
\frac{d_\tau^{\Sp_{2n(i+1)},S}(s)}{L^S(\tau,s+\frac{i}{2}+1)L^S(\tau,\wedge^2,2s+i+1)}=d_\tau^{\Sp^{(2)}_{2ni},S}(s).
\end{equation}
Let
$$
E^*_S(f_{\Delta(\tau,i+1),s})=d_\tau^{\Sp_{2n(i+1)},S}(s)E(f_{\Delta(\tau,i+1),s}).
$$
As in Theorem \ref{thm 6.5}, we get
\begin{thm}\label{thm 7.4}
	The descent to $\Sp^{(2)}_{2ni}(\BA)$ of the normalized Eisenstein series\\ $E^*_S(f_{\Delta(\tau,i+1),s})$, $\mathcal{D}^\phi_{\psi,ni}(E^*_S(f_{\Delta(\tau,i+1),s}))$, is the normalized (outside $S$) Eisenstein series on
	$\Sp^{(2)}_{2ni}(\BA)$ corresponding to the section $\Lambda(d_\tau^{\Sp_{2n(i+1)},S}(s)f_{\Delta(\tau,i+1),s},\phi)$.
\end{thm}

\section{Application of Fourier-Jacobi coefficients to $E(f_{\Delta(\tau,i+1)\gamma_\psi,s})$ on $\Sp^{(2)}_{2n(1+i)}(\BA)$: descent to $\Sp_{2ni}(\BA)$}

We keep the notations of the previous section. We consider the parabolic induction
\begin{equation}\label{8.1}
\rho_{\Delta(\tau,i+1)\gamma_\psi,s}=\Ind_{Q^{(2)}_{n(i+1)}(\BA)}^{\Sp^{(2)}_{2n(i+1)}(\BA)}\Delta(\tau,i+1)\gamma_\psi|\det\cdot|^s.
\end{equation}
Let $f_{\Delta(\tau,i+1)\gamma_\psi,s}$ be a smooth, holomorphic section of $\rho_{\Delta(\tau,i+1)\gamma_\psi,s}$ in \eqref{8.1}. Consider the corresponding Eisenstein series $E(f_{\Delta(\tau,i+1)\gamma_\psi,s})$, and apply to it the Fourier-Jacobi coefficient similar to \eqref{7.5}\\
\\
$\mathcal{D}^\phi_{\psi,ni}(E(f_{\Delta(\tau,i+1)\gamma_\psi,s}))(h)$
\begin{equation}\label{8.2}
=\int_{U_{1^n}(F)\backslash U_{1^n}(\BA)}E(f_{\Delta(\tau,i+1)\gamma_\psi,s},u\tilde{t}(h))\psi_{n-1}^{-1}(u)\theta^\phi_{\psi^{-1}}(\beta(u)\tilde{h})du.
\end{equation}
Here, $h\in \Sp_{2ni}(\BA)$, and $\tilde{h}$ is any element of $\Sp^{(2)}_{2ni}(\BA)$, which projects to $h$. Let $\tilde{h}=C'\Pi'_v(h_v,\mu_v)$. Then $\tilde{t}(h)=C'\Pi'_v(t(h_v),\mu_v)\in \Sp_{2n(i+1)}^{(2)}(\BA)$. Recall agai that we identify $u\in U_{1^n}(\BA)$ with $(u,1)$ (see right after \eqref{1.1.d}). Define the function on $\Sp_{2n(i+1)}^{(2)}(\BA)$, $f^\psi_{\Delta(\tau,i+1)\gamma_\psi,s}$, in a similar way to \eqref{7.14}. The proof of the following theorem is the same as that of Theorem \ref{thm 7.1}, with obvious modifications. We just remark that in the proof that the section $\Lambda(f_{\Delta(\tau,i+1)\gamma_\psi,s},\phi)$ below corresponds to the parabolic data $(Q_{ni},\Delta(\tau,i)
|\det\cdot|^s )$, we use the fact that
$$
\gamma_\psi(t)\gamma_{\psi^{-1}}(t)=\gamma_\psi(t)^2(t,-1)=(t,-t)=1,
 $$ 
where $(t,t')$ denotes the Hilbert symbol.
\begin{thm}\label{thm 8.1}
	For $Re(s)$ sufficiently large, $h\in \Sp_{2ni}(\BA)$, 
	\begin{equation}\label{8.3} 
	\mathcal{D}^\phi_{\psi,ni}(E(f_{\Delta(\tau,i+1)\gamma_\psi,s}))(h)=\sum_{\gamma\in Q_{ni}\backslash \Sp_{2ni}}\Lambda(f_{\Delta(\tau,i+1)\gamma_\psi,s},\phi)(\gamma h),
	\end{equation}
	where
	$$
	\Lambda(f_{\Delta(\tau,i+1)\gamma_\psi,s},\phi)(h)=\int_{U'_n(\BA)}\omega_{\psi^{-1}}(\beta(u)\tilde{h})\phi(0)f^\psi_{\Delta(\tau,i+1)\gamma_\psi,s}(\alpha_0 u\tilde{t}(h))du.
	$$
	The function $\Lambda(f_{\Delta(\tau,i+1)\gamma_\psi,s},\phi)$, defined, for $Re(s)$ sufficiently large, by the last integral, admits analytic continuation
	to a  meromorphic function of $s$ in the whole plane. It defines a
	smooth meromorphic section of
	$$
	\rho_{\Delta(\tau,i),s}=\Ind_{Q_{ni}(\BA)}^{\Sp_{2ni}(\BA)}\Delta(\tau,i)
	|\det\cdot|^s .
	$$
	Thus, $\mathcal{D}^\phi_{\psi,ni}(E(f_{\Delta(\tau,i+1)\gamma_\psi,s}))$ is the Eisenstein series on
	$\Sp_{2ni}(\BA)$ corresponding to the section $\Lambda(f_{\Delta(\tau,i+1)\gamma_\psi,s},\phi)$ of
	$\rho_{\Delta(\tau,i),s}$.
\end{thm}

At a place $v$ of $F$, consider, as in \eqref{7.16}, with the same notation,
\begin{equation}\label{8.4}
\Lambda_v(f_{\Delta(\tau_v,i;1)\gamma_{\psi_v},s},\phi_v)(1)=\int_{U'_n(F_v)}\phi_v(x_n)f_{\Delta(\tau_v,i;1)\gamma_{\psi_v},s}(\alpha_0u;I_{ni},I_n)\psi^{-1}_v(y_{n,1})du,
\end{equation}
where $f_{\Delta(\tau_v,i;1)\gamma_{\psi_v},s}$ is a section
of  
$$
\rho_{\Delta(\tau_v,i;1)\gamma_{\psi_v},s}=\Ind_{Q^{(2)}_{ni,n}(F_v)}^{\Sp^{(2)}_{2n(i+1)}(F_v)}\Delta(\tau_v,i)\gamma_{\psi_v}|\det\cdot|^{s-\frac{1}{2}}\times
\tau_v\gamma_{\psi_v}|\det\cdot|^{s+\frac{i}{2}},
$$
See also \eqref{4.2.1}. Now, the same proof as that of Prop. \ref{prop 7.2} shows that there is a section $f'_{\Delta(\tau_v,i;1)\gamma_{\psi_v},s}$, which
depends on the smoothness of $f_{\Delta(\tau_v,i;1)\gamma_{\psi_v},s}$ and $\phi_v$,	such that 
\begin{equation}\label{8.5}
\Lambda_v(f_{\Delta(\tau_v,i;1)\gamma_{\psi_v},s},\phi_v)(1)=\int_{U''_n(F_v)}
f'_{\Delta(\tau_v,i;1)\gamma_{\psi_v},s}(\alpha_0u;I_{ni},I_n)\psi^{-1}_v(u_{n,n+2ni+1})du;
\end{equation}
and if $f_{\Delta(\tau_v,i;1)\gamma_{\psi_v},s}$ and $\phi_v$ are spherical, then
$f'_{\Delta(\tau_v,i;1)\gamma_{\psi_v},s}=f_{\Delta(\tau_v,i;1)\gamma_{\psi_v},s}$. Thus, for such a finite place $v$, where all data are spherical and normalized, define, for $(h,\epsilon)\in \Sp^{(2)}_{2n}(F_v)$,
$$
f^0_{\tau_v\gamma_{\psi_v},s+\frac{i}{2}}((h,\epsilon))=f^0_{\Delta(\tau_v,i;1)\gamma_{\psi_v},s}((diag(I_{ni}, h,I_{ni}),\epsilon)).
$$
This is the spherical, normalized element of
$$
\rho_{\tau_v\gamma_{\psi_v},s+\frac{i}{2}}=\Ind_{Q^{(2)}_n(F_v)}^{\Sp^{(2)}_{2n}(F_v)}\tau_v\gamma_{\psi_v}|\det\cdot|^{s+\frac{i}{2}}.
$$
By \eqref{8.5}, we have, for $Re(s)$ sufficiently large, 
\begin{equation}\label{8.6}
\Lambda_v(f^0_{\Delta(\tau_v,i;1)\gamma_{\psi_v},s},\phi^0_v)(1)=\int f^0_{\tau_v\gamma_{\psi_v},s+\frac{i}{2}}((\begin{pmatrix}I_n\\y&I_n\end{pmatrix},1))\psi^{-1}(y_{n,1})dy,
\end{equation}
where the domain of integration is as in \eqref{7.21}. This is the Jacquet integral applied to $f^0_{\tau_v\gamma_{\psi_v},s+\frac{i}{2}}$, with respect to the standard $\psi_v$-Whittaker character. The analog of the Casselman-Shalika formula (\cite{BFH91}, Theorem 1.2) in this case gives then the following.
\begin{thm}\label{thm 7.2}
	$$
	\Lambda_v(f^0_{\Delta(\tau_v,i;1)\gamma_{\psi_v},s},\phi^0_v)(1)=\frac{L(\tau_v,s+\frac{i+1}{2})}{L(\tau_v,sym^2,2s+i+1)}.
	$$
\end{thm}

Finally, let us multiply our Eisenstein series $E(f_{\Delta(\tau,i+1)\gamma_\psi,s})$ on
$\Sp^{(2)}_{2n(i+1)}(\BA)$ by its normalizing factor outside $S$, $d_\tau^{\Sp^{(2)}_{2n(i+1)},S}(s)$. See \eqref{1.10.6}, \eqref{1.10.7}. By \eqref{1.10.4}, \eqref{1.10.5}, we check that
\begin{equation}\label{8.7}
\frac{d_\tau^{\Sp^{(2)}_{2n(i+1)},S}(s)L^S(\tau,s+\frac{i+1}{2})}{L^S(\tau,sym^2,2s+i+1)}=d_\tau^{\Sp_{2ni},S}(s).
\end{equation}
Let
$$
E^*_S(f_{\Delta(\tau,i+1)\gamma_\psi,s})=d_\tau^{\Sp^{(2)}_{2n(i+1)},S}(s)E(f_{\Delta(\tau,i+1)\gamma_\psi,s}).
$$
As in Theorem \ref{thm 7.4}, we get
\begin{thm}\label{thm 8.3}
	The descent to $\Sp_{2ni}(\BA)$ of the normalized Eisenstein series\\ $E^*_S(f_{\Delta(\tau,i+1)\gamma_\psi,s})$, $\mathcal{D}^\phi_{\psi,ni}(E^*_S(f_{\Delta(\tau,i+1)\gamma_\psi,s}))$, is the normalized (outside $S$) Eisenstein series on
	$\Sp_{2ni}(\BA)$ corresponding to the section $\Lambda(d_\tau^{\Sp^{(2)}_{2n(i+1)},S}(s)f_{\Delta(\tau,i+1)\gamma_\psi,s},\phi)$.
\end{thm}

\end{document}